\pdfoutput=1
\documentclass[hidelinks, reqno]{amsart}
\usepackage[english,french]{babel}
\usepackage{graphicx}
\usepackage{subcaption}
\usepackage{tabularx}
\usepackage{booktabs}
\usepackage{float}
\usepackage{array}
\usepackage{amsmath}
\usepackage{amsfonts}
\usepackage{amssymb}
\usepackage{amsthm}
\usepackage[scr]{rsfso}
\usepackage{enumitem}
\usepackage{permute}
\usepackage{slashed}
\usepackage[usenames,dvipsnames]{xcolor}
\usepackage[pagebackref = true, colorlinks, linkcolor = Red, citecolor = Green, bookmarksdepth=2, linktocpage=true]{hyperref}
\usepackage[capitalize]{cleveref}
\usepackage{comment}
\usepackage{changepage}
\usepackage{dsfont}

\usepackage[T1]{fontenc}
\usepackage{makecell}

\usepackage{tikz}
\usetikzlibrary{calc}

\allowdisplaybreaks

\DeclareMathOperator{\interior}{int}

\DeclareMathOperator{\Hom}{Hom}

\DeclareMathOperator{\Id}{Id}
\DeclareMathOperator{\tr}{tr}
\DeclareMathOperator{\Stab}{Stab}

\DeclareMathOperator{\vol}{vol}
\DeclareMathOperator{\supp}{supp}

\DeclareMathOperator{\SL}{SL}

\DeclareMathOperator{\SO}{SO}

\DeclareMathOperator{\Lip}{Lip}

\DeclareMathOperator{\diam}{diam}
\DeclareMathOperator{\Leb}{Leb}

\DeclareMathOperator{\ad}{ad}
\DeclareMathOperator{\Ad}{Ad}

\DeclareMathOperator{\inj}{inj}

\DeclareMathOperator{\Span}{span}

\DeclareMathOperator{\Fix}{Fix}

\DeclareMathOperator{\codim}{codim }
\DeclareMathOperator{\Gr}{Gr}

\newcommand{\N}{\mathbb{N}}
\newcommand{\Z}{\mathbb{Z}}
\newcommand{\Q}{\mathbb{Q}}
\newcommand{\R}{\mathbb{R}}
\newcommand{\C}{\mathbb{C}}

\newcommand{\LieG}{\mathfrak{g}}
\newcommand{\LieA}{\mathfrak{a}}

\newcommand{\LieU}{\mathfrak{u}}
\newcommand{\LieK}{\mathfrak{k}}
\newcommand{\LieM}{\mathfrak{m}}
\newcommand{\LieP}{\mathfrak{p}}
\newcommand{\LieQ}{\mathfrak{q}}

\newcommand{\LieH}{\mathfrak{h}}

\newcommand{\height}{\mathrm{ht}}
\newcommand{\disc}{\mathrm{disc}}

\newcounter{consta}
\renewcommand{\theconsta}{{A_{\arabic{consta}}}}
\newcommand{\consta}{\refstepcounter{consta}{\color{red}\theconsta}}

\newcounter{constm}
\renewcommand{\theconstm}{{M_{\arabic{constm}}}}
\newcommand{\constm}{\refstepcounter{constm}{\color{red}\theconstm}}

\newcounter{constc}
\renewcommand{\theconstc}{{C_{\arabic{constc}}}}
\newcommand{\constc}{\refstepcounter{constc}{\color{red}\theconstc}}

\newcounter{constE}
\renewcommand{\theconstE}{{{E}_{\arabic{constE}}}}
\newcommand{\constE}{\refstepcounter{constE}{\color{red}\theconstE}}

\newcounter{constd}
\renewcommand{\theconstd}{{D_{\arabic{constd}}}}
\newcommand{\constd}{\refstepcounter{constd}{\color{red}\theconstd}}

\makeatletter
\newcommand{\mytag}[2]{%
\text{#1}%
\@bsphack
\begingroup
\@onelevel@sanitize\@currentlabelname
\edef\@currentlabelname{%
\expandafter\strip@period\@currentlabelname\relax.\relax\@@@%
}%
\protected@write\@auxout{}{%
\string\newlabel{#2}{%
{#1}%
{\thepage}%
{\@currentlabelname}%
{\@currentHref}{}%
}%
}%
\endgroup
\@esphack
}
\makeatother

\DeclareFontFamily{U}{mathb}{\hyphenchar\font45}
\DeclareFontShape{U}{mathb}{m}{n}{
<5> <6> <7> <8> <9> <10> gen * mathb
<10.95> mathb10 <12> <14.4> <17.28> <20.74> <24.88> mathb12
}{}
\DeclareSymbolFont{mathb}{U}{mathb}{m}{n}
\DeclareMathSymbol{\bigast}{2}{mathb}{"06}
\DeclareMathAlphabet{\mathpzc}{OT1}{pzc}{m}{it}

\def\XXint#1#2#3{{\setbox0=\hbox{$#1{#2#3}{\int}$}
\vcenter{\hbox{$#2#3$}}\kern-.5\wd0}}

\theoremstyle{plain}
\newtheorem{theorem}{Theorem}[section]
\newtheorem{proposition}[theorem]{Proposition}
\newtheorem{lemma}[theorem]{Lemma}
\newtheorem*{theorem*}{Theorem}
\newtheorem*{proposition*}{Proposition}
\newtheorem*{lemma*}{Lemma}
\newtheorem{corollary}[theorem]{Corollary}

\newtheorem*{claim*}{Claim}
\theoremstyle{definition}
\newtheorem{definition}[theorem]{Definition}

\newtheorem{example}{Example}[section]

\theoremstyle{remark}
\newtheorem{remark}[theorem]{Remark}

\Crefname{enumi}{Property}{Properties}
\Crefname{alternativei}{Alternative}{Alternatives}
\Crefname{subsection}{Subsection}{Subsections}

\begin{document}
\selectlanguage{english}


\title[Quadratic forms of signature $(2, 2)$ or $(3, 1)$ I]{Quadratic forms of signature $(2, 2)$ or $(3, 1)$ I: effective equidistribution in quotients of $\mathrm{SL}_4(\mathbb{R})$}

\author{Zuo Lin}
\address{Department of Mathematics, University of California, San Diego, CA 92093}
\email{zul003@ucsd.edu}

\date{\today}

\begin{abstract}
We prove an effective equidistribution theorem for orbits of horospherical subgroups of $\mathrm{SO}(2, 2)$ and $\mathrm{SO}(3, 1)$ in quotients of $\mathrm{SL}_4(\mathbb{R})$ with a polynomial error term. In a forthcoming paper, we will use this theorem to prove an effective version of the Oppenheim conjecture for indefinite quadratic forms of signature $(2, 2)$ or $(3, 1)$ with a polynomial error rate.
\end{abstract}
\maketitle

\setcounter{tocdepth}{1}
\tableofcontents
\setcounter{part}{-1}

\part{Introduction}
\section{Introduction}
An important theme in homogeneous dynamics is the behavior of orbits of $\Ad$-unipotent subgroups from \emph{any} initial point. More precisely, let $G$ be a Lie group, $\Gamma < G$ be a lattice and $U \leq G$ be an $\Ad$-unipotent subgroup. Raghunathan conjectured that for \emph{any} initial point $x \in X = G/\Gamma$, the orbit closure $\overline{U.x}$ is a periodic orbit $L.x$ of some subgroup $U \leq L \leq G$. We say $L.x$ is periodic if $\Stab(x) \cap L$ is a lattice in $L$. In the literature the conjecture was first stated in the paper \cite{Dan81} and in a more general form in \cite{Mar90} where the subgroup $U$ is not necessarily $\Ad$-unipotent but generated by $\Ad$-unipotent elements. 

Raghunathan's conjecture was proved in full generality by Ratner in \cite{Rat90a,Rat90b,Rat91a,Rat91b}. In her landmark work, Ratner also classified all ergodic invariant probability measures under the action of $U$ and proved an equidistribution theorem for orbits of $U$. These remarkable theorems have been highly influential and have led to a lot of important applications. 

Prior to Ratner's proof, the conjecture was known in certain cases. We refer to the book by Morris \cite{Mor05} for a detailed historical background. We mention the following important special case related to Oppenheim conjecture on distribution of values of indefinite quadratic forms on integer points. In his seminal work \cite{Mar89}, Margulis proved Oppenheim conjecture by showing every $\SO(2, 1)$-orbit in $\SL_3(\R)/\SL_3(\Z)$ is either periodic or unbounded. Later, Dani and Margulis \cite{DM89,DM90} showed that any $\SO(2, 1)$-orbit is either periodic or dense. They also classified possible orbit closures of a one-parameter unipotent subgroup of $\SO(2, 1)$ in $\SL_3(\R)/\SL_3(\Z)$. 

Based on equidistribution results for unipotent subgroups, information on asymptotics of distribution of values of indefinite quadratic forms with signature $(p, q)$ on integer points when $p \geq 3$ or $(p, q) = (2, 2)$ is provided by Eskin, Margulis and Mozes in \cite{EMM98,EMM05}. Recently, Kim extended the ideas by Eskin, Margulis and Mozes to indefinite quadratic forms with signature $(2, 1)$ in \cite{Kim24}. 

Because of its intrinsic interest and in view of the applications, \emph{effective} results on distribution of the orbits of unipotent groups have been sought after for some time. We briefly review the progress on this problem related to applications on distribution of values of indefinite quadratic forms on integer points and refer to \cite{Moh23}(and also \cite[Section 1.4]{LMW22}) for a throughout survey on both historical background and recent progress. We refer to \cite{LM23,OS25} and references therein for recent progress related to hyperbolic geometry. 

An effective version of equidistribution theorem for a one-parameter unipotent subgroup in $\SL_3(\R)/\SL_3(\Z)$ with a poly-logarithmic rate was proved by Lindenstrauss and Margulis in \cite{LM14}. It lead to an effective proof of the Oppenheim conjecture with a poly-logarithmic rate. 

In the landmark works by Lindenstrauss and Mohammadi \cite{LM23} and later with Wang and Yang and by Yang \cite{LMW22,Yan24,LMWY25}, effective density and equidistribution theorems with polynomial rate for orbits of unipotent subgroups is established in quotients
of quasi-split, almost simple linear algebraic groups of absolute rank $2$. In \cite{LMWY25}, they established an effective Oppenheim conjecture with a polynomial rate when the dimension $d = 3$ building on their effective equidistribution theorem. 

Motivated by the above problems and results, we prove an effective equidistribution theorem with a polynomial rate (Theorem~\ref{thm:main equidistribution}) in quotients of $\SL_4(\R)$ and discuss effective results on Oppenheim conjecture for indefinite quadratic forms of signature $(2, 2)$ or $(3, 1)$ in this series of papers. This first paper is devoted to the proof of the effective equidistribution theorem (Theorem~\ref{thm:main equidistribution}). Before we state the main theorem, let us introduce the following notions. 

Let $G = \SL_4(\R)$ and $\LieG = \mathrm{Lie}(G)$.
Let $Q_1$ be the following quadratic form on $\R^4$,
\begin{align*}
    Q_1(x_1, x_2, x_3, x_4) = x_2 x_3 - x_1 x_4,
\end{align*}
and put $H_1 = \SO(Q_1)^\circ \subset \SL_4(\R)$. Note that $Q_1$ is of signature $(2, 2)$ and $H_1 \cong \SO(2, 2)^\circ$. Let $\LieH_1 = \mathrm{Lie}(H_1)$. 
Let $Q_2$ be the following quadratic form on $\R^4$,
\begin{align*}
    Q_2(x_1, x_2, x_3, x_4) = x_2^2 +  x_3^2 - 2x_1 x_4,
\end{align*}
and put $H_2 = \SO(Q_2)^\circ \subset \SL_4(\R)$. Note that $Q_2$ is of signature $(3, 1)$ and $H_2 \cong \SO(3, 1)^\circ$. Let $\LieH_2 = \mathrm{Lie}(H_2)$. If a definition/statement/proof can be formulated simultaneously to $H_1$ and $H_2$, we drop the subscripts and denote them by $Q$, $H$ and $\LieH$ for simplicity. 

Let $a_t$ be the one-parameter diagonal subgroup in both $H_1$ and $H_2$ defined by
\begin{align}\label{eqn:def of a}
    a_t = \begin{pmatrix}
        e^t & & & \\
         & 1 & & \\
         & & 1 & \\
         & & & e^{-t}
    \end{pmatrix}.
\end{align}
The corresponding horospherical subgroups $U_1 \leq H_1$ and $U_2 \leq H_2$ consists of the following elements respectively:
\begin{align}\label{eqn:def of U}
    u_{r, s}^{(1)} = \begin{pmatrix}
        1 & r & s & sr\\
         & 1 & & s\\
         & & 1 & r\\
         & & & 1
    \end{pmatrix}, \qquad u_{r, s}^{(2)} = \begin{pmatrix}
        1 & r & s & \frac{r^2 + s^2}{2}\\
         & 1 & & r\\
         & & 1 & s\\
         & & & 1
    \end{pmatrix}.
\end{align}
As before, if a definition/statement/proof can be formulated simultaneously to $U_1$ and $U_2$, we drop the subscripts for $U$ and superscripts for $u_{r, s}$ for simplicity. 

Let $\Gamma \subset G$ be a lattice. By Margulis' arithmeticity theorem, $\Gamma$ is arithemtic. Let $X = G/\Gamma$ and let $\mu_X$ be the probability Haar measure on $X$. Let $d$ be a right $G$-invariant left $\SO(4)$-invariant Riemmanian metric coming from the Killing form of $G$. It induces a Riemmanian metric $d_X$ on $X$ and natural volume forms on $X$ and its embedded submanifolds. Let $\Leb$ be the standard Lebesgue measure on $\R^2$. 

\begin{theorem}\label{thm:main equidistribution}
There exist constants $\consta\label{a:main equidistribution1} > \consta\label{a:main equidistribution2} \geq 1$ and $\kappa > 0$ depending only on $X$ so that the following holds. For all $x_0 \in X$ and large enough $R$ depending explicitly on $x_0$, for any $T \geq R^{\ref{a:main equidistribution1}}$, at least one of the following is true. 
\begin{enumerate}
\item For all $\phi \in \mathrm{C}_c^\infty(X)$, 
\begin{align*}
\Biggl|\int_{[0, 1]^2} \phi(a_{\log T} u_{r, s}.x_0) \,\mathrm{d}\Leb(r, s) - \int_X \phi \,\mathrm{d}\mu_X\Biggr| \leq \mathcal{S}(\phi) R^{-\kappa}
\end{align*}
where $\mathcal{S}(\phi)$ is a certain Sobolev norm. 
\item There exists $x \in X$ so that $H.x$ is periodic with $\vol(H.x) \leq R$ and 
\begin{align*}
    d_X(x_0, x) \leq T^{-\frac{1}{\ref{a:main equidistribution2}}}.
\end{align*}
\end{enumerate}
\end{theorem}

\begin{remark}
We remark that the dependence of $R$ on $x_0$ is of the form $R \gg \inj(x_0)^{-\star}$. See Section~\ref{sec:Global prelim} for the precise definition of $\inj(x_0)$ and the convention on $\star$-notations. The reader can trace the implied constants from \cref{eqn:FinalProofConditionR}.
\end{remark}

\begin{remark}
All unipotent elements in both $H_1$ and $H_2$ are not $\R$-regular (see \cite{And75}) in $\SL_4(\R)$. In other words, there is no principal $\SL_2(\R)$ of $G = \SL_4(\R)$ in either $H_1$ or $H_2$ (cf. \cite[Chapter VIII, \S 11, Exercise 4)]{Bou05}). 
\end{remark}

In a sequel paper, we will investigate the applications of Theorem~\ref{thm:main equidistribution} to distribution of value of indefinite quadratic forms on integer points and further counting results. In particular, we will obtain an effective version of Oppenheim conjecture for quadratic forms with $4$ variables. 
\begin{theorem}\label{thm:main quadratic form}
There exist absolute constants $A_1 > A_2 \geq 1$ and $\kappa > 0$ so that the following holds. Let $Q$ be an indefinite quadratic form of signature $(2, 2)$ or $(3, 1)$ with $\det Q = 1$. For all $R$ large enough depending on $\|Q\|$ and all $T \geq R^{A_1}$, at least one of the following is true. 
\begin{enumerate}
    \item For every $s \in [-R^{\kappa}, R^{\kappa}]$, there exists a primitive vector $v \in \Z^4$ with $0 < \|v\| \leq T$ so that 
    \begin{align*}
        |Q(v) - s| \leq R^{-\kappa}.
    \end{align*}
    \item There exists an integral quadratic form $Q'$ with $\|Q'\| \leq R$ so that 
    \begin{align*}
        \|Q - \lambda Q'\| \leq T^{-\frac{1}{A_2}} \text{ where } \lambda = \det(Q')^{-\frac{1}{4}}.
    \end{align*}
\end{enumerate}
\end{theorem}

Combining Theorem~\ref{thm:main quadratic form} with the work by Lindenstrauss, Mohammadi, Wang and Yang for quadratic forms of signature $(2, 1)$ \cite[Theorem 2.5]{LMWY25} and the work by Buterus, G\"{o}tze, Hille and Margulis for quadratic forms with at least $5$ variables \cite[Corollary 1.4]{BGHM22}, we conclude the following theorem. It establish an effective Oppenheim conjecture with a polynomial rate regarding the Diophantine inequality $|Q(x)| < \epsilon$ in all dimension $d \geq 3$. 

\begin{theorem}
For all integer $d \geq 3$, there exist constants $A_1 > A_2 \geq 1$ and $\kappa > 0$ depending only on $d$ so that the following holds. Let $Q$ be a non-degenerate indefinite quadratic form with $d$ variables and $\det Q = 1$. For all $R$ large enough depending on $\|Q\|$ and all $T \geq R^{A_1}$, at least one of the following is true. 
\begin{enumerate}
    \item There exists a primitive vector $v \in \Z^d$ with $0 < \|v\| \leq T$ so that 
    \begin{align*}
        |Q(v)| \leq R^{-\kappa}.
    \end{align*}
    \item There exists an integral quadratic form $Q'$ with $\|Q'\| \leq R$ so that 
    \begin{align*}
        \|Q - \lambda Q'\| \leq T^{-\frac{1}{A_2}} \text{ where } \lambda = \det(Q')^{-\frac{1}{d}}.
    \end{align*}
\end{enumerate}
Moreover, if the dimension $d \geq 5$, case~(1) is always true. 
\end{theorem}

We now discuss the proof of Theorem~\ref{thm:main equidistribution}. For the convenience to later discussions related to \cite{LMWY25}, we extend the content of notations $G$ and $H$ as the following. This extension is only for the rest of the introduction. 

Let $\mathbf{G}$ be a connected semisimple real linear algebraic group. Let $G = \mathbf{G}(\R)^\circ$ be the connected component of the identity under the Hausdorff topology. It is a connected semisimple Lie group. Let $\LieG = \mathrm{Lie}(G)$ be the Lie algebra of $G$. Suppose $G$ is noncompact. Let $H < G$ be a noncompact semisimple connected proper Lie subgroup and let $\LieH = \mathrm{Lie}(H)$ be its Lie algebra. By semisimplicity of $H$, there exists an $\Ad(H)$-invariant complement $\mathfrak{r}$ of $\LieH$ in $\LieG$ so that $\LieG = \LieH \oplus \mathfrak{r}$. We fix such $\mathfrak{r}$ once and for all. We remark that in our case $\mathfrak{r}$ is unique although in general it might not. Let $\{a_t\}_{t \in \R}$ be a one-parameter subgroup of $H$ generated by a semisimple element in $\LieH$ and let $U$ be the expanding horospherical subgroup of $H$ corresponding to the one-parameter subgroup $\{a_t\}_{t \in \R}$. 

In our case, $(G, H) = (\SL_4(\R), \SO(2, 2)^\circ)$ or $(G, H) = (\SL_4(\R), \SO(3, 1)^\circ)$. The subgroups $\{a_t\}_{t \in \R}$ and $U$ are defined in \cref{eqn:def of a,eqn:def of U}. In \cite{LMWY25}, $\mathbf{G}$ is a semisimple connected real linear algebraic group with absolute rank $2$ which is $\R$-quasi-split. The subgroup $H$ is a principal $\SL_2(\R)$ in $G = \mathbf{G}(\R)^\circ$. The subgroups $\{a_t\}_{t \in \R}$ and $U$ are the standard diagonal subgroup and strictly upper-triangular subgroup in $\SL_2(\R)$. An important common feature in both our work and \cite{LMWY25} is that the $\Ad(H)$-invariant complement $\mathfrak{r}$ is an \emph{irreducible} representation of $H$. 

The strategy of the proof of Theorem~\ref{thm:main equidistribution} is similar to the general strategy developed in \cite{LMW22,LMWY25}. However, due to the complication from $H$ and the $\Ad(H)$-invariant complement $\mathfrak{r}$, the achievement to higher dimension is harder. Before we point out the difficulties and the solutions, let us recall the general strategy developed in \cite{LMW22,LMWY25}. 

In \cite{LMWY25} (see also \cite{LMW22}), the proof can be roughly divided into three phases: 
\begin{enumerate}
    \item Initial dimension from effective closing lemma;
    \item Improving dimension using ingredients from projection theorems;
    \item From large dimension to equidistribution. 
\end{enumerate}

The major difficulties in our setting come from phase~(1) and (2). Due to the complexity of $H$, especially the case where $H = H_1 \cong \SO(2, 2)^\circ$ which is not simple, phase~(1) cannot be proved directly as in \cite[Section 4]{LMWY25}. However, thanks to the effective closing lemma for long unipotent orbits proved by Lindenstrauss, Margulis, Mohammadi, Shah and Wieser in \cite{LMMSW24}, we obtain a similar initial phase. This phase is done in Part~\ref{part:closinglemma}. The reader can compare Theorem~\ref{thm:closing lemma initial dim} with \cite[Proposition 4.6]{LMWY25} and also \cite[Lemma 8.2]{LMWY25}. 

The difficulty from phase~(2) is more severe. In \cite{LMWY25}, phase~(2) can be very roughly further divided into three steps. First, they established a dimension improving result for the linear $\Ad(H)$-action on $\mathfrak{r}$, see \cite[Theorem 6.1]{LMWY25}. Building on this, they established a Margulis function estimate which provides dimension improvement in the transverse direction in $X$, see \cite[Lemma 7.2]{LMWY25}. With the Margulis function estimate, they ran a bootstrap process to get a high dimension (close to $\dim(\mathfrak{r})$) in the transverse direction to $H$, see \cite[Section 8]{LMWY25}. The major difficulty comes from the first step. 

In \cite{LMWY25}, the first step is proved in turn using an optimal projection theorem proved in \cite{GGW24}. Roughly speaking, we say a family of maps $\pi_r:\R^n \to \R^m$ is optimal if for all set $A$ and almost all parameter $r$ one has $\dim \pi_r(A) = \min\{\dim A, m\}$. The $\dim$ here stands for a suitable dimension notion for fractal-like set. 

The $\Ad(H)$-invariant complement $\mathfrak{r}$ can be decomposed into weight spaces $\mathfrak{r}_\lambda$ for $a_t$. Let $\mathfrak{r}^{(\mu)} = \bigoplus_{\lambda \geq \mu} \mathfrak{r}_\lambda$ and let $\pi^{(\mu)}$ be the orthogonal projection to $\mathfrak{r}^{(\mu)}$. An important feature in the setting in \cite{LMWY25} is that for \emph{all} $\mu$, the family of projections $\{\pi^{(\mu)} \circ \Ad(u)\}_{u \in U}$ are optimal thanks to the work by Gan, Guo and Wang in \cite{GGW24}. However, for some $\mathfrak{r}^{(\mu)}$ in our setting, the family of projections $\{\pi^{(\mu)} \circ \Ad(u)\}_{u \in U}$ is \emph{never} optimal due to \emph{algebraic} obstructions. See Example~\ref{example:optimal fail for so 2 2} for a discussion for that algebraic obstruction. If $\mu$ corresponds to the \emph{fastest} expanding direction in $\mathfrak{r}$, we establish that the family of projections $\{\pi^{(\mu)} \circ \Ad(u)\}_{u \in U}$ is optimal by the work in \cite{GGW24}, see Theorem~\ref{thm:optimal}. For all the other $\mu$'s, we apply ideas from representation theory and recent developments on Bourgain's discretized projection theorem \cite{He20,Shm23a,BH24} to establish subcritical estimates (see Section~\ref{sec:Subcritical}). Combining those estimates, we prove a dimension improving result for the linear $\Ad(H)$-action on $\mathfrak{r}$. This is the main novel part of this paper and the whole Part~\ref{part:projection} is devoted to it. 

Part~\ref{part:closinglemma} and Part~\ref{part:projection} are independent. Part~\ref{part:closinglemma} is devoted to phase~(1) and Part~\ref{part:projection} is devoted to the linear dimension improvement result in phase~(2). In Part~\ref{part:bootstrap}, we adapt the framework in \cite{LMWY25} with ingredients proved in \cref{part:projection,part:closinglemma} to prove Theorem~\ref{thm:main equidistribution}. In Part~\ref{part:bootstrap}, we only need the results stated in the introductory parts in \cref{part:projection,part:closinglemma} so it can be read without digging into \cref{part:closinglemma,part:projection}. 

\subsection*{Acknowledgment}
I am extremely grateful to my advisor Amir Mohammadi for introducing the topic and for his guidance, supports, and many helpful discussions throughout. 

\section{Notations and Preliminaries}\label{sec:Global prelim}
As indicated in the introduction, \cref{part:projection,part:closinglemma} can be read independently and Part~\ref{part:bootstrap} only needs the results stated in the introductory parts in \cref{part:projection,part:closinglemma}. In this section, we introduce the notations and preliminaries used in the above indicated region. New notations and preliminaries needed inside \cref{part:projection,part:closinglemma} will be introduced in those 'preparation' sections. We remark that \emph{inside} \cref{part:projection,part:closinglemma}, we might slightly change the conventions for simplicity. We will always clarify the changes at the beginning of each section. 

\subsection{Constants and \texorpdfstring{$\star$-notations}{⋆-notations}}
For $A \ll B^{\star}$, we mean there exist constants $C > 0$ and $\kappa > 0$ depending at most on the $(G, H, \Gamma)$ such that $A \leq C B^{\kappa}$. For $A \asymp B$, we mean $A \ll B$ and $B \ll A$. We also use the notion of $O(\cdot)$ where $f = O(g)$ is the same as $|f| \ll g$. For $A \ll_D B$, we mean there exist constant $C_D > 0$ depending on $D$ and at most $(G, H, \Gamma)$ so that $A \leq C_D B$. For example, in Part~\ref{part:projection}, we will heavily use the notation $A \ll_\epsilon B^{O(\sqrt{\epsilon})}$. This is equivalent to the following. There exists a constant $C_\epsilon$ depending on $\epsilon$ and at most also on $(G, H, \Gamma)$ and a constant $E$ depending at most on $(G, H, \Gamma)$ so that $A \leq C_\epsilon B^{E\sqrt{\epsilon}}$. 

\subsection{Lie groups and Lie algebras}
We use corresponding Fraktur letters for the Lie algebras of Lie groups throughout the paper. For example, $\mathfrak{s}$ is the Lie algebra of Lie group $S$. For a Lie group $S$, we use $S^\circ$ to denote its identity component under the \emph{Hausdorff topology}. For a group $G$ acting on a space $X$, we use $g.x$ to denote this action. Sometimes the action is clear from the context and we will use $g.x$ without introducing it explicitly. For example, for $v \in \LieG$ and $g \in G$, we write $g.v = \Ad(g)v$. 

Throughout Part~\ref{part:bootstrap} and the introductory parts in \cref{part:closinglemma,part:projection}, we fix the group $G$ and $H$ as in the introduction. The rest of this subsection is devoted to their basic properties and related decompositions. 

Recall that $G = \SL_4(\R)$ and $\LieG = \mathrm{Lie}(G)$. Recall $Q_1(x_1, x_2, x_3, x_4) = x_2 x_3 - x_1 x_4$ is an indefinite quadratic form of signature $(2, 2)$ on $\R^4$. There exists a symmetric matrix which we also denote as $Q_1$ so that $Q_1(x) = \langle x, Q_1x\rangle$ where $\langle \cdot, \cdot\rangle$ is the standard Euclidean inner product on $\R^4$. Recall $Q_2(x_1, x_2, x_3, x_4) = x_2^2 +  x_3^2 - 2x_1 x_4$ is an indefinite quadratic form of signature $(3, 1)$ on $\R^4$. Similarly, there exists a symmetric matrix which we also denote as $Q_2$ so that $Q_2(x) = \langle x, Q_2x\rangle$. 

Let $\sigma_i:\LieG \to 
\LieG$ defined to be $\sigma_i(x) = -Q_i x^t (Q_i)^{-1}$. This is an involution of the Lie algebra $\LieG$. Moreover we have $\LieH_i = \Fix(\sigma_i)$. Let $\mathfrak{r}_i$ be the eigenspace of $\sigma_i$ with eigenvalue $-1$. They are $\Ad(H_i)$-invariant complements of $\LieH_i$ in $\LieG$ respectively. Moreover, $\dim(\mathfrak{r}_i) = 9$ and they are irreducible representations of $H_i$-respectively. 

If a definition/result/proof in this paper can be stated simultaneously to $H_1$ and $H_2$ respectively, we drop the subscripts and denote them by $Q$, $H$ and $\mathfrak{r}$. 

Let $\theta: \LieG \to \LieG$ be the involution defined by $\theta(x) = -x^t$. It is a Cartan involution for the Lie algebra $\LieG = \mathfrak{sl}_4(\R)$. Moreover, $\theta$ commutes with $\sigma_i$. Therefore, $\theta|_{\LieH_i}$ is also a Cartan involution. We use $\LieH_i = \LieK_i \oplus \LieP_i$ to denote the corresponding Cartan decomposition. The involution $\theta$ induced an inner product on both $\LieG$ and $\LieH$ and hence a Riemannian metric $d_X$ on $X$ and a volume form on periodic $H$-orbits as in the introduction. 

Let $\LieA_i$ be the subspaces in $\LieH_i$ consists of diagonal matrices. We have $\LieA_i \subset \LieP_i$,  $\dim \LieA_1 = 2$ and $\dim \LieA_2 = 1$. Let $\LieM_i = \mathfrak{Z}_{\LieK_i}(\LieA_i)$ be the centralizer of $\LieA_i$ in $\LieK_i$. We have $\LieM_1 = \{0\}$ and $\dim \LieM_2 = 1$. Let $\LieU_i = \mathrm{Lie}(U_i)$ and $\LieU_i^- = \theta(\LieU_i)$. 
A direct calculation shows that 
\begin{align*}
\LieH_i = \LieM_i \oplus \LieA_i \oplus \LieU \oplus \LieU^-
\end{align*}
and this is a restricted root space decomposition of $\LieH_i$. Let $A_i = \exp(\LieA_i) \leq H_i$, $M_i = \exp(\LieM_i)$ and $U_i^- = \exp(\LieU_i^-)$. Let $U_i^+ = U_i$ and $\LieU_i^+ = \LieU_i$. 

As before, if a definition/result/proof can be state simultaneously for both $i = 1, 2$, we drop the subscript for simplicity \emph{except} $M_i$ and $\LieM_i$. For $M_i$ and $\LieM_i$, all definition/result/proof can be state simultaneously for both $i = 1, 2$, we use $M_0$ and $\LieM_0$ to denote it so that we can use compatible notations with the one in \cite{LMMSW24}, see Theorem~\ref{thm:closing lemma long unipotent}. 

\subsection{Norms and balls}
Let $\|\cdot\| = \|\cdot\|_\infty$ be the maximum norm from $\mathrm{Mat}_4(\R)$. For any subspace $V \subseteq \LieG \subset \mathrm{Mat}_4(\R)$ and $v \in V$, we define
\begin{align*}
B_r^{V}(v) = \{w \in V:\|w - v\| \leq r\}.
\end{align*}
If $v = 0$, we often omit it and denote the ball by $B_r^{V}$. 

We define
\begin{align*}
    \mathsf{B}^U_r = \exp(B^{\LieU}_r), \mathsf{B}^A_r = \exp(B^{\LieA}_r), \mathsf{B}^{M_0}_r = \exp(B^{\LieM_0}_r), \mathsf{B}^{U^-}_r = \exp(B^{\LieU^-}_r)
\end{align*}
and
\begin{align*}
\mathsf{B}_r^{M_0A} = \mathsf{B}^{M_0}_r\mathsf{B}^A_r = \exp(B^{\LieA \oplus \LieM_0}_r).
\end{align*}
We set
\begin{align*}
    \mathsf{B}_r^H = \mathsf{B}^{U^-}_r\mathsf{B}_r^{M_0A} \mathsf{B}^U_r,\\
    \mathsf{B}_r^{s,H} = \mathsf{B}^{U^-}_r\mathsf{B}_r^{M_0A},
\end{align*}
and $\mathsf{B}_r^G = \mathsf{B}_r^H\exp(B_r^{\mathfrak{r}})$. 

\subsection{Natural measures}
Note that $U = \exp(\LieU)$. Since $U$ is abelian, the exponential map $\exp$ is an isomorphism between Lie groups if we identify $\LieU$ with $\R^2$ using the standard coordinate in $\mathrm{Mat}_4(\R)$. Let $\tilde{m}_U$ be the push-forward of the standard Lebesgue measure $\Leb$ under the exponential map. Let $m_U$ be the rescaling of $\tilde{m}_U$ so that it assign $\mathsf{B}_1^U$ with measure $1$. This is a $U$-invariant measure on $U$. For the ball $\mathsf{B}_1^U \subset U$, we use $m_{\mathsf{B}_1^U}$ to denote the restriction of $m_U$ to $\mathsf{B}_1^U$. For simplicity, we use $\mathrm{d}u$ to denote $\mathrm{d}m_{\mathsf{B}_1^U}$ in any related integration. 

Similarly, we can define $m_A$, $m_{M_0}$, $m_{U^{-}}$ via the push-forward of the standard Lebesgue measure on subspaces in $\mathrm{Mat}_4(\R)$. They are Haar measures on the corresponding groups. Let $m_H$ be the corresponding Haar measure on $H$. It is proportional to the measure defined by the volume form induced by the Riemannian metric from the Cartan involution $\theta$. 

Recall that since $\Gamma$ is a lattice in $G$, there is a unique probability $G$ invariant measure $\mu_X$ on $X = G/\Gamma$. This measure is proportional to the measure defined by the volume form induced by the Riemannian metric $d_X$. 

\subsection{Commutation relations}
We record the following consequences of Baker--Campbell--Hausdorff formula. 

\begin{lemma}\label{lem:BCH}
    There exists $\eta_0 > 0$ and $C_0 > 0$ so that the following holds for all $0 < \eta \leq \eta_0$. For all $w_1, w_2 \in B_{\eta}^{\mathfrak{r}}(0)$, there exists $h \in H$ and $\bar{w} \in \mathfrak{r}$ with
    \begin{align*}
        \|h - \Id\| \leq C_0\eta \text{, and }\quad \|\bar{w} - (w_1 - w_2)\| \leq C_0\eta\|w_1 - w_2\|
    \end{align*}
    so that 
    \begin{align*}
        \exp(w_1) \exp(-w_2) = h \exp(\bar{w}).
    \end{align*}
    In particular, 
    \begin{align*}
        \frac{1}{2}\|w_1 - w_2\| \leq \|\bar{w}\| \leq \frac{3}{2}\|w_1 - w_2\|.
    \end{align*}
\end{lemma}
\begin{proof}
    This is a direct application of Baker--Campbell--Hausdorff formula. See \cite[Lemma 2.1]{LM23}. 
\end{proof}

We take a further minimum so that for all $\eta \leq \eta_0$ the following holds. 
\begin{enumerate}
\item The exponential map restrict to $B_\eta^{\LieG}$ is a bi-analytic map. 
\item The maps
\begin{align}
\begin{aligned}
&B_\eta^{\LieU^+} \times B_\eta^{\LieM_0} \times B_\eta^{\LieA} \times 
B_\eta^{\LieU^-} \to H\\
&(X_{\LieU^+}, X_{\LieM_0}, X_\LieA, X_{\LieU^-}) \mapsto \exp(X_{\LieU^+})\exp(X_{\LieM_0})\exp(X_\LieA)\exp(X_{\LieU^-}),
\end{aligned}
\end{align}
\begin{align}
\begin{aligned}
&B_\eta^{\LieU^-} \times B_\eta^{\LieM_0} \times B_\eta^{\LieA} \times 
B_\eta^{\LieU^+} \to H\\
&(X_{\LieU^-}, X_{\LieM_0}, X_\LieA, X_{\LieU^+}) \mapsto \exp(X_{\LieU^-})\exp(X_{\LieM_0})\exp(X_\LieA)\exp(X_{\LieU^+}),
\end{aligned}
\end{align}
\begin{align}
\begin{aligned}
&B_\eta^{\LieU^+} \times B_\eta^{\LieU^-} \times B_\eta^{\LieM_0} \times 
B_\eta^{\LieA} \to H\\
&(X_{\LieU^+}, X_{\LieU^-}, X_\LieA, X_{\LieM_0}) \mapsto \exp(X_{\LieU^+})\exp(X_{\LieU^-})\exp(X_{\LieM_0})\exp(X_{\LieA})
\end{aligned}
\end{align}
are bi-analytic map to their images.
\item The map 
\begin{align*}
    &\mathsf{B}_\eta^{H} \times B_\eta^{\mathfrak{r}} \to G\\
    &(\mathsf{h}, X_{\mathfrak{r}}) \mapsto \mathsf{h}\exp(X_{\mathfrak{r}})
\end{align*}
is a bi-analytic map to its image.
\item Lemma~\ref{lem:BCH} holds. 
\end{enumerate}
For a parameter $\eta \leq \eta_0$ and $\beta = \eta^2$, we set
\begin{align*}
    \mathsf{E} = \mathsf{B}_{\beta}^{U^-} \mathsf{B}_\beta^{M_0A} \mathsf{B}_\eta^{U^+}.
\end{align*}
The choice of the parameter $\eta$ will always be clear from the context. 

\subsection{Injectivity radius}
For all $x \in X = G/\Gamma$, we set
\begin{align*}
    \inj(x) = \sup\{\eta: \mathsf{B}_{100C_0\eta}^G\to \mathsf{B}_{100C_0\eta}^G.x \text{ is a diffeomorphism}\}.
\end{align*}
The constant $C_0$ comes from \cref{lem:BCH}. Taking a further minimum if necessary, we always assume that the injectivity
radius of $x$ defined using the Riemannian metric $d_X$ dominates $\inj(x)$. 

For all $\eta > 0$, let
\begin{align*}
    X_\eta = \{x \in X: \inj(x) \geq \eta\}.
\end{align*}

\subsection{Different formulations for Theorem~\ref{thm:main equidistribution}}\label{subsection:different form on balls}
Recall that we set $\mathsf{B}^U_1 = \exp(B^{\LieU}_1)$ and assign $m_U$ to be the Haar measure on $U$ so that $m_U(\mathsf{B}^U_1) = 1$. We write $\mathrm{d}u = \mathrm{d}m_U(u)$ in integrals for simplicity. The following theorem is a slightly different formulation of Theorem~\ref{thm:main equidistribution}. 

\begin{theorem}\label{thm:main equidistribution different form}
There exist constants $\ref{a:main equidistribution1} > \ref{a:main equidistribution2} \geq 1$ and $\kappa > 0$ depending only on $X$ so that the following holds. For all $x_0 \in X$ and large enough $R$ depending explicitly on $x_0$, for any $T \geq R^{\ref{a:main equidistribution1}}$, at least one of the following is true. 
\begin{enumerate}
\item For all $\phi \in \mathrm{C}_c^\infty(X)$, 
\begin{align*}
\Biggl|\int_{\mathsf{B}^U_1} \phi(a_{\log T} u.x_0) \,\mathrm{d}u - \int_X \phi \,\mathrm{d}\mu_X\Biggr| \leq \mathcal{S}(\phi) R^{-\kappa}
\end{align*}
where $\mathcal{S}(\phi)$ is a certain Sobolev norm. 
\item There exists $x \in X$ so that $H.x$ is periodic with $\vol(H.x) \leq R$ and 
\begin{align*}
    d_X(x_0, x) \leq T^{-\frac{1}{\ref{a:main equidistribution2}}}.
\end{align*}
\end{enumerate}
\end{theorem}

Theorem~\ref{thm:main equidistribution different form} is equivalent to Theorem~\ref{thm:main equidistribution}. Therefore we will focus on the study of the orbit of $a_{\log T}\mathsf{B}_1^U$ in this paper. 

\begin{proof}[Sketch of the proof that Theorem~\ref{thm:main equidistribution} is equivalent to Theorem~\ref{thm:main equidistribution different form}]
Recall that we set $\|\cdot\| = \|\cdot\|_\infty$ on $\mathrm{Mat}_4(\R)$. Therefore $B^{\LieU}_1$ is identified with $[-1, 1]^2$ under the standard coordinate of $\mathrm{Mat}_4(\R)$. Note that we have $[0, 1]^2 = \frac{1}{2}[-1, 1]^2 + (\frac{1}{2}, \frac{1}{2})$ in $\LieU \cong \R^2$, the rest follows from the change of variables formula and the fact that $\inj(u_{\frac{1}{2}, \frac{1}{2}}^{\pm 1}.x) \asymp \inj(x)$. 
\end{proof}

\part{Closing lemma and initial dimension}\label{part:closinglemma}
The main result of this part is Theorem~\ref{thm:closing lemma initial dim}. Before we state the result, let us fix some parameters. 

Let $0 < \epsilon' < 0.001$ be a small constant. In particular, it will be chosen to depend only on $(G, H, \Gamma)$ in Part~\ref{part:bootstrap}. Let $\beta = e^{-\epsilon' t}$ and $\eta = \beta^{1/2}$. We assume that $t$ is large enough so that $t^{100} \leq e^{\epsilon't}$ and $100C_0\eta \leq \eta_0$ where $\eta_0$ is defined in Section~\ref{sec:Global prelim} and $C_0$ is from Lemma~\ref{lem:BCH}. Recall we set
\begin{align*}
    \mathsf{E} = \mathsf{B}_{\beta}^{U^-} \mathsf{B}_\beta^{M_0A} \mathsf{B}_\eta^{U^+}.
\end{align*}

Let us introduce the notions of \emph{sheeted sets} and a \emph{Margulis function} to state the main result. A subset $\mathcal{E} \subseteq X$ is called a \emph{sheeted set} if there exists a base point $y \in X_\eta$ and a finite set of transverse cross-section $F \subset B_{\eta}^{\mathfrak{r}}$ so that the map $(\mathsf{h}, w) \mapsto \mathsf{h}\exp(w).y$ is bi-analytic on $\mathsf{E} \times B_{\eta}^{\mathfrak{r}}$ and 
\begin{align*}
    \mathcal{E} = \bigsqcup_{w \in F}\mathsf{E}\exp(w).y.
\end{align*}
For all $z \in \mathcal{E}$, let
\begin{align*}
    I_{\mathcal{E}}(z) = \{w \in \mathfrak{r}: \|w\| < \inj(z), \exp(w).z \in \mathcal{E}\}.
\end{align*}

Let us recall the (modified) Margulis function defined in \cite{LMWY25}. For every $0 < \delta < 1$ and $0 < \alpha < \dim \mathfrak{r}$, we define the (modified) Margulis function of a sheeted set $\mathcal{E}$:
\begin{align*}
    f^{(\alpha)}_{\mathcal{E}, \delta}(z) = \sum_{w \in I_\mathcal{E}(z) \setminus \{0\}} \max\{\|w\|, \delta\}^{-\alpha}.
\end{align*}
Roughly speaking, the Margulis function provides a measurement on the dimension in the transverse direction of the sheeted set $\mathcal{E}$ for scales at least $\delta$. We refer to Subsection~\ref{subsec:Margulis function} for discussion on its connection with Frostman-type condition and $\alpha$-energy. 

The statement of following theorem also needs the notion of \emph{admissible measures} introduced in \cite{LMW22}. We refer to Subsection~\ref{subsec:sheeted set admissible measure} for its precise definition, see also \cite[Appendix D]{LMWY25} or \cite[Section 7]{LMW22}. Informally, an admissible measure $\mu_{\mathcal{E}}$ associated to a sheeted set $\mathcal{E}$ is a probability measure on $\mathcal{E}$ that is equivalent to Haar measure of $H$ on each sheet. Moreover, each sheet is assigned with roughly equal weight. 

Let $\lambda$ be the normalized Haar measure on
\begin{align*}
    \mathsf{B}^{s, H}_{\beta + 100 \beta^2} = \mathsf{B}_{\beta + 100 \beta^2}^{U^-} \mathsf{B}_{\beta + 100 \beta^2}^{M_0A},
\end{align*}
and
\begin{align*}
    \nu_t = (a_t)_\ast m_{\mathsf{B}_1^U}.
\end{align*}

The following theorem is the main result of this part. 

\begin{theorem}\label{thm:closing lemma initial dim}
There exist constants $\consta\label{a:closing lemma main} > 1$, $\constc\label{c:closing lemma main} > 1$, $\constd\label{d:closing lemma main} > 1$, $\constE\label{e:closing lemma main1}, \constE\label{e:closing lemma main2}> 1$, $\constm\label{m:closing lemma main} > 1$, $\epsilon_0 > 0$, and $\mathsf{L}$ depending only on $(G, H, \Gamma)$ so that the following holds. For all $x_1 \in X_\eta$ and $R \gg \eta^{-\ref{e:closing lemma main1}}$, let $\delta_0 = R^{-\frac{1}{\ref{a:closing lemma main}}}$. For all $D \geq \ref{d:closing lemma main} + 1$, let $t = M\log R$ where $M = \ref{m:closing lemma main} + \ref{c:closing lemma main}D$ and $\mu_t = \nu_{t} \ast \delta_{x_1}$. 
    
Suppose that for all periodic orbit $H.x'$ with $\vol(H.x') \leq R$, we have
\begin{align*}
    d_X(x_1, x') > R^{-D}.
\end{align*}
Then there exists a family of sheeted sets $\mathcal{F} = \{\mathcal{E}\}$ with associated $\mathsf{L}$-admissible measures $\{\mu_{\mathcal{E}}: \mathcal{E} \in \mathcal{F}\}$ so that the following holds. 
\begin{enumerate}
    \item There exists $\{c_\mathcal{E}\}$ with $c_\mathcal{E} > 0$ and $\sum_{\mathcal{E}} c_\mathcal{E} = 1$ so that for all $u' \in \mathsf{B}_1^U$, $d \geq 0$ and all $\phi \in \mathrm{C}^{\infty}_c(X)$
    \begin{align}
        \int_{X} \phi(a_d u' x) \,\mathrm{d}(\lambda \ast \mu_t)(x) = \sum_{\mathcal{E}} c_\mathcal{E}\int_{X} \phi(a_du'.x) \,\mathrm{d}\mu_{\mathcal{E}}(x) + O(\mathcal{S}(\phi)(\beta^\star)).
    \end{align}
    \item For all sheeted set $\mathcal{E} \in \mathcal{F}$ with cross-section $F \subset B_\eta^{\mathfrak{r}}$. 
    The number of sheets satisfies 
    \begin{align}
        \beta^{29}\delta_0^{-2\epsilon_0} \leq \#F \leq \beta^{-2}e^{2t}.
    \end{align}
    Moreover, we have the Margulis function estimate
    \begin{align}
        f^{(\epsilon_0)}_{\mathcal{E}, \delta_0}(z) \leq  \beta^{-\ref{e:closing lemma main2}}\#F \quad \forall z \in \mathcal{E}.
    \end{align}
\end{enumerate}
\end{theorem}
The proof of Theorem~\ref{thm:closing lemma initial dim} relies on the following lemma. 
It asserts that for an initial point with suitable Diophantine condition, if the expanding time $t$ is long enough, then the measure $\lambda \ast \mu_t$ has a small coarse dimension in the transverse direction. Moreover, the weaker Diophantine condition is provided, the longer the time is needed. 
\begin{lemma}\label{lem:Closing lemma many scale}
    There exist constants $\consta\label{a:closinglemma frostman} > 1$, $\constc\label{c:closinglemma frostman} > 1$, $\constd\label{d:closinglemma frostman} > 1$, $\constm\label{m:closinglemma frostman} > 1$, and $\epsilon_1 > 0$ depending only on $(G, H, \Gamma)$ so that the following holds. For all $D \geq \ref{d:closinglemma frostman} + 1$, $x_1 \in X_\eta$ and $R \gg \eta^{-\star}$, let $M = \ref{m:closinglemma frostman} + \ref{c:closinglemma frostman} D$, $t = M\log R$, $\mu_t = \nu_{t} \ast \delta_{x_1}$ and $\delta_0 = R^{-\frac{1}{\ref{a:closinglemma frostman}}}$. 

    Suppose that for all periodic orbit $H.x'$ with $\vol(H.x') \leq R$, we have
    \begin{align*}
        d(x_1, x') > R^{-D}.
    \end{align*}
    Then for all $y \in X_{3\eta}$, $r_H \leq \frac{1}{4}\min\{\inj(y), \eta_0\}$, $r \in [\delta_0, \eta]$, we have
    \begin{align*}
        (\lambda \ast \mu_t)((\mathsf{B}_{r_H}^H)^{\pm 1}\exp(B_{r}^{\mathfrak{r}}).y) \ll \eta^{-\star}r^{\epsilon_1}.
    \end{align*}
\end{lemma}
The proof of the lemma heavily relies on the effective closing lemma proved in \cite{LMMSW24}. We record it in \cref{thm:closing lemma long unipotent}. Let us introduce notions related to the lattice $\Gamma \leq G = \SL_4(\R)$. 

By Margulis' arithmeticity theorem, $\Gamma$ is an arithmetic lattice. Without loss of generality, we assume that there exists a $\Q$-group $\mathbf{G} \subseteq \SL_N$ with $\mathbf{G}(\R)^\circ \cong G = \SL_4(\R)$ and $\Gamma \leq G \cap \SL_N(\Z)$. Later in this paper, when we say \emph{$\mathbf{M}$ is a $\Q$-subgroup of $G$}, we refer to this $\Q$-structure from $\Gamma$. Write $\LieG_\Z = \LieG \cap \mathfrak{sl}_N(\Z)$. It is invariant under $\Gamma$-action. For any $\Q$-subgroup $\mathbf{M}$ of $\mathbf{G}$, let $\LieM$ be its Lie algebra. It is a $\Q$-subspace of $\LieG$. We define $\mathpzc{v}_M \in \wedge^{\dim \LieM} \LieG$ to be one of the primitive integral vector in the line $\wedge^{\dim \LieM} \LieM$. 

For any subspace (not necessarily $\Q$-subspace) $\mathfrak{s} \subseteq \LieG$, we define $\hat{v}_{\mathfrak{s}}$ to be the corresponding point in $\mathbb{P}(\wedge^{\dim \mathfrak{s}} \LieG)$. For any $0 < r \leq \dim \LieG$, we equip $\mathbb{P}(\wedge^{r} \LieG)$ with the Fubini-Study metric $\mathrm{d}$ where $\mathrm{d}(\hat{v}, \hat{w})$ is the angle between the corresponding lines in $\mathbb{P}(\wedge^{r} \LieG)$. 

The following is the main result in \cite{LMMSW24}. Note that since $\SL_4(\R)$ has no connected normal subgroup, the case~(2) in \cite[Theorem 2]{LMMSW24} does not appear. 

\begin{theorem}[Lindenstrauss--Margulis--Mohammadi--Shah--Wieser]\label{thm:closing lemma long unipotent}
    
    There exist constants $\consta\label{a:long unipotent1}, \consta\label{a:long unipotent2} > 1$ and $\constE\label{e:long unipotent} > 1$ depending on $(G, \Gamma)$ so that the following holds. Let $\tau \in (0, 1)$ and $e^t > S \geq \ref{e:long unipotent} \tau^{-\ref{a:long unipotent1}}$. Let $x = g\Gamma \in X_\tau$ be a point. 

    Suppose there exists $\mathcal{E} \subseteq \mathsf{B}_1^U$ with the following properties.
    \begin{enumerate}
        \item $|\mathcal{E}| > S^{-\frac{1}{\ref{a:long unipotent1}}}$.
        \item For any $u, u' \in \mathcal{E}$, there exists $\gamma \in \Gamma$ with \begin{align*}
            \|a_t u a_{-t} g \gamma g^{-1} a_t (u')^{-1} a_{-t} \| \leq S^{\frac{1}{\ref{a:long unipotent1}}},\\
            \mathrm{d}(a_t u a_{-t} g \gamma g^{-1} a_t (u')^{-1} a_{-t}.\hat{v}_{\mathfrak{h}}, \hat{v}_{\mathfrak{h}}) \leq S^{-1}.
        \end{align*}
    \end{enumerate}
    Then there exists a non-trivial proper $\Q$-subgroup $\mathbf{M}$ so that 
    \begin{align*}
        \sup_{u \in \mathsf{B}_1^U} \|a_t u a_{-t} g.\mathpzc{v}_M\| {}&\leq S^{\ref{a:long unipotent2}}, \\
        \sup_{\mathpzc{z} \in B_1^{\LieU}, u \in \mathsf{B}_1^U} \|\mathpzc{z} \wedge (a_t u a_{-t} g.\mathpzc{v}_M)\| {}&\leq e^{-\frac{t}{\ref{a:long unipotent2}}}S^{\ref{a:long unipotent2}}.
    \end{align*}
\end{theorem}

\begin{remark}
    We remark that in \cite{LMMSW24} the notion $X_\tau$ is defined via the \emph{heights} of points in $X$ instead of the notion $\inj$ defined in this paper. However, the transition between them is well-known, see for example \cite[Proposition 26]{SS24}.
\end{remark}

Another key ingredient for Lemma~\ref{lem:Closing lemma many scale} is the following avoidance principle. It is similar to \cite[Theorem 2]{SS24}. It will also play an important role later in Part~\ref{part:bootstrap}. 
\begin{proposition}\label{pro:avoidance}
    There exist $\mathsf{m}$, $s_0$, $\consta\label{a:avoidance}$, $\constc\label{c:avoidance}$, and $\constd\label{d:avoidance}$ depending only on $(G, H, \Gamma)$, so that the following holds. Let $R_1, R_2 \geq 1$. Suppose $x_0 \in X$ is so that
    \begin{align*}
        d_X(x_0, x) \geq(\log R_2)^{\ref{d:avoidance}} R_2^{-1}
    \end{align*}
    for all $x$ with $\vol(H.x) \leq R_1$. Then for all $s \geq \ref{a:avoidance}\max\{\log R_2, |\log \inj(x_0)|\} + s_0$ and all $\eta \in (0, 1]$, we have
    \begin{align*}
        m_U\Biggl(\Biggl\{u \in \mathsf{B}_1^U: \begin{aligned}{}&\inj(a_su.x_0) \leq \eta \text{ or }\exists x \text { with }\vol(H.x) \leq R_1\\ {}&\text { and } d_X(a_s u. x_0, x) \leq \ref{c:avoidance}^{-1} R_1^{-\ref{d:avoidance}}\end{aligned}\Biggr\}\Biggr) \leq \ref{c:avoidance} (R_1^{-1} + \eta^{\frac{1}{\mathsf{m}}}).
    \end{align*}
\end{proposition}

We now sketch an outline for Part~\ref{part:closinglemma}. In Section~\ref{sec:preparation for discriminant}, we recall the relations between different measurements for complexity of a periodic orbit. With this preparation, we start by proving a single scale version for Lemma~\ref{lem:Closing lemma many scale} (namely, Lemma~\ref{lem:Closing lemma one scale}) in Section~\ref{sec:Initial one scale}. This is the main part of this section and the proof relies heavily on the effective closing lemma (recorded in Theorem~\ref{thm:closing lemma long unipotent}) proved in \cite{LMMSW24}. Then, we apply the avoidance principle (Proposition~\ref{pro:avoidance}) to prove Lemma~\ref{lem:Closing lemma many scale} in Section~\ref{sec: Improve many scales}. The last two sections in Part~\ref{part:closinglemma} are devoted to the transition from Lemma~\ref{lem:Closing lemma many scale} to Theorem~\ref{thm:closing lemma initial dim}. Section~\ref{sec:boxes} provides suitable preparation from \cite[Section 7, 8]{LMW22}. In Section~\ref{sec:Margulis function} we prove Theorem~\ref{thm:closing lemma initial dim}. Roughly speaking, one can view these two sections as a transition between two notions of dimension, i.e., from Frostman-type condition to $\alpha$-energy estimate. It roughly follows from the process in \cite[Section 11]{LMW22} with a mild modification. See Section~\ref{sec:Margulis function} for a detailed exposition.  

\section{Preparation I: Measurement for complexity of periodic orbits}\label{sec:preparation for discriminant}
For a periodic orbit, there are various ways to measure its complexity. We briefly recall their relations in this section. For a periodic orbit $Hg\Gamma$ inside $X = G/\Gamma$, one can attach the following quantities to measure its complexity. 

First, from the Riemannian metric on $X = G/\Gamma$, there is a natural volume form on all its embedded submanifolds. Therefore, we can define the volume of a periodic orbit $Hg \Gamma$. We use $\vol(Hg\Gamma)$ to denote this quantity. 

Second, we can define its discriminant which measures its arithmetic complexity. Since $Hg\Gamma$ is periodic, $g^{-1} H g \cap \Gamma$ is a lattice in $g^{-1} H g$ and therefore it is Zariski dense in $g^{-1} H g$. There exists a $\Q$-subgroup $\mathbf{M} \subseteq \mathbf{G}$ so that $g^{-1} H g = \mathbf{M}(\R)^\circ$. This implies that $\Ad(g^{-1})\LieH$ is a $\Q$-subspace of $\LieG$. Let $B$ be the Killing form of $\LieG$. Let
\begin{align*}
    V = (\wedge^{\dim(H)} \LieG)^{\otimes 2}, \qquad V_\Z = (\wedge^{\dim(H)} \LieG_\Z)^{\otimes 2}
\end{align*}
and let
\begin{align*}
    v_{Hg} = \frac{1}{\det(B(e_i, e_j))} (e_1 \wedge \cdots \wedge e_{\dim H})^{\otimes 2} \in V
\end{align*}
where $e_1, \cdots, e_{\dim H}$ is a $\Q$-basis of the $\Q$-subspace $\Ad(g^{-1})\LieH$. The discriminant of $Hg\Gamma$ is defined to be
\begin{align*}
    \disc(Hg\Gamma) = \min\{m \in \Z_{> 0}: mv_{Hg} \in V_\Z\}.
\end{align*}
Note that although the $\Q$-subspace $\Ad(g^{-1})\LieH$ \emph{does} depends on the choice of the representative $g$, $\disc(Hg\Gamma)$ is well-defined. Indeed, a different representative $g\gamma$ gives a possibly different $\Q$-subspace $\Ad(\gamma^{-1}g^{-1})\LieH$. However, $\Ad(\gamma^{-1})$ maps primitive vectors in $V_\Z$ to primitive vectors, the discriminant $\disc(Hg\Gamma)$ is unchanged. 

Lastly, recall that $\mathpzc{v}_M$ is defined as one of the primitive integer vector of the line $\wedge^{\dim \mathbf{M}} \LieM$ inside $\wedge^{\dim \mathbf{M}} \LieG$. The height of $\mathbf{M}$ is defined to be $\height(\mathbf{M}) = \|\mathpzc{v}_M\|$. However, the group $\mathbf{M}$ \emph{does} depends on the choice of representative $g$: if we change $g$ to $g\gamma$, then we need to change $\mathbf{M}$ to $\gamma^{-1}\mathbf{M} \gamma$. The length $\|\Ad(\gamma^{-1})\mathpzc{v}_M\|$ can be significantly different from $\|\mathpzc{v}_M\|$. 

By \cite[Proposition 17.1]{EMV09}, we have the following relation between volume and discriminant:
\begin{align*}
    \vol(Hg\Gamma) \ll \disc(Hg\Gamma)^{\star}.
\end{align*}
The connection between $\disc(Hg\Gamma)$ and $\height(\mathbf{M})$ is recorded in the following lemma. 

\begin{lemma}
    For all $\Q$-subgroup $\mathbf{M}$ so that $\mathbf{M}(\R)^\circ = g^{-1} H g$, we have
    \begin{align*}
        \disc(Hg\Gamma) \ll \height(\mathbf{M})^2.
    \end{align*}
    The implied constant does not depend on the choice of representative $g$. 
\end{lemma}
\begin{proof}
By taking a $\Z$-basis of $\LieM = \Ad(g^{-1})\LieH$, we have
\begin{align*}
\disc(Hg\Gamma) = \det (B(e_i, e_j)) \ll \|e_1 \wedge \cdots \wedge e_{\dim H}\|^2 = \height(\mathbf{M})^2.
\end{align*}
\end{proof}
We have the following direct corollary. 
\begin{corollary}\label{cor:Volume and Height}
    There exists $c > 1$ depends only on $(G, H, \Gamma)$ so that the following holds. For all $\Q$-subgroup $\mathbf{M}$ with $\mathbf{M}(\R)^\circ = g^{-1} H g$, we have
    \begin{align*}
        \vol(Hg\Gamma) \ll \height(\mathbf{M})^{c}.
    \end{align*}
\end{corollary}

\section{Dimension estimate in one scale}\label{sec:Initial one scale}
This section is devoted to prove the following weaker version of \cref{lem:Closing lemma many scale}. It provides a dimension estimate in a single scale. Later in \cref{sec: Improve many scales}, using \cref{pro:avoidance}, we are able to extend it to all larger scales and prove \cref{lem:Closing lemma many scale}. 
\begin{lemma}\label{lem:Closing lemma one scale}
    There exist constants $\consta\label{a:closing lemma frostman one scale}, \constc\label{c:closing lemma frostman one scale}, \constE\label{e:closing lemma frostman one scale}, \constm\label{m:closing lemma frostman one scale} > 1$ and $\epsilon_2 > 0$ depending only on $(G, H, \Gamma)$ so that the following holds. For all $D > 0$, $x_1 \in X_\eta$ and $R \gg \eta^{-\ref{e:closing lemma frostman one scale}}$, let $M = \ref{m:closing lemma frostman one scale} + \ref{c:closing lemma frostman one scale} D$, $t = M\log R$, $\mu_t = \nu_{t} \ast \delta_{x_1}$ and $\delta = R^{-\frac{1}{\ref{a:closing lemma frostman one scale}}}$. 

    Suppose that for all periodic orbit $H.x'$ with $\vol(H.x') \leq R$, we have
    \begin{align*}
        d(x_1, x') > R^{-D}.
    \end{align*}
    Then for all $y \in X_{3\eta}$ and all $r_H \leq \frac{1}{2}\min\{\inj(y), \eta_0\}$, we have
    \begin{align*}
        \mu_t((\mathsf{B}_{r_H}^H)^{\pm 1}\exp(B_{\delta}^{\mathfrak{r}}).y) \leq \delta^{\epsilon_2}.
    \end{align*}
\end{lemma}

\subsection{Linear algebra lemma}
The main lemma for this subsection is of the following. 
\begin{lemma}\label{lem:Important Linear Algebra}
    There exists an absolute constant $\constc\label{c:linear algebra} > 0$ depending only on $(G, H)$ so that the following holds for all $\tilde{\eta} \in (0, 1)$, $\tilde{R} \gg \tilde{\eta}^{-2}$, and $\tilde{t} \geq \ref{c:linear algebra}(\ref{c:linear algebra} + 1)\log R$. The implied constant here depends only on $(G, H)$. 
    
    Suppose there exist a connected proper $\R$-subgroup $\mathbf{M}$ of $\mathbf{SL}_4$ and $g \in G$ with the following properties. Let $M = \mathbf{M}(\R)$ and let $v_M$ be a non-zero vector in the line corresponding to $\LieM = \mathrm{Lie}(M)$ in $\wedge^{\dim M} \LieG$, it satisfies
    \begin{align}
    \begin{aligned}
        \|g.v_M\| {}&\geq \tilde{\eta},\\
        \sup_{u \in \mathsf{B}^U_1} \|a_{\tilde{t}} u g.v_M\| &{}\leq \tilde{R}. 
    \end{aligned}
    \end{align}
    Then we have
    \begin{align*}
        \|g.v_M\| \ll \tilde{R}.
    \end{align*}
    Moreover, if $H \cong \SO(3, 1)^\circ$, there exists $g' \in G$ with $\|g' - I\| \leq \tilde{R}^{\ref{c:linear algebra}} e^{-\frac{1}{\ref{c:linear algebra}}\tilde{t}}$ so that 
    \begin{align*}
        g'gM^\circ g^{-1}(g')^{-1} = H.
    \end{align*}
    If $H \cong \SO(2, 2)^\circ$, assume further that there exists $A > 1$ so that
    \begin{align}\label{eqn:Horosphere SO 2 2}
        \sup_{\mathpzc{z} \in B_1^{\LieU}, u \in \mathsf{B}_1^U} \|\mathpzc{z} \wedge (a_{\tilde{t}} u g.v_M)\| \leq e^{-\frac{\tilde{t}}{A}}\tilde{R}.
    \end{align}
    Then if $\tilde{t} \geq A\ref{c:linear algebra}(\ref{c:linear algebra} + 1) \log \tilde{R}$, there exists $g' \in G$ with $\|g' - I\| \leq \tilde{R}^{\ref{c:linear algebra}} e^{-\frac{1}{\ref{c:linear algebra}}\tilde{t}}$ so that 
    \begin{align*}
        g'gM^\circ g^{-1}(g')^{-1} = H.
    \end{align*}
\end{lemma}

\subsubsection{$H$-invariant subspaces of $\LieG$}
Recall that we write $\LieG = \LieH \oplus \mathfrak{r}$ as decomposition of representation of $H$. We classify all $H$-invariant subspaces in $\LieG$ and show that the complement $\mathfrak{r}$ in our setting is far from being a subalgebra.  

\begin{lemma}\label{lem:H invariant SO 3 1}
    Let $\LieG = \mathfrak{sl}_4(\R)$ and $\LieH \cong \mathfrak{so}(3, 1)$ as in the introduction. Then $\LieH$ is a simple Lie algebra. If $W$ is a proper non-trivial $H$-invariant subspace, we have $W = \LieH$ or $W = \mathfrak{r}$.
\end{lemma}
\begin{proof}
The first claim is standard. The second claim follows from $\LieG = \LieH \oplus \mathfrak{r}$ and the fact that $\LieH$ and $\mathfrak{r}$ are \emph{non-isomorphic} irreducible representations of $\LieH$. 
\end{proof}

\begin{lemma}\label{lem:H invariant SO 2 2}
    Let $\LieG = \mathfrak{sl}_4(\R)$ and $\LieH \cong \mathfrak{so}(2,2)$ as in the introduction. There exists a decomposition $\LieH = \LieH_1 \oplus \LieH_2$ where $\LieH_1$, $\LieH_2$ are ideals of $\LieH$ isomorphic to $\mathfrak{sl}_2(\R)$. If $W$ is a proper non-trivial $H$-invariant subspace, then $W$ satisfies one of the following:
    \begin{enumerate}
        \item $W = \LieH_i$ for some $i = 1, 2$,
        \item $W = \LieH$, 
        \item $W \supseteq \mathfrak{r}$.
    \end{enumerate}
\end{lemma}
\begin{proof}
    Recall that $Q_1(x_1, x_2, x_3, x_4) = x_2x_3 - x_1x_4$. Let
    \begin{align*}
        \tilde{Q}_1 = \begin{pmatrix}
             & & & 1\\
             & & -1 & \\
             & -1 & &\\
             1 & & &
        \end{pmatrix}
    \end{align*}
    be half of the corresponding matrix. Then 
    \begin{align*}
        \LieH = \{X \in \mathfrak{sl}_4(\R): X = -\tilde{Q}_1X^t\tilde{Q}_1\}
    \end{align*}
    Therefore, all element in $\LieH$ has the form
    \begin{align*}
        X = \begin{pmatrix}
            a_1 + a_2 & b_1 & b_2 & 0\\
            c_1 & -a_1+a_2 & 0 & b_2\\
            c_2 & 0 & a_1 - a_2 & b_1\\
            0 & c_2 & c_1 & -a_1 - a_2
        \end{pmatrix}.
    \end{align*}
    Let $\LieH_1$ and $\LieH_2$ be subspaces consist of following elements respectively:
    \begin{align*}
        X_1 = \begin{pmatrix}
            a_1 & b_1 & 0 & 0\\
            c_1 & -a_1 & 0 & 0\\
            0 & 0 & a_1 & b_1\\
            0 & 0 & c_1 & - a_1
        \end{pmatrix}, \qquad
        X_2 = \begin{pmatrix}
            a_2 & 0 & b_2 & 0\\
            0 & a_2 & 0 & b_2\\
            c_2 & 0 & -a_2 & 0\\
            0 & c_2 & 0 & -a_2
        \end{pmatrix}.
    \end{align*}
    A direct calculation shows that they are ideals of $\LieH$ and they both isomorphic to $\mathfrak{sl}_2(\R)$. For the second claim, it suffices to show that the only non-trivial proper $H$-invariant subspace of $\LieH$ are $\LieH_1$ and $\LieH_2$, which follows from the uniqueness of decomposition of semisimple Lie algebra to direct sum of ideals. 
\end{proof}

The following lemma asserts that the natural complement of a symmetric subalgebra is far from being a subalgebra. Recall from Section~\ref{sec:Global prelim} the maps $\sigma_i: x \mapsto -(Q_i)x^t(Q_i)^{-1}$ are Lie algebra involutions for $\LieG = \mathfrak{sl}_4(\R)$, we can apply the following lemma to $\LieG = \LieH \oplus \mathfrak{r}$ in our case. 
\begin{lemma}\label{lem:symmetric subgroup lie algebra}
    Let $\LieG$ be a semisimple Lie algebra. Let $\LieH \subset \LieG$ be a symmetric subalgebra, that is, there is a Lie algebra involution $\sigma$ so that $\LieH = \Fix(\sigma)$. Suppose $\LieG = \LieH \oplus \LieQ$ is the decompostion of eigenspace of $\sigma$ where $\LieH = \Fix(\sigma)$ and $\LieQ$ is the eigenspace for eigen value $-1$. Then there exist two elements $x_1, x_2 \in \LieQ$ with $\|x_1\| = \|x_2\| = 1$ and $[x_1, x_2] \in \LieH$ so that 
    \begin{align*}
        \|[x_1, x_2]\| \gg 1.
    \end{align*}
    The implied constant depends only on the pair $(\LieG, \LieH)$. 
\end{lemma}
\begin{proof}
    Note that $[\LieQ, \LieQ] \subseteq \LieH$. It suffices to show that $[\LieQ, \LieQ] \neq \{0\}$. Suppose not, then for all $x \in \LieQ$, the matrix of $\ad x$ under the decomposition $\LieG = \LieH \oplus \LieQ$ is of the following form
    \begin{align*}
        \ad x = \begin{pmatrix}
            0 & 0\\
            \ast & 0
        \end{pmatrix}.
    \end{align*}
    Since $[\LieH, \LieQ] \subseteq \LieQ$, for all $y \in \LieH$, its matrix representation under the decomposition $\LieG = \LieH \oplus \LieQ$ is of the following form
    \begin{align*}
        \ad y = \begin{pmatrix}
            \ast & 0\\
            0 & \ast
        \end{pmatrix}.
    \end{align*}
    Therefore, 
    \begin{align*}
        \kappa(x, y) = \tr(\ad x \ad y) = 0.
    \end{align*}
    Also, for all $z \in \LieQ$, we have $\kappa(x, z) = \tr(\ad x \ad z) = 0$. 
    This implies that the Killing form $\kappa$ is degenerate, contradicting to the fact that $\LieG$ is semisimple. 
\end{proof}

\subsubsection{An equivariant projection} 
We record an equivariant projection from \cite{EMV09}. Let $\bar{v}_H$ be a \textit{unit vector} in the line corresponding to $\LieH$ in $\wedge^{\dim \LieH} \LieG$. 
\begin{lemma}\label{lem:equivariant proj}
    There exists a neighborhood $\mathcal{N}_H$ of $\bar{v}_H$ and a projection map $\Pi: \mathcal{N}_H \to G.\bar{v}_H$ so that the following holds. 
    For all $v \in \mathcal{N}_H$ with $g.v = v$ for some $g \in B_1^G$, we have $g.\Pi(v) = \Pi(v)$. 
\end{lemma}

\begin{proof}
    See \cite[Lemma 13.2]{EMV09}. 
\end{proof}

\subsubsection{Proof of Lemma~\ref{lem:Important Linear Algebra}}
The proof uses an effective version of {\L}ojasiewiecz's inequality \cite{Loj59}. It asserts that the distance between a point to zero locus of an analytic function can be controlled by the value of that function. We use an effective version of this statement for polynomials proved in \cite{Sol91}, see also \cite[Theorem 3.2]{LMMSW24}. The height of a polynomial in $\Z[x_1, \ldots, x_n]$ is defined to be the maximum of its coefficients in absolute value. 
\begin{theorem}[Solern\'{o} \cite{Sol91}]\label{thm:Lojasiewiecz}
    For any $d \in \N$, there exists $C(d) > 1$ with the following property. 

    Let $h > 1$ and let $f_1, \ldots, f_r \in \Z[x_1, \ldots, x_n]$ have degree at most $d$ and height at most $h$. Let $\mathcal{V} \subseteq \R^n$ be the zero locus of $f_1, \ldots, f_r$. Then for $w \in \R^n$
    \begin{align*}
        \min\{1, d(w, \mathcal{V})\} \ll_d (1 + \|w\|)^{C(d)} h^{C(d)} \max_{1 \leq i \leq r} |f_i(w)|^{\frac{1}{C(d)}}
    \end{align*}
    where $d(w, \mathcal{V}) = \inf_{v \in \mathcal{V}} d(w, v)$. 
\end{theorem}

\begin{proof}[Proof of \cref{lem:Important Linear Algebra}]
    Let $m = \dim M$. Let $V = \wedge^m \LieG$. This is a representation of $H$ with the following decomposition
    \begin{align*}
        V = V^H \oplus V^{\mathrm{non}}
    \end{align*}
    where $V^H$ consists of fixed vectors of $H$ and $V^{\mathrm{non}}$ is the direct sum of non-trivial sub-representations. We write $g.v_M = v_0 + v^{\mathrm{non}}$ according to this decomposition. 
    Since
    \begin{align*}
        \sup_{u \in \mathsf{B}_1^U}\|a_{\tilde{t}} u g.v_M\| \asymp \|v_0\| + \sup_{u \in \mathsf{B}_1^U}\|a_{\tilde{t}} u.v^{\mathrm{non}}\| &{}\leq \tilde{R},
    \end{align*}
    we have 
    \begin{align*}
        \|v^{\mathrm{non}}\| \ll \tilde{R}e^{-\tilde{t}}, \qquad \|v_0\| \leq \tilde{R}.
    \end{align*}
    where the implied constants depend only on the representation $V$, cf. \cite[Section 5]{Sh96}. 
    Therefore, we have $\|g.v_M\| \ll \tilde{R}$, which proves the first assertion. 
    
    Let $\mathcal{V}$ be the variety in $V = \wedge^m \LieG$ consists of pure $m$-wedges. Let $\mathcal{W} = \mathcal{V}^H$, that is, the fixed points of $H$ in $\mathcal{V}$. This is an affine variety defined by polynomials with integral coefficients. Moreover, their degrees and heights are bounded by absolute constants. By Theorem~\ref{thm:Lojasiewiecz}, there exists an absolute constant $C > 1$ so that
    \begin{align*}
        \min\{1, d(v_0, \mathcal{W})\} \ll (1 + \|v_0\|)^C\max_i\{|f_i(v_0)|\}^{\frac{1}{C}}.
    \end{align*}
    where $f_i$ are those polynomials defining $\mathcal{V}$. Since $f_i(g.v_M) = 0$, we have
    \begin{align*}
        |f_i(v_0)| \ll \tilde{R}^{d_i} \|v^{\mathrm{non}}\| \leq \tilde{R}^{d_i + 1} e^{-t}
    \end{align*}
    where $d_i = \deg f_i$. Therefore,
    \begin{align*}
        \min\{1, d(v_0, \mathcal{W})\} \ll \tilde{R}^{C + \max_i d_i + 1} e^{-\frac{\tilde{t}}{C}}.
    \end{align*}

    Let $\ref{c:linear algebra} = C(10 + \max_i d_i) + 1$, if $\tilde{t} \geq \ref{c:linear algebra}(\ref{c:linear algebra} + 1)\log \tilde{R}$ and $\tilde{R} \gg 1$, we have 
    \begin{align*}
        d(v_0, \mathcal{W})\leq \tilde{R}^{\ref{c:linear algebra} - 1} e^{-\frac{1}{\ref{c:linear algebra}}\tilde{t}}.
    \end{align*}
    Let $v_W \in \mathcal{W}$ be the closest point to $v_0$, we have \begin{align}\label{eqn: close to H inv space}
        \|g.v_M - v_W\| \leq \tilde{R}^{\ref{c:linear algebra} - 1} e^{-\frac{1}{\ref{c:linear algebra}}\tilde{t}}.
    \end{align}
    Since $v_W \in \mathcal{W}$, it is a pure $m$-wedge of size $\ll \tilde{R}$ coming from a $H$-invariant subspace $W$. Moreover, since $\|g.v_M\| \geq \tilde{\eta}$ and $\tilde{R} \gg \tilde{\eta}^{-2}$, $\|v_W\| \gg \tilde{\eta}$. Applying \cref{lem:H invariant SO 3 1,lem:H invariant SO 2 2}, we have the following cases. 

    \medskip
    \noindent\textit{Case 1.} $W \supseteq \mathfrak{r}$. We exclude this case using the fact that $\mathfrak{r}$ is far from being a subalgebra. Recall that for a non-zero vector $v \in V$, we use $\hat{v}$ to denote the corresponding line in $\mathbb{P}(V)$. Let $\mathrm{d}$ be the Fubini-Study metric on $\mathbb{P}(V)$. Since $\|g.v_M\| \geq \tilde{\eta}$, we have
    \begin{align*}
        \mathrm{d}(\hat{v}_W, \hat{v}_{gMg^{-1}}) \ll \tilde{\eta}^{-1}\tilde{R}^{\ref{c:linear algebra}}e^{-\frac{1}{\ref{c:linear algebra}}\tilde{t}} \leq \tilde{\eta}^{-1} \tilde{R}^{-1}.
    \end{align*}
    Since $\tilde{R} \gg \tilde{\eta}^{-2}$, we have
    \begin{align}\label{eqn:Fubini Study q}
        \mathrm{d}(\hat{v}_W, \hat{v}_{gMg^{-1}}) \leq \tilde{R}^{-1/2}.
    \end{align}
    By \cref{lem:symmetric subgroup lie algebra}, there exist two elements $x_1, x_2 \in W = \mathfrak{r}$ with $\|x_1\| = \|x_2\| = 1$ and $[x_1, x_2] \in \LieH$ so that 
    \begin{align*}
        \|[x_1, x_2]\| \gg 1.
    \end{align*}
    By \cref{eqn:Fubini Study q}, there exist two elements $x_1', x_2' \in \Ad(g)\LieM$ with $\|x_1'\| = \|x_2'\| = 1$ and 
    \begin{align*}
        \|x_i' - x_i\| \leq \tilde{R}^{-1/2} \quad \forall i = 1, 2.
    \end{align*}
    Write $x_i = x_i' + \epsilon_i$ for $i = 1, 2$. 

    Since $\Ad(g) \LieM$ is a subalgebra, $[x_1', x_2'] \in \Ad(g)\LieM$. This implies
    \begin{align*}
        \mathrm{dist}([x_1, x_2], \Ad(g)\LieM) \ll \tilde{R}^{-1/2}
    \end{align*}
    where implied constant depends only on $\LieG$. 
    By \cref{eqn:Fubini Study q}, we have
    \begin{align*}
        \mathrm{dist}([x_1, x_2], \mathfrak{r}) \ll \tilde{R}^{-1/2}.
    \end{align*}
    Since $[x_1, x_2] \in \LieH$ and $\|[x_1, x_2]\| \gg 1$, we get
    \begin{align*}
        1 \ll \mathrm{dist}([x_1, x_2], \mathfrak{r}) \ll \tilde{R}^{-1/2}.
    \end{align*}
    If $\tilde{R}$ is large enough depending only on $\LieG$, we get a contradiction. 

    \medskip
    \noindent\textit{Case 2.} $\LieH \cong \mathfrak{so}(2, 2)$ and $W = \LieH_i \cong \mathfrak{sl}_2(\R)$ for some $i = 1, 2$ as in \cref{lem:H invariant SO 2 2}. We exclude this case via the additional condition (\cref{eqn:Horosphere SO 2 2}) for $H \cong \SO(2, 2)^\circ$. 
    Set
    \begin{align*}
        \mathpzc{z} = \begin{pmatrix}
            0 & 0 & 1 & 0\\
            0 & 0 & 0 & 1\\
            0 & 0 & 0 & 0\\
            0 & 0 & 0 & 0
        \end{pmatrix} \text{ if } W = \LieH_1; \quad \mathpzc{z} = \begin{pmatrix}
            0 & 1 & 0 & 0\\
            0 & 0 & 0 & 0\\
            0 & 0 & 0 & 1\\
            0 & 0 & 0 & 0
        \end{pmatrix} \text{ if } W = \LieH_2.
    \end{align*}
    Note that $\mathpzc{z} \in \LieU$. Moreover, if $W = \LieH_1$, $\mathpzc{z} \in \LieH_2$ and if $W = \LieH_2$, $\mathpzc{z} \in \LieH_1$. Since $\|v_W\| \gg \tilde{\eta}$, we have
    \begin{align}\label{eqn:wedge lower bound}
        \|\mathpzc{z}\wedge v_W\| \gg \tilde{\eta} \gg \tilde{R}^{-\frac{1}{2}}
    \end{align}
    
    Since $\mathpzc{z}$ is fixed under $U$-action and expanded by $a_t$ with a rate $e^t$, by \cref{eqn:Horosphere SO 2 2}, we have
    \begin{align*}
        \sup_{u \in B_1^U}\|a_tu(\mathpzc{z} \wedge (g.v_M))\| \leq e^{(1 - \frac{1}{A})\tilde{t}} \tilde{R}.
    \end{align*}
    
    Let $\tilde{V} = \wedge^{m + 1} \LieG$ where $m = \dim W = \dim \LieH_i = 3$. As in $V$, there is a decomposition
    \begin{align*}
        \tilde{V} = \tilde{V}^{\mathrm{non}} \oplus \tilde{V}^H
    \end{align*}
    where $\tilde{V}^H$ consists of fixed vectors of $H$ and $\tilde{V}^{\mathrm{non}}$ is the sum of all non-trivial sub-representation of $H$. Write $\mathpzc{z} \wedge (g.v_M) = \tilde{v}_0 + \tilde{v}^{\mathrm{non}}$ according to this decomposition. 
    Similar to the argument in $V$, we have
    \begin{align*}
        \|\tilde{v}^{\mathrm{non}}\| \ll \tilde{R} e^{-\frac{\tilde{t}}{A}}.
    \end{align*}

    Since $\|g.v_M\| \ll \tilde{R}$, we have
    \begin{align*}
        \|\mathpzc{z} \wedge (g.v_M)\| \ll \tilde{R}
    \end{align*}
    and hence
    \begin{align*}
        \|\tilde{v}_0\| \ll \tilde{R}.
    \end{align*}

    Let $\tilde{\mathcal{V}}$ be the variety in $\tilde{V} = \wedge^{m + 1} \LieG$ consists of pure $m + 1$-wedge. Let $\tilde{W} = \tilde{\mathcal{V}}^H$, which is the fixed point of $H$ in $\tilde{\mathcal{V}}$. This is an affine variety defined by polynomials with integral coefficients. Moreover, their degrees and heights are bounded by absolute constants. By \cref{thm:Lojasiewiecz}, there exists $\tilde{C} > 1$ so that 
    \begin{align*}
        \min\{1, d(\tilde{v}_0, \tilde{\mathcal{W}})\} \ll (1 + \|\tilde{v}_0\|)^{\tilde{C}} \max_i\{|\tilde{f}_i(v_0)|\}^{\frac{1}{\tilde{C}}}.
    \end{align*}
    where $\tilde{f}_i$ are the defining polynomials for pure $(m + 1)$-wedge. Since $f_i(\mathpzc{z} \wedge (g.v_M)) = 0$, we have
    \begin{align*}
        |\tilde{f}_i(v_0)| \ll \tilde{R}^{\tilde{d}_i}\|\tilde{v}^{\mathrm{non}}\| \leq \tilde{R}^{\tilde{d}_i + 1} e^{-\frac{1}{A}\tilde{t}}
    \end{align*}
    where $\tilde{d}_i = \deg \tilde{f}_i$. Therefore, 
    \begin{align*}
        \min\{1, d(\tilde{v}_0, \tilde{\mathcal{W}})\} \ll \tilde{R}^{\tilde{C} + \max_i \tilde{d}_i} e^{-\frac{1}{\tilde{C}A}\tilde{t}}.
    \end{align*}
    Enlarge $\ref{c:linear algebra}$ to $\ref{c:linear algebra} = \tilde{C}(10 + \max_i \tilde{d}_i) + C(10 + \max_i d_i) + 1$, 
    If $t \geq A\ref{c:linear algebra}(\ref{c:linear algebra} + 1)\log R$ and $R \gg 1$ is large enough, we have
    \begin{align*}
        d(\tilde{v}_0, \tilde{\mathcal{W}}) \leq \tilde{R}^{\ref{c:linear algebra}} e^{-\frac{t}{A\ref{c:linear algebra}}}.
    \end{align*}
    Therefore, there is a $H$-fixed pure wedge $\tilde{w}$ in $\tilde{V} = \wedge^{m + 1}\LieG$ so that
    \begin{align*}
        \|\mathpzc{z} \wedge (g.v_M) - \tilde{w}\| \leq \tilde{R}^{\ref{c:linear algebra}} e^{-\frac{t}{A\ref{c:linear algebra}}} \leq \tilde{R}^{-1}.
    \end{align*}
    Note that by Lemma~\ref{lem:H invariant SO 2 2}, there is no non-trivial $4$-dimensional proper $H$-invariant subspace of $\LieG$. This implies $\tilde{w} = 0$ and 
    \begin{align*}
        \|\mathpzc{z} \wedge (g.v_M)\| \leq \tilde{R}^{-1}.
    \end{align*}

    Recall that $\|g.v_M - v_W\| \leq \tilde{R}^{-1}$, we have
    \begin{align*}
        \|\mathpzc{z} \wedge v_W\| \ll \tilde{R}^{-1}
    \end{align*}
    On the other hand, recall from \cref{eqn:wedge lower bound} that our choice of $\mathpzc{z}$ ensures that $\|\mathpzc{z} \wedge v_W\| \gg \tilde{R}^{-\frac{1}{2}}$. We get a contradiction if $\tilde{R}$ is large enough depending on $\LieG$. 

    \medskip
    \noindent\textit{Case 3.} $W = \LieH$. Write $v_W = \lambda \bar{v}_H$ where $\lambda > 0$ and $\bar{v}_H$ is a unit vector in the line corresponding to $\LieH$ in $V$. Note that $\lambda = \|v_W\| \gg \tilde{\eta}$. By \cref{eqn: close to H inv space}, we have
    \begin{align*}
        \|\lambda^{-1}g.v_M - \bar{v}_H\| \leq \tilde{\eta}^{-1} \tilde{R}^{-1} \leq \tilde{R}^{-1/2}.
    \end{align*}
    
    We claim that if $\tilde{R}$ is large enough depending only on $(\LieG, \LieH)$, then the above inequality forces $\Ad(g) \LieM$ to be a reductive subalgebra. Let $B$ be the Killing form of $\LieG$ and let $\wedge^{\dim \LieH} B$ be the induced bilinear form on $V$. By semisimplicity of $\LieH$, we have $\wedge^{\dim \LieH}B(\bar{v}_H, \bar{v}_H) \neq 0$. If $\tilde{R} \gg 1$, we have
    \begin{align*}
        \wedge^{\dim \LieH}B(\lambda^{-1}g.v_M,  \lambda^{-1}g.v_M) \neq 0,
    \end{align*}
    which implies that the restriction Killing form $B$ of $\LieG$ to $\Ad(g)\LieM$ is non-degenerate. Therefore, $\Ad(g) \LieM$ is a reductive subalgebra\footnote{In fact, using \cite[Chapter VII, \S 1, no.3, Lemma 2]{Bou05}, one can show that $\Ad(g)\mathbf{M}$ is a reductive subgroup of $G$. Note that since $\Ad(g)\LieM$ is the Lie algebra of the algebraic subgroup $\Ad(g)\mathbf{M}$, condition~2) in that lemma satisfies automatically. }, cf. \cite[Chapter I, \S 6, no.4, Proposition 5]{Bou89}. 

    If $\tilde{R}$ is large enough, we have $\lambda^{-1}g.v_M \in \mathcal{N}_H$ where $\mathcal{N}_H$ is the neighborhood in Lemma~\ref{lem:equivariant proj}. Apply the equivariant projection $\Pi$ in Lemma~\ref{lem:equivariant proj} to the vector $\lambda^{-1}g.v_M$ and denote $\Pi(\lambda^{-1}g.v_M) = (g')^{-1}\bar{v}_H$. We have
    \begin{align}\label{eqn:small error}
    \begin{aligned}
        \|g' - \Id\| \ll{}& \|\lambda^{-1} g.v_M - \bar{v}_H\| \leq \tilde{R}^{\ref{c:linear algebra}}e^{-\frac{1}{\ref{c:linear algebra}}\tilde{t}}.
    \end{aligned}
    \end{align}
    The last inequality follows from \cref{eqn: close to H inv space} and $\tilde{R} \gg \tilde{\eta}^{-2}$. 

    Since $\Ad(g)\LieM$ is a reductive Lie algebra, for all $m \in M^\circ$, $gmg^{-1}$ fixes the \emph{vector} $g.v_M$. Lemma~\ref{lem:equivariant proj} implies that $(g')^{-1}\bar{v}_H$ is also fixed by all elements in $B_1^{G} \cap gMg^{-1}$, which generates $gM^\circ g^{-1}$. Therefore, $g'gM^\circ g^{-1}(g')^{-1} \subseteq \Stab(\bar{v}_H) = H$. Since they are both $6$-dimensional connected subgroup, we have
    \begin{align*}
        g'gM^\circ g^{-1}(g')^{-1} = H.
    \end{align*}
    Combining with \cref{eqn:small error}, the proof is complete. 
\end{proof}

\subsection{Applying effective closing lemma for large unipotent orbit}
In this subsection, we combine \cref{thm:closing lemma long unipotent,lem:Important Linear Algebra} to prove Lemma~\ref{lem:Closing lemma one scale}. For reader's convenience, we restate Lemma~\ref{lem:Closing lemma one scale} as the following. 

\begin{lemma*}
There exist constants $\ref{a:closing lemma frostman one scale}, \ref{c:closing lemma frostman one scale}, \ref{e:closing lemma frostman one scale}, \ref{m:closing lemma frostman one scale} > 1$ and $\epsilon_2 > 0$ depending only on $(G, H, \Gamma)$ so that the following holds. For all $D > 0$, $x_1 \in X_\eta$ and $R \gg \eta^{-\ref{e:closing lemma frostman one scale}}$, let $M = \ref{m:closing lemma frostman one scale} + \ref{c:closing lemma frostman one scale} D$, $t = M\log R$, $\mu_t = \nu_{t} \ast \delta_{x_1}$ and $\delta = R^{-\frac{1}{\ref{a:closing lemma frostman one scale}}}$. 

Suppose that for all periodic orbit $H.x'$ with $\vol(H.x') \leq R$, we have
\begin{align*}
    d(x_1, x') > R^{-D}.
\end{align*}
Then for all $y \in X_{3\eta}$ and all $r_H \leq \frac{1}{2}\min\{\inj(y), \eta_0\}$, we have
\begin{align*}
    \mu_t((\mathsf{B}_{r_H}^H)^{\pm 1}\exp(B_{\delta}^{\mathfrak{r}}).y) \leq \delta^{\epsilon_2}.
\end{align*}
\end{lemma*}

\begin{proof}[Proof of \cref{lem:Closing lemma one scale}]
We prove the lemma for $\mathsf{B}_{r_H}^H$. The proof for $(\mathsf{B}_{r_H}^H)^{-1}$ is exactly the same. 

Let $\ref{c:linear algebra}$ be the constant coming from \cref{lem:Important Linear Algebra} and $c$ be the constant coming from the comparison between volume and height in \cref{cor:Volume and Height}. Let $\ref{a:long unipotent1} > 1$, $\ref{a:long unipotent2} > 1$ and $\ref{e:long unipotent}$ be as in \cref{thm:closing lemma long unipotent}. Let $\ref{m:closing lemma frostman one scale} = \ref{a:long unipotent2}(\ref{c:linear algebra}(\ref{c:linear algebra} + 1) + 1)/c$, $\ref{a:closing lemma frostman one scale} = c\ref{a:long unipotent2}$, and $\epsilon_2 = \frac{1}{2\ref{a:long unipotent1}}$. 

For initial point $x_1 \in X_{\eta}$, by reduction theory, we can write $x_1 = g_1 \Gamma$ where
\begin{align}\label{eqn:reduction initial point 1}
    \max\{\|g_1\|, \|g_1^{-1}\|\} \leq \eta^{-A}.
\end{align}
The constant $A$ depends only on $(G, \Gamma)$. 
Let $R \geq \ref{e:long unipotent}^{c\ref{a:long unipotent1}\ref{a:long unipotent2}}\eta^{-2c\ref{a:long unipotent1}\ref{a:long unipotent2}A}$. Let $\delta = R^{-\frac{1}{\ref{a:closing lemma frostman one scale}}} \leq \eta$ and let $D > 0$, $M = \ref{m:closing lemma frostman one scale} + \ref{c:closing lemma frostman one scale}D$, and $t = M \log R$. Let $\tilde{R} = \eta^{A} R^{\frac{1}{c}}$ and $S = \tilde{R}^{\frac{1}{\ref{a:long unipotent2}}}$. Note that $S \geq \ref{e:long unipotent}\eta^{-\ref{a:long unipotent1}}$. 

Suppose there exists $y \in X_{3\eta}$ and $r_H \leq \frac{1}{2}\min\{\inj(y), \eta_0\}$ so that 
\begin{align*}
\mu_t(\mathsf{B}_{r_H}^H\exp(B_{\delta}^{\mathfrak{r}}). y) = \nu_t \ast \delta_{x_1} (\mathsf{B}_{r_H}^H\exp(B_{\delta}^{\mathfrak{r}}). y) > \delta^{\epsilon_2}.
\end{align*}
Let
\begin{align*}
    \mathcal{E}_y = \{u \in \mathsf{B}_1^U: a_t u.x_1 \in \mathsf{B}_{r_H}^H \exp(B_{\delta}^{\mathfrak{r}}). y\},
\end{align*}
we have 
\begin{align*}
    |\mathcal{E}_y| > \delta^{\epsilon_2} = R^{-\frac{1}{2c\ref{a:long unipotent1}\ref{a:long unipotent2}}} \geq (\eta^{-A/\ref{a:long unipotent2}} R^{-\frac{1}{c\ref{a:long unipotent2}}})^{\frac{1}{\ref{a:long unipotent1}}} = S^{\frac{1}{\ref{a:long unipotent1}}}.
\end{align*}
Fix a $u_1 \in \mathcal{E}$ and let $\mathcal{E} = \mathcal{E}_y u_1^{-1} \subseteq \mathsf{B}_2^U$. For all $u, u' \in \mathcal{E}$, we have
\begin{align*}
    a_t uu_1.x_1 = h\exp(w).a_{t} u'u_1.x_1
\end{align*}
where $h \in \mathsf{B}_{2r_H}^H$ and $w \in B_{2\delta}^{\mathfrak{r}}$. 
Re-centering at $x_2 = a_{t}u_1.x_1$, we have
\begin{align*}
    a_{t} ua_{-t} x_2 = h\exp(w).a_{t} u'a_{-t} x_2.
\end{align*}
Note that $x_2 \in \mathsf{B}^H_{r_H}\exp(B_{\delta}^{\mathfrak{r}}).y \subseteq X_{2\eta}$. Write $x_2 = g_2 \Gamma$ where $\max\{\|g_2\|, \|g_2\|\} \leq \eta^{-A}$. There exists $\gamma_{u, u'} \in \Gamma$ so that 
\begin{align*}
    a_{t} ua_{-t} g_2 \gamma_{u, u'} = h\exp(w)a_{t} u'a_{-t} g_2.
\end{align*}
Therefore, 
\begin{align*}
    a_{t} ua_{-t} g_2 \gamma_{u, u'} g_2^{-1}a_{t} (u')^{-1}a_{-t} = h\exp(w).
\end{align*}

In summary, we have
\begin{enumerate}
    \item $|\mathcal{E}| > \delta^{\epsilon_2} \geq S^{\frac{1}{\ref{a:long unipotent1}}}$.
    \item For all $u, u' \in \mathcal{E}$, there exists $\gamma \in \Gamma$ so that 
    \begin{align*}
        \|a_{t} ua_{-t} g_2 \gamma_{u, u'} g_2^{-1}a_{t} (u')^{-1}a_{-t}\| = \|h\exp(w)\| \ll 1,\\
        \mathrm{d}(a_{t} ua_{-t} g_2 \gamma_{u, u'} g_2^{-1}a_{t} (u')^{-1}a_{-t}.\hat{\mathpzc{v}}_{\mathfrak{h}}, \hat{\mathpzc{v}}_\mathfrak{h}) \leq \delta = R^{-\frac{1}{c\ref{a:long unipotent2}}} \leq S^{-1}.
    \end{align*}
\end{enumerate}
Applying \cref{thm:closing lemma long unipotent} with $e^t \geq S = \eta^{A/\ref{a:long unipotent2}}R^{\frac{1}{c\ref{a:long unipotent2}}} \geq \ref{e:long unipotent}\eta^{-\ref{a:long unipotent1}}$, the base point $x_2 = g_2\Gamma \in X_{2\eta}$, there exists a non-trivial proper $\Q$-subgroup $\mathbf{M}$ so that 
\begin{align*}
    \sup_{u \in \mathsf{B}_2^U} \|a_{t} u a_{-t} g_2.\mathpzc{v}_M\| {}&\leq S^{\ref{a:long unipotent2}} = \eta^{A} R^{\frac{1}{c}},\\
    \sup_{\mathpzc{z} \in B_1^{\LieU}, u \in \mathsf{B}_2^U} \|\mathpzc{z} \wedge (a_{t} u a_{-t} g_2.\mathpzc{v}_M)\| {}&\leq e^{-\frac{t}{\ref{a:long unipotent2}}}S^{\ref{a:long unipotent2}} = e^{-\frac{t}{\ref{a:long unipotent2}}}\eta^{A} R^{\frac{1}{c}}.
\end{align*}

Since $x_2 = g_2\Gamma = a_tu_1g_1\Gamma$, there exists $\gamma \in \Gamma$ so that
\begin{align*}
    g_2 \gamma = a_t u_1 g_1. 
\end{align*}
Therefore, we have
\begin{align*}
    \sup_{u \in \mathsf{B}_2^U} \|a_{t} u u_1  g_1 \gamma.\mathpzc{v}_M\| \leq S^{\ref{a:long unipotent2}} = \eta^{A}R^{\frac{1}{c}} = \tilde{R}.
\end{align*}
Since $u_1 \in \mathsf{B}_1^U$, we have
\begin{align*}
    \sup_{u \in \mathsf{B}_1^U} \|a_{t} u g_1 \gamma.\mathpzc{v}_M\| \leq \tilde{R}.
\end{align*}
Similarly, we have
\begin{align*}
    \sup_{\mathpzc{z} \in B_1^{\LieU}, u \in \mathsf{B}_1^U} \|\mathpzc{z} \wedge (a_t u g_1 \gamma.\mathpzc{v}_M)\| \leq e^{-\frac{t}{\ref{a:long unipotent2}}} S^{\ref{a:long unipotent2}} = e^{-\frac{t}{\ref{a:long unipotent2}}} \tilde{R}.
\end{align*}

Note that $\mathbf{M}$ is a $\Q$-group and $\mathpzc{v}_M$ is a primitive integer vector in $\LieG_{\Z}$, we have $\|\mathpzc{v}_M\| \geq 1$. Combine it with \cref{eqn:reduction initial point 1}, we have
\begin{align*}
    \|g_1\gamma.\mathpzc{v}_M\| \geq \eta^{A}.
\end{align*}

We now apply \cref{lem:Important Linear Algebra} with $\tilde{t} = t$, $\tilde{R} = \eta^{A}R^{\frac{1}{c}}$, $\ref{a:long unipotent2}$, $g = g_1 \gamma$ and $\mathbf{M}(\R)^\circ$. Note that 
\begin{align*}
    {}&\tilde{R} = \eta^{A}R^{1/c} \geq \eta^{-A},\\
    {}&\tilde{t} = M\log R \geq \ref{m:closing lemma frostman one scale} \log R \geq \ref{a:long unipotent2}(\ref{c:linear algebra}(\ref{c:linear algebra} + 1) + 1) \log \tilde{R},
\end{align*}
the condition of the lemma is satisfied. There exists $g' \in G$ with $\|g' - \Id\|\leq \tilde{R}^{\ref{c:linear algebra}}e^{-\frac{1}{\ref{c:linear algebra}}\tilde{t}}$ so that 
\begin{align*}
    g'g_1\gamma \mathbf{M}(\R)^\circ \gamma^{-1} g_1^{-1}(g')^{-1} = H.
\end{align*}
Therefore, the orbit $Hg'g_1\Gamma$ is periodic. Moreover, 
\begin{align*}
    \|g' - \Id\| \leq \tilde{R}^{\ref{c:linear algebra}}e^{-\frac{1}{\ref{c:linear algebra}}t} \leq R^{-D}.
\end{align*}
Note that $g_1 \gamma \mathpzc{v}_M$ satisfies 
\begin{align*}
    \|g_1 \gamma \mathpzc{v}_M\|\ll \tilde{R} = \eta^{A}R^{\frac{1}{c}}.
\end{align*}
Combining with \cref{eqn:reduction initial point 1}, we have
\begin{align*}
    \|\gamma \mathpzc{v}_M\|\ll R^{\frac{1}{c}}.
\end{align*}
By \cref{cor:Volume and Height}, We have
\begin{align*}
    \vol(Hg'g_1 \Gamma) = \vol(Hg'g_1\gamma \Gamma) \ll \height(\gamma\mathbf{M}\gamma^{-1})^c \leq R.
\end{align*}
We get a contradiction to the initial Diophantine condition. This proves the lemma. 
\end{proof}

\section{Dimension estimate in many scales}\label{sec: Improve many scales}
This section is devoted to prove Lemma~\ref{lem:Closing lemma many scale}. We improve Lemma~\ref{lem:Closing lemma one scale} to obtain information for larger scales. The key ingredient is the following avoidance principle, \cref{pro:avoidance}. 

\subsection{Avoiding periodic orbits}
Let us recall Proposition~\ref{pro:avoidance}. It asserts that the trajectory $a_t \mathsf{B}_1^U.x_0$ is away from periodic orbits most of the time. 

\begin{proposition*}
    There exist $\mathsf{m}$, $s_0$, $\ref{a:avoidance}$, $\ref{c:avoidance}$, and $\ref{d:avoidance}$ depending only on $(G, H, \Gamma)$,  so that the following holds. Let $R_1, R_2 \geq 1$. Suppose $x_0 \in X$ is so that
    \begin{align*}
        d_X(x_0, x) \geq(\log R_2)^{\ref{d:avoidance}} R_2^{-1}
    \end{align*}
    for all $x$ with $\vol(H.x) \leq R_1$. Then for all $s \geq \ref{a:avoidance}\max\{\log R_2, |\log \inj(x_0)|\} + s_0$ and all $\eta \in (0, 1]$, we have
    \begin{align*}
        \Biggl|\Biggl\{u \in \mathsf{B}_1^U: \begin{aligned}{}&\inj(a_su.x_0) \leq \eta \text{ or }\exists x \text { with }\vol(H.x) \leq R_1\\ {}&\text { and } d_X(a_s u. x_0, x) \leq \ref{c:avoidance}^{-1} R_1^{-\ref{d:avoidance}}\end{aligned}\Biggr\}\Biggr| \leq \ref{c:avoidance} (R_1^{-1} + \eta^{\frac{1}{\mathsf{m}}}).
    \end{align*}
\end{proposition*}

\begin{proof}
    See \cite[Proposition 4.2, 4.4]{LMWY25} and \cite[Theorem 2]{SS24}. See also \cite[Corollary 7.2]{LMMS}. 
\end{proof}

\subsection{\texorpdfstring{F{\o}lner property for $U$}{F{\o}lner property for U}}
The following lemma allow us to view $\mu_{t_2 + t_1}$ as a $2$-step random walk. 
\begin{lemma}\label{lem:Folner1}
For all $A \subseteq X$, we have
\begin{align*}
|\mu_{t_2 + t_1}(A) - (\nu_{t_2} \ast \mu_{t_1})(A)| \ll e^{-t_1}.
\end{align*}
\end{lemma}
\begin{proof}
The proof is a standard application of F{\o}lner property of $U$ as the following: 
\begin{align*}
    {}&|\mu_{t_2 + t_1}(A) - (\nu_{t_2} \ast \mu_{t_1})(A)|\\
    \leq{}& \int_{\mathsf{B}_1^U}\Biggl|\int_{\mathsf{B}_1^U} \mathds{1}_A(a_{t_2 + t_1} u_1. x_1) - \mathds{1}_A(a_{t_2 + t_1} a_{-t_1} u_2 a_{t_1} u_1. x_1)\,\mathrm{d}u_1\Biggr|\,\mathrm{d}u_2\\
    \leq{}& \sup_{u_1 \in \mathsf{B}_1^U} |\mathsf{B}_1^U \triangle a_{-t_1} u_2^{-1} a_{t_1} \mathsf{B}_1^U| \ll e^{-t_1}.
\end{align*}    
\end{proof}

\subsection{Proof of Lemma~\ref{lem:Closing lemma many scale}}
We restate Lemma~\ref{lem:Closing lemma one scale} in the following form by explicitly writing the condition $R \gg \eta^{-\ref{e:closing lemma frostman one scale}}$ for reader's convenience. 
\begin{lemma}\label{lem:Closing lemma one scale restate}
    There exist constants $\ref{a:closing lemma frostman one scale} > 1$, $\ref{c:closing lemma frostman one scale} > 1$, $\ref{e:closing lemma frostman one scale} > 1$, $\ref{m:closing lemma frostman one scale} > 1$, $\epsilon_2 > 0$ and $R_0$ depending only on $(G, H, \Gamma)$ so that the following holds. For all $R \geq R_0$, let $\eta = R_0^{\frac{1}{\ref{e:closing lemma frostman one scale}}} R^{-\frac{1}{\ref{e:closing lemma frostman one scale}}}$. For all $D > 0$, $x_1 \in X_{\eta}$, let $M = \ref{m:closing lemma frostman one scale} + \ref{c:closing lemma frostman one scale} D$, $t = M\log R$, $\mu_t = \nu_{t} \ast \delta_{x_1}$ and $\delta = R^{-\frac{1}{\ref{a:closing lemma frostman one scale}}}$. 

    Suppose that for all periodic orbit $H.x'$ with $\vol(H.x') \leq R$, we have
    \begin{align*}
        d(x_1, x') > R^{-D}.
    \end{align*}
    Then for all $y \in X_{3\eta}$ and all $r_H \leq \frac{1}{2}\min\{\inj(y), \eta_0\}$, we have
    \begin{align*}
        \mu_t((\mathsf{B}_{r_H}^H)^{\pm 1}\exp(B_{\delta}^{\mathfrak{r}}).y) \leq \delta^{\epsilon_2}.
    \end{align*}
\end{lemma}

\begin{proof}[Proof of Lemma~\ref{lem:Closing lemma many scale}]
    We will prove the lemma for $\mathsf{B}_{r_H}^H$. The proof for $(\mathsf{B}_{r_H}^H)^{-1}$ is exactly the same. 
    
    Recall that $\lambda$ is the normalized Haar measure on $\mathsf{B}^{s, H}_{\beta + 100\beta^2}$. We have
    \begin{align*}
        (\lambda \ast \mu_t)(\mathsf{B}_{r_H}^H\exp(B_{r}^{\mathfrak{r}}).y) \leq{}& \int_{\mathsf{B}^{s, H}_{\beta + 100\beta^2}} \mu_t(\mathsf{h}^{-1}\mathsf{B}_{r_H}^H\exp(B_{r}^{\mathfrak{r}}).y) \,\mathrm{d}\lambda(\mathsf{h})\\
        \leq{}& \mu_t(\mathsf{B}_{2r_H}^H\exp(B_{r}^{\mathfrak{r}}).y).
    \end{align*}
    It suffices to show that for all $y \in X_{3\eta}$, $r_H \leq \frac{1}{2}\min\{\inj(y), \eta_0\}$ and $r \in [\delta_0, \eta]$, we have
    \begin{align}\label{eqn:estimate before shear}
        \mu_t(\mathsf{B}_{r_H}^H \exp(B_{r}^{\mathfrak{r}}).y) \ll \eta^{-\star} r^{\epsilon_1}.
    \end{align}
    The rest of the proof will be devoted to prove \cref{eqn:estimate before shear}. 

    Write $t = t_2 + t_1$ where $t_2$ and $t_1$ will be explicated later. By \cref{lem:Folner1}, 
    \begin{align*}
        |(\nu_{t_2} \ast \mu_{t_1})(\mathsf{B}_{r_H}^H \exp(B_{r}^{\mathfrak{r}}).y) - \mu_{t}(\mathsf{B}_{r_H}^H \exp(B_{r}^{\mathfrak{r}}).y)| \ll e^{-t_1}.
    \end{align*}
    It suffices to estimate $(\nu_{t_2} \ast \mu_{t_1})(\mathsf{B}_{r_H}^H \exp(B_{r}^{\mathfrak{r}}).y)$ if $t_1$ is large enough. We will explicate this range later. 
    
    Let $\ref{a:closing lemma frostman one scale}$, $\ref{c:closing lemma frostman one scale}$, $\ref{e:closing lemma frostman one scale}$, $\ref{m:closing lemma frostman one scale}$, $\epsilon_2$ and $R_0$ be as in \cref{lem:Closing lemma one scale restate}. Let $\mathsf{m}$, $\ref{a:avoidance}$, $\ref{d:avoidance}$, $\ref{c:avoidance}$, and $s_0$ be as in \cref{pro:avoidance}. Let $\epsilon_1 = \min\{\frac{1}{\ref{a:closing lemma frostman one scale}\mathsf{m} \ref{e:closing lemma frostman one scale}}, \epsilon_0\}$. This is a small constant depends only on $(G, H, \Gamma)$. Let $\ref{m:closinglemma frostman} = 2\ref{a:avoidance} + \ref{m:closing lemma frostman one scale} + 1$, $\ref{c:closinglemma frostman} = 2\ref{a:avoidance} + \ref{c:closing lemma frostman one scale}$, $\ref{d:closinglemma frostman} = \ref{d:avoidance} + 1$, $\ref{a:closinglemma frostman} = \ref{a:closing lemma frostman one scale}$. Suppose $R \geq \ref{c:avoidance}e^{s_0}R_0\eta^{-2\ref{e:closing lemma frostman one scale}}$. We have $\log R \geq 2|\log \eta| + s_0 \geq 2|\log \inj(x_1) | + s_0$. Suppose $R \gg_{\ref{d:avoidance}} 1$ so that $R^D \geq \ref{d:avoidance} \log(2D + 1)$ and $R^{\ref{d:avoidance}} \geq (\log R)^{\ref{d:avoidance}}$. 

    For all $\delta_0 = R^{-\frac{1}{\ref{a:closing lemma frostman one scale}}} \leq r \leq R_0^{-\frac{1}{\ref{a:closing lemma frostman one scale}}}\eta^{\frac{\ref{e:closing lemma frostman one scale}}{\ref{a:closing lemma frostman one scale}}}$, let $R_1 = r^{-\ref{a:closing lemma frostman one scale}} \geq R_0 \eta^{-\ref{e:closing lemma frostman one scale}}$. Let $R_2 = R^{2D + 1}$. Then for all $x$ with $H.x$ periodic and $\vol(H.x) \leq R_1 \leq R$, we have
    \begin{align*}
        d(x_1, x) \geq R^{-D} \geq (\log R_2)^{\ref{d:avoidance}} R_2^{-1}.
    \end{align*}

    For all $D \geq \ref{d:closinglemma frostman} + 1$, let $M = \ref{m:closinglemma frostman} + \ref{c:closinglemma frostman}D$ and $t = M \log R$. Let $t_2 = (\ref{m:closinglemma frostman} + \ref{c:closinglemma frostman}D)\log R_1$ and $t_1 = t - t_2$. We have
    \begin{align*}
        t_1 ={}& t - t_2\\
        \geq{}& (2\ref{a:avoidance}D + 2\ref{a:avoidance} + 1)\log R\\
        \geq{}& \ref{a:avoidance}\max\{\log R_2, \log |\inj(x_1)|\} + s_0.
    \end{align*}
    Apply \cref{pro:avoidance} to $x_1$, $R_1$, $R_2$, $\tilde{\eta} = R_0^{\frac{1}{\ref{e:closing lemma frostman one scale}}}R_1^{-\frac{1}{\ref{e:closing lemma frostman one scale}}}$ and $t_1$, we have
    \begin{align*}
        \Biggl|\Biggl\{u \in \mathsf{B}_1^U: \begin{aligned}{}&\inj(a_su.x_1) \leq R_0^{\frac{1}{\ref{e:closing lemma frostman one scale}}}R_1^{-\frac{1}{\ref{e:closing lemma frostman one scale}}} \text{ or }\exists x \text { with } H.x \text{ periodic}\\ {}&\text{and}\vol(H.x) \leq R_1\text { so that  } d_X(a_s u. x_1, x) \leq R_1^{-\ref{d:avoidance} - 1}\end{aligned}\Biggr\}\Biggr| \leq C (R_1^{-1} + R_0^{\frac{1}{\mathsf{m}\ref{e:closing lemma frostman one scale}}}R_1^{-\frac{1}{\mathsf{m}\ref{e:closing lemma frostman one scale}}}).
    \end{align*}
    Let 
    \begin{align*}
        X_1 = \{x \in X_{\tilde{\eta}}: \forall x' \text{ with } H.x \text{ periodic and }\vol(H.x) \leq R_1, d(x, x') > R_1^{-D}\}
    \end{align*}
    and $X_2 = X \setminus X_1$. The above inequality is equivalent to 
    \begin{align}\label{eqn:manyscale bad point}
        \mu_{t_1}(X_2) \leq C(R_1^{-1} + R_0^{\frac{1}{\mathsf{m}\ref{e:closing lemma frostman one scale}}}R_1^{-\frac{1}{\mathsf{m}\ref{e:closing lemma frostman one scale}}}).
    \end{align}
    Apply \cref{lem:Closing lemma one scale restate} to $x \in X_1$, $R_1$ and $t_2$ and note that $R_1^{-\frac{1}{\ref{a:closing lemma frostman one scale}}} = r$, for all $y \in X_{3\tilde{\eta}}$ and $r_H \leq \frac{1}{2}\min\{\inj(y), \eta_0\}$, we have
    \begin{align*}
        (\nu_{t_2} \ast \delta_{x})(\mathsf{B}_{r_H}^H \exp(B_{r}^{\mathfrak{r}}).y) \ll r^{\epsilon_2}.
    \end{align*}
    In particular, note that
    \begin{align*}
        \tilde{\eta} = R_0^{\frac{1}{\ref{e:closing lemma frostman one scale}}}R_1^{-\frac{1}{\ref{e:closing lemma frostman one scale}}} \leq \eta,
    \end{align*}
    we have
    \begin{align}\label{eqn:many scale good point}
        (\nu_{t_2} \ast \delta_{x})(\mathsf{B}_{r_H}^H \exp(B_{r}^{\mathfrak{r}}).y) \ll r^{\epsilon_2}
    \end{align}
    for all $y \in X_{3\eta}$ and $r_H \leq \frac{1}{2}\min\{\inj(y), \eta_0\}$. 

    Combine \cref{eqn:many scale good point,eqn:manyscale bad point}, for all $y \in X_{3\eta}$, $r_H \leq \frac{1}{2}\min\{\inj(y), \eta_0\}$ we have
    \begin{align*}
        (\nu_{t_2} \ast \mu_{t_1})(\mathsf{B}_{r_H}^H\exp(B_{r}^{\mathfrak{r}}).y) ={}& \int_X (\nu_{t_2} \ast \delta_x)(\mathsf{B}_{r_H}^H\exp(B_{r}^{\mathfrak{r}}).y) \, \mathrm{d}\mu_{t_1}(x)\\
        ={}& \int_{X_1} (\nu_{t_2} \ast \delta_x)(\mathsf{B}_{r_H}^H\exp(B_{r}^{\mathfrak{r}}).y) \, \mathrm{d}\mu_{t_1}(x)\\
        {}&+ \int_{X_2} (\nu_{t_2} \ast \delta_x)(\mathsf{B}_{r_H}^H\exp(B_{r}^{\mathfrak{r}}).y) \, \mathrm{d}\mu_{t_1}(x)\\
        \ll{}& r^{\epsilon_2} + \mu_{t_1}(X_2) \leq r^{\epsilon_2} + C(R_1^{-1} + R_0^{\frac{1}{\mathsf{m}\ref{e:closing lemma frostman one scale}}}R_1^{-\frac{1}{\mathsf{m}\ref{e:closing lemma frostman one scale}}}).
    \end{align*}
    Note that $C$, $R_0$ and $\ref{e:closing lemma frostman one scale}$ depends only on $(G, H, \Gamma)$, we have
    \begin{align*}
        (\nu_{t_2} \ast \mu_{t_1})(\mathsf{B}_{r_H}^H\exp(B_{r}^{\mathfrak{r}}).y) \ll r^{\epsilon_1}
    \end{align*}
    where $\epsilon_1 = \min\{\frac{1}{\mathsf{m}\ref{a:closing lemma frostman one scale}\ref{e:closing lemma frostman one scale}}, \epsilon_2\}$. 
    Since $t_1 \geq \log R$, we have
    \begin{align*}
        \mu_{t}(\mathsf{B}_{r_H}^H\exp(B_{r}^{\mathfrak{r}}).y) = \mu_{t_2 + t_1}(\mathsf{B}_{r_H}^H\exp(B_{r}^{\mathfrak{r}}).y) \ll r^{\epsilon_1} + R^{-1} \ll r^{\epsilon_1}.
    \end{align*}
    This proves \cref{eqn:estimate before shear} for all $r$ satisfying 
    \begin{align*}
        \delta_0 = R^{-\frac{1}{\ref{a:closing lemma frostman one scale}}} \leq r \leq R_0^{-\frac{1}{\ref{a:closing lemma frostman one scale}}}\eta^{\frac{\ref{e:closing lemma frostman one scale}}{\ref{a:closing lemma frostman one scale}}}.
    \end{align*}
    For all $r \geq R_0^{-\frac{1}{\ref{a:closing lemma frostman one scale}}}\eta^{\frac{\ref{e:closing lemma frostman one scale}}{\ref{a:closing lemma frostman one scale}}}$, we have
    \begin{align*}
        \mu_{t}(\mathsf{B}_{r_H}^H\exp(B_{r}^{\mathfrak{r}}).y) \leq 1 \leq R_0^{\frac{1}{\ref{a:closing lemma frostman one scale}}}\eta^{-\frac{\ref{e:closing lemma frostman one scale}}{\ref{a:closing lemma frostman one scale}}} r^{\epsilon_1}.
    \end{align*}
    Therefore, for all $y \in X_{3\eta}$, $r_H \leq \frac{1}{2}\min\{\inj(y), \eta_0\}$ and all $r \in [\delta_0, \eta]$, we have
    \begin{align*}
        \mu_{t}(\mathsf{B}_{r_H}^H\exp(B_{r}^{\mathfrak{r}}).y) \ll \eta^{-\star} r^{\epsilon_1}.
    \end{align*}
    This proves the lemma. 
\end{proof}

\section{Preparation II: Boxes, sheeted sets and admissible measures}\label{sec:boxes}
The deduction of \cref{thm:closing lemma initial dim} from \cref{lem:Closing lemma many scale} is straight forward. See the sketch in the introduction part of \cref{sec:Margulis function}. However, due to multiplicity of covering for $X$ and boundary effect of ball in $H$, the detail is lengthy and tedious. We collect needed results in \cite[Section 7]{LMW22} in this section and proceed the proof in the next section. The results there are stated for 
$H = \SL_2(\R)$, but their proof work in far more generality. In particular, in the case where $H\cong \SO(2, 2)^\circ$ or $H\cong \SO(3, 1)^\circ$, the expanding rate of $a_t$ on $\LieU$ is uniform. The proofs in \cite[Section 7]{LMW22} can be adapted easily. We will indicate the needed change in the proof. 

\subsection{Covering lemma}
Let
\begin{align*}
    \mathsf{Q}^H_{\eta, \beta^2, m} = \mathsf{B}_{e^{-m}\beta^2}^{U^-} \mathsf{B}_{\beta^2}^{M_0A} \mathsf{B}_{\eta}^{U^+}.
\end{align*}
and let
\begin{align*}
    \mathsf{Q}^G_{\eta, \beta^2, m} = \mathsf{Q}^H_{\eta, \beta^2, m} \cdot \exp(B_{2\beta^2}^{\mathfrak{r}}).
\end{align*}
For simplicity, we will denote them by $\mathsf{Q}^H_m$ and $\mathsf{Q}^G_m$ respectively. 

We also introduce the notion
\begin{align*}
    \Check{\mathsf{Q}}^H_m = (\mathsf{Q}^H_m)^{-1}
\end{align*}
and 
\begin{align*}
    \Check{\mathsf{Q}}^G_m = (\mathsf{Q}^H_m)^{-1} \exp(B_{2\beta^2}^{\mathfrak{r}}).
\end{align*}

\begin{lemma}[\text{\cite[Lemma 7.1]{LMW22}}]\label{lem: covering multiplicity}
    There exists $K \geq 1$ depends only on $X$ so that for all $m \geq 0$, there is a covering 
    \begin{align*}
        \{\mathsf{Q}^G_{\eta, \beta^2, m}.y_j: j \in \mathcal{J}_m, y_j \in X_{\frac{3}{2}\eta}\}
    \end{align*}
    of $X_{2\eta}$ with multiplicity $\leq K$. In particular, $\#\mathcal{J}_m \ll \eta^{-2}\beta^{-26}e^{2m}$. 
\end{lemma}

\begin{proof}
    The proof is exactly the same as in \cite[Lemma 7.1]{LMW22}. Note that $\dim \LieU^- = 2$, $\dim \LieM \oplus \LieA = 2$, $\dim \mathfrak{r} = 9$, and $\Ad(a_t) v = e^t v$ for all $v \in \LieU$. 
\end{proof}

Similarly, we have the following lemma. 
\begin{lemma}\label{lem: inverse covering initial}
    There exists $K \geq 1$ depends only on $X$ so that for all $m \geq 0$, there is a covering 
    \begin{align*}
        \{\Check{\mathsf{Q}}^G_{\eta, \beta^2, m}.y_j: j \in \Check{\mathcal{J}}_m, y_j \in X_{\frac{3}{2}\eta}\}
    \end{align*}
    of $X_{2\eta}$ with multiplicity $\leq K$. In particular, $\#\Check{\mathcal{J}}_m \ll \eta^{-2}\beta^{-26}e^{2m}$. 
\end{lemma}

From now in this paper, we fix such covers
\begin{align*}
    \{\mathsf{Q}^G_{\eta, \beta^2, m}.y_j: j \in \mathcal{J}_m, y_j \in X_{\frac{3}{2}\eta}\}
\end{align*}
as in \cref{lem: covering multiplicity} and also fix 
\begin{align*}
    \{\Check{\mathsf{Q}}^G_{0}.y_j: j \in \Check{\mathcal{J}}_0, y_j \in X_{\frac{3}{2}\eta}\}
\end{align*}
as in \cref{lem: inverse covering initial}. Let $\mathsf{k}_m(z) := \#\{j \in \mathcal{J}_m: z \in \mathsf{Q}^G_{m}.y_j\}$, then $1 \leq \mathsf{k}_m(z) \leq K$. Define $\rho_m(z) := \frac{1}{\mathsf{k}_m(z)}$ and 
\begin{align*}
    \rho_{m, j} = \rho_m \cdot \mathds{1}_{\mathsf{Q}^G_{m}.y_j},
\end{align*}
we have
\begin{align*}
    &1/K \leq \rho_{m, j} \leq 1\\
    &\sum_{j \in \mathcal{J}_m} \rho_{m, j}(z) = 1 \quad \forall z \in X_{2\eta}.
\end{align*}
Let $\Check{\mathsf{k}}_0(z) := \#\{j \in \Check{\mathcal{J}}_0: z \in \Check{\mathsf{Q}}^G_{0}.y_j\}$, then $1 \leq \Check{\mathsf{k}}_0(z) \leq K$. Define $\Check{\rho}_0(z) := \frac{1}{\Check{\mathsf{k}}_0(z)}$ and 
\begin{align*}
    \Check{\rho}_{0, j} = \Check{\rho}_0 \cdot \mathds{1}_{\Check{\mathsf{Q}}^G_{0}.y_j},
\end{align*}
we have
\begin{align*}
    &1/K \leq \Check{\rho}_{0, j} \leq 1\\
    &\sum_{j \in \Check{\mathcal{J}}_0} \Check{\rho}_{0, j}(z) = 1 \quad \forall z \in X_{2\eta}.
\end{align*}

\subsection{Boxes and complexity}
Let $\mathsf{prd}: \LieH \to H$ be the map defined by
\begin{align*}
    \mathsf{prd}: \LieH = \LieU^- \oplus \LieM_0 \oplus \LieA \oplus \LieU^+ {}&\to H\\
    (X_{\LieU^-}, X_{\LieM_0}, X_{\LieA}, X_{\LieU^+}) {}&\mapsto \exp(X_{\LieU^-})\exp(X_{\LieM_0})\exp(X_{\LieA})\exp(X_{\LieU^+}).
\end{align*}
A subset $\mathsf{D} \subseteq H$ is called a \textit{box} if there exist cubes $\mathsf{B}_{\LieU^-} \subset \LieU^-$, $\mathsf{B}_{\LieM_0} \subset \LieM_0$, $\mathsf{B}_{\LieA} \subset \LieA$, and $\mathsf{B}_{\LieU^+} \subset \LieU^+$ so that 
\begin{align*}
    \mathsf{D} = \mathsf{prd}(\mathsf{B}_{\LieU^-} \times \mathsf{B}_{\LieM_0} \times \mathsf{B}_{\LieA} \times \mathsf{B}_{\LieU^+}).
\end{align*}
\begin{example}
The set $\mathsf{Q}^H_m$ is a box. 
\end{example}
\begin{example}
Note that since we set $\|\cdot\| = \|\cdot\|_{\infty}$ on $\LieG = \mathfrak{sl}_4(\R)$, intersection of boxes is still a box. 
\end{example}

We say that a subset $\Xi \subset H$ has complexity bounded by $L$ (or at most $L$) if $\Xi$ can be written as union of at most $L$ boxes. 
We adapt the convention that the empty set is a box so that all sets of complexity at most $L$ can be written as $\Xi = \cup_{i = 1}^L \Xi_i$ where $\Xi_i$'s are boxes. 

For all ball $\mathsf{B}$ in $\LieU^-$, $\LieM_0$, $\LieA$ or $\LieU^+$, we define its (coarse) boundary to be
\begin{align*}
    \partial \mathsf{B} = \partial_{100\eta \diam(\mathsf{B})} \mathsf{B}.
\end{align*}
We define its (coarse) interior to be $\mathring{\mathsf{B}} = \mathsf{B} \setminus \partial \mathsf{B}$. For a box $\mathsf{D} = \mathsf{prd}(\mathsf{B}_{\LieU^-} \times \mathsf{B}_{\LieM_0} \times \mathsf{B}_{\LieA} \times \mathsf{B}_{\LieU^+})$, we define
\begin{align*}
    {}&\mathring{\mathsf{D}} = \mathsf{prd}(\mathring{\mathsf{B}}_{\LieU^-} \times \mathring{\mathsf{B}}_{\LieM_0} \times \mathring{\mathsf{B}}_{\LieA} \times \mathring{\mathsf{B}}_{\LieU^+}) \text{ and }\\
    {}&\partial \mathsf{D} = \mathsf{D} \setminus \mathring{\mathsf{D}}. 
\end{align*}

More generally, if $\mathsf{D} = \mathsf{prd}(\mathsf{B}_{\LieU^-} \times \mathsf{B}_{\LieM_0} \times \mathsf{B}_{\LieA} \times \mathsf{B}_{\LieU^+})$ is a box and $\Xi \subseteq \mathsf{D}$ has complexity bounded by $L$, we define 
\begin{align*}
    {}&\mathring{\Xi} := \bigcup_{i} \mathring{\Xi}_i \text{ and }\\
    {}&\partial \Xi = \bigcup_{i} \partial \Xi_i 
\end{align*}
where the union is taken over those $i$ with the following property. Writing $\Xi_i = \mathsf{prd}(\mathsf{B}_{\LieU^-, i} \times \mathsf{B}_{\LieM_0, i} \times \mathsf{B}_{\LieA, i} \times \mathsf{B}_{\LieU^+, i})$, we have 
\begin{align*}
    \diam \mathsf{B}_{\boldsymbol{\cdot}, i} \geq 100\eta \diam \mathsf{B}_{\boldsymbol{\cdot}}
\end{align*}
where $\boldsymbol{\cdot} = \LieU^-, \LieM_0, \LieA, \LieU^+$. 
\begin{lemma}[\text{\cite[Lemma 7.3]{LMW22}}]
    There exists $K'$ depending only on $X$ so that the following holds. Let $j \in \mathcal{J}_m$ and $w \in B_{2\beta^2}^{\mathfrak{r}}$. Then for all $1 \leq \mathsf{k} \leq K$, there exists $\Xi^{\mathsf{k}} = \Xi^{\mathsf{k}}(j, w) \subseteq \mathsf{Q}^H_m$ with complexity at most $K'$ so that 
    \begin{align*}
        {}&\rho_{m, j}(z) = 1/\mathsf{k} \text{ for all } z \in \Xi^{\mathsf{k}} \exp(w).y_j \text{ and }\\ 
        {}&|\{\mathsf{h} \in \mathsf{Q}^H_m: \rho_{m, j}(\mathsf{h}\exp(w)y_j) = 1/\mathsf{k}\} \setminus \Xi^{\mathsf{k}}| \ll \eta |\mathsf{Q}^H_m|
    \end{align*}
    where the implied constant depends only on $X$. 
\end{lemma}
\begin{proof}
    The proof is the same as in \cite{LMW22}. Note that \cite[Equation (7.9)]{LMW22} is a formula in the case $H = \SL_2(\R)$, but the proof of \cite[Lemma 7.3]{LMW22} only uses the fact that they are analytic functions. In general it follows from the fact that near identity, the map $\mathsf{prd}$ is a bi-analytic homeomorphism. 
\end{proof}

We also introduce the notion of inverse box. It is a similar notion to boxes in the coordinate $UM_0AU^-$. 
Let $\Check{\mathsf{prd}}: \LieH \to H$ be the map defined by
\begin{align*}
    \Check{\mathsf{prd}}: \LieH = \LieU^+ \oplus \LieM_0 \oplus \LieA \oplus \LieU^- {}&\to H\\
    (X_{\LieU^+}, X_{\LieM}, X_{\LieA}, X_{\LieU^-}) {}&\mapsto \exp(X_{\LieU^+})\exp(X_{\LieM_0})\exp(X_{\LieA})\exp(X_{\LieU^-}).
\end{align*}
A subset $\Check{\mathsf{D}} \subseteq H$ is called an \textit{inverse box} if there exist cubes $\mathsf{B}_{\LieU^+} \subset \LieU^+$, $\mathsf{B}_{\LieM_0} \subset \LieM_0$, $\mathsf{B}_{\LieA} \subset \LieA$, and $\mathsf{B}_{\LieU^+} \subset \LieU^-$ so that 
\begin{align*}
    \Check{\mathsf{D}} = \Check{\mathsf{prd}}(\mathsf{B}_{\LieU^+} \times \mathsf{B}_{\LieM_0} \times \mathsf{B}_{\LieA} \times \mathsf{B}_{\LieU^-}).
\end{align*}
\begin{example}
The set $\Check{\mathsf{Q}}^H_m$ is a box. 
\end{example}
\begin{example}
If $\mathsf{D}$ is a box, then $\mathsf{D}^{-1}$ is an inverse box. 
\end{example}
\begin{example}
Note that since we are using $\|\cdot\|_{\infty}$ on $\LieG = \mathfrak{sl}_4(\R)$, intersection of inverse boxes is still an inverse box. 
\end{example}

We say that a subset $\Check{\Xi} \subset H$ has inverse complexity bounded by $L$ (or at most $L$) if $\Check{\Xi}$ can be written as union of at most $L$ inverse boxes. 
We adapt the convention that the empty set is also an inverse box so that all sets of inverse complexity at most $L$ can be written as $\Check{\Xi} = \cup_{i = 1}^L \Check{\Xi}_i$  where $\Check{\Xi}_i$'s are boxes. 

Similarly, for an inverse box $\Check{\mathsf{D}} = \Check{\mathsf{prd}}(\mathsf{B}_{\LieU^+} \times \mathsf{B}_{\LieM_0} \times \mathsf{B}_{\LieA} \times \mathsf{B}_{\LieU^-})$, we define
\begin{align*}
    {}&\mathring{\Check{\mathsf{D}}} = \Check{\mathsf{prd}}(\mathring{\mathsf{B}}_{\LieU^+} \times \mathring{\mathsf{B}}_{\LieM_0} \times \mathring{\mathsf{B}}_{\LieA} \times \mathring{\mathsf{B}}_{\LieU^-}) \text{ and }\\
    {}&\partial \mathsf{D} = \mathsf{D} \setminus \mathring{\mathsf{D}}. 
\end{align*}

More generally, if $\Check{\mathsf{D}} = \Check{\mathsf{prd}}(\mathsf{B}_{\LieU^+} \times \mathsf{B}_{\LieM_0} \times \mathsf{B}_{\LieA} \times \mathsf{B}_{\LieU^-})$ is a box and $\Check{\Xi} \subseteq \Check{\mathsf{D}}$ has inverse complexity bounded by $L$, we define 
\begin{align*}
    {}&\mathring{\Check{\Xi}} := \bigcup_{i} \mathring{\Check{\Xi}}_i \text{ and }\\
    {}&\partial \Check{\Xi} = \bigcup_{i} \partial \Check{\Xi}_i 
\end{align*}
where the union is taken over those $i$ with the following property. Writing $\Check{\Xi}_i = \Check{\mathsf{prd}}(\mathsf{B}_{\LieU^+, i} \times \mathsf{B}_{\LieM_0, i} \times \mathsf{B}_{\LieA, i} \times \mathsf{B}_{\LieU^-, i})$, we have 
\begin{align*}
    \diam \mathsf{B}_{\boldsymbol{\cdot}, i} \geq 100\eta \diam \mathsf{B}_{\boldsymbol{\cdot}}
\end{align*}
where $\boldsymbol{\cdot} = \LieU^+, \LieM_0, \LieA, \LieU^-$.  

Similar to \cite[Lemma 7.3]{LMW22}, we have the following lemma. 
\begin{lemma}\label{lem: inverse complexity bound}
    There exists $K'$ depending only on $X$ so that the following holds. Let $j \in \mathcal{J}_0$ and $w \in B_{2\beta^2}^{\mathfrak{r}}$. Then for all $1 \leq \mathsf{k} \leq K$, there exists $\Check{\Xi}^{\mathsf{k}} = \Check{\Xi}^{\mathsf{k}}(j, w) \subseteq \Check{\mathsf{Q}}^H_m$ with complexity at most $K'$ so that 
    \begin{align*}
        {}&\Check{\rho}_{0, j}(z) = 1/\mathsf{k} \text{ for all } z \in \Check{\Xi}^{\mathsf{k}} \exp(w).y_j \text{ and }\\ 
        {}&|\{\mathsf{h} \in \Check{\mathsf{Q}}^H_0: \Check{\rho}_{0, j}(\mathsf{h}\exp(w)y_j) = 1/\mathsf{k}\} \setminus \Check{\Xi}^{\mathsf{k}}| \ll \eta |\Check{\mathsf{Q}}^H_0|
    \end{align*}
    where the implied constant depends only on $X$. 
\end{lemma}
\begin{proof}
    The proof is the same as the previous one. 
\end{proof}

\subsection{Sheeted set and admissible measure}\label{subsec:sheeted set admissible measure}
Recall that $\eta \leq \frac{1}{100C_0}\eta_0$ be a small parameter where $\eta_0$ and $C_0$ are from Lemma~\ref{lem:BCH}. Recall
\begin{align*}
    \mathsf{E} = \mathsf{B}_{\beta}^{U^-} \mathsf{B}_\beta^{M_0A} \mathsf{B}_\eta^{U^+}.
\end{align*}
Recall that a subset $\mathcal{E} \subseteq X$ is called a sheeted set if there exists a base point $y \in X_\eta$ and a finite set of transverse cross-section $F \subset B_{\eta}^{\mathfrak{r}}$ so that the map $(h, w) \mapsto h\exp(w).y$ is injective on $\mathsf{E} \times B_{\eta}^{\mathfrak{r}}$ and 
\begin{align*}
    \mathcal{E} = \bigsqcup_{w \in F}\mathsf{E}\exp(w).y.
\end{align*}

We now recall the definition of $\Lambda$-admissible measure in \cite[Appendix D]{LMWY25}. A probability measure $\mu_{\mathcal{E}}$ on $\mathcal{E}$ is called $\Lambda$-admissible if 
\begin{align*}
    \mu_\mathcal{E} = \frac{1}{\sum_{w \in F}\mu_w(X)} \sum_{w \in F} \mu_w
\end{align*}
where $\mu_w$ are measures on $\mathsf{E}\exp(w).y$ with the following properties. For all $w \in F$, there exists a function $\varrho_w$ defined on $\mathsf{E}$ with $\frac{1}{\Lambda} \leq \varrho_w \leq \Lambda$ so that for all $\mathsf{E}' \subseteq \mathsf{E}$, we have
\begin{align*}
    \mu_w(\mathsf{E}'\exp(w).y) = \int_{\mathsf{E}'} \varrho_w(\mathsf{h}) \,\mathrm{d}m_H(\mathsf{h}).
\end{align*}
Moreover, there exists $\mathsf{E}_w = \cup_{i = 1}^{\Lambda}\mathsf{E}_{w, i} \subseteq \mathsf{E}$ so that 
\begin{enumerate}
    \item $\mu_w\bigl((\mathsf{E}\setminus\mathsf{E}_{w, i})\exp(w).y\bigr) \leq \Lambda\eta \mu_w(X)$,
    \item the complexity of $\mathsf{E}_{w, i}$ is bounded by $\Lambda$ for all $i$, and
    \item $\Lip(\varrho_w|_{\mathsf{E}_{w, i}}) \leq \Lambda$.
\end{enumerate}

\section{Construction of sheeted sets}\label{sec:Margulis function}
This whole section is devoted to the proof of \cref{thm:closing lemma initial dim} from \cref{lem:Closing lemma many scale}. The idea is straight-forward, cf. \cite[Section 8]{LMW22}. We decompose the measure $\lambda \ast \mu_t$ into local pieces. Then \cref{lem:Closing lemma many scale} provides dimension estimate that can be translate into Margulis function estimate. Due to the difference in closing lemma (comparing \cref{lem:Closing lemma many scale} with \cite[Proposition 4.8]{LMW22}), there are two major differences comparing to \cite[Section 8]{LMW22}. 

In \cite[Proposition 4.8]{LMW22}, it is proved that the map $B_\beta^H a_t U_1 \to B_\beta^H a_t U_1 a_{8t}u_rx_1$ is injective for most $r \in [0,1]$. This ensures that the local pieces are roughly renormalized Haar measure on $H$-sheets and each $H$-sheet contribute roughly the same amount of measure. These are not guaranteed by \cref{lem:Closing lemma many scale}. Instead, locally $\lambda \ast \mu_t$ might not looks like a renormalized Haar measure. Moreover, locally $\lambda \ast \mu_t$ might assign different weight for each $H$-sheet. 

We resolve the problems in two steps. First, we decompose $\mu_t$ into local pieces of size $\beta^2$ in $U^-M_0A$-direction and then smear it using $\lambda$ which is of size $\beta + 100\beta^2$ in $U^-M_0A$-direction. This ensures that in size $\beta$ ball near the origin, it looks roughly like Haar measure and the boundary contributes only small error. Next, we decompose the measure once again according to the weight on each $H$-sheets. This ensures that we get admissible measures at the end. 

\subsection{Dimension, energy and Margulis function}\label{subsec:Margulis function}
For a finite set $F \subset \mathfrak{r}$, we set $\mu_F$ be the normalized counting measure on $F$. It is a probability measure. We say that the set $F$ has dimension $\geq \alpha$ for scales larger than $\delta$ if there exists $C > 1$ so that 
\begin{align*}
    \mu_F(B(x, r)) \leq Cr^\alpha \quad \forall x \in \mathfrak{r} \text{ and }r \geq \delta.
\end{align*}
In literatures, this is always denoted by $(C, \alpha)$-Frostman-type condition or $(C, \alpha)$-nonconcentration condition. We also define the (modified) $\alpha$-energy of $F$ as follows. 
\begin{align*}
    \mathcal{G}^{(\alpha)}_{F, \delta}(w) = \sum_{w' \in F: w' \neq w} \max\{\|w' - w\|, \delta\}^{-\alpha}.
\end{align*}

We recall the notion of (modified) Margulis function in \cite[Section 7]{LMWY25}. Suppose $\mathcal{E}$ is a sheeted set. For all $z \in \mathcal{E}$, let
\begin{align*}
    I_{\mathcal{E}}(z) = \{w \in \mathfrak{r}: \|w\| < \inj(z), \exp(w).z \in \mathcal{E}\}.
\end{align*}
For every $0 < \delta < 1$ and $0 < \alpha < \dim \mathfrak{r}$, we define the (modified) Margulis function as follows. 
\begin{align*}
    f^{(\alpha)}_{\mathcal{E}, \delta}(z) = \sum_{w \in I_\mathcal{E}(z) \setminus \{0\}} \max\{\|w\|, \delta\}^{-\alpha}.
\end{align*}

We have the following connection between those notions. 
\begin{proposition}\label{pro:frostman energy Margulis function}
    Suppose $F \subset B^{\mathfrak{r}}_1$ is a finite set and suppose $\mathcal{E} = \mathsf{E} \exp(F).y$ is a sheeted set. We have the following properties. 
    \begin{enumerate}
        \item Suppose $F$ is a set of dimension $\geq \alpha$ for scales larger than $\delta$, then for all $w \in F$ and $0 < \beta < \alpha$, 
        \begin{align*}
            \mathcal{G}_{F, \delta}^{(\beta)}(w) \leq 2^{\dim\mathfrak{r}}C\Bigl(1 + \frac{1}{1 - 2^{\beta - \alpha}}\Bigr) \#F.
        \end{align*}
        \item Suppose for all $w \in F$ we have
        \begin{align*}
            \mathcal{G}_{F, \delta}^{(\alpha)}(w) \leq C \#F,
        \end{align*}
        then for all $z \in \mathcal{E}$, we have
        \begin{align*}
            f_{\mathcal{E}, \delta}^{(\alpha)}(z) \ll C \#F.
        \end{align*}
        \item Let $\hat{\mathcal{E}} = (\overline{\mathsf{E} \setminus \partial_{5\beta^2}\mathsf{E}})\exp(F).y$. Suppose for all $z \in \mathcal{E}$, we have
        \begin{align*}
            f_{\mathcal{E}, \delta}^{(\alpha)}(z) \leq \Upsilon.
        \end{align*}
        Then for all $z \in \hat{\mathcal{E}}$ and all $w \in I_{\mathcal{E}}(z)$, we have
        \begin{align*}
            \mathcal{G}_{I_{\mathcal{E}}(z), \delta}^{(\alpha)}(w) \ll \Upsilon.
        \end{align*}
    \end{enumerate}
\end{proposition}
\begin{proof}
For property~(1), note that
\begin{align*}
    \mathcal{G}^{(\beta)}_{F, \delta}(w) ={}& \sum_{w' \in F: w' \neq w} \max\{\|w' - w\|, \delta\}^{-\beta}\\
    ={}& \sum_{k = 0}^{\lceil|\log \delta|\rceil} \sum_{2^{-k} \leq \|w' - w\| < 2^{-k + 1}} \max\{\|w' - w\|, \delta\}^{-\beta} + \delta^{-\beta}C\delta^{\alpha}\#F\\
    \leq{}& \sum_{k = 0}^{\lceil|\log \delta|\rceil} 2^\beta 2^{k\beta}C2^{-k\alpha}\#F + \delta^{-\beta}C\delta^{\alpha}\#F\\
    \leq{}& 2^{\dim\mathfrak{r}}C\Bigl(1 + \frac{1}{1 - 2^{\beta - \alpha}}\Bigr) \#F.
\end{align*}
For property~(2), we first show that the value of $f_{\mathcal{E}, \delta}^{(\alpha)}$ remains roughly the same on a single $H$-sheet. In particular, for all $z \in \mathcal{E}$ and $\mathsf{h} \in \mathsf{B}^{s, H}_{\beta}\mathsf{B}^U_{\eta}$ so that $\mathsf{h}.z \in \mathcal{E}$, we claim that
\begin{align*}
    2^{-\dim \mathfrak{r}} f_{\mathcal{E}, \delta}^{(\alpha)}(z) \leq f_{\mathcal{E}, \delta}^{(\alpha)}(\mathsf{h}.z) \leq 2^{\dim \mathfrak{r}} f_{\mathcal{E}, \delta}^{(\alpha)}(z).
\end{align*}
Indeed, note that $\|\Ad(\mathsf{h})\|_{\mathrm{op}} \leq 2$ for all $\mathsf{h} \in \mathsf{E}$, we have
\begin{align*}
    f_{\mathcal{E}, \delta}^{(\alpha)}(\mathsf{h}.z) ={}& \sum_{w \in I_{\mathcal{E}}(\mathsf{h}.z) \setminus \{0\}} \max\{\|w\|, \delta\}^{-\alpha}\\
    \leq{}& \sum_{w \in I_{\mathcal{E}}(z) \setminus \{0\}} \max\{\|\Ad(\mathsf{h})w\|, \delta\}^{-\alpha}\\
    \leq{}& 2^{\dim \mathfrak{r}} f_{\mathcal{E}, \delta}^{(\alpha)}(z).
\end{align*}
It now suffices to estimate $f_{\mathcal{E}, \delta}^{(\alpha)}(\exp(w_0).y)$ for $w_0 \in F$. For all $w \in I_{\mathcal{E}}(\exp(w_0).y)$, by Lemma~\ref{lem:BCH}, we have
\begin{align*}
    \exp(w)\exp(w_0) = \mathsf{h}_w\exp(w')
\end{align*}
where $\|\mathsf{h}_w - \Id\| \leq C_0\eta$ and $\|w' - w - w_0\| \leq C_0\|w_0\|\|w\|$. If $\eta$ is small enough, we have
\begin{align*}
    \frac{1}{2}\|w\| \leq \|w' - w_0\| \leq 2\|w\|.
\end{align*}
We also have $\mathsf{h}_w\exp(w').y \in \mathcal{E}$. Using the local injectivity, we have $w' \in I_{\mathcal{E}}(y) = F$ and also that the map $w \mapsto w'$ is injective. Therefore, 
\begin{align*}
    f_{\mathcal{E}, \delta}^{(\alpha)}(\exp(w_0).y) \leq \sum_{w \in I_{\mathcal{E}}(\exp(w_0).y)\setminus \{0\}} \max\{2\|w' - w_0\|, \delta\}^{-\alpha} \ll \mathcal{G}_{F, \delta}^{(\alpha)}(w_0).
\end{align*}
The proof for property~(2) is complete. Property~(3) follows directly from \cite[Lemma 7.1]{LMWY25}. 
\end{proof}

\subsection{Non-divergence result}
The following result assert that the trajectory is away from cusp most of the time. 

\begin{proposition}\label{pro:Non divergence}
    There exists $\mathsf{m} > 0$ depending only on $(G, H)$, $\kappa > 0$ and $C \geq 1$ depending only on $X$ with the following property. Let $0 < \delta, \eta < 1$ and let $B \subseteq \mathsf{B}_{10}^{U}$ be an open ball with radius $\geq \delta$. For all $x \in X$ and $t \geq \mathsf{m} |\log (\delta\inj(x))| + C$, we have
    \begin{align*}
        |\{u \in B: a_t u.x \notin X_{\eta} \}| \leq C\eta^{\frac{1}{\mathsf{m}}} |B|.
    \end{align*}
\end{proposition}
\begin{proof}
    It follows from \cite[Proposition 26, Theorem 16]{SS24} and Chebyshev inequality. See also \cite[Proposition 4.2]{LMWY25}. 
\end{proof}

\subsection{Proof of Theorem~\ref{thm:closing lemma initial dim}}We now proceed the proof. Let all parameter be as in Lemma~\ref{lem:Closing lemma many scale}. By Lemma~\ref{lem:Closing lemma many scale}, for all $y \in X_{3\eta}$, $r_H \leq \frac{1}{4}\min\{\inj(y), \eta_0\}$, $r \in [\delta_0, \eta]$, we have
\begin{align*}
    (\lambda \ast \mu_t)((\mathsf{B}_{r_H}^H)^{\pm 1}\exp(B_{r}^{\mathfrak{r}}).y) \ll \eta^{-\star}r^{\epsilon_1}.
\end{align*}

\subsubsection{\text{Boundary effect for $\nu_t$ and $\lambda$}}Due to the boundary effect of balls in $H$, we consider the (coarse) \textit{interior} of $\nu_t$ and $\lambda$. Recall that $\lambda$ is the normalized Haar measure on $\mathsf{B}^{s, H}_{\beta + 100 \beta^2}$. Let 
\begin{align*}
    \lambda_1 = \lambda|_{\mathsf{B}^{s, H}_{\beta - 100 \beta^2}}, \quad \mathring{\lambda} = \lambda|_{\mathsf{B}^{s, H}_{\beta}}
\end{align*}
and write
\begin{align*}
    \lambda = \lambda_1 + \lambda_2, \quad\lambda = \mathring{\lambda} + \partial\lambda.
\end{align*}

Recall that $\nu_t = a_t.m_{\mathsf{B}_1^U}$. Let $\nu_{t, 1}'$ be the restriction of $\nu_t$ to $a_t \mathsf{B}_{1 - e^{-t}}^U$. Note that for every $h \in \supp (\nu_{t, 1}')$, we have $\mathsf{B}_1^U h \in \supp (\nu_t)$. 

By \cref{pro:Non divergence} applying to $10\eta$ and $x_1 \in X_\eta$ and $B = \mathsf{B}_{1 - e^{-t}}^U$, we can decompose
\begin{align*}
    \nu_t = \nu_{t, 1} + \nu_{t, 2}
\end{align*}
where $\supp(\nu_{t, 1} \ast \delta_{x_1}) \subset X_{10\eta}$, for all $h \in \supp(\nu_{t, 1})$, we have $\mathsf{B}_1^U.h \subseteq \supp(\nu_t)$ and $\nu_{t, 2}(H) \ll \eta^\star$. 

Similarly, write $\nu_t = \mathring{\nu}_{t} + \partial\nu_{t}$ where $\supp(\mathring{\nu}_t \ast \delta_{x_1}) \subset X_{10\eta}$, for all $h \in \supp(\mathring{\nu}_{t})$, we have $\mathsf{B}_{1 - 100\eta}^U.h \subseteq \supp(\nu_t)$ and $\partial\nu_{t}(H) \ll \eta^\star$. 
Note that 
\begin{align*}
    \supp(\nu_{t, 1}) \subset \supp(\mathring{\nu}_t) \quad \supp(\lambda_{1}) \subset \supp(\mathring{\lambda}).
\end{align*}

\subsubsection{Decomposition of the space}
Recall that 
\begin{align*}
    \Check{\mathsf{Q}}^H_0 = \mathsf{B}_\eta^{U^+} \mathsf{B}_{\beta^2}^{M_0A} \mathsf{B}_{\beta^2}^{U^-}
\end{align*}
and
\begin{align*}
    \Check{\mathsf{Q}}^G_0 = \Check{\mathsf{Q}}^H_0\exp(B_{2\beta^2}^{\mathfrak{r}}).
\end{align*}
Recall that in \cref{lem: inverse covering initial}, there is a covering \begin{align*}
        \{\Check{\mathsf{Q}}^G_{0}.y_j: j \in \Check{\mathcal{J}}_0, y_j \in X_{5\eta}\}
\end{align*}
of $X_{10\eta}$ with multiplicity $\ll 1$. We fix this covering. 

For every $j \in \mathcal{J}_0$ and every $z \in \supp(\nu_{t, 1} \ast \delta_{x_1}) \cap \Check{\mathsf{Q}}^G_{0}.y_j$, we have that 
\begin{align*}
    z = \mathsf{u}\mathsf{ma} \mathsf{u^-} \exp(w).y_j
\end{align*}
for $\mathsf{u} \in \mathsf{B}_\eta^U$, $\mathsf{mau^-} \in \mathsf{B}_{\beta^2}^{s, H}$ and $w \in B_{2\beta^2}^{\mathfrak{r}}$. Note that 
\begin{align*}
    \mathsf{B}_{2\eta}^U.z \subset \supp(\mathring{\nu}_{t} \ast \delta_{x_1}),
\end{align*}
which implies
\begin{align*}
    \mathsf{B}_{\eta}^{U}\mathsf{ma} \mathsf{u^-} \exp(w).y_j \subseteq \supp(\mathring{\nu}_{t} \ast \delta_{x_1}) \cap \Check{\mathsf{Q}}^G_{0}.y_j.
\end{align*}
Therefore, for all $j \in \mathcal{J}_0$, we have a decomposition
\begin{align*}
    (\mu_t)|_{\mathsf{Q}^G_0.y_j} = \mu_j' + \sum_{i = 1}^{N_j} \sum_{k = 1}^{M_{j, i}} \bar{\mu}_{j, i, k}
\end{align*}
where for all $i, k$ there exist $w_i \in B_{2\beta^2}^{\mathfrak{r}}$ and $\mathsf{h}_{j, i, k}  \in \mathsf{B}^{s, H}_{\beta^2}$ so that 
\begin{align*}
    \bar{\mu}_{j, i, k} = (\mathring{\nu}_t \ast \delta_{x_1})|_{\mathsf{B}_{\eta}^U\mathsf{h}_{j, i, k} \exp(w_i).y_j}.
\end{align*}
We also have
\begin{align*}
    \mu_j'(X) \leq (\partial \nu_t \ast \delta_{x_1})(X) \leq \partial \nu_t(H) \ll \eta^{\star}.
\end{align*}

For all $j \in \mathcal{J}_0$, consider the set 
\begin{align*}
    \mathfrak{F}_j = \{(w_i, \mathsf{h}_{j, i, k} ): \bar{\mu}_{j, i, k} = (\mathring{\nu}_t \ast \delta_{x_1})|_{\mathsf{B}_{\eta}^U\mathsf{h}_{j, i, k} \exp(w_i).y_j}\}.
\end{align*}
\begin{lemma}\label{lem:number of sheet 1}
    We have 
    \begin{align*}
        \# \mathfrak{F}_j \ll \eta^{-2}e^{2t}.
    \end{align*}
\end{lemma}
\begin{proof}
    This is proved directly by volume counting. See \cite[Lemma 8.1]{LMW22}. 
\end{proof}
For all $j \in \mathcal{J}_0$, $1 \leq i \leq N_j$ and $1 \leq k \leq M_{j,i}$, define $d\mu_{j, i, k}(z) = \Check{\rho}_{0, j}(z)d\bar{\mu}_{j, i, k}(z)$. We have
\begin{align*}
    \mu_t = \mu' + \sum_{j \in \mathcal{J}_0} \sum_{i = 1}^{N_j} \sum_{k = 1}^{M_{i, k}} \mu_{j, i, k}
\end{align*}
where $\mu'(X) \ll \eta^\star$. Let 
\begin{align}
    \hat{c}_j = \sum_{i = 1}^{N_j} \sum_{k = 1}^{M_{i, k}} \mu_{j, i, k}(X).
\end{align}

\begin{lemma}\label{lem:throw away box with small weight}
    If $\hat{c}_j \geq \beta^{28}$, then $\#\mathfrak{F}_j \gg e^{2t} \beta^{27}$. Moreover, 
    \begin{align*}
        1 - \sum_{\hat{c}_j \geq \beta^{28}} \hat{c}_j = O(\beta).
    \end{align*}
\end{lemma}

\begin{proof}
    Recall that $\bar{\mu}_{j, i, k} = (\mathring{\nu}_t \ast \delta_{x_1})|_{\mathsf{B}_{\eta}^U\mathsf{h}_{j, i, k} \exp(w_i).y_j}$ and $d\mu_{j, i, k} = \Check{\rho}_{0, j}d\bar{\mu}_{j, i, k}$, we have
    \begin{align*}
        \hat{c}_j \asymp \#\mathfrak{F}_j \eta^2e^{-2t}.
    \end{align*}
    If $\hat{c}_j \geq \beta^{28}$, we have $\#\mathfrak{F}_j \gg \beta^{28} \eta^{-2} e^{2t} = e^{2t}\beta^{27}$. For the second statement, recall that $\#\mathcal{J}_0 \ll \eta^{-2}\beta^{-26}$, we have
    \begin{align*}
        \sum_{\hat{c}_j < \beta^{28}} \hat{c}_j \leq \hat{c}_j\#\mathcal{J}_0 \ll \beta.
    \end{align*}
\end{proof}

\subsubsection{\text{Smearing along the $H$-direction}}
We now smear along the $H$-direction. Recall that $\lambda$ is the normalized Haar measure on
\begin{align*}
    \mathsf{B}^{s, H}_{\beta + 100 \beta^2} = \mathsf{B}_{\beta + 100 \beta^2}^{U^-} \mathsf{B}_{\beta + 100 \beta^2}^{M_0A}.
\end{align*}
Let 
\begin{align*}
    \bar{\mu}_{j, i} = \sum_{k = 1}^{M_{j, i}} \bar{\mu}_{j, i, k}, \quad \bar{\mu}_{j} = \sum_{i = 1}^{N_j} \bar{\mu}_{j, i},
\end{align*}
and 
\begin{align*}
    \mu_{j, i} = \sum_{k = 1}^{M_{j, i}} \mu_{j, i, k}, \quad \mu_{j} = \sum_{i = 1}^{N_j} \mu_{j, i}.
\end{align*}
Recall that by definition, $\bar{\mu}_{j, i, k}$ is proportional to the push-forward of the Haar measure on $\mathsf{B}_\eta^{U}$ under $\mathsf{B}_\eta^{U} \to \mathsf{B}_\eta^{U}\mathsf{h}_{j, i, k} \exp(w_i).y_j$. Moreover, the factor is independent to $i$ and $k$. In fact, we have $\bar{\mu}_{j, i, k}(X) \asymp e^{-2t}\eta^2$. 

\begin{lemma}\label{lem:shear no multiplicity}
    Let $\bar{\mu}_{j, i, k}^U$ be the Haar measure on $\mathsf{B}_\eta^{U}\mathsf{h}_{j, i, k} $ with 
    \begin{align*}
        \bar{\mu}_{j, i, k}^U(H) = \bar{\mu}_{j, i, k}(X) \asymp e^{-2t}\eta^2.
    \end{align*}
    \begin{enumerate}
        \item For all $j, i$, there exists a function $\sigma_{j, i}$ so that
        \begin{align*}
            \mathrm{d}\Biggl(\lambda \ast \Biggl(\sum_{k = 1}^{M_{i, k}} \bar{\mu}_{j, i, k}^U\Biggr)\Biggr)(h) = \sigma_{j, i}(h) \,\mathrm{d} m_H(h)
        \end{align*}
        where 
        \begin{align*}
            0 \leq \sigma_{j, i}(h) \leq \frac{\bar{\mu}_{j, i}(X)}{m_{H}(\mathsf{B}^{s, H}_{\beta + 100\beta^2}\mathsf{B}^U_\eta)}.
        \end{align*}
        Moreover, for $h \in \mathsf{B}^{s, H}_{\beta}\mathsf{B}^U_{\eta - O(\beta^2\eta^2)}$, we have
        \begin{align*}
            \sigma_{j, i}(h) = \frac{\bar{\mu}_{j, i}(X)}{m_{H}(\mathsf{B}^{s, H}_{\beta + 100\beta^2}\mathsf{B}^U_\eta)}.
        \end{align*}
        \item We have
        \begin{align*}
            \Biggl(\lambda \ast \Biggl(\sum_{k = 1}^{M_{i, k}} \bar{\mu}_{j, i, k}^U\Biggr)\Biggr)\Bigl(H \setminus \mathsf{B}^{s, H}_{\beta} B_{\eta - O(\beta^2\eta^2)}^U\Bigr) \ll \eta \bar{\mu}_{j, i}(X).
        \end{align*}
    \end{enumerate}
    All implied constant depends only on $(G, H, \Gamma)$. 
\end{lemma}
\begin{proof}
    For all $\phi \in \mathrm{C}_c^{\infty}(H)$, we have
    \begin{align*}
        \int_H \phi \,\mathrm{d}(\lambda \ast \bar{\mu}^U_{j, i, k}) ={}& \frac{\bar{\mu}_{j, i, k}(X)}{m_U(\mathsf{B}^U_\eta) m_{U^-M_0A}(\mathsf{B}^{s, H}_{\beta + 100\beta^2})}\int_{\mathsf{B}^{s, H}_{\beta + 100\beta^2}} \int_{\mathsf{B}_\eta^U} \phi(\mathsf{h}u\mathsf{h}_{j, i, k} ) \,\mathrm{d}u \,\mathrm{d}m_{U^-M_0A}(\mathsf{h})\\
        ={}& \frac{\bar{\mu}_{j, i, k}(X)}{m_{H}(\mathsf{B}^{s, H}_{\beta + 100\beta^2}\mathsf{B}^U_\eta)} \int_{\mathsf{B}^{s, H}_{\beta + 100\beta^2}\mathsf{B}_\eta^U} \phi(h\mathsf{h}_{j, i, k} ) \,\mathrm{d}m_H(h)\\
        ={}& \frac{\bar{\mu}_{j, i, k}(X)}{m_{H}(\mathsf{B}^{s, H}_{\beta + 100\beta^2}\mathsf{B}^U_\eta)} \int_{\mathsf{B}^{s, H}_{\beta + 100\beta^2}\mathsf{B}_\eta^U \mathsf{h}_{j, i, k} } \phi(h) \,\mathrm{d}m_H(h).
    \end{align*}
    Therefore, we have
    \begin{align}\label{eqn:shear alon H calculation 1}
    \begin{aligned}
        {}&\int_H \phi \,\mathrm{d}\Biggl(\lambda \ast \Biggl(\sum_{k = 1}^{M_{i, k}} \bar{\mu}_{j, i, k}^U\Biggr)\Biggr)\\
        ={}& \frac{1}{m_{H}(\mathsf{B}^{s, H}_{\beta + 100\beta^2}\mathsf{B}^U_\eta)} \int_{H} \phi(h) \Biggl(\sum_{k = 1}^{M_{i, k}} \bar{\mu}_{j, i, k}(X) \mathds{1}_{\mathsf{B}^{s, H}_{\beta + 100\beta^2}\mathsf{B}_\eta^U \mathsf{h}_{j, i, k} }\Biggr) \,\mathrm{d}m_H(h).
    \end{aligned}
    \end{align}
    Let 
    \begin{align*}
        \sigma_{j, i} := \sum_{k = 1}^{M_{i, k}} \bar{\mu}_{j, i, k}(X) \mathds{1}_{\mathsf{B}^{s, H}_{\beta + 100\beta^2}\mathsf{B}_\eta^U \mathsf{h}_{j, i, k} }.
    \end{align*}

    Note that the following two maps are bi-analytic in $\eta_0$-neighborhood of $0$: 
    \begin{align*}
        \mathsf{prd}: \LieU^{-} \oplus \LieM \oplus \LieA \oplus \LieU^+ {}&\to H\\
        (X_{\LieU^-}, X_{\LieM}, X_{\LieA}, X_{\LieU^+}) {}&\mapsto \exp(X_{\LieU^-})\exp(X_{\LieM})\exp(X_{\LieA})\exp(X_{\LieU^+}),
    \end{align*}
    \begin{align*}
        \mathsf{prd}': \LieU^- \oplus \LieM \oplus \LieA \oplus \LieU^{+} {}&\to H\\
        (X_{\LieU^-}, X_{\LieM}, X_{\LieA}, X_{\LieU^+}) {}&\mapsto \exp(X_{\LieU^+})\exp(X_{\LieU^-})\exp(X_{\LieM})\exp(X_{\LieA}).
    \end{align*}
    Since $\mathsf{h}_{j, i, k}  \in \mathsf{B}^{s, H}_{\beta^2}$, we have
    \begin{align}\label{eqn:shear along H 2 support}
        \mathsf{B}^{s, H}_{\beta} \mathsf{B}_{\eta - O(\beta^2\eta^2)}^U \subseteq \mathsf{B}^{s, H}_{\beta + 100\beta^2} \mathsf{B}_\eta^U \mathsf{h}_{j, i, k} .
    \end{align}
    Combining \cref{eqn:shear alon H calculation 1,eqn:shear along H 2 support}, we prove property~(1). Property~(2) follows from a direct calculation. 
\end{proof}

The previous lemma implies the following. 
\begin{lemma}\label{lem:admissible complexity}
    The measure $\hat{\mu}_{j, i}$ satisfies the following properties. 
    \begin{enumerate}
        \item For all $\phi \in \mathrm{C}_c^\infty(X)$ we have
        \begin{align*}
            \int \phi(z) \,\mathrm{d}(\lambda \ast \mu_{j, i})(z) = \int_{\mathsf{E}} \phi(\mathsf{h}\exp(w_i).y_j)\Check{\rho}_{0, j}(z)\sigma_{j, i}(\mathsf{h}) \,\mathrm{d}m_H(\mathsf{h})
        \end{align*}
        where 
        \begin{align*}
            0 \leq \sigma_{j, i}(h) \leq \frac{\bar{\mu}_{j, i}(X)}{m_{H}(\mathsf{B}^{s, H}_{\beta + 100\beta^2}\mathsf{B}^U_\eta)}.
        \end{align*}
        \item For all $1 \leq \mathsf{k} \leq K$, there exists $\mathsf{E}^{\mathsf{k}} \subseteq \mathsf{E}$ with complexity $\ll 1$ so that 
        \begin{align*}
            {}&\Check{\rho}_{0, j}(z) = 1/\mathsf{k} \text{ for all } z \in \mathsf{E}^{\mathsf{k}} \exp(w_i).y_j, \\
            {}&\sigma_{j, i}(\mathsf{h}) = \frac{\bar{\mu}_{j, i}(X)}{m_{H}(\mathsf{B}^{s, H}_{\beta + 100\beta^2}\mathsf{B}^U_\eta)} \text{ for all } \mathsf{h} \in \mathsf{E}^k \text{ and }\\ 
            {}& (\lambda \ast \mu_{j, i})\Bigl(\Bigl\{z \in \mathsf{E}\exp(w_i).y_j: \Check{\rho}_{0, j}(z) = \frac{1}{\mathsf{k}}\Bigr\}\setminus \mathsf{E}^{\mathsf{k}}.\exp(w_i).y_j\Bigr) \ll \eta \mu_{j, i}(X).
        \end{align*}
    \end{enumerate}
    The implied constants depend only on $(G, H, \Gamma)$. 
\end{lemma}
\begin{proof}
    Property~(1) follows from the definition of $\mu_{j, i}$ and property~(1) of Lemma~\ref{lem:shear no multiplicity}. 

    For property~(2), let $\Check{\Xi}^{\mathsf{k}} = \Check{\Xi}^{\mathsf{k}}(j, w_i)$ be the subset of $\Check{\mathsf{Q}}^H_0$ with inverse complexity $K'$ from \cref{lem: inverse complexity bound}. The number $K'$ depends only on $(G, H, \Gamma)$. Write 
    \begin{align*}
        \Check{\Xi}^{\mathsf{k}} = \bigcup_{l = 1}^{K'} \Check{\Xi}^{\mathsf{k}}_l
    \end{align*}
    where each $\Check{\Xi}^{\mathsf{k}}_l$ is an inverse box. Let $\mathsf{B}^{s, H}_{\beta}\mathsf{B}^U_{\eta - C\beta^2\eta^2}$ be as in the property~(1) of Lemma~\ref{lem:shear no multiplicity} where $C$ is a constant depends only on $(G, H, \Gamma)$. Since $\Check{\Xi}_l^{\mathsf{k}}$ is an inverse box, there exists a cube $\mathsf{B}_{\LieU, l}^{\mathsf{k}} \subseteq B^\LieU_\eta$ so that for all $1 \leq k \leq M_{j, i}$, 
    \begin{align*}
        \Check{\Xi}_l^{\mathsf{k}} \cap \mathsf{B}^U_\eta \mathsf{h}_{i, k} = \exp(\mathsf{B}_{\LieU, l}^{\mathsf{k}})\mathsf{h}_{i, k}.
    \end{align*}

    \medskip
    \noindent\textit{Claim.} There exists a box $\mathsf{D}_l^{\mathsf{k}} \subseteq \mathsf{E}$ with the following two properties. 
    \begin{enumerate}
        \item For all $1 \leq k \leq M_{j, i}$, $\mathsf{D}_l^{\mathsf{k}} \subseteq \mathsf{B}^{s, H}_{\beta - 100\beta^2}\exp(\mathsf{B}_{\LieU, l}^{\mathsf{k}})\mathsf{h}_{i, k}$. 
        \item We have
        \begin{align*}
            m_H\Biggl(\bigcup_{i = 1}^{M_{j, i}}\mathsf{B}^{s, H}_{\beta - 100\beta^2}\exp(\mathsf{B}_{\LieU, l}^{\mathsf{k}})\mathsf{h}_{i, k} \setminus \mathsf{D}_l^{\mathsf{k}}\Biggr) \ll \eta m_H(\mathsf{E}).
        \end{align*}
    \end{enumerate}
    Indeed, the first property follows from the bi-analyticity of the map $\mathsf{prd}$ and $\mathsf{prd}'$ in as the following. Let $x_1$ be the center of the cube $\mathsf{B}_{\LieU, l}^{\mathsf{k}}$ and write $\mathsf{B}_{\LieU, l}^{\mathsf{k}} = \mathsf{B}^{\LieU}_r(x_1)$. The bi-analyticity implies that 
        \begin{align}\label{eqn:change box}
            \mathsf{B}^{s, H}_{\beta - 200\beta^2} \exp(\mathsf{B}^{\LieU}_{r - O(\beta^2r^2)}(x_1)) \subseteq \mathsf{B}^{s, H}_{\beta - 100\beta^2}\exp(\mathsf{B}_{\LieU, l}^{\mathsf{k}})\mathsf{h}_{i, k} \quad \forall 1 \leq k \leq M_{j, i}.
        \end{align}
        The second claim follows from a direct calculation. 

    Let 
    \begin{align*}
        \mathsf{E}^{\mathsf{k}}_i := \bigcup_{l = 1}^{K'} \biggl(\mathsf{D}^{\mathsf{k}}_l \cap \mathsf{B}^{s, H}_{\beta}\mathsf{B}^U_{\eta - C\beta^2\eta^2}\biggr).
    \end{align*}
    It is a subset of $\mathsf{E}$ with complexity $\ll 1$. By the construction of $\mathsf{D}^{\mathsf{k}}_l$ and $\mathsf{B}^{s, H}_{\beta}\mathsf{B}^U_{\eta - C\beta^2\eta^2}$, 
    \begin{align*}
        {}&\Check{\rho}_{0, j}(z) = 1/\mathsf{k} \text{ for all } z \in \mathsf{E}^{\mathsf{k}}_i \exp(w_i).y_j, \text{ and }\\
        {}&\sigma_{j, i}(\mathsf{h}) = \frac{\bar{\mu}_{j, i}(X)}{m_{H}(\mathsf{B}^{s, H}_{\beta + 100\beta^2}\mathsf{B}^U_\eta)} \text{ for all } \mathsf{h} \in \mathsf{E}^k_i.
    \end{align*}
    To prove the last estimate, note that by property~(1) and $\frac{1}{K} \leq \Check{\rho}_{0, j} \leq 1$, it suffices to show that
    \begin{align*}
        m_{H}\Bigl(\Bigl\{\mathsf{h} \in \mathsf{E}: \Check{\rho}_{0, j}(\mathsf{h}\exp(w_i).y_j) = \frac{1}{\mathsf{k}}\Bigr\}\setminus \mathsf{E}^{\mathsf{k}}_i\Bigr) \ll \eta m_H(\mathsf{E}).
    \end{align*}
    Note that
    \begin{align*}
        {}&m_H\Bigl(\Bigl\{\mathsf{h} \in \mathsf{E}: \Check{\rho}_{0, j}(\mathsf{h}\exp(w_i).y_j) = \frac{1}{\mathsf{k}}\Bigr\}\setminus \mathsf{E}^{\mathsf{k}}_i\Bigr)\\
        \leq{}& m_H\Bigl(\mathsf{B}^{s, H}_{\beta + 100 \beta^2} \cdot \Bigl\{\mathsf{h} \in \Check{\mathsf{Q}}^H_0: \Check{\rho}_{0, j}(\mathsf{h}\exp(w_i).y_j) = \frac{1}{\mathsf{k}}\Bigr\}\Bigr) \setminus \Bigl(\mathsf{B}^{s, H}_{\beta + 100 \beta^2} \cdot \Check{\Xi}^{\mathsf{k}}\Bigr)\Bigr)\\
        {}&+ \sum_{l = 1}^{K'}m_H\Bigl(\mathsf{B}^{s, H}_{\beta + 100\beta^2} \cdot \Check{\Xi}_l^{\mathsf{k}} \setminus \bigcup_{i = 1}^{M_{j, i}}\mathsf{B}^{s, H}_{\beta - 100\beta^2}\exp(\mathsf{B}_{\LieU, l}^{\mathsf{k}})\mathsf{h}_{i, k}\Bigr)\\
        {}&+ \sum_{l = 1}^{K'} m_H\Bigl(\bigcup_{i = 1}^{M_{j, i}}\mathsf{B}^{s, H}_{\beta - 100\beta^2}\exp(\mathsf{B}_{\LieU, l}^{\mathsf{k}})\mathsf{h}_{i, k} \setminus \mathsf{D}_l^{\mathsf{k}}\Bigr).
    \end{align*}
    The last term is estimated by property~(2) in the claim. It suffices to deal with the first two term. 
    
    We now deal with the second term. By definition, $\Check{\Xi}^{\mathsf{k}}_l \cap \mathsf{B}^U_\eta \mathsf{h} = \exp(\mathsf{B}_{\LieU, l}^{\mathsf{k}})$, similar to \cref{eqn:change box},  we have
    \begin{align*}
        m_H\Bigl(\mathsf{B}^{s, H}_{\beta + 100\beta^2} \cdot \Check{\Xi}_l^{\mathsf{k}} \setminus \bigcup_{i = 1}^{M_{j, i}}\mathsf{B}^{s, H}_{\beta - 100\beta^2}\exp(\mathsf{B}_{\LieU, l}^{\mathsf{k}})\mathsf{h}_{i, k}\Bigr) \ll \eta m_H(\mathsf{E}).
    \end{align*}
    
    For the first term, note that 
    \begin{align*}
        d(\lambda \ast m_H)(\mathsf{h}) = \hat{\rho}(\mathsf{h}) \,\mathrm{d}m_H(\mathsf{h})
    \end{align*}
    where $\hat{\rho}(\mathsf{h}) \asymp 1$ for all $\mathsf{h} \in \mathsf{B}^{s, H}_{\beta + 100 \beta^2}\mathsf{B}^U_\eta$, it suffices to show
    \begin{align*}
        {}&\lambda \ast m_H\Bigl(\mathsf{B}^{s, H}_{\beta + 200 \beta^2} \cdot \Bigl\{\mathsf{h} \in \Check{\mathsf{Q}}^H_0: \Check{\rho}_{0, j}(\mathsf{h}\exp(w_i).y_j) = \frac{1}{\mathsf{k}}\Bigr\}\Bigr) \setminus \Bigl(\mathsf{B}^{s, H}_{\beta + 100 \beta^2} \cdot \Check{\Xi}^{\mathsf{k}}\Bigr)\Bigr)\\
        \ll{}& \eta m_H(\mathsf{E}),
    \end{align*}
    which follows from \cref{lem: inverse complexity bound}.     
\end{proof}

\subsubsection{Decomposition of the local measure according to the weight on $H$-sheets}
Recall that for all $j \in \mathcal{J}_0$, we have a decomposition
\begin{align*}
    (\mu_t)|_{\mathsf{Q}^G_0.y_j} = \mu_j' + \sum_{i = 1}^{N_j} \sum_{k = 1}^{M_{j, i}} \bar{\mu}_{j, i, k}
\end{align*}
where for all $i, k$ there exists $w_i \in B_{2\beta^2}^{\mathfrak{r}}$ and $\mathsf{h}_{j, i, k}  \in \mathsf{B}^{s, H}_{\beta^2}$ so that 
\begin{align*}
    \bar{\mu}_{j, i, k} = (\mathring{\nu}_t \ast \delta_{x_1})|_{\mathsf{B}_{\eta}^U\mathsf{h}_{j, i, k} \exp(w_i).y_j}.
\end{align*}
We also have
\begin{align*}
    \mu_j'(X) \leq (\partial \nu_t \ast \delta_{x_1})(X) \leq \partial \nu_t(H) \ll \eta^{\star}.
\end{align*}
Recall that we set
\begin{align*}
    \bar{\mu}_{j, i} = \sum_{k = 1}^{M_{j, i}} \bar{\mu}_{j, i, k}
\end{align*}
and 
\begin{align*}
    d\mu_{j, i}(z) = \Check{\rho}_{0, j}(z)d\hat{\mu}_{j, i}(z)
\end{align*}
Note that by the dimension estimate in Lemma~\ref{lem:Closing lemma many scale}, for all $j \in \mathcal{J}_0$ and $1 \leq i \leq M_{j, i}$,
\begin{align*}
    \eta^2 e^{-2t} \ll \bar{\mu}_{j, i}(X) \ll \delta_0^{\epsilon_1}.
\end{align*}
Since $\frac{1}{K} \leq \Check{\rho}_{0, j} \leq 1$, there exists a large integer $L$ depending only on $(G, H, \Gamma)$ so that 
\begin{align}\label{eqn:weight on each sheet estimate 1}
    L^{-1}\eta^2 e^{-2t} \leq \mu_{j, i}(X) \leq L \delta_0^{\epsilon_1}.
\end{align}

For all $j \in \mathcal{J}_0$, let
\begin{align*}
    F_j = \{w_i: \bar{\mu}_{j, i, k} = (\mathring{\nu}_t \ast \delta_{x_1})|_{\mathsf{B}_{\eta}^U\mathsf{h}_{j, i, k} \exp(w_i).y_j} \forall k\}.
\end{align*}
By Lemma~\ref{lem:number of sheet 1}, we have
\begin{align*}
    \#F_j \leq \mathfrak{F}_j \ll \eta^{-2}e^{2t}.
\end{align*}

Let $\mathsf{L}$ be an integer so that $\mathsf{L} > L$ and also takes care of all constants in Lemma~\ref{lem:admissible complexity}. Note that $\mathsf{L}$ depends only on $(G, H, \Gamma)$. We now decompose the measure according to its weight on each sheet. 
For all integer $m \geq 0$, let
\begin{align*}
    F_{j, m} = \{w_i \in F_j: \mathsf{L}^{-m}\delta_0^{\epsilon_1} \leq \mu_{j, i}(X) < \mathsf{L}^{-m + 1}\delta_0^{\epsilon_1}\}.
\end{align*}
Since $\mu_{j, i}(X) \geq \mathsf{L}^{-1}\eta^2e^{-2t}$, the set $F_{j, m} = \emptyset$ for all $m > \lceil2t/\log(\mathsf{L})\rceil$. From now on we only consider $F_{j, m}$ for $1 \leq m \leq \lceil2t/\log(\mathsf{L})\rceil$ and $j \in \mathcal{J}_0$ with 
\begin{align}\label{eqn:condition on weight of box}
    \hat{c}_j = \sum_{i = 1}^{N_j} \sum_{k = 1}^{M_{i, k}} \mu_{j, i, k}(X) \geq \beta^{28}.
\end{align}
Denote the set consists of such index $j$ by $\mathcal{J}_0'$

For all $1 \leq m \leq \lceil2t/\log(\mathsf{L})\rceil$ so that 
\begin{align}\label{eqn:condition on weight of sheets}
    \sum_{i:w_i \in F_{j, m}} \mu_{j, i}(X) \geq \beta\hat{c}_j \geq \beta^{29}, 
\end{align}
we have
\begin{align}\label{eqn:number of sheet lower bound}
    \#F_{j, m} \geq \beta^{29}\delta_0^{-\epsilon_1}.
\end{align}
Denote the set consists of index $m$ satisfying \cref{eqn:condition on weight of sheets} by $\mathcal{M}_j'$

Let 
\begin{align*}
    \hat{c}_{j, m} = \sum_{i:w_i \in F_{j, m}} \mu_{j, i}(X), 
\end{align*}
and
\begin{align*}
    c_{j, m} = \Biggl(\sum_{j \in \mathcal{J}_0'} \sum_{m \in \mathcal{M}_j'} \hat{c}_{j, m}\Biggr)^{-1} \hat{c}_{j, m}.
\end{align*}

\begin{lemma}
    We have 
    \begin{align*}
        \sum_{j \in \mathcal{J}_0'} \sum_{m \in \mathcal{M}_j'} \hat{c}_{j, m} \geq 1 - O(\beta^\star).
    \end{align*}
\end{lemma}
\begin{proof}
    Recall that we take $j, m$ so that they satisfy the following properties:
    \begin{align*}
        \hat{c}_{j, m} \geq \beta \hat{c}_j, \quad \hat{c}_j \geq \beta^{28}.
    \end{align*}
    Therefore, 
    \begin{align*}
        \sum_{j \in \mathcal{J}_0'} \sum_{m \notin \mathcal{M}_j'} \hat{c}_{j, m} \leq \sum_{j \in \mathcal{J}_0} \sum_{1 \leq m \leq \lceil2t/\log \mathsf{L}\rceil:m \notin \mathcal{M}_j'} \beta\hat{c}_{j} \leq \lceil2t/\log \mathsf{L}\rceil\beta \leq \beta^{\frac{1}{2}}.
    \end{align*}
    By \cref{lem:throw away box with small weight}, we also have
    \begin{align*}
        \sum_{j \notin \mathcal{J}_0} \sum_{1 \leq m \leq \lceil2t/\log \mathsf{L}\rceil} \hat{c}_{j, m} = \sum_{j \notin \mathcal{J}_0} \hat{c}_j = O(\beta). 
    \end{align*}
    Combine both estimates, we prove the lemma. 
\end{proof}

For $j \in \mathcal{J}_0'$ and $m \in \mathcal{M}_j'$, we set
\begin{align*}
    \mathcal{E}_{j, m} = \mathsf{E}\exp(F_{j, m}).y_j.
\end{align*}

\begin{lemma}
    For all $j \in \mathcal{J}_0'$ and $m \in \mathcal{M}_j'$, there exists a $\mathsf{L}$-admissible measure $\mu_{\mathcal{E}_{j, m}}$ so that for all $\phi \in \mathrm{C}_c(X)$, 
    \begin{align*}
        \Biggl|\int_X \phi \,\mathrm{d}\mu_{\mathcal{E}_{j, m}} - \int_X \phi \,\mathrm{d}\Biggl(\lambda \ast \Biggl(\sum_{i: w_i \in F_{j, m}}\mu_{j, i}\Biggr)\Biggr)\Biggr| \ll \hat{c}_{j, m}\|\phi\|_\infty\eta^\star
    \end{align*}
\end{lemma}

\begin{proof}
    Let $\mathsf{B}^{s, H}_{\beta}\mathsf{B}^U_{\eta - C\beta^2\eta^2}$ be as in the property~(1) of Lemma~\ref{lem:shear no multiplicity} where $C$ is a constant depends only on $(G, H, \Gamma)$. Let $\mu_{\mathcal{E}_{j, m}}$ be the restriction of 
    \begin{align*}
        \lambda \ast \Biggl(\sum_{i: w_i \in F_{j, m}}\mu_{j, i}\Biggr)
    \end{align*}
    to $\mathsf{B}^{s, H}_{\beta}\mathsf{B}^U_{\eta - C\beta^2\eta^2}\exp(F_{j, m}).y_j$ and normalized to probability measure. By Lemma~\ref{lem:admissible complexity}, we have
    \begin{align*}
        \Biggl|\int_X \phi \,\mathrm{d}\mu_{\mathcal{E}_{j, m}} - \int_X \phi \,\mathrm{d}\Biggl(\lambda \ast \Biggl(\sum_{i: w_i \in F_{j, m}}\mu_{j, i}\Biggr)\Biggr)\Biggr| \ll \hat{c}_{j, m}\|\phi\|_\infty\eta^\star.
    \end{align*}
    It suffices to show that $\mu_{\mathcal{E}_{j, m}}$ is $\mathsf{L}$-admissible. 
    
    For all $w = w_i \in F_{j, m}$, let 
    \begin{align*}
        \mu_{w} := \frac{m_H(\mathsf{B}^{s, H}_{\beta + 100 \beta^2}\mathsf{B}^U_\eta)}{\mathsf{L}^{-m}\delta_0^{\epsilon_1}} \lambda \ast \mu_{j, i}|_{\mathsf{B}^{s, H}_{\beta}\mathsf{B}^U_{\eta - C\beta^2\eta^2}\exp(w).y_j},
    \end{align*}
    we have
    \begin{align*}
        \mu_{\mathcal{E}_{j, m}} = \frac{1}{\sum_{w \in F_{j, m}}\mu_w(X)}\sum_{w \in F_{j, m}} \mu_w.
    \end{align*}
    By Lemma~\ref{lem:admissible complexity}~(1), we have
    \begin{align*}
        \mathrm{d}\mu_{w_i}(z) = \frac{m_H(\mathsf{B}^{s, H}_{\beta + 100 \beta^2}\mathsf{B}^U_\eta)}{\mathsf{L}^{-m}\delta_0^{\epsilon_1}}\Check{\rho}_{0, j}(z)\sigma_{j, i}(\mathsf{h}) \, \mathrm{d}m_H(\mathsf{h})
    \end{align*}
    where $z = \mathsf{h}\exp(w_i).y_j$. Moreover, we have
    \begin{align*}
        \mathsf{L}^{-1} \leq \frac{m_H(\mathsf{B}^{s, H}_{\beta + 100 \beta^2}\mathsf{B}^U_\eta)}{\mathsf{L}^{-m}\delta_0^{\epsilon_1}}\Check{\rho}_{0, j}(z)\sigma_{j, i}(\mathsf{h}) \leq \mathsf{L}
    \end{align*}
    for all $\mathsf{h} \in \mathsf{B}^{s, H}_{\beta}\mathsf{B}^U_{\eta - C\beta^2\eta^2}$    
    
    Let $\mathsf{E}^{\mathsf{k}}$ be as in Lemma~\ref{lem:admissible complexity}~(2). It has complexity $\ll 1$ and the function $\Check{\rho}_{0, j}\sigma_{j, i}$ is constant on $\mathsf{E}^{\mathsf{k}}$. This proves the remaining properties of $\mathsf{L}$-admissible measure. 
\end{proof}

For $j \in \mathcal{J}_0'$ and $m \in \mathcal{M}_j'$, let
\begin{align*}
    c_{\mathcal{E}_{j, m}} = c_{j, m}.
\end{align*}
From now on, to reduce complicated subscript, we will drop $j,m$ in the subscript. The sum $\sum_{\mathcal{E}}$ will be the same as $\sum_{j \in \mathcal{J}_0'} \sum_{m \in \mathcal{M}_j'}$. 

The above lemmas provides a decomposition
\begin{align*}
    \lambda \ast \mu_t = \mu'' + \sum_{\mathcal{E}} c_{\mathcal{E}} \mu_{\mathcal{E}}
\end{align*}
with $\mu''(X) \ll \eta^\star$. 

Therefore, for all $d \geq 0$ and $u' \in \mathsf{B}_1^U$, we have
\begin{align*}
    \int_{X} \phi(a_d u' x) \,\mathrm{d}(\lambda \ast \mu_t) = \sum_{\mathcal{E}} c_{\mathcal{E}} \int_{X} \phi(a_d u' x) \,\mathrm{d}\mu_{\mathcal{E}} + O(\|\phi\|_\infty \eta^\star).
\end{align*}
This proves property~(1) in Theorem~\ref{thm:closing lemma initial dim}. 

Let $\epsilon_0 = \epsilon_1/2$. We show Theorem~\ref{thm:closing lemma initial dim} property~(2) holds for this $\epsilon_0$. 
\begin{lemma}
    For all $j$ and $m$ satisfying \cref{eqn:condition on weight of box,eqn:condition on weight of sheets}, Write $\mathcal{E} = \mathcal{E}_{j, m} = \mathsf{E}\exp(F_{j, m}).y_j$ and $F = F_{j, m}$. It satisfies the following conditions. 
    \begin{enumerate}
        \item The number of sheets satisfies
        \begin{align*}
            \beta^{29}\delta_0^{-2\epsilon_0} \leq \#F \leq \beta^{-2}e^{2t}.
        \end{align*}
        \item We have the Margulis function estimate
        \begin{align*}
            f^{(\epsilon_0)}_{\mathcal{E}, \delta_0}(x) \ll \beta^{-\star} \#F \quad \forall x \in \mathcal{E}.
        \end{align*}
    \end{enumerate}
\end{lemma}

\begin{proof}
    Property~(1) follows from \cref{lem:number of sheet 1,eqn:number of sheet lower bound}. 

    For property~(2), by \cref{pro:frostman energy Margulis function} and $F \subset B_\eta^{\mathfrak{r}}$, it suffices to show that $\mu_{F}$ satisfies
    \begin{align}\label{eqn:final set frostman}
        \mu_F(B_r^{\mathfrak{r}}(w)) \ll \beta^{-\star}r^{\epsilon_1} \quad \forall w \in \mathfrak{r} \text{ and } \delta_0 \leq r \leq \eta.
    \end{align}
    Indeed, if \cref{eqn:final set frostman} is satisfied, applying \cref{pro:frostman energy Margulis function}~(1) with $\epsilon_0 = \epsilon_1/2$ and then applying \cref{pro:frostman energy Margulis function}~(2), we prove the Margulis function estimate. 
    Moreover, it suffices to show \cref{eqn:final set frostman} holds for all $w \in F$. 

    Recall that 
    \begin{align*}
        F = F_{j, m} = \{w_i \in F_j: \mathsf{L}^{-m}\delta_0^{\epsilon_1} \leq \mu_{j, i}(X) < \mathsf{L}^{-m + 1}\delta_0^{\epsilon_1}\}.
    \end{align*}
    For all $w_i \in F_{j, m}$, we have
    \begin{align*}
        \mu_{F_{j, m}}(B_r^{\mathfrak{r}}(w_i)) ={}& \frac{\#\{w_{i'} \in F_{j, m}: \|w_{i'} - w_i\| \leq r\}}{\#F_{j, m}}\\
        \leq{}& \frac{\mathsf{L}^{m}\delta_0^{-\epsilon_1}\sum_{i':w_{i'} \in F_{j, m}, \|w_{i'} - w_i\| \leq r} \mu_{j, i'}(X)}{\mathsf{L}^{m - 1}\delta_0^{-\epsilon_1}\sum_{i':w_{i'} \in F_{j, m}}\mu_{j, i'}(X)}\\
        \leq{}& \mathsf{L}\beta^{-29}\sum_{i':w_{i'} \in F_{j, m}, \|w_{i'} - w_i\| \leq r} \mu_{j, i'}(X).
    \end{align*}
    The last inequality follows from \cref{eqn:condition on weight of sheets}. 

    Recall that $d\mu_{j, i'}(z) = \Check{\rho}_{0, j}(z)d\bar{\mu}_{j, i'}(z)$ where $\Check{\rho}_{0, j} \leq 1$, we have
    \begin{align*}
        \mu_{F_{j, m}}(B_r^{\mathfrak{r}}(w_i)) \leq \mathsf{L} \beta^{-29} \sum_{i':w_{i'} \in F_{j, m}, \|w_{i'} - w_i\| \leq r} \bar{\mu}_{j, i'}(X).
    \end{align*}
    Since $\bar{\mu}_{j, i'} = \mu_t|_{\Check{\mathsf{Q}}^H_0\exp(w_i).y_j}$, we have
    \begin{align*}
        \mu_{F_{j, m}}(B_r^{\mathfrak{r}}(w_i)) \leq \mathsf{L} \beta^{-29} \mu_t(\Check{\mathsf{Q}}^H_0 \exp(B_{r}^{\mathfrak{r}}(w_i)).y_j).
    \end{align*}

    Using \cref{lem:BCH}, for all $w \in B_{r}^{\mathfrak{r}}(w_i)$, we have
    \begin{align*}
        \exp(w).y_j = \exp(w)\exp(-w_i)\exp(w_i).y_j = \mathsf{h}\exp(\bar{w})\exp(w_i).y_j
    \end{align*}
    where $\|\bar{w}\| \leq 2 \|w - w_i\| \leq 2r$, $\|\mathsf{h} - \Id\| \leq C_0 \eta$. Therefore, 
    \begin{align*}
        \mu_{F_{j, m}}(B_r^{\mathfrak{r}}(w_i)) \leq \mathsf{L} \beta^{-29} \mu_t(\mathsf{B}_{C\eta}^H \exp(B_{2r}^{\mathfrak{r}}(0))\exp(w_i).y_j) \ll \beta^{-\star} r^{\epsilon_1}.
    \end{align*}
    The last inequality follows from \cref{lem:Closing lemma many scale} and $100C_0\eta \leq \eta_0$.  
\end{proof}

\part{Dimension improvement in the transverse complement}\label{part:projection}

The main result of this part is Theorem~\ref{thm:energy Improvement}. It is a linear dimension improvement result in the representations $\mathfrak{r}_1$ and $\mathfrak{r}_2$ of $H_1$ and $H_2$ respectively. It is an analog of \cite[Theorem 6.1]{LMWY25}. We first fix some notations. 

Recall $G = \SL_4(\R)$ and $\LieG = \mathrm{Lie}(G)$. For $v \in \LieG$ and $g \in G$, we write $g.v = \Ad(g)v$. 

Recall that $H_1$ preserves the quadratic form $Q_1(x_1, x_2, x_3, x_4) = x_2 x_3 - x_1 x_4$ and $H_1 \cong \SO(2, 2)^\circ$. Recall that $H_2$ preserves the quadratic form $Q_2(x_1, x_2, x_3, x_4) = x_2^2 +  x_3^2 - 2x_1 x_4$ and $H_2 \cong \SO(3, 1)^\circ$. For both $\LieH_1$ and $\LieH_2$, there exist unique $\Ad(H_i)$-invariant complements $\mathfrak{r}_1$ and $\mathfrak{r}_2$ of $\LieH_1$ and $\LieH_2$ respectively in $\LieG$. Moreover, they are $9$-dimensional irreducible representations for $H_i$ correspondingly. 

If a definition/result/proof in this part can be state simultaneously to $H_1$ and $H_2$ respectively, we drop the subscripts and denote them by $Q$, $H$ and $\mathfrak{r}$. 

Recall that both $H_1$ and $H_2$ contain the following one-parameter diagonal subgroup:
\begin{align*}
    a_t = \begin{pmatrix}
        e^t & & & \\
         & 1 & & \\
         & & 1 & \\
         & & & e^{-t}
    \end{pmatrix}.
\end{align*}
The corresponding horospherical subgroups $U_1 \leq H_1$ and $U_2 \leq H_2$ consists of the following elements respectively:
\begin{align*}
    u_{r, s}^{(1)} = \begin{pmatrix}
        1 & r & s & sr\\
         & 1 & & s\\
         & & 1 & r\\
         & & & 1
    \end{pmatrix}, \qquad u_{r, s}^{(2)} = \begin{pmatrix}
        1 & r & s & \frac{r^2 + s^2}{2}\\
         & 1 & & r\\
         & & 1 & s\\
         & & & 1
    \end{pmatrix}.
\end{align*}
As before, if a definition/statement/proof can be formulated simultaneously to $U_1$ and $U_2$, we drop the subscripts for $U$ and superscripts for $u_{r, s}$ for simplicity. When the explicit parametrization is not needed, we drop the $(r, s)$ in the subscripts and use $u$ to denote the elements in $U$. Recall that $\mathsf{B}^U_1 = \exp(B_1^\LieU(0))$ and $m_U$ is the Haar measure on $U$ so that $m_U(\mathsf{B}^U_1) = 1$. 

Recall that we fix a norm on $\LieG$ by restricting the maximum norm on $\mathrm{Mat}_4(\R)$. We will us $|\cdot|_\delta$ to denote the $\delta$-covering number according to this metric. We remark that for the results in this part, changing to a different norm will only affect the estimate by a constant factor. 

For a finite set $F$, let $\mu_F$ be the uniform probability measure on $F$. For all $\alpha \in (0, \dim(\mathfrak{r}))$ and scale $\delta \in (0, 1)$, recall we defined the following (modified) $\alpha$-energy of the set $F$ in Subsection~\ref{subsec:Margulis function}:
\begin{align*}
    \mathcal{G}^{(\alpha)}_{F, \delta}(w) = \sum_{w' \in F, w' \neq w} \max\{\|w' - w\|, \delta\}^{-\alpha}.
\end{align*}
Let $\hat{\varphi}$ be the following function:
\begin{align*}
    \hat{\varphi}(\alpha) = \min\{\alpha, 1\} - \frac{1}{9}\alpha = \begin{cases}
        \frac{8}{9} \alpha & \text{ if } 0 \leq \alpha \leq 1;\\
        1 - \frac{1}{9} \alpha & \text{ if } 1 < \alpha \leq 9.
    \end{cases}
\end{align*}
Let $\varphi = \frac{1}{36}\hat{\varphi}$. 

The following is the main result of this part. 
\begin{theorem}\label{thm:energy Improvement}
Let $\alpha \in (0, \dim(\mathfrak{r}))$, $\delta \in (0, 1)$ and $\epsilon \in (0, 10^{-10}\alpha)$. Suppose there exists a finite set $F \subset B_1^{\mathfrak{r}}(0)$ with $\#F \gg_\epsilon 1$ satisfying
\begin{align*}
    \mathcal{G}_{F, \delta}^{(\alpha)}(w) \leq \Upsilon \quad \forall w \in F.
\end{align*}

Then for all $\ell \gg_\epsilon 1$, there exists $J \subset \mathsf{B}_1^U$ with $m_U(\mathsf{B}_1^U \setminus J) \ll_\epsilon |\log \delta|e^{-\epsilon \ell}$ so that the following holds. For all $u \in J$ there exists $F_{u} \subseteq F$ with $\#(F \setminus F_u) \ll_\epsilon |\log \delta|e^{-\epsilon \ell}\#F$ so that for all $w \in F_u$
\begin{align*}
    \mathcal{G}^{(\alpha)}_{F_{u}(w), \delta'}(a_\ell u.w) \ll_\epsilon e^{-\varphi(\alpha)\ell} \delta^{-O(\sqrt{\epsilon})}\Upsilon
\end{align*}
where the new scale $\delta' = e^{2\ell}\max\{\delta, \#F^{-\frac{1}{\alpha}}\}$ and the set
\begin{align*}
    F_{u}(w) = \{a_{\ell} u.w': w' \in F_{u}, \|a_{\ell} u.w' - a_{\ell} u.w\|\leq e^{-2\ell}\}.
\end{align*}
\end{theorem}

Theorem~\ref{thm:energy Improvement} follows from the following theorem which is inspired by \cite[Lemma 6.2]{LMWY25}. 
\begin{theorem}\label{thm:ImprovementMain}
    Let $F \subset B^{\mathfrak{r}}_1(0)$ be a finite set satisfying
    \begin{align*}
        \mu_F(B^{\mathfrak{r}}_\delta(x)) \leq C\delta^\alpha \quad \forall x \in \mathfrak{r}
    \end{align*}
    for some $C \geq 1$, $\alpha \in (0, \dim(\mathfrak{r}))$ and all $\delta \geq \delta_0$. 

    Let $\epsilon \in (0, 10^{-10}\alpha)$. For all $\ell \gg_{\epsilon} 1$ and $\delta \in [e^{2\ell}\delta_0, e^{-2\ell}]$, there exists $J_{\ell, \delta} \subseteq \mathsf{B}_1^U$ with $m_U(\mathsf{B}_1^U \setminus J_{\ell, \delta}) \ll_\epsilon e^{-\epsilon\ell}$ so that the following holds. Let $u \in J_{\ell, \delta}$, there exists $F_{\ell, \delta, u} \subseteq F$ with 
    \begin{align*}
        \mu_F(F \setminus F_{\ell, \delta, u}) \ll_\epsilon e^{-\epsilon\ell}
    \end{align*}
    such that for all $w \in F_{\ell, \delta, u}$ we have
    \begin{align*}
        \mu_F(\{w' \in F_{\ell, \delta, u}: \|a_\ell u.w' - a_\ell u.w\| \leq \delta\}) \ll_\epsilon Ce^{-\varphi(\alpha)\ell}\delta^{\alpha - O(\sqrt{\epsilon})}.
    \end{align*}
\end{theorem}
\cref{thm:ImprovementMain} follows from the following theorem on covering numbers. 

\begin{theorem}\label{thm:ImprovementSlabExpansion}
    Let $F \subset B^{\mathfrak{r}}_1(0)$ be a finite set satisfying
    \begin{align*}
        \mu_F(B^{\mathfrak{r}}_\delta(x)) \leq C\delta^\alpha \quad \forall x \in \mathfrak{r}
    \end{align*}
    for some $C \geq 1$, $\alpha \in (0, \dim(\mathfrak{r}))$ and all $\delta \geq \delta_0$. 

    Then for all $\epsilon \in (0, 10^{-10}\alpha)$, there exists $C_{\epsilon} > 0$ so that the following holds. 
    For all $\ell \gg_{\epsilon} 1$ and $\delta \in [e^{2\ell}\delta_0, e^{-2\ell}]$, we define the exceptional set $\mathcal{E}(F)$ to be
    \begin{align*}
        \mathcal{E}(F) = \{u \in \mathsf{B}_1^U: \exists F' \subseteq F \text{ with }& \mu_F(F') \geq e^{-\epsilon\ell}\\
        \text{ and }&|a_\ell u.F'|_{\delta} < C_{\epsilon}^{-1}C^{-1}e^{(\varphi(\alpha) - O(\sqrt{\epsilon}))\ell} \delta^{-\alpha}\}.
    \end{align*}

    We have
    \begin{align*}
        m_U(\mathcal{E}(F)) \leq C_{\epsilon}e^{-\epsilon\ell}.
    \end{align*}
\end{theorem}

A key step in the proof of the above theorem is an estimate of covering number using certain anisotropic tubes explicated later, see Theorem~\ref{thm:Multislicing}. Similar anisotropic tubes were studied in the case of irreducible representations of $\SL_2(\R)$ in \cite{LMWY25,OL25}. Before we state Theorem~\ref{thm:Multislicing}, let us introduce some notations. 

The one-parameter diagonal subgroup $\{a_t\}_{t \in \R}$ is generated by the following element $\mathbf{a} \in \LieH \subset \LieG$:
\begin{align*}
    \mathbf{a} = \begin{pmatrix}
        1 & & & \\
         & 0 & & \\
         & & 0 & \\
         & & & -1
    \end{pmatrix}.
\end{align*}

As a representation of $\LieH$, $\mathfrak{r}$ can be decomposed into eigenspaces of $\ad \mathbf{a}$. Here the eigenvalues are exactly $-2, -1, 0, 1, 2$. We denote those eigenspaces by $\mathfrak{r}_{\lambda}$, where $\lambda$ is the corresponding eigenvalue. Let $\pi_{\lambda}$ be the orthogonal projection to $\mathfrak{r}_{\lambda}$ with respect to the standard inner product on $\mathrm{Mat}_4(\R)$. We also use the notion $\mathfrak{r}^{(\lambda)}$ as sum of eigenspaces with eigenvalues $\geq \lambda$. Let $\pi^{(\lambda)}$ be the orthogonal projection to $\mathfrak{r}^{(\lambda)}$. Note that those $\mathfrak{r}^{(\lambda)}$'s are $U$-submodules. We use $\pi_{r, s}^{(\lambda)}$ to denote the projections $\pi^{(\lambda)} \circ u_{r, s}$. When the exact parametrization for $U$ is not important, we use $\pi_{u}^{(\lambda)}$ to denote the projections $\pi^{(\lambda)} \circ u$. Those eigenspaces form a flag with dimension $(9, 8, 6, 3, 1)$:
\begin{align*}
    \mathfrak{r} = \mathfrak{r}^{(-2)} \supset \mathfrak{r}^{(-1)} \supset \mathfrak{r}^{(0)} \supset 
    \mathfrak{r}^{(1)} \supset 
    \mathfrak{r}^{(2)} = \mathfrak{r}_2.
\end{align*}
For simplicity, we write $\mathbf{d} = (\mathsf{d_1}, \mathsf{d_2}, \mathsf{d_3}, \mathsf{d_4}, \mathsf{d_5}) = (1, 2, 3, 2, 1)$ as the dimension difference for the above flag. 

We adapt the notations in \cite{BH24} for partitions using anisotropic tubes associated to the above flag. Let $\mathcal{D}_\delta$ be the partition of $\mathfrak{r}$ via $\delta$-cubes. For a $5$-tuple $\mathbf{r} = (\mathsf{r_1}, \mathsf{r_2}, \mathsf{r_3}, \mathsf{r_4}, \mathsf{r_5})$ satisfying $0 \leq \mathsf{r_1} \leq \mathsf{r_2} \leq \mathsf{r_3} \leq \mathsf{r_4} \leq \mathsf{r_5} = 1$, we define
\begin{align*}
    \mathcal{D}_{\delta}^{\mathbf{r}} = \vee_{i} (\pi^{(i - 3)})^{-1} \mathcal{D}_{\delta^{\mathsf{r_i}}}
\end{align*}
to be the partition consisting of (possibly anisotropic) tubes. We will use $T$ to denote an atom in $\mathcal{D}_{\delta}^{\mathbf{r}}$. Roughly, $T$ is a tube of size
\begin{align*}
    \delta^{\mathsf{r_1}} \times \delta^{\mathsf{r_2}} \times \delta^{\mathsf{r_2}} \times
    \delta^{\mathsf{r_3}} \times
    \delta^{\mathsf{r_3}} \times
    \delta^{\mathsf{r_3}} \times
    \delta^{\mathsf{r_4}} \times
    \delta^{\mathsf{r_4}} \times
    \delta^{\mathsf{r_5}}
\end{align*}
with edges parallel to an orthogonal basis compatible with the weight space decomposition $\mathfrak{r} = \mathfrak{r}_{-2} \oplus \mathfrak{r}_{-1} \oplus \mathfrak{r}_0 \oplus \mathfrak{r}_1 \oplus \mathfrak{r}_2$. Its volume satisfies
\begin{align*}
    \vol(T) \sim \delta^{\sum_{i = 1}^5\mathsf{d_i} \mathsf{r_i}}.
\end{align*}

In this paper, we always assume the $5$-tuple $\mathbf{r} = (\mathsf{r_1}, \mathsf{r_2}, \mathsf{r_3}, \mathsf{r_4}, \mathsf{r_5})$ satisfies $0 \leq \mathsf{r_1} \leq \mathsf{r_2} \leq \mathsf{r_3} \leq \mathsf{r_4} \leq \mathsf{r_5} \leq 1$. We remark that to prove \cref{thm:ImprovementSlabExpansion}, one only needs to focus on the case where
\begin{align*}
    \mathbf{r} = \Bigl(0, \frac{1}{4}, \frac{1}{2}, \frac{3}{4}, 1\Bigr),
\end{align*}
which is compatible with the expanding rates of $a_\ell$ on $\mathfrak{r}$. 

\begin{theorem}\label{thm:Multislicing}
    Let $F \subset B^{\mathfrak{r}}_{1}(0)$ be a finite set satisfying
    \begin{align*}
        \mu_F(B^{\mathfrak{r}}_{\rho}(x)) \leq C\rho^\alpha, \forall x \in \mathfrak{r}
    \end{align*}
    for some $C \geq 1$, $\alpha \in (0, 9)$ and all $\rho \geq \rho_0$. 

    Fix a $5$-tuple $\mathbf{r}$. Then for all $0 < \epsilon \ll \mathsf{r_5} - \mathsf{r_4}$, there exists $C_{\epsilon, \mathbf{r}}$ so that the following holds. 
    
    For all $\rho_0 \leq \rho \ll_{\epsilon, \mathbf{r}}1$, we define the exceptional set $\mathcal{E}(F)$ to be
    \begin{align*}
        \mathcal{E}(F) = \{u \in \mathsf{B}_1^U: &\exists F' \subseteq F \text{ with }\mu_F(F') \geq \rho^\epsilon \text{ and }\\
        &|u.F'|_{\mathcal{D}_{\rho}^{\mathbf{r}}} < C_{\epsilon, \mathbf{r}}^{-1}C^{-1}\vol(T)^{-\frac{1}{9}\alpha}\rho^{-(\mathsf{r_5} - \mathsf{r_4})\varphi(\alpha) + O_{\mathbf{r}}(\sqrt{\epsilon})}\}.
    \end{align*}

    We have
    \begin{align*}
        m_U(\mathcal{E}(F)) \leq C_{\epsilon, \mathbf{r}}\rho^{\epsilon}.
    \end{align*}
\end{theorem}

Part~\ref{part:projection} is organized as the following. We first deduce \cref{thm:energy Improvement,thm:ImprovementMain,thm:ImprovementSlabExpansion} from \cref{thm:Multislicing} in Section~\ref{sec:Anisotropic implies expansion}. The arguments are similar to \cite{LMWY25,OL25}. In Section~\ref{sec:prepareRegular}, we collect results for regular sets and measures needed later. In Section~\ref{sec:prepareIrrep}, we collect the properties of certain class of irreducible representation of semisimple Lie groups. In section~\ref{sec:dim 1 or codim 1}, we study behaviors of lines or hyperplanes in irreducible representation of semisimple Lie groups and prove subcritical estimates for $\{\pi^{(\lambda)}_{u}\}_{\lambda = 2, -1}$. In Section~\ref{sec:Subcritical}, we study the representation $\mathfrak{r}$ in details and prove subcritical estimates for $\{\pi^{(\lambda)}_{u}\}_{\lambda = 1, 0}$. In Section~\ref{sec:optimal}, we prove an optimal projection theorem for $\pi^{(2)}_{u}$. The key ingredient is the restricted projection theorem proved by Gan, Guo and Wang in \cite{GGW24}. In Section~\ref{sec:proof of multislice}, we adapt the arguments in the Multislicing theorem proved by B{\'e}nard and He in \cite[Theorem 2.1]{BH24} to combine the above ingredients to prove Theorem~\ref{thm:Multislicing}. 

\section{\texorpdfstring{Proof of \cref{thm:energy Improvement,thm:ImprovementMain,thm:ImprovementSlabExpansion} assuming Theorem~\ref{thm:Multislicing}}{Proof of Theorem~\nameref{thm:energy Improvement,thm:ImprovementMain,thm:ImprovementSlabExpansion} assuming Theorem~\ref{thm:Multislicing}}}\label{sec:Anisotropic implies expansion}
We first deduce \cref{thm:energy Improvement} from \cref{thm:ImprovementMain}. 
\begin{proof}[Proof of \cref{thm:energy Improvement} assuming \cref{thm:ImprovementMain}]
    The statement can be proved by following the proof of \cite[Theorem 6.1]{LMWY25} step-by-step and replacing \cite[Lemma 6.2]{LMWY25} by \cref{thm:ImprovementMain}. 
\end{proof}

We now deduce \cref{thm:ImprovementMain} from \cref{thm:ImprovementSlabExpansion}. This procedure is well-known. We reproduce it here for completeness. 
\begin{proof}[Proof of \cref{thm:ImprovementMain} assuming \cref{thm:ImprovementSlabExpansion}]
Applying Theorem~\ref{thm:ImprovementSlabExpansion} with $\epsilon$, there exists $\mathcal{E} \subset \mathsf{B}^U_1$ with $m_U(\mathcal{E}) \ll_\epsilon e^{-\epsilon \ell}$ such that for all $u \notin \mathcal{E}$ and all $F'$ with $\mu_F(F') \geq e^{-\epsilon\ell}$, we have
\begin{align*}
    |a_{\ell}u.F'|_{\delta} \geq C_{\epsilon}^{-1} C^{-1} e^{(\varphi(\alpha) - O(\sqrt{\epsilon}))\ell} \delta^{-\alpha}.
\end{align*}

We define
\begin{align*}
    \mathcal{D}_{\delta, \text{bad}}^{\ell, u} = \{Q \in \mathcal{D}_{\delta}: (a_{\ell}u)_{*} \mu_F (Q) > C_{\epsilon}^{-1} C e^{-(\varphi(\alpha) - O(\sqrt{\epsilon}))\ell} \delta^{\alpha}\}.
\end{align*}

Let
\begin{align*}
    F_{\ell, \delta, u}' = (a_{\ell} u)^{-1} \bigcup_{Q \in \mathcal{D}_{\delta, \text{bad}}^{\ell, u}} ((a_{\ell} u.F) \cap Q).
\end{align*}
Since $(a_{\ell} u)_{*} \mu_F$ is a probability measure, we have
\begin{align*}
    \# \mathcal{D}_{\delta, \text{bad}}^{\ell, u} < C_{\epsilon}^{-1} C^{-1} e^{(\varphi(\alpha) - O(\sqrt{\epsilon}))\ell} \delta^{-\alpha}, 
\end{align*}
which is equivalent to
\begin{align*}
    |a_{\ell} u F_{\ell, \delta, u}'|_{\delta} < C_{\epsilon}^{-1} C^{-1} e^{(\varphi(\alpha) - O(\sqrt{\epsilon}))\ell} \delta^{-\alpha}.
\end{align*}
Therefore, we have
\begin{align*}
    \mu_F(F_{\ell, \delta, u}') \leq e^{-\epsilon\ell}.
\end{align*}

Let $F_{\ell, \delta, u} = F \backslash F_{\ell, \delta, u}'$. For all $\delta$-(dyadic) cube $Q$, we have
\begin{align*}
    (a_{\ell} u)_{*}(\mu_F|_{F_{\ell, \delta, u}}) (Q) \leq{}& C_{\epsilon} C e^{-(\varphi(\alpha) - O(\sqrt{\epsilon}))\ell} \delta^{\alpha}\\
    \leq{}& C_{\epsilon} C e^{-\varphi(\alpha)\ell} \delta^{\alpha - O(\sqrt{\epsilon})}
\end{align*}
which proves the theorem. 
\end{proof}

We now prove \cref{thm:ImprovementSlabExpansion} assuming \cref{thm:Multislicing}. Before we proceed the proof, let us introduce the following notations. For a dyadic cube $Q$ in $\R^n$, $\Hom_Q$ is the unique homothety that map $Q$ to $[0, 1)^n$. For a space $X$, a partition $\mathcal{P}$ of it and a subset $A \subseteq X$, we use $|A|_{\mathcal{P}}$ to denote the number of atoms needed in $\mathcal{P}$ to cover $A$. Also, we use the notion $\mathcal{P}(A)$ to denote the atoms in $\mathcal{P}$ intersecting $A$ non-trivially. 

\begin{proof}[Proof of \cref{thm:ImprovementSlabExpansion} assuming \cref{thm:Multislicing}]
   We will only use the case where $\mathbf{r} = (0, \frac{1}{4}, \frac{1}{2}, \frac{3}{4}, 1)$ in \cref{thm:Multislicing}. In the rest of the proof, $\mathbf{r}$ will always be this $5$-tuple. 
   
   For simplicity, let $\tilde{\delta} = e^{2\ell} \delta$ and $\rho = e^{-4\ell}$. For all $u \in \mathsf{B}^U_1$ all subset $F' \subseteq F$ with $\mu_F(F') \geq e^{-\epsilon\ell}$, we have
    \begin{align*}
        |a_{\ell}u. F'|_{\delta} ={}& |u. F'|_{\tilde{\delta}\mathcal{D}_{\rho}^{\mathbf{r}}}\\
        \gg{}& \sum_{Q \in \mathcal{D}_{\tilde{\delta}}} |u. F'_Q|_{\tilde{\delta}\mathcal{D}_{\rho}^{\mathbf{r}}}\\
        ={}& \sum_{Q \in \mathcal{D}_{\tilde{\delta}}} |u. \Hom_Q F'_Q|_{\mathcal{D}_{\rho}^{\mathbf{r}}}.
    \end{align*}

    We use $F^Q$ to denote $\Hom_Q F_Q$. We note that $F^Q$ satisfies the following Frostman-type condition:
    \begin{align*}
        \mu_{F^Q} (B^{\mathfrak{r}}_{\rho'}(x)) ={}& \frac{1}{\mu_F(Q)} \mu_F(B^{\mathfrak{r}}_{\tilde{\delta}\rho'}(x'))\\
        \leq{}& \frac{C(\tilde{\delta})^\alpha}{\mu_F(Q)} (\rho')^{\alpha}
    \end{align*}
    for all $\rho' \geq (\tilde{\delta})^{-1} \delta_0$. 

    Note that by our restriction to $\delta$, $\rho \geq (\tilde{\delta})^{-1} \delta_0$. Therefore, for all $Q \in \mathcal{D}_{\tilde{\delta}}$ so that $\mu_{F_Q}(F'_Q) \geq \rho^{\epsilon}$, applying \cref{thm:Multislicing} to $\mu_{F^Q}$, there exists $\mathcal{E}_Q \subset \mathsf{B}^U_1$ for all $Q$ with $m_U(\mathcal{E}_Q) \leq C_{\epsilon} \rho^{\epsilon}$, and for all $u \notin \mathcal{E}_Q$, we have  \begin{align}\label{eqn:SlabExpansionLocal}
        |u. \Hom_Q F'_Q|_{\mathcal{D}_{s}^{\mathbf{r}}} \geq{}& C_\epsilon^{-1}\frac{\mu_F(Q)}{C(\tilde{\delta})^{\alpha}} \rho^{-\frac{1}{2}\alpha - \frac{1}{4}\varphi(\alpha) + O(\sqrt{\epsilon})}\\
        ={}& C_\epsilon^{-1}\mu_F(Q)C^{-1}e^{(\varphi(\alpha) - O(\sqrt{\epsilon}))\ell}\delta^{-\alpha}.
    \end{align}
    Let
    \begin{align*}
        \mathcal{D}_{\tilde{\delta}}(u) = \{Q \in \mathcal{D}_{\tilde{\delta}}(F): u \in \mathcal{E}_Q\}
    \end{align*}
    and let
    \begin{align*}
        \mathcal{D}_{\tilde{\delta}}^{\mathrm{large}}(F') = \{Q \in \mathcal{D}_{\tilde{\delta}}: \mu_{F_Q}(F_Q') \geq e^{-2\epsilon\ell}\}.
    \end{align*}
    Since $\mu_F(F') \geq e^{-\epsilon\ell}$, we have
    \begin{align*}
        \sum_{Q \notin \mathcal{D}_{\tilde{\delta}}^{\mathrm{large}}(F')} \mu_F(Q) \geq e^{-2\epsilon\ell}.
    \end{align*}
    By Fubini's theorem, there exists $\mathcal{E} \subseteq \mathsf{B}^U_1$ with $m_U(\mathcal{E}) \ll_\epsilon e^{-\epsilon\ell}$ so that for all $u \notin \mathcal{E}$, we have
    \begin{align*}
        \sum_{Q \in \mathcal{D}_{\tilde{\delta}}(u)} \mu_F(Q) \geq e^{-\epsilon\ell}.
    \end{align*}
    Therefore, we have
    \begin{align*}
        |a_\ell u.F'|_{\delta} \gg{}& \Biggl(\sum_{Q \in \mathcal{D}_{\tilde{\delta}}^{\mathrm{large}}(F') \setminus \mathcal{D}_{\tilde{\delta}}(u)} \mu_F(Q)\Biggr) C_\epsilon^{-1}C^{-1} e^{(\varphi(\alpha) - O(\sqrt{\epsilon}))\ell} \delta^{-\alpha}\\
        \geq{}& (e^{-\epsilon\ell} - e^{-2\epsilon\ell}) C_\epsilon^{-1}C^{-1} e^{(\varphi(\alpha) - O(\sqrt{\epsilon}))\ell} \delta^{-\alpha}\\
        \gg{}& C_\epsilon^{-1}C^{-1} e^{(\varphi(\alpha) - O(\sqrt{\epsilon}))\ell} \delta^{-\alpha}.
    \end{align*}
    This completes the proof of the theorem. 
\end{proof}

\section{Preparation III: Regular sets and regular measures}\label{sec:prepareRegular}
\subsection{Covering numbers, measures and projections}
For a space $X$, a partition $\mathcal{P}$ of it and a subset $A \subseteq X$, we use $|A|_{\mathcal{P}}$ to denote the number of atoms needed in $\mathcal{P}$ to cover $A$. Also, we use the notion $\mathcal{P}(A)$ to denote the atoms in $\mathcal{P}$ intersecting $A$ non-trivially. For a finite set $F$, we use $\mu_F$ to denote the uniform probability measure on $F$. For any measure $\mu$ on $X$ and any partition $\mathcal{P}$ of $X$, we use $\mathcal{P}(\mu)$ to denote the collection of atoms in $\mathcal{P}$ with positive measure. For a dyadic cube $Q$ in $\R^n$, we set $\Hom_Q$ to be the unique homothety that map $Q$ to $[0, 1)^n$. 

We say $\mathcal{Q}$ roughly refines $\mathcal{P}$ with a parameter $L \geq 1$, and write $\mathcal{P} \overset{L}\prec \mathcal{Q}$, if 
\begin{align*}
    \max_{Q \in \mathcal{Q}}|Q|_{\mathcal{P}} \leq L.
\end{align*}
We say $\mathcal{Q}$ and $\mathcal{P}$ are roughly equivalent with a parameter $L \geq 1$, and write $\mathcal{P} \overset{L}\sim \mathcal{Q}$ if $\mathcal{P} \overset{L}\prec \mathcal{Q}$ and $\mathcal{Q} \overset{L}\prec \mathcal{P}$. This is the same as each atom of $P$ is contained in at most $L$ atoms in $\mathcal{Q}$ and vice versa. 
\subsubsection{Regular sets and regular measures}
Fix a filtration $\mathcal{P}_0 \prec \cdots \prec \mathcal{P}_n$, we set $d_i = \log_2 \max_{P \in \mathcal{P}_{i - 1}} |P|_{\mathcal{P}_i}$ for all $i = 1, \cdots, n$. Fix an $n$-tuple $(\sigma_1, \cdots, \sigma_n)$ with $\sigma_i \in [1, d_i + 1]$ for all $i$. For a set $A \subseteq X$, we say it is $(\sigma_1, \cdots, \sigma_n)$-regular with respect to the filtration $\mathcal{P}_0 \prec \cdots \prec \mathcal{P}_n$ if for all $i = 1, \cdots, n$ and all $P \in \mathcal{P}_{i - 1}(A)$, we have
\begin{align*}
    2^{\sigma_i - 1} \leq |A \cap P|_{\mathcal{P}_{i}} < 2^{\sigma_i}.
\end{align*}
We omit the $n$-tuple $(\sigma_1, \cdots, \sigma_n)$ and just call it regular throughout the paper for simplicity. We remark that this is slightly weaker than the usual notion of regular sets, cf. \cite{Shm23b}, but the following lemma shows that they are closely related. 
\begin{lemma}\label{lem:weaker regular}
    Suppose $A$ is regular with respect to a filtration $\mathcal{P}_0 \prec \cdots \prec \mathcal{P}_n$, then for all $i = 1, \ldots, n$ and $P \in \mathcal{P}_{i - 1}$, we have
    \begin{align*}
        \frac{1}{2}\frac{|A|_{\mathcal{P}_i}}{|A|_{\mathcal{P}_{i - 1}}} \leq |A \cap P|_{\mathcal{P}_i} \leq 2\frac{|A|_{\mathcal{P}_i}}{|A|_{\mathcal{P}_{i - 1}}}.
    \end{align*}
    Moreover, for any subset $A' \subseteq A$, we have
    \begin{align*}
        \frac{|A'|_{\mathcal{P}_{i - 1}}}{|A|_{\mathcal{P}_{i - 1}}} \geq \frac{1}{2}\frac{|A'|_{\mathcal{P}_{i}}}{|A|_{\mathcal{P}_{i}}}.
    \end{align*}
\end{lemma}
\begin{proof}
    Note that
    \begin{align*}
        |A|_{\mathcal{P}_i} = \sum_{P \in \mathcal{P}_{i - 1}(A)} |A \cap P|_{\mathcal{P}_i},
    \end{align*}
    we have
    \begin{align*}
        2^{\sigma_i - 1} |A|_{\mathcal{P}_{i - 1}} \leq \sum_{P \in \mathcal{P}_{i - 1}(A)} |A \cap P|_{\mathcal{P}_i} < 2^{\sigma_i}|A|_{\mathcal{P}_{i - 1}}.
    \end{align*}
    Therefore, 
    \begin{align*}
        2^{\sigma_i - 1} \leq \frac{|A|_{\mathcal{P}_i}}{|A|_{\mathcal{P}_{i - 1}}} < 2^{\sigma_i}.
    \end{align*}
    Combining with the definition of regularity, this proves the first statement. 

    For the second statement, note that
    \begin{align*}
        |A'|_{\mathcal{P}_i} = \sum_{P \in \mathcal{P}_{i - 1}(A')} |A' \cap P|_{\mathcal{P}_i} \leq \sum_{P \in \mathcal{P}_{i - 1}(A')} |A \cap P|_{\mathcal{P}_i} \leq |A'|_{\mathcal{P}_{i - 1}} 2\frac{|A|_{\mathcal{P}_i}}{|A|_{\mathcal{P}_{i - 1}}}.
    \end{align*}
    This proves the second statement. 
\end{proof}

For a probability measure $\mu$ on $X$, we say it is $(\sigma_1, \cdots, \sigma_n)$-regular with respect to the same filtration if for all $i = 1, \cdots, n$ and all $\widehat{P} \in \mathcal{P}_{i - 1}(\mu)$ and all $P \in \mathcal{P}_{i}(\mu)$ with $P \subseteq \widehat{P}$, we have
\begin{align*}
    2^{-\sigma_i} < \frac{\mu(P)}{\mu(\widehat{P})} \leq 2^{-\sigma_i + 1}.
\end{align*}
We omit the $n$-tuple $(\sigma_1, \cdots, \sigma_n)$ and just call it regular throughout the paper for simplicity. 

The connection between being regular for a set $F$ and the corresponding measure $\mu_F$ is recorded in the following lemma. 
\begin{lemma}\label{lem:regular measure vs regular set}
    For a finite set $F$, if $\mu_F$ is regular with respect to the filtration $\mathcal{P}_0 \prec \cdots \prec \mathcal{P}_n$, then the set $F$ is also regular with respect to the filtration. 

    Moreover, if $F$ lies in one atom of $\mathcal{P}_0$, then for any subset $F' \subseteq F$, we have
    \begin{align*}
        \frac{|F'|_{\mathcal{P}_n}}{|F|_{\mathcal{P}_n}} \geq \frac{1}{2^n} \mu_F(F').
    \end{align*}
    Conversely, let
    \begin{align*}
        F'' = \cup_{P \in \mathcal{P}_n(F')} (P \cap F) \supseteq F',
    \end{align*}
    we have
    \begin{enumerate}
        \item $|F''|_{\mathcal{P}_i} = |F'|_{\mathcal{P}_i}$ for all $i$, 
        \item $\mu_F(F'') \geq \frac{1}{2^n} \frac{|F'|_{\mathcal{P}_n}}{|F|_{\mathcal{P}_n}}$.
    \end{enumerate}
\end{lemma}
\begin{proof}
    For all $\widehat{P} \in \mathcal{P}_{i - 1}(F)$, we have
\begin{align*}
    &1 = \sum_{P \in \mathcal{P}_{i}(F \cap \widehat{P})} \frac{\mu_F(P)}{\mu_F(\widehat{P})} \leq 2^{-\sigma_i + 1} |F \cap \widehat{P}|_{\mathcal{P}_i},\\
    &1 = \sum_{P \in \mathcal{P}_{i}(F \cap \widehat{P})} \frac{\mu_F(P)}{\mu_F(\widehat{P})} > 2^{-\sigma_i} |F \cap \widehat{P}|_{\mathcal{P}_i},
\end{align*}
which implies
\begin{align*}
    2^{\sigma_i - 1} \leq |F \cap \widehat{P}|_{\mathcal{P}_i} < 2^{\sigma_i}.
\end{align*}
Therefore, $F$ is also regular. 

We now suppose $F$ lies in just one atom of $\mathcal{P}_0$. This implies that for all $P \in \mathcal{P}_n(F)$, we have
\begin{align*}
    2^{-(\sigma_1 + \cdots + \sigma_n)} < \mu_F(P) \leq 2^n 2^{-(\sigma_1 + \cdots + \sigma_n)}.
\end{align*}
This implies
\begin{align*}
    \frac{|F'|_{\mathcal{P}_n}}{|F|_{\mathcal{P}_n}} \geq \frac{1}{2^n} \mu_F(F').
\end{align*}
For $F'' = \cup_{P \in \mathcal{P}_n(F')} (P \cap F) \supseteq F'$, we have
\begin{align*}
    \mu_F(F'') = \sum_{P \in \mathcal{P}_n(F')} \mu_F(P) \geq 2^{\sigma_1 + \cdots + \sigma_n} |F'|_{\mathcal{P}_n},
\end{align*}
which implies the last statement. 
\end{proof}

We have the following regularization process due to Bourgain. 

\begin{lemma}\label{lem:regularization set}
    Let $\mathcal{P}_0 \prec \cdots \prec \mathcal{P}_n$ be a filtration of partitions of $X$. Let $A$ be a subset of $X$. Then there exists $A' \subseteq A$ so that $A'$ is regular with respect to the filtration and 
    \begin{align*}
        |A'|_{\mathcal{P}_n} \geq |A|_{\mathcal{P}_n} \prod_{i = 1}^n\frac{1}{2(1 + \log_2\max_{P \in \mathcal{P}_{i - 1}}|P|_{\mathcal{P}_{i}})}.
    \end{align*}
    Moreover, the subset $A'$ can be taken as intersection of $A$ with disjoint union of atoms $\mathcal{P}_n(A)$. 
\end{lemma}

\begin{proof}
    See \cite[Section 2]{Bou10} or \cite[Lemma 2.5]{BH24}. 
\end{proof}

We also have the following variant of Bourgain’s regularization argument for measure. 
\begin{lemma}\label{lem:regularization measure}
    Let $\mathcal{P}_0 \prec \cdots \prec \mathcal{P}_n$ be a filtration of partitions of $X$. Let $F$ be a finite subset of $X$. Then there exists $F' \subseteq F$ so that the conditional measure $\mu_{F'}$ is regular with respect to the filtration and 
    \begin{align*}
        \mu_F(F') \geq \prod_{i = 1}^n\frac{1}{2(1 + \log_2\max_{P \in \mathcal{P}_{i - 1}}|P|_{\mathcal{P}_{i}})}.
    \end{align*}
     Moreover, $F'$ can be taken as intersection of $F$ with disjoint union of atoms in $\mathcal{P}_n(F)$. 
\end{lemma}
\begin{proof}
    See \cite[Lemma 3.4]{KS19}. 
\end{proof}

For a finite set $F$, iterating the above process with $\mu_F$, we can decompose a large portion of $F$ into regular pieces as in the following lemma. 
\begin{lemma}\label{lem:regularization exhaust}
    Let $\mathcal{P}_0 \prec \cdots \prec \mathcal{P}_n$ be a filtration of partitions of $X$. Let $d_i = \log_2 \max_{P \in \mathcal{P}_{i - 1}} |P|_{\mathcal{P}_i}$ for all $i = 1, \cdots, n$. Let $F$ be a finite subset of $X$. 
    
    For all $\mathsf{c} \in (0, 1)$, there exists a family of disjoint subsets $\{F_j\}_{j = 1}^N$ so that the following holds. 
    \begin{enumerate}
        \item For all $j$, the measure $\mu_{F_j}$ is regular. 
        \item We have $\mu_F(\sqcup F_j) \geq 1 - \mathsf{c}$. 
        \item For each $F_j$, we have $\mu_F(F_j) \geq \mathsf{c}\prod_{i = 1}^n\frac{1}{2(1 + d_i)}$.
    \end{enumerate}
    Moreover, all $F_j$ can be taken as intersection of $F$ with disjoint union of atoms in $\mathcal{P}_n(F)$. 
\end{lemma}
\begin{proof}
    This is essentially \cite[Corollary 3.5]{KS19}. We reproduce the argument here. For simplicity, we let 
    \begin{align*}
        \lambda = \prod_{i = 1}^n\frac{1}{2(1 + d_i)}.
    \end{align*}
    Applying \cref{lem:regularization measure} to $F$, we get a regular subset $F_1$ with $\mu_F(F_1) \geq \lambda$. Let $B_0 = F$ and $B_1 = F \setminus F_1$. We now construct $\{B_j\}$ and $\{F_j\}$ inductively. Suppose $B_j$ is constructed, applying \cref{lem:regularization measure}, we get $F_{j + 1}$ with $\mu_{B_j}(F_{j + 1}) \geq \lambda$. Let $B_{j + 1} = B_{j} \setminus F_{j + 1}$. Note that by construction, we have
    \begin{align*}
        \mu_{B_j}(B_{j + 1}) \leq (1 - \lambda).
    \end{align*}
    Therefore, 
    \begin{align*}
        \mu_{F}(B_{j}) \leq (1 - \lambda)^j.
    \end{align*}
    Let $N$ be the smallest integer so that $(1 - \lambda)^N < \mathsf{c}$. We now show that $\{F_j\}_{j = 1}^N$ is a family of subset satisfying the lemma. The regularity of each $F_j$ follows directly from \cref{lem:regularization measure}. They are disjoint by the construction. Note that $B_{N} = F \setminus (\sqcup_{j = 1}^N F_j)$, we have
    \begin{align*}
        \mu_F(F \setminus (\sqcup_{j = 1}^N F_j)) = \mu_F(B_N) \leq (1 - \lambda)^N < \mathsf{c}.
    \end{align*}
    For each $F_j$ where $j \in \{1, \cdots, N\}$, we have $F_j \subseteq B_{j - 1}$ with $\mu_{B_{j - 1}}(F_j) \geq \lambda$. Since $j - 1 < N$, $\mu_F(B_{j - 1}) \geq \mathsf{c}$, we have
    \begin{align*}
        \mu_{F}(F_j) = \mu_F(B_{j - 1})\mu_{B_{j - 1}}(F_j) \geq \mathsf{c}\lambda.
    \end{align*}
    The last claim follows directly from the construction and \cref{lem:regularization measure}. 
\end{proof}
\subsubsection{Submodularity inequality}The following inequality is taken from \cite[Lemma 2.6]{BH24}. This provides us tools to connect covering number of tubes of different sizes. 
\begin{lemma}[{\cite[Lemma 2.6]{BH24}}]\label{lem:submodularity}
Let $\mathcal{P}$, $\mathcal{Q}$, $\mathcal{R}$, $\mathcal{S}$ be partitions of some space $X$ and
$A$ a subset of $X$. Assume that $\mathcal{R} = \mathcal{P}\vee\mathcal{Q}$, $\mathcal{S} \prec \mathcal{P}$
and $\mathcal{S} \prec \mathcal{Q}$. Then for every $c > 0$, there is a subset $A' \subseteq A$ such that $|A'|_{\mathcal{R}} \geq (1 - c)|A|_{\mathcal{R}}$ and
\begin{align*}
    |A|_{\mathcal{P}}\cdot|A|_{\mathcal{Q}} \geq \frac{c^2}{4}|A|_{\mathcal{R}}\cdot|A'|_{\mathcal{S}}.
\end{align*}
Moreover, the subset $A'$ can be taken as intersection of $A$ with disjoint union of atoms $\mathcal{S}(A)$. 
\end{lemma}

\section{Preparation IV: Irreducible representations of semisimple Lie groups}\label{sec:prepareIrrep}
We discuss properties of irreducible representations of semisimple Lie groups in this section. We remark that the notations in this section is compatible with the notations for $H = \SO(Q_1)^\circ$ or $\SO(Q_2)^\circ$ with irreducible representation $\mathfrak{r}$ introduced in Section~\ref{sec:Global prelim}. Let $\mathbf{H}$ be a connected semisimple $\R$-group and let $H = \mathbf{H}(\R)^\circ$ be the identity component of its $\R$-points under the Hausdorff topology. Suppose $H$ is noncompact. Let $\LieH = \mathrm{Lie}(H)$ be its Lie algebra. Fix a maximal split $\R$-torus $\mathbf{A}$ in $\mathbf{H}$. Let $\LieA$ be its corresponding Lie algebra. Let $\Phi \subset \LieA^*$ be the associated restricted root system. Let $\Phi^\pm \subset \Phi$ be sets of positive and negative roots with respect to some lexicographic order on $\LieA^*$ and $\Pi \subset \Phi^+$ be the set of simple roots. Let $\LieA^+ \subset \LieA$ be the corresponding closed positive Weyl chamber. Then, we have the restricted root space decomposition
\begin{align*}
\LieH = \LieA \oplus \LieM_0 \oplus \LieU^+ \oplus \LieU^- = \LieA \oplus \LieM_0 \oplus \bigoplus_{\alpha \in \Phi} \LieH_\alpha
\end{align*}
where $\LieM_0 = Z_{\LieK}(\LieA) \subset \LieK$ and $\LieU^\pm = \bigoplus_{\alpha \in \Phi^+} \LieH_{\pm\alpha}$. Define the Lie subgroups
\begin{align}
\label{eqn:SubgroupsOfG}
A &= \exp(\LieA) < G, & U^\pm &= \exp(\LieU^\pm) < G, & M_0 = Z_K(A) < K < H.
\end{align}
Define the closed subset $A^+ = \exp(\LieA^+) \subset A$. Denote
\begin{align*}
a_v = \exp(v) \in A, \qquad \text{$v \in \LieA$}.
\end{align*}
The middle two subgroups in \cref{eqn:SubgroupsOfG} are the maximal expanding and contracting horospherical subgroups in $H$, i.e.,
\begin{align*}
U^\pm = \Bigl\{u^\pm \in H: \lim_{t \to \pm\infty} a_{-tv}u^\pm a_{tv} = e\Bigr\}
\end{align*}
for any $v \in \interior(\LieA^+)$. We often denote $U := U^+$. Note that the set $U^+M_0AU^-$ is an open dense subset of $H$. 

we fix a norm $\|\cdot\|$ on $\LieH$ and for any subalgebra $\mathfrak{s} \subseteq \LieH$ let
\begin{align*}
    B^{\mathfrak{s}}_r(x_0) = \{x \in W: \|x - x_0\| \leq r\}.
\end{align*}
The choice of the norm $\|\cdot\|$ will only affect the result in this part by a constant factor. We set $\mathsf{B}^S_r = \exp(B^{\mathfrak{s}}_r(0))$ and $m_S$ is the left invariant Haar measure on $S$ so that $m_S(\mathsf{B}^S_1) = 1$. 

Let $(\rho, V)$ be an irreducible representation of $\mathbf{H}$. For weights $\lambda$ associated to $\LieA$, we use $V_\lambda$ to denote the corresponding weight space. By the fixed choice of positive roots, we have a partial order on the set of weights. As for $\mathfrak{r}$, we denote $V^{(\lambda)} = \oplus_{\mu \geq \lambda} V_{\mu}$. The representation has the following property. 
\begin{theorem}\label{thm:good basis}
    There exists a $K$-invariant inner product on $V$ so that for all $a \in A$, $\rho(a)$ is symmetric. 
\end{theorem}
\begin{proof}
    This is a direct consequence of Mostow's simultaneous Cartan decomposition theorem, see \cite[Theorem 6]{Mos55}. 
\end{proof}
With the above theorem, we can find an orthonormal basis of $V$ so that all elements $a \in A$ acts diagonally and $\rho(U)$ consists of strictly upper-triangular matrices and  $\rho(U^-)$ consists of strictly lower-triangular matrices. For all $h \in H$, the matrix transpose $\rho(h)^t$ is the adjoint operator of $\rho(h)$ with respect to this inner product. 

We will only consider the case where $\dim \Fix(U) = \dim \Fix(U^-) = 1$. In this case, there is a highest weight $\chi$. 

\begin{remark}
    There exists irreducible representation of semisimple $\R$-groups with $\dim_{\R} \Fix(U) > 1$. For example, the adjoint representation of $\SO(n, 1)$ is irreducible but $\dim_{\R} \Fix(U) = \dim_{\R} U = n - 1$. 
\end{remark} 

\subsection*{Constants and \texorpdfstring{$\star$-notations}{⋆-notations}}
Since we will discuss results on irreducible representation of general semisimple Lie groups in this part, we make the following convention on implied constants and $\star$-notations. For $A \ll B^{\star}$, we mean there exist constants $C > 0$ and $\kappa > 0$ depend at most on $H$ and the representation $V$ such that $A \leq C B^{\kappa}$. For $A \ll_D B$, we mean there exist constant $C_D > 0$ depending on $D$ and at most on $H$ and the representation $V$ so that $A \leq C_D B$. We will apply those results to $H$ and $\mathfrak{r}$. It is compatible with our previous convention. 

\subsection{Projections}\label{subsec:Equivalent Projections}
For any representation $V$ of $H$ in this paper, we fix an inner product from \cref{thm:good basis}. For a subspace $W \subseteq V$, we set $\pi_W$ to be the orthogonal projection to $W$ with respect to this inner product. For a linear operator $A$ on $V$, we use $(A)^t$ to denote the adjoint operator of $A$ under this inner product. In the particular representation $\mathfrak{r}$, we take the inner product on $\mathfrak{r}$ by the restricting of the inner product on $\LieG = \mathfrak{sl}_4$ defined by the Cartan involution $\theta: x \mapsto -x^t$. Under this inner product, the matrix transpose $h^t$ acting on the representation $\mathfrak{r}$ is the adjoint action of $h$. Later in this paper we will use this inner product and $h^t$ stands for the matrix transpose of $h$. 

Recall that in the introduction in Part~\ref{part:projection}, we define the projections
\begin{align*}
    \pi_{r, s}^{(\lambda)} = \pi^{(\lambda)} \circ u_{r, s}.
\end{align*}
In general, for an irreducible representation $V$ of a semisimple Lie group $H$, we can define
\begin{align*}
    \pi_{u}^{(\lambda)} = \pi^{(\lambda)} \circ u
\end{align*}
where $\pi^{(\lambda)}$ is the orthogonal projection to $V^{(\lambda)}$ under the above inner product and $u \in U$ is an element of the horospherical subgroup of $H$ defined in the previous section. There are also the following closely related orthogonal projections:
\begin{align*}
    \pi_{u^t.V^{(\lambda)}}.
\end{align*}
This is the orthogonal projection to the subspace $u^t.V^{(\lambda)}$. We define the following linear map
\begin{align*}
    f: V^{(\lambda)} &\to V^{(\lambda)}\\
    w &\mapsto \pi_{V^{(\lambda)}}((uu^t).w).
\end{align*}

\begin{lemma}
The linear map $f$ satisfies the following properties. 
    \begin{enumerate}
        \item We have $\pi^{(\lambda)}_u = f \circ (u^{-1})^t\pi_{u^t.V^{(\lambda)}}$.
        \item There exists $\beta > 0$ depends only on $H$ so that for all $u \in B_\beta^U$, the map $f$ is an invertible linear map with 
        \begin{align*}
            \max\{\|f\|, \|f^{-1}\|\} \ll 1
        \end{align*}
        where the constant depends only on the ambient representation. 
    \end{enumerate}
\end{lemma}
\begin{proof}
For property~(1), note that we have the following orthogonal decomposition
\begin{align*}
    V = u^t. V^{(\lambda)} \oplus u^{-1}. (\oplus_{\mu < \lambda} V_\mu).
\end{align*}
If we write $v = u^t.w + u^{-1}.w'$ where $w \in V^{(\lambda)}$ and $w' \in \oplus_{\mu < \lambda} V_\mu$, then we have
\begin{align*}
    \pi^{(\lambda)}_u(v) = \pi_{V^{(\lambda)}}((uu^t).w),\\
    (u^{-1})^t\pi_{u^t.V^{(\lambda)}}(v) = w,
\end{align*}
which proves property~(1). 

For property~(2), note that both $\|f\|^2$ and $\det f$ are polynomial on $u$ and $u \in B_1^U$, it suffices to show that $f$ is invertible for all $u \in B_1^U$. Suppose $f(w) = 0$ for some $w \in V^{(\lambda)}$. Then we have $uu^t.w \in \oplus_{\mu < \lambda} V_\mu$ and in particular, 
\begin{align*}
    \langle uu^t.w , w \rangle = 0
\end{align*}
where $\langle \cdot , \cdot \rangle$ is the inner product compatible to the weight space decomposition chosen in the beginning of the section. Note that $u^t$ is the adjoint operator of $u$ under this inner product, we have
\begin{align*}
    \langle u^t.w , u^t.w \rangle = 0
\end{align*}
which implies $u^t.w = 0$ hence $w = 0$. This shows that $f$ is injective and therefore invertible. 
\end{proof}

The above lemma implies that the projections $\pi^{(\lambda)}_u$ and $\pi_{u^t.V^{(\lambda)}}$ differs by a bi-Lipschitz map. Moreover, when we pick $u \in B_1^U$, the Lipschitz constants depend only on the ambient representation. Therefore, the estimate on covering numbers after projections $\pi^{(\lambda)}_u$ and $\pi_{u^t.V^{(\lambda)}}$ are equivalent up to an absolute constant. We will not distinguish them in this paper. 

\subsection{Non-degenerate measures on Grassmannians}\label{subsec:nondegeneracydef}
Recall that we identify $V$ with $\R^n$ under the basis given by Theorem~\ref{thm:good basis}. For two subspaces $U$ and $W$ of $V$, we define
\begin{align*}
    d_{\measuredangle}(U, W) = \|u_1 \wedge \cdots \wedge u_k \wedge w_1 \wedge \cdots \wedge w_l\|
\end{align*}
where $\{u_i\}_{i = 1}^k$ and $\{w_j\}_{j = 1}^l$ are orthonormal basis of $U$ and $W$ respectively. This is independent to the choice of $\{u_i\}_{i = 1}^k$ and $\{w_j\}_{j = 1}^l$. Similarly, for subspaces $V_1, \ldots, V_q$ of $V$, we define
\begin{align*}
    d_{\measuredangle}(V_1, \ldots, V_q) = \|\mathbf{v}_1 \wedge \cdots \wedge \mathbf{v}_q\|
\end{align*}
where $\mathbf{v}_i$'s are wedge of an orthonormal basis of $V_i$. This is independent to the choice of $\{\mathbf{v}_i\}_{i = 1}^q$. 

Let $W \in \Gr(n, n - k)$, we define
\begin{align*}
    \mathcal{V}(W, \rho) = \{U \in \Gr(n, k): d_{\measuredangle}(U, W) \leq \rho\}.
\end{align*}
If $\rho = 0$, $\mathcal{V}(W, 0)$ is the collection of $k$-dimensional subspaces intersecting $W$ non-trivially. It belongs
to the class of algebraic subvarieties of the grassmannian known as Schubert varieties. 

\begin{definition}[$(C, \kappa)$-non-degeneracy]
    
For a probability measure $\sigma$ on $\Gr(n, m)$, we say it satisfies $(C, \kappa)$-non-degeneracy condition at scales larger than $\delta$ if the following holds. 

There exist constants $C \geq 1$, $\kappa > 0$ such that for all $\rho \geq \delta$ and all $W \in \Gr(n, n - m)$, one has
\begin{align}\label{eqn:Non-degenerate}
    \sigma(\mathcal{V}(W, \rho)) \leq C\rho^{\kappa}.
\end{align}
\end{definition}
\begin{remark}
Most literature use the terminology non-concentration condition. We use the terminology non-degeneracy here to distinguish it from the non-concentration condition on the set or the measure in the representation space $V$. Also, due to the polynomial nature of unipotent flow, this condition corresponds to non-degeneracy for some polynomials, as we will show in the next lemma. 
\end{remark}

In practice, we always allow $C = O(\delta^{-O(\epsilon)})$. 
We will say a family of subspaces satisfies the non-degeneracy condition if the measure and scale are clear in the context. 

In this note, we will only consider the following family of subspaces. Let $(\psi, V)$ be an irreducible representation of semisimple Lie group $H$. Recall that we set $V^{(\lambda)} = \oplus_{\mu \geq \lambda} V_{\mu}$. We will mainly consider the family $\{u^t.V^{(\lambda)}\}_{u \in U}$. The associated measure is the push forward of $m_U|_{\mathsf{B}^U_1}$. The following lemma relate the non-degeneracy condition of this measure to non-degeneracy of some polynomial. 

\begin{lemma}\label{lem:angle proportional to poly}
    For all $r > 0$, there exists a constant $A_r > 1$ depending only on $r$, $V$ and $H$ so that the following holds.  For all $q \geq 2$ and subspaces $V_1, \ldots, V_q$ of $V$, there exists a polynomial $P$ on $H^q$ so that for all $(h_1, \ldots, h_q) \in (B_{r}^{H})^q$
    \begin{align*}
        \frac{1}{A_r}d_{\measuredangle}(h_1.V_1, \ldots, h_q.V_q)^2 \leq P(h_1, \ldots, h_q) \leq A_r d_{\measuredangle}(h_1.V_1, \ldots, h_q.V_q)^2.
    \end{align*}
\end{lemma}

\begin{proof}
    Let $\mathbf{v}_i$ be the wedge of an orthonormal basis of $V_i$ and let 
    \begin{align*}
        P(h_1, \ldots, h_q) = \|h_1.\mathbf{v}_1\wedge \cdots \wedge h_q.\mathbf{v}_q\|^2.
    \end{align*}
    The rest follows from the fact that $h_i$'s are invertible and $(B_r^{H})^q$ is relatively compact. 
\end{proof}

For the families of subspaces $u^-.V^{(\lambda)}$, we have the following lemma. It says for generic $u_1^-$ and $u_2^-$, $u_1^-.V^{(\lambda)}$ and $u_2^-.V^{(\lambda)}$ are in general position. 
\begin{lemma}\label{lem:General two piece transverse}
    We have the following properties for the family of subspaces $\{u^-.V^{(\lambda)}\}$. 
    \begin{enumerate}
        \item If $2\dim V^{(\lambda)} \leq \dim V$, then the set of $(u_1, u_2)$ so that 
        \begin{align*}
            u_1^-.V^{(\lambda)} \cap u_2^-.V^{(\lambda)}\neq \{0\}
        \end{align*}
        is a proper Zariski closed subset of $U^-$. 
        \item If $2\dim V^{(\lambda)} > \dim V$, then the set of $(u_1, u_2)$ so that 
        \begin{align*}
            u_1^-.V^{(\lambda)} + u_2^-.V^{(\lambda)} \neq V
        \end{align*}
        is a proper Zariski closed subset of $U^-$.  
    \end{enumerate}
\end{lemma}
\begin{proof}
    We first prove property~(1). Let $\mathbf{v}$ be the wedge of an orthonormal basis of $V^{(\lambda)}$. By Lemma~\ref{lem:angle proportional to poly},  $d_{\measuredangle}(u_1^-. V^{(\lambda)}, u_2^-. V^{(\lambda)})^2$ is proportional to $\|u_1^-.\mathbf{v} \wedge u_2^-.\mathbf{v}\|^2$ which is polynomial in $u_1^-$ and $u_2^-$. It suffices to show the latter is a non-zero polynomial. Suppose not, then for all $u^- \in U^-$, we have
    \begin{align*}
        \|u^-.\mathbf{v} \wedge \mathbf{v}\|^2 = 0
    \end{align*}
    Consider the following Zariski closed subset $\mathcal{V}$ of $\mathbf{H}(\R)$:
    \begin{align*}
        \mathcal{V} = \{h \in \mathbf{H}(\R): \|h.\mathbf{v} \wedge \mathbf{v}\|^2 = 0\} = \{h \in \mathbf{H}(\R): h.V^{(\lambda)} \cap V^{(\lambda)} \neq \{0\}\}.
    \end{align*}
    Note that since $V^{(\lambda)}$ is sum of weight spaces, $A$ and $M$ leave $V^{(\lambda)}$ invariant. Moreover, since $V^{(\lambda)} = \oplus_{\mu \geq \lambda} V_{\mu}$, the subgroup $U^+$ leaves $V^{(\lambda)}$ invariant. Therefore, for all $h \in U^-M_0AU^+$, it lies in $\mathcal{V}$. Since $U^-M_0AU^+$ is a Zariski dense subset of $H$, we have $\mathcal{V} = \mathbf{H}(\R)$. Let $w \in \mathbf{H}(\R)$ be a representative of the longest element in the Weyl group $W = \mathbf{N_H(A)}/\mathbf{C_H(A)}$. Note that since $V^{(\lambda)}$ is a sum of weight space associated to $\LieA$, $w.V^{(\lambda)}$ does not depend on the choice of the representative. Also, we have 
    \begin{align*}
        w.V^{(\lambda)} = \oplus_{\mu \leq -\lambda} V_{\mu}
    \end{align*}
    intersect $V^{(\lambda)}$ trivially by the dimension condition. This leads to a contradiction. 

    For property~(2), note that the condition $2\dim V^{(\lambda)} > \dim V$ is equivalent to $2\dim \oplus_{\mu < \lambda} V_\mu < \dim V$. Also, the conclusion that $u_1^-.V^{(\lambda)} + u_2^-.V^{(\lambda)} = V$ holds for generic $u_1^-, u_2^-$ is equivalent to the statement that
    \begin{align*}
        u_1.(\oplus_{\mu < \lambda} V_\mu) \cap u_2.(\oplus_{\mu < \lambda} V_\mu) = \{0\}
    \end{align*}
    holds for generic $u_1, u_2$. 

    Conjugating via the longest element in the Weyl group, the above proof of property~(1) works in same words if we replace $u^-.V^{(\lambda)}$ by $u^+.(\oplus_{\mu \leq \lambda} V_\mu)$. This completes the proof of property~(2). 
\end{proof}

\section{Projections to lines and hyperplanes in irreducible representations}\label{sec:dim 1 or codim 1}
This section is devoted to the subcritical estimates for projections to the families of lines of the shape $\{u^-.\ell\}_{u^- \in U^-}$ in irreducible representations of semisimple Lie groups. We also discuss its $\codim 1$ analog, projections to the families of hyperplanes of the shape $\{u^-.W\}_{u^- \in U^-}$. Roughly speaking, we provide algebraic criteria to the following estimates. For most of $u^-$, we have
\begin{align*}
    |\pi_{u^-.\ell}(A)|_\delta \geq |A|_{\delta}^{\frac{1}{\dim V}}, \quad |\pi_{u^-.W}(A)|_\delta \geq |A|_{\delta}^{\frac{\dim V - 1}{\dim V}}.
\end{align*}We will make use of the polynomial nature of actions by unipotent groups. 

As in Section~\ref{sec:prepareIrrep}, we set $H = \mathbf{H}(\R)^\circ$ where $\mathbf{H}$ is a semisimple connected real linear algebraic group and $V$ is an irreducible representation of $H$ with $\dim \Fix(U) = 1$. Let $\chi$ to be the highest weight of $V$. Note that $V_\chi = \Fix(U)$. We fix an inner product and a basis from \cref{thm:good basis} to identify $V$ with $\R^n$. Under this basis, the weight spaces are orthogonal and $\rho(U^-) = \rho(U)^t$ where $(\cdot)^t$ is the matrix transpose. Therefore, the families $\{u^-.\ell\}_{u^- \in U^-}$ and $\{u^-.W\}_{u^- \in U^-}$ are the same as $\{u^t.\ell\}_{u \in U}$ and $\{u^t.W\}_{u \in U}$ respectively. We will mainly use the latter notations.

Recall that we set $\mathsf{B}^U_1 = \exp(B^{\LieU}_1(0))$ and $\mathsf{B}^{U^-}_1 = \exp(B^{\LieU^-}_1(0))$. We also set $m_U$ and $m_{U^-}$ to be the Haar measure on $U$ and $U^-$ respectively so that $m_U(\mathsf{B}^U_1) = 1$ and $m_{U^-}(\mathsf{B}^{U^-}_1) = 1$. 

The following two theorems are the main results of this section. 

\begin{theorem}\label{thm:subcritical 9 to 1}
    Let $v \in V$ be a nonzero unit vector satisfying 
    \begin{align*}
        \pi_{V_{\chi}}(v) \neq 0.
    \end{align*}
    Then there exist $M > 1$ depending only on $V$ and $E > 1$ depending only on the dimension of $U$ and $V$ so that the following holds.  
    
    For all $0 < \epsilon \ll 1$, $\delta \ll_\epsilon \|\pi_{V_\chi}(v)\|^{\frac{E}{\epsilon}}$ and $A \subseteq B^{\R^n}_1(0)$, we define the following exceptional set:
    \begin{align*}
        \mathcal{E}(A) = \biggl\{u \in \mathsf{B}^U_1: \exists A' \subseteq A \text{ with } |A'|_\delta \geq \delta^\epsilon|A|_\delta \text{ and } |\pi_{u^t.\R v}(A')|_\delta < \delta^{M\epsilon}|A|_{\delta}^{\frac{1}{n}} \biggr\}.
    \end{align*}

    We have
    \begin{align*}
        m_U(\mathcal{E}(A)) \leq \delta^{\epsilon}.
    \end{align*}
\end{theorem}
The codimension $1$ analog of Theorem~\ref{thm:subcritical 9 to 1} is of the following. 
\begin{theorem}\label{thm:subcritical 9 to 8}
    Let $W \in \Gr(n, n - 1)$ be a hyperplane with unit normal vector $\nu$ satisfying 
    \begin{align*}
        \pi_{V_{-\chi}}(\nu) \neq 0.
    \end{align*}    
    Then there exist $M > 1$ depending only on $V$ and $E > 1$ depending only on the dimension of $U$ and $V$ so that the following holds.  
    
    For all $0 < \epsilon \ll 1$, $\delta \ll_\epsilon \|\pi_{V_{-\chi}}(\nu)\|^{\frac{E}{\epsilon}}$ and $A \subseteq B^{\R^n}_1(0)$, we define the following exceptional set:
    \begin{align*}
        \mathcal{E}(A) = \biggl\{u \in \mathsf{B}^{U}_1: \exists A' \subseteq A \text{ with } |A'|_\delta \geq \delta^\epsilon|A|_\delta \text{ and } |\pi_{u^t.W}(A')|_\delta < \delta^{M\epsilon}|A|_{\delta}^{\frac{n - 1}{n}} \biggr\}.
    \end{align*}

    We have
    \begin{align*}
        m_U(\mathcal{E}(A)) \leq \delta^{\epsilon}.
    \end{align*}
\end{theorem}

Note that $\dim \mathfrak{r}^{(2)} = 1$ and $\dim \mathfrak{r}^{(-1)} = 8 = \dim \mathfrak{r} - 1$, the results in this subsection hold for the families of projections $\{\pi_{u^t.\mathfrak{r}^{(\lambda)}}\}_{u \in \mathsf{B}_1^{U}}$ where $\lambda = 2$ or $\lambda = -1$. By the discussions in Subsection~\ref{subsec:Equivalent Projections}, same subcritical estimates hold for the families of projections $\{\pi^{(\lambda)}_{u} = \pi^{(\lambda)} \circ u\}_{u \in \mathsf{B}^U_1}$ where $\lambda = 2$ or $\lambda = -1$. 

We now proceed the proof of \cref{thm:subcritical 9 to 1,thm:subcritical 9 to 8}. Recall from Subsection~\ref{subsec:nondegeneracydef}, we say a probability measure $\sigma$ on $\Gr(n, m)$ satisfies $(C, \kappa)$-non-degeneracy condition at scales larger than $\delta$ if there exist constants $C \geq 1$ and $\kappa > 0$ so that for all $W \in \Gr(n, n - m)$, 
\begin{align*}
    \sigma(\mathcal{V}(W, \rho)) \leq C\rho^{\kappa}, \quad \forall \rho \geq \delta.
\end{align*}
Under the setting of \cref{thm:subcritical 9 to 1,thm:subcritical 9 to 8}, we will show that the push-forward of $m_U|_{\mathsf{B}^U_1}$ via $u \mapsto u^{t}.\R v \in \Gr(n, 1)$ or $u \mapsto u^{t}.W \in \Gr(n, n - 1)$ satisfies the non-degeneracy condition. The rest follows from \cite[Proposition 29]{He20} recorded in the following proposition. 

\begin{proposition}[{\cite[Proposition 29]{He20}}]\label{pro:He subcritical}
Given $0 < m \leq n$, $0 < \alpha < n$ and $\kappa > 0$, there exists $M > 1$
such that for all $0 < \epsilon < \kappa/M$, the following is true for all $\delta > 0$ sufficiently small depending on $\epsilon$. 

Let $A \subseteq \R^n$ be a subset contained in the unit ball and $\sigma$ a probability measure on
$\Gr(n, m)$. Let the exceptional set be defined as the following:
\begin{align*}
    \mathcal{E}(A) = \biggl\{V \in \Gr(n,m):\exists A' \subseteq A \text{ with } |A'|_\delta \geq \delta^\epsilon|A|_\delta \text{ and } |\pi_{V}(A')|_\delta < \delta^{M\epsilon}|A|_{\delta}^{\frac{m}{n}}\biggr\}.
\end{align*}

If $m < n$, suppose $\sigma$ satisfies $(\delta^{-\epsilon}, \kappa)$-non-degeneracy condition for all scales larger than $\delta$. Then
\begin{align*}
    \sigma(\mathcal{E}(A)) \leq \delta^\epsilon.
\end{align*}
\end{proposition}

We now provide two criteria of non-degeneracy for family of dimension $1$ or codimension $1$ subspaces in irreducible representation $V$ under the condition $\dim \Fix(U) = 1$. \cref{thm:subcritical 9 to 1,thm:subcritical 9 to 8} will be direct consequences of the following criteria and Proposition~\ref{pro:He subcritical}. 
\begin{theorem}\label{thm:Non-degenerate-dim1}
    Let $v \in V$ be a unit vector satisfying 
    \begin{align*}
        \pi_{V_{\chi}}(v) \neq 0.
    \end{align*}
    Consider the lines $\{u^t.\R v\}_{u \in \mathsf{B}^U_1} \subset \Gr(V, 1)$ and let the measure $\sigma$ be the push-forward of $m_U|_{\mathsf{B}^U_1}$ under the map $u \mapsto u^t.\R v$. 
    
    Then $\sigma$ satisfies $(C, \kappa)$-non-degeneracy condition for some $C = O(\|\pi_{V_{\chi}}(v)\|^{-\star})$ and $\kappa$ depending only on the dimension of $U$ and $V$. The implied constant depends only on the ambient representation. 
\end{theorem}

\begin{theorem}\label{thm:Non-degenerate-codim1}
    Let $W \in \Gr(n, n - 1)$ be a hyperplane with normal vector $\nu$ satisfying 
    \begin{align*}
        \pi_{V_{-\chi}}(\nu) \neq 0.
    \end{align*}
    Consider the family of hyperplanes $\{u^t.W\}_{u \in \mathsf{B}^U_1} \subset \Gr(V, \dim V - 1)$ and let the measure $\sigma$ be the push-forward of $m_U|_{\mathsf{B}^U_1}$ under the map $u \mapsto u^t.W$. 
    
    Then $\sigma$ satisfies $(C, \kappa)$-non-degeneracy condition for some $C = O(\|\pi_{V_{-\chi}}(\nu)\|^{-\star})$ and $\kappa$ depending only on the dimension of $U$ and $V$. The implied constant depends only on the ambient representation. 
\end{theorem}

The idea of the above criteria is straight-forward. Note that for a hyperplane $W$ with normal vector $w$, by Lemma~\ref{lem:angle proportional to poly}, on $\mathsf{B}^{U}_1$ we have
\begin{align*}
    d_{\measuredangle}(u^t.\R v, W) \asymp \langle u^t.v, w \rangle.
\end{align*}
Due to the polynomial nature of actions of unipotent groups, we need an estimate on the size of the set where
the polynomial function $(u \mapsto \langle u^t.v, w \rangle)$ is small. This is known as Remez's inequality and is used by Kleinbock and Margulis and later Kleinbock and Tomanov in \cite{KM98,KT07} to verify the '$(C, \alpha)$-good' property. We record the form we need in the following lemma. 

\begin{lemma}[\text{\cite[Lemma 3.4]{KT07}}]\label{lem:RemezKT}
    For all $d, k \in \N$, there exists a constant $C = C_{d, k} > 0$ so that the following holds. Let $P \in \R[x_1, \ldots, x_d]$ be a polynomial with degree at most $k$. For all ball $B \subset \R^d$ and $\epsilon > 0$, we have
    \begin{align*}
        \Leb\{x \in B:|P(x)| < \epsilon\} \leq C\biggl(\frac{\epsilon}{\|P\|_{L^\infty(B)}}\biggr)^{\frac{1}{dk}}\Leb(B).
    \end{align*}
\end{lemma}

By Remez's inequality, it suffices to estimate the supreme of the polynomial $\langle u^t.v, w \rangle$ on $\mathsf{B}^{U}_1$, which is done in the following lemma. It is a variant of \cite[Lemma 5.1]{Sh96}, see also \cite[Lemma 3.1, 3.2]{Kat23}. Recall that we fix an inner product $\langle \cdot, \cdot \rangle$ and a basis of the representation $(\rho, V)$ from Theorem~\ref{thm:good basis} so that the weight spaces are orthogonal. Under this basis, $\rho(U^-) = \rho(U)^t$ where $(\cdot)^t$ is the matrix transpose. 

\begin{lemma}\label{lem:VariantOfShah}
    Suppose $v, w$ are unit vectors in $V$. We have
    \begin{align*}
        \sup_{u^- \in \mathsf{B}^{U^-}_1} \langle u^-.v, w \rangle \gg \|\pi_{V_\chi}(v)\|^{\dim V}.
    \end{align*}
    Equivalently, we have
    \begin{align*}
        \sup_{u \in \mathsf{B}^U_1} \|\pi_{\R v}(u.w)\| = \sup_{u \in \mathsf{B}^U_1} \langle v, u.w \rangle \gg \|\pi_{V_\chi}(v)\|^{\dim V}.
    \end{align*}
\end{lemma}

\begin{proof}
The proof is also a variant of \cite[Lemma 5.1]{Sh96}. We include it for completeness. We will show that 
\begin{align*}
    \sup_{u^- \in \mathsf{B}^U_1} \langle u^-.v, w \rangle \gg \|\pi_{V_\chi}(v)\|^{\dim V}.
\end{align*}

For any Lie algebra $\mathfrak{s}$, let $\mathcal{U}(\mathfrak{s})$ be the universal enveloping algebra of $\mathfrak{s}$. Let $e_\chi \in V_\chi$ be a unit vector so that $\langle v, e_\chi\rangle = \|\pi_{V_\chi}(v)\|$. Let $v_{< \chi} = v - \langle v, e_\chi\rangle e_\chi$ be the orthogonal projection of $v$ to $\oplus_{\lambda < \chi} V_{\lambda}$. 

Write $\Phi^+ = \{\alpha_1, \ldots, \alpha_{l}\}$. Later when we write product over $\alpha \in \Phi^+$, we refer to this order. For all positive root $\alpha \in \Phi^+$, let $\LieU^{-}_\alpha = \LieH_{-\alpha}$. We have $\LieU^- = \oplus_{\alpha \in \Phi^+} \LieU^{-}_\alpha$. Suppose $\dim \LieU^{-}_\alpha = m_{\alpha}$. Let $\{\mathpzc{z}_{\alpha, k}\}_{k = 1}^{m_\alpha}$ be an orthonormal basis of $\LieU^{-}_\alpha$. We introduce the following multi-index
\begin{align*}
    I_{\alpha} = (i_{\alpha, k})_{k = 1, \ldots, m_{\alpha}}, \quad J = (I_{\alpha})_{\alpha \in \Phi^+} = (I_{\alpha_1}, \ldots, I_{\alpha_l}).
\end{align*}
For $t_{\alpha} = (t_{\alpha, 1}, \ldots, t_{\alpha, m_{\alpha}}) \in \R^{m_{\alpha}}$, we define
\begin{align*}
    t_{\alpha}^{I_{\alpha}} = t_{\alpha, 1}^{i_{\alpha, 1}} \cdots t_{\alpha, m_{\alpha}}^{i_{\alpha, m_{\alpha}}}, \qquad \mathpzc{z}_{\alpha}^{I_\alpha} = \mathpzc{z}_{\alpha, 1}^{i_{\alpha, 1}} \cdots \mathpzc{z}_{\alpha, m_{\alpha}}^{i_{\alpha, m_{\alpha}}} \in \mathcal{U}(\LieU^-).
\end{align*}
For all $t = (t_\alpha)_{\alpha \in \Phi^+} = (t_{\alpha_1}, \ldots, t_{\alpha_l})$ and $J \in \mathcal{J}$, we define
\begin{align*}
    t^J = \prod_{\alpha \in \Phi^+} t_\alpha^{I_\alpha}, \qquad \mathpzc{z}^J = \prod_{\alpha \in \Phi^+} \mathpzc{z}_\alpha^{I_\alpha} = \mathpzc{z}_{\alpha_1}^{I_{\alpha_1}} \cdots \mathpzc{z}_{\alpha_{l}}^{I_{\alpha_{l}}} \in \mathcal{U}(\LieU^-).
\end{align*}
By Poincar{\'e}--Birkhoff--Witt's theorem, $\{\mathpzc{z}^J\}_J$ forms a basis of $\mathcal{U}(\LieU^-)$. 

Note that we have
\begin{align*}
    V = \mathcal{U}(\LieU^-).e_\chi.
\end{align*}
There exists a finite set $\mathcal{J}$ of multi-indices $J$ so that $\{\mathpzc{z}^J.e_\chi\}_{J \in \mathcal{J}}$ forms a basis of $V$. For all $u^- \in U^-$ that can be written as
\begin{align*}
    u^- = \prod_{\alpha \in \Phi^+} \prod_{k = 1}^{m_{\alpha}} \exp(t_{\alpha, k}\mathpzc{z}_{\alpha, k}) = \sum_{J \in \mathcal{J}} t^J \mathpzc{z}^J,
\end{align*}
we calculate $\langle u^-.v, w\rangle$ as the following. 
\begin{align*}
    \langle u^-.v, w\rangle = \sum_{J  \in \mathcal{J}} t^J\langle\mathpzc{z}^J.v, w\rangle.
\end{align*}
Consider the map 
\begin{align*}
    T:V &\to \R^{\mathcal{J}}\\
    w &\mapsto (\langle\mathpzc{z}^J.v, w\rangle)_{J \in \mathcal{J}}.
\end{align*}
We have $\|T\| \ll 1$. The partial order on the set of weights associated to $V$ defined by $\Phi^+$ ensures that $T$ can be written as an upper-triangular matrix with diagonal entries $\|\pi_{V_\chi}(v)\|$. Therefore, $|\det T| \gg \|\pi_{V_\chi}(v)\|^{\dim V}$ and 
\begin{align*}
    \|T(w)\| \gg \|\pi_{V_\chi}(v)\|^{\dim V}\|w\|.
\end{align*}
This implies that $\langle u^-.v, w\rangle$ is a polynomial with maximum coefficient $\gg \|\pi_{V_\chi}(v)\|^{\dim V}$ and 
\begin{align*}
    \sup_{u^- \in \mathsf{B}_1^{U^-}} \langle u^-.v, w \rangle \gg \|\pi_{V_\chi}(v)\|^{\dim V}.
\end{align*}
\end{proof}

\begin{proof}[Proof of Theorem~\ref{thm:Non-degenerate-dim1}]
For all $W \in \Gr(n, n - 1)$, let $w$ be its normal vector. Note that by Lemma~\ref{lem:angle proportional to poly}, on $B^U_1$ we have
\begin{align*}
    d_{\measuredangle}(u^t.\R v, W)^2 \asymp \langle u^t.v, w \rangle
\end{align*}
where $\langle u^t.v, w \rangle$ is a polynomial on $u$. Lemma~\ref{lem:VariantOfShah} implies that 
\begin{align*}
    \sup_{u \in B_1^U} |\langle u^t.v, w \rangle| \gg \|\pi_{V_{\chi}}(v)\|^{\dim V}.
\end{align*}
Remez's inequality (Lemma~\ref{lem:RemezKT}) implies that $\sigma$ satisfies a $(C, \kappa)$-non-degeneracy condition for some $C = O(\|\pi_{V_{\chi}}(v)\|^{-\star})$ and $\kappa$ depends only on the dimension of $U$ and $V$. 
\end{proof}

\begin{proof}[Proof of Theorem~\ref{thm:Non-degenerate-codim1}]
    Note that all hyperplanes $\{u^t.W\}_{u \in U}$ containing a line is the same as the normal vectors $\{u.\nu\}_{u \in U}$ lies in the orthogonal hyperplane of that line. The rest follows from the same line as the previous \cref{thm:Non-degenerate-dim1}. 
\end{proof}

\section{\texorpdfstring{Subcritical estimates for projections to $\mathfrak{r}^{(1)}$ and $\mathfrak{r}^{(0)}$}{Subcritical estimates for projections to \unichar{"1D597}\unichar{"005E}(1) and \unichar{"1D597}\unichar{"005E}(0)}}\label{sec:Subcritical}
This section is devoted to the subcritical estimates for the families of projections $\{\pi^{(\lambda)}_{u}\}_{u \in \mathsf{B}^U_1}$ where $\lambda = 0, 1$. These are the cases with algebraic obstructions so that the method in the previous section does not work. We will make use of the properties of the specific representation $\mathfrak{r}$. Recall that $\dim \mathfrak{r}^{(1)} = 3$ and $\dim \mathfrak{r}^{(0)} = 6$. The subcritical estimates we expect are
\begin{align*}
    |\pi^{(1)}_{u}(A)|_\delta \geq |A|_{\delta}^{\frac{3}{9}}, \quad |\pi^{(0)}_{u}(A)|_\delta \geq |A|_{\delta}^{\frac{6}{9}}.
\end{align*}

The following two theorems are the main results of this section. Recall that $\mathsf{B}^U_1 = \exp(B_1^\LieU(0))$ and $m_U$ is the Haar measure on $U$ so that $m_U(\mathsf{B}^U_1) = 1$. 

\begin{theorem}\label{thm: subcritical 9 to 3}
    There exists $M$ depending only on the ambient representation so that the following holds for all $0 < \epsilon \ll 1$ and $\delta \ll_{\epsilon} 1$. 

    For all $A \subseteq B^{\mathfrak{r}}_1(0)$, we define the following exceptional set:
    \begin{align*}
        \mathcal{E}(A) = \biggl\{u \in \mathsf{B}^U_1: \exists A' \subseteq A \text{ with } |A'|_\delta \geq \delta^\epsilon|A|_\delta \text{ and } |\pi^{(1)}_{u}(A')|_{\delta} < \delta^{M\epsilon}|A|_\delta^{\frac{3}{9}}\biggr\}.
    \end{align*}

    We have
    \begin{align*}
        m_U(\mathcal{E}(A)) \leq \delta^{\epsilon}.
    \end{align*}
\end{theorem}

\begin{theorem}\label{thm: subcritical 9 to 6}
    There exists $M$ depending only on the ambient representation so that the following holds for all $0 < \epsilon \ll 1$ and $\delta \ll_{\epsilon} 1$. 

    For all $A \subseteq B^{\mathfrak{r}}_1(0)$, we define the following exceptional set:
    \begin{align*}
        \mathcal{E}(A) = \biggl\{u \in \mathsf{B}^U_1: \exists A' \subseteq A \text{ with } |A'|_\delta \geq \delta^\epsilon|A|_\delta \text{ and } |\pi^{(0)}_{u}(A')|_{\delta} < \delta^{M\epsilon}|A|_\delta^{\frac{6}{9}}\biggr\}.
    \end{align*}

    We have
    \begin{align*}
        m_U(\mathcal{E}(A)) \leq \delta^{\epsilon}.
    \end{align*}
\end{theorem}

\subsection{\texorpdfstring{Properties of the representation $\mathfrak{r}$}{Properties of the representation \unichar{"1D597}}}\label{subsec:specialrep}
This subsection is devoted to the study of $\mathfrak{r}_1$ and $\mathfrak{r}_2$. In this subsection $H = H_1 = \SO(Q_1)^\circ$ or $H = H_2 = \SO(Q_2)^\circ$. 

We first give a convenient coordinate of $\mathfrak{r}_1$. Using the coordinate from $\mathfrak{sl}_4(\R)$, we can write elements of $\mathfrak{r}$ as the following $4 \times 4$ matrices:
\begin{align}\label{eqn:coordinate so 2 2}
    \begin{pmatrix}
        A & B\\
        C & -A
    \end{pmatrix}
\end{align}
where $A, B, C \in \mathfrak{sl}_2$. We use $A^{\pm}$, $A^0$, $B^{\pm}$, $B^0$, $C^{\pm}$, $C^0$ to denote the corresponding subspaces to strictly upper(lower)-triangular matrices and diagonal matrices. In this coordinate, $\mathfrak{r}^{(1)}$ is spanned by $A^+$, $B^0$ and $B^+$. 

We now provide the algebraic obstruction for getting the optimal dimension estimate for projections $\{\pi^{(1)}_{r, s} = \pi^{(1)} \circ \Ad(u_{r, s})\}_{r, s \in [-1, 1]^2}$. 
\begin{example}\label{example:optimal fail for so 2 2}
Let $W$ be the subspace of $\mathfrak{r}$ as the following
\begin{align*}
W = \biggl\{\begin{pmatrix}
        0 & B\\
        0 & 0
    \end{pmatrix}: B \in \mathfrak{sl}_2(\R)\biggr\}.
\end{align*}
We have $\dim W = 3$. The action of $U_1$ leaves $W$ invariant. Therefore, $\pi_{r, s}^{(1)}(W) = \pi^{(1)}(W) = \R B^0 \oplus \R B^+$. We have $\dim \pi_{r, s}^{(1)}(W) = 2 < 3 =\min\{\dim W, \dim \mathfrak{r}^{(1)}\}$. This implies that the family of projections $\{\pi_{r, s}^{(1)}\}_{r, s}$ is never optimal. 

We now give a slightly more conceptual interpretation of $W$. As a representation of $\mathfrak{so}(2, 2) \cong \mathfrak{sl}_2(\R) \oplus \mathfrak{sl}_2(\R)$, $\mathfrak{r}$ is isomorphic to $\mathfrak{sl}_2(\R) \otimes \mathfrak{sl}_2(\R)$. Let $e$ be the fixed vector in $\mathfrak{sl}_2(\R)$ by the adjoint action of strictly upper triangular matrices. Then $W$ is identified to $\mathfrak{sl}_2(\R) \otimes \R e$. It is invariant under the action of $U_1$ but does not contain the expanding direction coming from the second copy of $\mathfrak{sl}_2(\R)$. 

The obstruction for $\{\pi^{(0)}_{r, s}\}_{r, s}$ can be constructed in a similar way. 
\end{example}

We now show that the non-degeneracy condition in the previous section does not hold for the family of subspace $\{u_{r, s}^t.\mathfrak{r}^{(1)}_1\}$. 
\begin{example}\label{example:nondegenerate fails}
    Let 
    \begin{align*}
        W = \Biggl\{\begin{pmatrix}
            0 & B\\
            C & 0
        \end{pmatrix}: B, C \in \mathfrak{sl}_2(\R)\Biggr\}.
    \end{align*}
    This is a $6$-dimensional subspace of $\mathfrak{r}$. We will show that 
    \begin{align*}
        u_{r, s}^t. \mathfrak{r}^{(1)} \cap W \neq \{0\}
    \end{align*}
    for all $r, s \in \R$. For simplicity, we write $u_r = \begin{pmatrix}
            1 & r\\
            0 & 1
        \end{pmatrix}$ in this example. 

    We can calculate that 
    \begin{align*}
        u_{r, s}^t. \mathfrak{r}^{(1)} = \Biggl\{\begin{pmatrix}
            u_r^t A^+ u_{-r}^t - s u_r^t B^{0, +} u_{-r}^t & u_r^t B^{0, +} u_{-r}^t\\
            s(2u_r^t A^+ u_{-r}^t - s u_r^t B^{0, +} u_{-r}^t)  & - u_r^t A^+ u_{-r}^t + s u_r^t B^{0, +} u_{-r}^t
        \end{pmatrix}\Biggr\}.
    \end{align*}

    Therefore, we have
    \begin{align*}
        \Biggl\{\begin{pmatrix}
            0 & u_r^t B^+ u_{-r}^t\\
            s^2 u_r^t B^{+} u_{-r}^t& 0
        \end{pmatrix}\Biggr\} \subseteq u_{r, s}^t. \mathfrak{r}^+ \cap W.
    \end{align*}
    This shows that $u_{r, s}^t.\mathfrak{r}^{(1)}$ lies in $\mathcal{V}(W, 0)$ for all $r, s$. 
\end{example}

We now give a convenient coordinate of $\mathfrak{r}_2$. Using the coordinate from $\mathfrak{sl}_4(\R)$, we can write elements of $\mathfrak{r}$ as the following $4 \times 4$ matrices:
\begin{align}\label{eqn:coordinate so 3 1}
    \begin{pmatrix}
        a_4 & a_2 & a_3 & a_1\\
        a_7 & a_5 & a_6 & -a_2\\
        a_8 & a_6 & -2a_4 - a_5 & -a_3\\
        a_9 & -a_7 & -a_8 & a_4
    \end{pmatrix}.
\end{align}
Under this coordinate, $\mathfrak{r}^{(1)}_2$ is spanned by $\R a_1 \oplus \R a_2 \oplus \R a_3$ and $\mathfrak{r}^{(2)}_2$ is spanned by $\R a_1 \oplus \cdots \oplus \R a_6$. The matrix of the adjoint action of $u_{r, s}$ under this coordinate can be written as a strictly upper-triangular matrix. 

Using this coordinate, one can also show that for the family of $3$-dimensional subspaces $\{u_{r, s}.\mathfrak{r}_2^{(1)}\}_{r, s}$, the non-degeneracy condition is not satisfied. One can also construct dual obstructions for the families of $6$-dimensional subspaces $\{u_{r, s}^t.\mathfrak{r}^{(0)}\}_{r, s}$. 

Nevertheless, the family of $3$-dimensional subspaces $\{u^t. \mathfrak{r}^{(1)}\}$ and the family of $6$-dimensional subspaces $\{u^t. \mathfrak{r}^{(0)}\}$ satisfies some weaker non-degeneracy condition recorded in the following four lemmas. 

Recall that we say two partitions $\mathcal{Q}$ and $\mathcal{P}$ are roughly equivalent with a parameter $L \geq 1$, and write $\mathcal{P} \overset{L}\sim \mathcal{Q}$ if each atom of $\mathcal{P}$ is contained in at most $L$ atoms in $\mathcal{Q}$ and vice versa. Recall that $\mathcal{D}_\delta$ is the partition of the ambient space by $\delta$-cubes. 

The following lemma is a consequence of Lemma~\ref{lem:General two piece transverse}. It says that $u_{1}^t. \mathfrak{r}^{(1)}$ and $u_{2}^t. \mathfrak{r}^{(1)}$ are transversal for generic $(u_1, u_2)$. Similar result holds for the family $\{u^t.\mathfrak{r}^{(0)}\}_u$. 
\begin{lemma}\label{lem:Special Calculation1}
There exist constant $E$ and polynomials $P_1, P_2$ on $U^2$ satisfying $\sup_{(\mathsf{B}_1^U)^2} |P_i| \gg 1$ for $i = 1, 2$ so that the following holds. 
\begin{enumerate}
    \item We have 
    \begin{align*}
        \{(u_1, u_2) \in U^2: \dim u_1^t.\mathfrak{r}^{(1)} + u_2^t.\mathfrak{r}^{(1)} = 6\} = \{P_1(u_1, u_2) \neq 0\}.
    \end{align*}
    Moreover, for $(u_1, u_2) \in (\mathsf{B}_1^U)^2$ so that $P_1(u_1, u_2) \geq c_1 > 0$,
    \begin{align*}
        \pi_{u_1^t.\mathfrak{r}^{(1)}}^{-1} \mathcal{D}_\delta \vee \pi_{u_2^t.\mathfrak{r}^{(1)}}^{-1} \mathcal{D}_\delta \overset{O(c_1^{-E})}{\sim}  \pi_{u_1^t.\mathfrak{r}^{(1)} \oplus u_2^t.\mathfrak{r}^{(1)}}^{-1} \mathcal{D}_\delta.
    \end{align*}
    \item We have
    \begin{align*}
        \{(u_1, u_2) \in U^2: \dim u_1^t.\mathfrak{r}^{(0)} + u_2^t.\mathfrak{r}^{(0)} = 9\} = \{P_2(u_1, u_2) \neq 0\}.
    \end{align*}
    Moreover, for $(u_1, u_2) \in (\mathsf{B}_1^U)^2$ so that $P_2(u_1, u_2) \geq c_2 > 0$,
    \begin{align*}
        \pi_{u_1^t.\mathfrak{r}^{(0)}}^{-1} \mathcal{D}_\delta \vee \pi_{u_2^t.\mathfrak{r}^{(0)}}^{-1} \mathcal{D}_\delta \overset{O(c_2^{-E})}{\sim} \pi_{u_1^t.\mathfrak{r}^{(0)} \oplus u_2^t.\mathfrak{r}^{(0)}}^{-1} \mathcal{D}_\delta.
    \end{align*}
\end{enumerate}
The constant $E$ depends only on dimension of $\mathfrak{r}$ and $U$. 
\end{lemma}

\begin{proof}
    By Lemma~\ref{lem:General two piece transverse}, such sets are Zariski open dense subsets in $U^2$. We now show the condition on partitions in property~(1) via the following construction of $P_1$. Property~(2) can be proved in a similar way. 

    Consider the map
    \begin{align*}
        T_{u_1, u_2}: \mathfrak{r} &\to \mathfrak{r}^{(1)} \times \mathfrak{r}^{(1)}\\
        v &\mapsto (\pi_{u_1}^{(1)}(v), \pi_{u_2}^{(1)}(v)).
    \end{align*}
    Using the coordinates in \cref{eqn:coordinate so 2 2,eqn:coordinate so 3 1}, $T_{u_1, u_2}$ can be written as a $6 \times 9$ matrix which we also denote by $T_{u_1, u_2}$. Let $P_1$ be the sum of squares of its $6 \times 6$ minors. Note that the columns of $T_{u_1, u_2}^t$ spans $u_1^t.\mathfrak{r}^{(1)} + u_2^t.\mathfrak{r}^{(1)}$, Lemma~\ref{lem:General two piece transverse} implies that $P_1$ is non-zero and its construction implies $\sup_{(\mathsf{B}_1^U)^2} |P_1| \gg 1$. Also note that $P_1 = 0$ if and only if $\mathrm{rank} (T_{u_1, u_2}) = \dim u_1^t.\mathfrak{r} + u_2^t.\mathfrak{r} < 6$, we have
    \begin{align*}
        \{(u_1, u_2) \in U^2: \dim u_1^t.\mathfrak{r}^{(1)} + u_2^t.\mathfrak{r}^{(1)} = 6\} = \{P_1(u_1, u_2) \neq 0\}.
    \end{align*}
    Note that $P_1 \asymp d_{\measuredangle}(u_1^t.\mathfrak{r}^{(1)}, u_2^t.\mathfrak{r}^{(1)})^2$, this implies the statement on partitions. 
\end{proof}

The following lemma says that $u_1^t. \mathfrak{r}^{(1)}, u_2^t. \mathfrak{r}^{(1)}$ and $u_3^t. \mathfrak{r}^{(1)}$ span an $8$-dimensional subspace for generic $u_1, u_2, u_3$. 
Similar result holds for the family of $6$-dimensional subspaces $\{u^t. \mathfrak{r}^{(0)}\}$. For convenience in later applications, we consider $\mathfrak{r}^{(1)}, u_1^t. \mathfrak{r}^{(1)}$ and $u_2^t. \mathfrak{r}^{(1)}$ for generic $(u_1, u_2)$. 

\begin{lemma}\label{lem:SpecialCalculation2}
    There exist constant $E$ and polynomials $R_1, R_2$ on $U^2$ satisfying $\sup_{(\mathsf{B}^U_1)^2} |R_i| \gg 1$ for $i = 1, 2$ so that the following holds. 
    \begin{enumerate}
        \item We have
        \begin{align*}
            \{(u_1, u_2) \in U^2: \dim \mathfrak{r}^{(1)} + u_1^t.\mathfrak{r}^{(1)} + u_2^t.\mathfrak{r}^{(1)} = 8\} = \{R_1(u_1, u_2) \neq 0\}.
        \end{align*}
        Moreover, for $(u_1, u_2) \in (\mathsf{B}^U_1)^2$ with $R_1(u_1, u_2) \geq c > 0$, 
        \begin{align*}
            \pi_{\mathfrak{r}^{(1)}}^{-1} \mathcal{D}_{\delta} \vee \pi_{u_1^t.\mathfrak{r}^{(1)}}^{-1} \mathcal{D}_\delta \vee \pi_{u_2^t.\mathfrak{r}^{(1)}}^{-1} \mathcal{D}_\delta \overset{O(c^{-E})}{\sim} \pi_{\mathfrak{r}^{(1)} + u_1^t.\mathfrak{r}^{(1)} + u_2^t.\mathfrak{r}^{(1)}}^{-1} \mathcal{D}_\delta.
        \end{align*}
        \item We have
        \begin{align*}
            \{(u_1, u_2) \in U^3: \dim  \mathfrak{r}^{(0)} \cap u_1^t.\mathfrak{r}^{(0)} \cap u_2^t.\mathfrak{r}^{(0)} = 1\} = \{R_2(u_1, u_2) \neq 0\}.
        \end{align*}
        Moreover, for $(u_1, u_2) \in (\mathsf{B}^U_1)^2$ with $R_2(u_1, u_2) \geq c > 0$, 
        \begin{align*}
            \pi_{\mathfrak{r}^{(0)}}^{-1} \mathcal{D}_{\delta} \vee \pi_{u_1^t.\mathfrak{r}^{(0)} \cap u_2^t.\mathfrak{r}^{(0)}}^{-1} \mathcal{D}_\delta \overset{O(c^{-E})}{\sim} \pi_{\mathfrak{r}^{(0)} + (u_1^t.\mathfrak{r}^{(0)} \cap u_2^t.\mathfrak{r}^{(0)})}^{-1} \mathcal{D}_\delta.
        \end{align*}
    \end{enumerate}
    The constant $E$ depends only on dimension of $\mathfrak{r}$ and $U$. 
\end{lemma}

\begin{figure}
\centering
\includegraphics{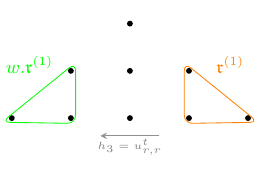}
\caption{This figure depicts the decomposition of $\mathfrak{r}$ into irreducible representations of $S \cong \SL_2(\R)$. }
\label{fig:decompositionRep}
\end{figure}

\begin{remark}\label{rem:3 copies dim 8}
We remark that using the coordinates introduced in \cref{eqn:coordinate so 2 2,eqn:coordinate so 3 1}, one can show by calculation that \emph{for all} $(u_1, u_2, u_3)$, we have $\dim \sum_{i = 1}^3 u_{i}^t.\mathfrak{r}^{(1)} \leq 8$ and $\dim \cap_{i = 1}^3 u_{i}^t.\mathfrak{r}^{(0)} \geq 1$. 
\end{remark}

\begin{proof}[Proof of Lemma~\ref{lem:SpecialCalculation2}]
    We prove property~(1). Property~(2) can be obtained in a similar way. 
    
    We first show that $\mathfrak{r}^{(1)}$, $u_1^t.\mathfrak{r}^{(1)}$ and $u_2^t.\mathfrak{r}^{(1)}$ span an $8$-dimensional subspace for $(u_1, u_2)$ from a Zariski open dense subset of $U^2$. Suppose not, then for all $(u_1, u_2) \in U^2$, 
    \begin{align*}
        \dim \mathfrak{r}^{(1)} + u_{1}^t.\mathfrak{r}^{(1)} + u_{2}^t.\mathfrak{r}^{(1)} < 8.
    \end{align*}
    As in Lemma~\ref{lem:General two piece transverse}, we consider the following subset of $\SO(Q)^2$:
    \begin{align*}
        \mathcal{V} = \biggl\{(h_1, h_2) \in \SO(Q)^2: \dim \mathfrak{r}^{(1)} + h_1.\mathfrak{r}^{(1)} + h_2.\mathfrak{r}^{(1)}< 8\biggr\}.
    \end{align*}
    This is a Zariski closed subset of $\SO(Q)^2$. Since $U^-M_0AU^+$ forms Zariski dense subset of $\SO(Q)$ and $M_0AU^+$ leaves $\mathfrak{r}^{(1)}$ invariant, we must have $\mathcal{V} = \SO(Q)^2$. 
    
    Note that both $H_1$ and $H_2$ contains a copy of $\SL_2(\R)$ generated by 
    \begin{align*}
        u_{r, r}^{t} = \begin{pmatrix}
        1 &  &  & \\
        r & 1 & & \\
        r & & 1 & \\
        r^2 & r & r & 1
    \end{pmatrix}, \quad 
    a_t = \begin{pmatrix}
        e^t & & & \\
         & 1 & & \\
         & & 1 & \\
         & & & e^{-t}
    \end{pmatrix}, \quad 
    u_{r, r} = \begin{pmatrix}
        1 & r & r & r^2\\
         & 1 & & r\\
         & & 1 & r\\
         & & & 1
    \end{pmatrix}.
    \end{align*}
    We denote this subgroup to be $S$.\footnote{It is a principal $\SL_2(\R)$ of both $H_1$ and $H_2$. } The representation $\mathfrak{r}$ is decomposed into irreducible representation of $S$ as in Figure~\ref{fig:decompositionRep}. Let $h_1 = w$ where $w \in \SO(Q)$ is a representative of of the longest element in the Weyl group. Let $h_2 = u_{r,r}^t$, we get
    \begin{align*}
        \dim \mathfrak{r}^{(1)} + h_1.\mathfrak{r}^{(1)} + h_2.\mathfrak{r}^{(1)} = 8 \text{ for generic }r,
    \end{align*}
    contradicting to the fact $\mathcal{V} = \SO(Q)^2$.
    
    We now construct $R_1$ explicitly. For simplicity of the notations, let $u_0 = \Id$. Consider the map
    \begin{align*}
        T_{u_1, u_2}: \mathfrak{r} &\to \mathfrak{r}^{(1)} \times \mathfrak{r}^{(1)} \times \mathfrak{r}^{(1)}\\
        v &\mapsto (\pi^{(1)}(v), \pi_{u_1}^{(1)}(v), \pi_{u_2}^{(1)}(v)).
    \end{align*}
    Under the coordinates in \cref{eqn:coordinate so 2 2,eqn:coordinate so 3 1}, the map can be written as a $9 \times 9$ matrix which we also denote by $T_{u_1, u_2}$. Let $R_1$ be the sum of squares of its $8 \times 8$ minors. The above argument shows that $R_1$ is non-zero. By construction, $\sup_{(\mathsf{B}^U_1)^2}|R_1| \gg 1$. Note that under the same coordinates, the span of columns of $T_{u_1, u_2}^t$ is $\sum_{i = 0}^2 u_i^t. \mathfrak{r}^{(1)}$. Therefore, 
    \begin{align*}
        \{R_1(u_1, u_2) \neq 0\} = \biggl\{(u_1, u_2) \in U^2: \dim \sum_{i = 0}^2 u_i^t. \mathfrak{r}^{(1)} = 8\biggr\}.
    \end{align*}

    We now show the statement on the partition. By Lemma~\ref{subsec:Equivalent Projections}, we can replace the projection $\pi_{u^t.\mathfrak{r}^{(1)}}$ by $\pi_{u}^{(1)} = \pi^{(1)} \circ u$. It suffices to show that
    \begin{align*}
        \bigvee_{i = 0}^2 (\pi_{u_i}^{(1)})^{-1} \mathcal{D}_\delta
        \overset{O(c^{-\star})}{\sim} \pi_{\sum_{i = 0}^2 u_i^t.\mathfrak{r}^{(1)}}^{-1} \mathcal{D}_\delta.
    \end{align*}
    Note that
    \begin{align*}
        \ker T_{u_1, u_2} = \bigcap_{i = 0}^2 u_i^{-1} (\bigoplus_{\lambda \leq 0} \mathfrak{r}_{\lambda}),
    \end{align*}
    which implies
    \begin{align*}
        (\ker T_{u_1, u_2})^{\perp} = \sum_{i = 0}^2 u_i^{t} \mathfrak{r}^{(1)}.
    \end{align*}
    Therefore, the restriction of $T_{u_1, u_2}$ to $\sum_{i = 0}^2 u_i^{t} \mathfrak{r}^{(1)}$ is a linear isomorphism. Since $(u_1, u_2) \in (\mathsf{B}^U_1)^2$, $\|T_{u_1, u_2}\| \ll 1$. It suffices to show that if $R_1(u_1, u_2) \geq c > 0$,
    \begin{align*}
        \|T_{u_1, u_2}v\| \gg c^{\frac{1}{2}}\|v\|
    \end{align*}
    for all $v \in \sum_{i = 0}^2 u_i^{t} \mathfrak{r}^{(1)}$. Take an orthonormal basis $w_1, \ldots, w_8$ of $\sum_{i = 0}^2 u_i^{t} \mathfrak{r}^{(1)}$ and a unit vector $w_9$ in $\cap_{i = 0}^2 u_i^{-1} (\oplus_{\lambda \leq 0} \mathfrak{r}_{\lambda})$. This forms an orthonormal basis of $\mathfrak{r}$. Let $k = (w_1, \ldots, w_8)$ and $\tilde{k} = (w_1, \ldots, w_9)$. View $k$ as an $8 \times 9$ matrix, it suffices to estimate $\|T_{u_1, u_2}k v\|$ for all $v \in \R^8$. By singular value decomposition, $\|T_{u_1, u_2}k v\| \gg \lambda^{\frac{1}{2}}\|v\|$ where $\lambda$ is the sum of squares of $8 \times 8$ minors of $T_{u_1, u_2, u_3}k$. Note that since $\tilde{k}$ is an orthogonal matrix, $c$ is also the sum of squares of $8 \times 8$ minors of $T_{u_1, u_2}\tilde{k} = (T_{u_1, u_2}k, 0)$. Therefore, $c = \lambda$ and 
    \begin{align*}
        \|T_{u_1, u_2}v\| \gg c^{\frac{1}{2}}\|v\|
    \end{align*}
    for all $v \in \sum_{i = 0}^2 u_i^{t} \mathfrak{r}^{(1)}$. 
\end{proof}

The discussion in the introductory part suggests to study families of the subspaces as $W_3 \cap (W_2 \oplus W_1)$ and $W_3 + W_2 + W_1$ where $W_i$ are taken from $\{u^t. \mathfrak{r}^{(1)}\}_{u \in \mathsf{B}^U_1}$. This is the goal of the following lemma. 

\begin{lemma}\label{lem:Nondegenerate of intersection sum}
There exist polynomial maps $V_1: U^2 \to \mathfrak{r}$ and $V_2:U^2 \to \mathfrak{r}$ with $\sup_{(\mathsf{B}^U_1)^2}\|V_i\| \gg 1$ for $i = 1, 2$ satisfy the following properties. Recall $P_1$ from Lemma~\ref{lem:Special Calculation1} and $R_1$ from Lemma~\ref{lem:SpecialCalculation2}. 
\begin{enumerate}
    \item For all $(u_1, u_2)$ so that $R_1(u_1, u_2) \neq 0$, we have
    \begin{align*}
        V_1(u_1, u_2) \perp \mathfrak{r}^{(1)} + u_{1}^t\mathfrak{r}^{(1)} + u_{2}^t\mathfrak{r}^{(1)}.
    \end{align*}
    Moreover, let $S_1$ be the $\mathfrak{r}_{-2}$ component of $V_1(u_1, u_2)$. The polynomial $S_1$ satisfies
    \begin{align*}
        \sup_{(u_1, u_2) \in (\mathsf{B}^U_1)^2} |S_1(u_1, u_2)| \gg 1.
    \end{align*}
    \item For all $(u_1, u_2)$ so that $P_1(u_1, u_2)R_1(u_1, u_2) \neq 0$ and $V_2(u_1, u_2) \neq 0$, we have
    \begin{align*}
        \Span V_2(u_1, u_2) = \mathfrak{r}^{(1)} \cap (u_{1}^t\mathfrak{r}^{(1)} + u_{2}^t\mathfrak{r}^{(1)}).
    \end{align*}
    Moreover, let $S_2$ be the $\mathfrak{r}_2$ component of $V_2(u_1, u_2)$. The polynomial $S_2$ satisfies
    \begin{align*}
        \sup_{(u_1, u_2) \in (\mathsf{B}^U_1)^2} |S_2(u_1, u_2)| \gg 1.
    \end{align*}
\end{enumerate}
\end{lemma}

\begin{proof}
    We write $u_0 = \Id$ for simplicity. 
    
    We start by constructing the map $V_1$. The idea is straight-forward: finding normal vector is the same as solving linear equations. Recall the following map from the proof of the previous lemma:
    \begin{align*}
        T_{u_1, u_2}: \mathfrak{r} &\to \mathfrak{r}^{(1)} \times \mathfrak{r}^{(1)} \times \mathfrak{r}^{(1)}\\
        v &\mapsto (\pi^{(1)}(v), \pi^{(1)}_{u_1}(v), \pi^{(1)}_{u_2}(v)).
    \end{align*}
    We have
    \begin{align*}
        \ker T_{u_1, u_2} = \bigcap_{i = 0}^2 u_i^{-1} (\bigoplus_{\lambda \leq 0} \mathfrak{r}_{\lambda}), \quad \text{ and } (\ker T_{u_1, u_2})^{\perp} = \sum_{i = 0}^2 u_i^{t} \mathfrak{r}^{(1)}.
    \end{align*}
    Therefore, to find a normal vector of $\mathfrak{r}^{(1)} + u_1^t.\mathfrak{r}^{(1)} + u_1^t.\mathfrak{r}^{(2)}$, it suffices to solve the homogeneous linear equation $T_{u_1, u_2}(v) = 0$. We use the same notation $T_{u_1, u_2}$ to denote the corresponding matrix of $T_{u_1, u_2}$ under the basis constructed in \cref{eqn:coordinate so 2 2} for the $(2,2)$-case or \cref{eqn:coordinate so 3 1} for the $(3, 1)$-case. 
    
    By Remark~\ref{rem:3 copies dim 8}, $\ker T_{u_1, u_2} \neq \{0\}$ for all $(u_1, u_2)$. Therefore, $\det T_{u_1, u_2}$ is a zero polynomial. Let $C_{u_1, u_2}$ be the co-factor matrix of $T_{u_1, u_2}$. Its entries are polynomials of $(u_1, u_2)$. By Lemma~\ref{lem:SpecialCalculation2}, when $R_1(u_1, u_2) \neq 0$, there exists one $8 \times 8$-minor of $T_{u_1, u_2}$, i.e., one entry $C_{ij}$ of $C_{u_1, u_2}$ which is a nonzero polynomial. Let $V_1(u_1,u_2)$ be the column containing that entry. Since
    \begin{align*}
        T_{u_1, u_2} C_{u_1, u_2} = \det (T_{u_1, u_2}) \Id_{\mathfrak{r}} = 0,
    \end{align*}
    the vector $V_1(u_1, u_2)$ lies in $\ker T_{u_1, u_2}$. Moreover, it has a non-zero entry $C_{ij}$ and by construction ($8\times 8$ minor of $T_{u_1, u_2}$), 
    \begin{align*}
        \sup_{(u_1, u_2) \in (\mathsf{B}^U_1)^2}\|V_1(u_1, u_2)\| \geq \sup_{(u_1, u_2) \in (\mathsf{B}^U_1)^2}|C_{ij}| \gg 1.
    \end{align*}

    Let $S_1(u_1, u_2)$ the $\mathfrak{r}_{-2}$ entry of $V_1(u_1, u_2)$. We now establish the estimate on $S_1$. It can be calculated directly using coordinates in \cref{eqn:coordinate so 2 2,eqn:coordinate so 3 1}. We present a proof that can be adapted to the general cases. The idea is also straight-forward: the action by $U^-$ shears $V_1(u_1, u_2)$ to $\mathfrak{r}_{-2}$ and the $U^-AM_0U^+$ decomposition ensures that $u^t.V_1(u_1, u_2)$ is parallel to $V_1(u_1', u_2')$ for some other $(u_1', u_2')$. 
    
    Since both $R_1$ and $V_1$ are polynomial on $(u_1, u_2)$, there exist $(u_1^0, u_2^0) \in (\mathsf{B}^U_\frac{1}{2})^2$ and $\rho_0 > 0$ so that the following holds. For all $(u_1', u_2') \in (\mathsf{B}^U_{\rho_0})^2$, 
    \begin{align*}
        \|V_1(u_1'u_1^0, u_2'u_2^0)\| \gg 1, \quad |R_1(u_1'u_1^0, u_2'u_2^0)| \gg 1.
    \end{align*}

    Recall that since the map
    \begin{align*}
        \LieH = \LieU^- \oplus \LieA \oplus \LieM_0 \oplus \LieU^+ &\to H\\
        (X_{\LieU^-}, X_{\LieA}, X_{\LieM_0}, X_{\LieU^+}) &\mapsto \exp(X_{\LieU^-})\exp(X_{\LieA})\exp(X_{\LieM_0})\exp(X_{\LieU^+})
    \end{align*}
    is bi-analytic near $0$, there exist analytic maps $u \mapsto \hat{u}_i(u) \in U$, $u \mapsto \hat{a}_i^+(u) \in A$, $u \mapsto \hat{m}_i^+(u) \in M_0$, $u \mapsto \hat{u}_i'(u) \in U$ for $i = 1, 2$ so that the following holds:
    \begin{align}\label{eqn:analytic commute u u}
    \begin{aligned}
        &u(u_1^0)^t = (u_1^0)^t (\hat{u}_1(u))^t \hat{a}_1(u) \hat{m}_1(u) \hat{u}_1'(u)\\
        &u(u_2^0)^t = (u_2^0)^t (\hat{u}_2(u))^t \hat{a}_2(u) \hat{m}_2(u) \hat{u}_2'(u).
    \end{aligned}
    \end{align}
    Moreover, there exists constant $C > 1$ depending only on $H$ so that for all $\eta \leq \eta_0$ and $u \in \mathsf{B}_{\eta}^{U}$, $\hat{u}_i(u) \in \mathsf{B}_{C\eta}^{U}$. 

    By Lemma~\ref{lem:VariantOfShah} (or apply \cite[Lemma 5.1]{Sh96} directly), we have
    \begin{align*}
        \sup_{u \in \mathsf{B}^{U^+}_{C^{-1}\rho_0}}\|\pi_{\mathfrak{r}_{-2}}(u^t.V_1(u_1^0, u_2^0))\| \gg \|V_1(u_1^0, u_2^0)\| \gg 1.
    \end{align*}
    Fix $u \in \mathsf{B}^{U^+}_{C^{-1}\rho_0}$ so that 
    \begin{align*}
        \|\pi_{\mathfrak{r}_{-2}}((u^{-1})^t.V_1(u_1^0, u_2^0))\| \gg \|V_1(u_1^0, u_2^0)\| \gg 1.
    \end{align*}Note that we have $(u^{-1})^t.V_1(u_1^0, u_2^0)$ is orthogonal to the subspace
    \begin{align*}
        u(\mathfrak{r}^{(1)} + (u_1^0)^t. \mathfrak{r}^{(1)} + (u_2^0)^t. \mathfrak{r}^{(1)}) = \mathfrak{r}^{(1)} + u(u_1^0)^t. \mathfrak{r}^{(1)} + u(u_2^0)^t. \mathfrak{r}^{(1)}.
    \end{align*}
    By \cref{eqn:analytic commute u u}, we have
    \begin{align*}
        \mathfrak{r}^{(1)} + u(u_1^0)^t. \mathfrak{r}^{(1)} + u(u_2^0)^t. \mathfrak{r}^{(1)}
        = \mathfrak{r}^{(1)} + (u_1^0)^t (\hat{u}_1(u))^t \mathfrak{r}^{(1)} + (u_2^0)^t (\hat{u}_2(u))^t \mathfrak{r}^{(1)}.
    \end{align*}
    Therefore, $(u^{-1})^t.V_1(u_1^0, u_2^0)$ is parallel to $V_1(\hat{u}_1(u)u_1^0, \hat{u}_2(u)u_2^0)$. Since $u \in \mathsf{B}^{U^+}_{C^{-1}\rho_0}$, both $\hat{u}_1(u)$ and $\hat{u}_2(u)$ lies in $\mathsf{B}^U_{\rho_0}$. This implies that 
    \begin{align*}
        \|V_1(\hat{u}_1(u)u_1^0, \hat{u}_2(u)u_1^0)\| \asymp (u^{-1})^t.V_1(u_1^0, u_2^0)
    \end{align*}
    and hence
    \begin{align*}
        \sup_{(u_1, u_2) \in \mathsf{B}^U_1}|S_1(u_1, u_2)| \geq \|\pi_{\mathfrak{r}_{-2}}(V_1(\hat{u}_1(u)u_1^0, \hat{u}_2(u)u_1^0)\| \gg 1.
    \end{align*}

    Property~(2) can be proved in a similar way. Let $\mathfrak{r}_{\leq 0} = \oplus_{\mu \leq 0} \mathfrak{r}_\mu$ and let $\pi_{\leq 0}$ be the orthogonal projection to it. We start by constructing $V_2$. Consider the following map:
    \begin{align*}
        S_{u_1, u_2}: \mathfrak{r}^{(1)} \times \mathfrak{r}^{(1)} &\to \bigoplus_{\lambda \leq 0} \mathfrak{r}_{\lambda}\\
        (v_1, v_2) &\mapsto \pi_{\leq 0}(u_1^t.v_1 + u_2^t.v_2).
    \end{align*}
    We use the same notation $S_{u_1, u_2}$ to denote the matrix corresponding to $S_{u_1, u_2}$ under the basis constructed in \cref{eqn:coordinate so 2 2,eqn:coordinate so 3 1}. Note that a vector $v$ lies in $\mathfrak{r}^{(1)} \cap (u_1^t.\mathfrak{r}^{(1)} + u_2^t.\mathfrak{r}^{(1)})$ if and only if 
    \begin{align*}
        v = u_1^t.v_1 + u_2^t.v_2
    \end{align*}
    for some $(v_1, v_2) \in \ker S_{u_1, u_2}$. For all $(u_1, u_2)$ so that $P_1(u_1, u_2)R_1(u_1, u_2) \neq 0$, we have
    \begin{align*}
        \dim u_1^t.\mathfrak{r}^{(1)} + u_2^t.\mathfrak{r}^{(1)} = 6, \quad \dim \mathfrak{r}^{(1)} + u_1^t.\mathfrak{r}^{(1)} + u_2^t.\mathfrak{r}^{(1)} = 8.
    \end{align*}
    For such $(u_1, u_2)$, $\dim \mathrm{Im} (S_{u_1, u_2}) = 5$. As in property~(1), we can construct $(\tilde{v_1}(u_1, u_2),\tilde{v_2}(u_1, u_2)) \in \ker S_{u_1, u_2}$ via the nonzero $5 \times 5$-minors. Let
    \begin{align*}
        V_2(u_1, u_2) = u_1^t.\tilde{v_1}(u_1, u_2) + u_2^t.\tilde{v_2}(u_1, u_2) \in \mathfrak{r}^{(1)} \cap (u_1^t.\mathfrak{r}^{(1)} + u_2^t.\mathfrak{r}^{(1)}).
    \end{align*}
    The construction implies that
    \begin{align*}
        \sup_{(u_1, u_2) \in (\mathsf{B}^U_1)^2}\|V_2(u_1, u_2)\| \gg 1.
    \end{align*}

    Let $S_2(u_1, u_2)$ the $\mathfrak{r}_{2}$ entry of $V_2(u_1, u_2)$. We now establish the estimate on $R_2$. It can be calculated directly using coordinates in \cref{eqn:coordinate so 2 2,eqn:coordinate so 3 1}. A conceptual proof can be obtained in a similar way as the following. 

    Since $P_1$, $R_1$ and $V_2$ are polynomial on $(u_1, u_2)$, there exist $(u_1^0, u_2^0) \in (\mathsf{B}^U_\frac{1}{2})^2$ and $\rho_0 > 0$ so that the following holds. For all $(u_1', u_2') \in (\mathsf{B}^U_{\rho_0})^2$, 
    \begin{align*}
        \|V_2(u_1'u_1^0, u_2'u_2^0)\| \gg 1, \quad |P_1(u_1'u_1^0, u_2'u_2^0)| \gg 1, \quad |R_1(u_1'u_1^0, u_2'u_2^0)| \gg 1.
    \end{align*}

    By Lemma~\ref{lem:VariantOfShah} (or apply \cite[Lemma 5.1]{Sh96} directly), we have
    \begin{align*}
        \sup_{u \in \mathsf{B}^{U}_{C^{-1}\rho_0}}\|\pi_{\mathfrak{r}_{2}}(u.V_2(u_1^0, u_2^0))\| \gg \|V_2(u_1^0, u_2^0)\| \gg 1.
    \end{align*}
    Fix $u \in \mathsf{B}^{U^+}_{C^{-1}\rho_0}$ so that 
    \begin{align*}
        \|\pi_{\mathfrak{r}_{2}}(u.V_2(u_1^0, u_2^0))\| \gg \|V_2(u_1^0, u_2^0)\| \gg 1.
    \end{align*}
    Note that by \cref{eqn:analytic commute u u} we have
    \begin{align*}
        u.V_2(u_1^0, u_2^0) ={}& u(u_1^0)^t.\tilde{v_1}(u_1^0, u_2^0) + u(u_2^0)^t.\tilde{v_2}(u_1^0, u_2^0)\\
        ={}& (u_1^0)^t(\hat{u}_1(u))^t \hat{a}_1(u) \hat{m}_1(u) \hat{u}_1'(u).\tilde{v_1}(u_1^0, u_2^0)\\
        {}&+ (u_2^0)^t(\hat{u}_2(u))^t \hat{a}_2(u) \hat{m}_2(u) \hat{u}_2'(u).\tilde{v_2}(u_1^0, u_2^0).
    \end{align*}
    Since $\mathfrak{r}^{(0)}$ is invariant under the action of $AM_0U^+$, we have
    \begin{align*}
         u.V_2(u_1^0, u_2^0) \in \mathfrak{r}^{(1)} \cap (u_1^0)^t(\hat{u}_1(u))^t.\mathfrak{r}^{(1)} + (u_2^0)^t(\hat{u}_2(u))^t. \mathfrak{r}^{(1)})
    \end{align*}
    Therefore, $u.V_2(u_1^0, u_2^0)$ is parallel to $V_2(\hat{u}_1(u)u_1^0, \hat{u}_2(u)u_2^0)$. Since $u \in \mathsf{B}^{U^+}_{C^{-1}\rho_0}$, both $\hat{u}_1(u)$ and $\hat{u}_2(u)$ lies in $\mathsf{B}^U_{\rho_0}$. This implies that 
    \begin{align*}
        \|V_2(\hat{u}_1(u)u_1^0, \hat{u}_2(u)u_2^0)\| \asymp u.V_2(u_1^0, u_2^0)
    \end{align*}
    and hence
    \begin{align*}
        \sup_{(u_1, u_2) \in \mathsf{B}^U_1}|S_2(u_1, u_2)| \geq \|\pi_{\mathfrak{r}_{2}}(V_2(\hat{u}_1(u)u_1^0, \hat{u}_2(u)u_2^0)\| \gg 1.
    \end{align*}
\end{proof}

We also have the following lemma for intersection and sum of the family $\{u^t.\mathfrak{r}^{(0)}\}_{u \in \mathsf{B}^U_1}$. 
\begin{lemma}\label{lem:nondegenerate intersection sum2}
    There exist polynomial maps $W_1: U^2 \to \mathfrak{r}$ and $W_2:U^2 \to \mathfrak{r}$ with $\sup_{(\mathsf{B}^U_1)^2}\|W_i\| \gg 1$ for $i = 1, 2$ satisfy the following properties. Recall $P_2$ from Lemma~\ref{lem:Special Calculation1} and $R_2$ from Lemma~\ref{lem:SpecialCalculation2}. 
\begin{enumerate}
    \item For all $(u_1, u_2)$ so that $R_2(u_1, u_2) \neq 0$ and $W_1(u_1, u_2) \neq 0$, we have
    \begin{align*}
        \Span W_1(u_1, u_2) = \mathfrak{r}^{(0)} \cap u_{1}^t\mathfrak{r}^{(0)} \cap u_{2}^t\mathfrak{r}^{(0)}.
    \end{align*}
    Moreover, let $L_1$ be the $\mathfrak{r}_2$ component of $W_1(u_1, u_2)$. The polynomial $L_1$ satisfies
    \begin{align*}
        \sup_{(u_1, u_2) \in (\mathsf{B}^U_1)^2} |L_1(u_1, u_2)| \gg 1.
    \end{align*}
    \item For all $(u_1, u_2)$ so that $P_2(u_1, u_2)R_2(u_1, u_2) \neq 0$, we have
    \begin{align*}
        W_2(u_1, u_2) \perp \mathfrak{r}^{(0)} + (u_{1}^t\mathfrak{r}^{(0)} \cap u_{2}^t\mathfrak{r}^{(0)}).
    \end{align*}
    Moreover, let $L_2$ be the $\mathfrak{r}_{-2}$ component of $W_2(u_1, u_2)$. The polynomial $L_2$ satisfies
    \begin{align*}
        \sup_{(u_1, u_2) \in (\mathsf{B}^U_1)^2} |L_2(u_1, u_2)| \gg 1.
    \end{align*}
\end{enumerate}
\end{lemma}

\begin{proof}
    Let $\mathfrak{r}_{< \lambda} = \oplus_{\mu < \lambda} \mathfrak{r}_\mu$ and $\mathfrak{r}_{\leq \lambda} = \oplus_{\mu \leq \lambda} \mathfrak{r}_\mu$. Note that $(\mathfrak{r}_{< \lambda})^\perp = \mathfrak{r}^{(\lambda)}$. Conjugating via a representative $w$ of the longest element of the Weyl group $W$, one can show the similar results in Lemma~\ref{lem:Nondegenerate of intersection sum} hold for the family of subspace $u.\mathfrak{r}_{< 0} = u.\mathfrak{r}_{\leq -1}$. Note that 
    \begin{align*}
        &\Bigl(\mathfrak{r}^{(0)} \cap u_{1}^t\mathfrak{r}^{(0)} \cap u_{2}^t\mathfrak{r}^{(0)}\Bigr)^\perp = \mathfrak{r}_{\leq -1} + u_{1}\mathfrak{r}_{\leq -1} + u_{2}\mathfrak{r}_{\leq -1}\\
        &\Bigl(\mathfrak{r}^{(0)} + (u_{1}^t\mathfrak{r}^{(0)} \cap u_{2}^t\mathfrak{r}^{(0)})\Bigr)^\perp = \mathfrak{r}_{\leq -1} \cap (u_{1}\mathfrak{r}_{\leq -1} + u_{2}\mathfrak{r}_{\leq -1}),
    \end{align*}
    the rest follows from Lemma~\ref{lem:Nondegenerate of intersection sum}. 
\end{proof}

\subsection{Proof of Theorem~\ref{thm: subcritical 9 to 3} and \ref{thm: subcritical 9 to 6}}\label{sec:Subcritical dim 3 and 6}

With preparations in the previous subsection, we now proceed the proof. 
\begin{proof}[Proof of \cref{thm: subcritical 9 to 3}]
    Let $C > 0$ and $M > 0$ be two large constants which will be determined later in the proof. In particular, $C$ will be chosen depending only on $\dim \mathfrak{r}$ and $\dim U$. Let $M_1$ and $E_1$ be the maximum of the constants $M$'s and $E$'s respectively from \cref{thm:subcritical 9 to 1,thm:subcritical 9 to 8}. Let $E_2$ be the maximum of the constants $E$'s in \cref{lem:Special Calculation1,lem:SpecialCalculation2}. Let $0 < \epsilon < \frac{1}{100CE_1^2E_2^2}$. We let $\delta$ be small enough depending explicitly on $\epsilon$ so that all implied constants appeared later in the proof are dominated by $\delta^{-\epsilon}$. 

    Recall that we defined $\mathcal{E}(A)$ as
    \begin{align*}
        \mathcal{E}(A) = \biggl\{u \in \mathsf{B}^U_1: \exists A' \subseteq A \text{ with } |A'|_\delta \geq \delta^\epsilon|A|_\delta \text{ and } |\pi^{(1)}_{u}(A')|_{\delta} < \delta^{M\epsilon}|A|_\delta^{\frac{3}{9}}\biggr\}.
    \end{align*}  
    Suppose the theorem does not hold, then $m_U(\mathcal{E}(A)) > \delta^{\epsilon}$. 

    We now collect and briefly review the polynomials constructed in \cref{lem:Special Calculation1,lem:SpecialCalculation2,lem:Nondegenerate of intersection sum}. Recall $P_1$ on $U^2$ from Lemma~\ref{lem:Special Calculation1} with the following property. For $(u_1, u_2) \in (\mathsf{B}^U_2)^2$ with $P_1(u_1, u_2) > \delta^{C\epsilon} > 0$, 
    \begin{align*}
        \pi_{u_1^t.\mathfrak{r}^{(1)}}^{-1} \mathcal{D}_\delta \vee \pi_{u_2^t.\mathfrak{r}^{(1)}}^{-1} \mathcal{D}_\delta \overset{O(\delta^{-CE_2\epsilon})}{\sim}  \pi_{u_1^t.\mathfrak{r}^{(1)} \oplus u_2^t.\mathfrak{r}^{(1)}}^{-1} \mathcal{D}_\delta.
    \end{align*}
    Recall $R_1$ on $U^2$ from Lemma~\ref{lem:SpecialCalculation2} with the following property. For $(u_1, u_2) \in (\mathsf{B}^U_2)^2$ with $R_1(u_1, u_2) > \delta^{C\epsilon} > 0$, we have
    \begin{align*}
        \dim \mathfrak{r}^{(1)} + u_1^t.\mathfrak{r}^{(1)} + u_2^t.\mathfrak{r}^{(1)} = 8
    \end{align*}
    and moreover, 
    \begin{align*}
        \pi_{\mathfrak{r}^{(1)}}^{-1} \mathcal{D}_{\delta} \vee \pi_{u_1^t.\mathfrak{r}^{(1)}}^{-1} \mathcal{D}_\delta \vee \pi_{u_2^t.\mathfrak{r}^{(1)}}^{-1} \mathcal{D}_\delta \overset{O(\delta^{-CE_2\epsilon})}{\sim} \pi_{\mathfrak{r}^{(1)} + u_1^t.\mathfrak{r}^{(1)} + u_2^t.\mathfrak{r}^{(1)}}^{-1} \mathcal{D}_\delta.
    \end{align*}
    Recall $V_1$ and $V_2$ on $U^2$ from Lemma~\ref{lem:Nondegenerate of intersection sum} with the following property. For all $(u_1, u_2)$ with $P_1(u_1, u_2)R_1(u_1, u_2) \neq 0$, 
    \begin{align*}
        &V_1(u_1, u_2) \perp \mathfrak{r}^{(1)} + u_1^t.\mathfrak{r}^{(1)} + u_2^t.\mathfrak{r}^{(1)}, \text{ and }\\
        &\Span V_2(u_1, u_2) = \mathfrak{r}^{(1)} \cap (u_1^t.\mathfrak{r}^{(1)} + u_2^t.\mathfrak{r}^{(1)}).
    \end{align*}
    The polynomial $S_1$ is the lowest weight component of $V_1$ and the polynomial $S_2$ is the highest weight component of $V_2$. All the above polynomials satisfy
    \begin{align*}
        \sup_{(u_1, u_2) \in (\mathsf{B}^U_1)^2} |(\cdot)(u_1, u_2)| \gg 1 \quad \text{ where } (\cdot) = P_1, R_1, S_1, S_2.
    \end{align*}
    
    Let $\mathcal{J}_1$ be the subset of $(\mathsf{B}^U_2)^2$ so that for all $(u_1, u_2) \in \mathcal{J}_1$, we have:
    \begin{enumerate}
        \item $P_1(u_1, u_2) > \delta^{C\epsilon}$,
        \item $R_1(u_1, u_2) > \delta^{C\epsilon}$,
        \item $S_1(u_1, u_2) > \delta^{C\epsilon}$,
        \item $S_2(u_1, u_2) > \delta^{C\epsilon}$.
    \end{enumerate}
    Roughly speaking, $\mathcal{J}_1$ is the set of $(u_1, u_2)$ so that the subspaces $\mathfrak{r}^{(1)}$, $u_1^t.\mathfrak{r}^{(1)}$ and $u_2^t.\mathfrak{r}^{(1)}$ are quantitatively in general position in $\mathfrak{r}$. By Remez's inequality, the measure  $m_U((\mathsf{B}^U_2)^2 \setminus \mathcal{J}_1) < \delta^{\frac{C}{d}\epsilon}$ for some constant $d$ depending only on the ambient representation. Let
    \begin{align*}
        \tilde{\mathcal{E}} = \{(u_1, u_2, u_3) \in (\mathsf{B}^U_1)^3: (u_1, u_2) \in \mathcal{J}_1, u_1u_3, u_2u_3, u_3 \in \mathcal{E}(A)\}.
    \end{align*}
    We now estimate the measure of $\tilde{\mathcal{E}}$. For all $u_3 \in U$, we define
    \begin{align*}
        \mathcal{E}(A)u_3^{-1} = \{u u_3^{-1}: u \in \mathcal{E}(A)\}
    \end{align*}
    to be the translation of $\mathcal{E}(A)$ by $u_3^{-1}$. By Fubini's theorem, we have
    \begin{align*}
        m_U(\tilde{\mathcal{E}}) ={}& \int_{\mathcal{E}(A)} m_U\bigl(\mathcal{J}_1 \cap (\mathcal{E}(A)u_3^{-1} \times \mathcal{E}(A)u_3^{-1})\bigr) \,\mathrm{d}u_3\\
        \geq{}& \int_{\mathcal{E}(A)} \bigl(m_U(\mathcal{E}(A))^2 - \delta^{\frac{C}{d}\epsilon}\bigr)\,\mathrm{d}u_3 \geq \delta^{4\epsilon}
    \end{align*}
    if $C \geq 4d + 1$. Applying Fubini's theorem again, there exists $\mathcal{J}_1' \subseteq \mathcal{J}_1$ so that for all $(u_1, u_2) \in \mathcal{J}_1'$, we have
    \begin{align*}
        m_U\bigl(\{u_3 \in \mathsf{B}^U_1: u_1u_3, u_2u_3, u_3 \in \mathcal{E}(A)\}\bigr) > \delta^{5\epsilon}.
    \end{align*}
    From now on, we fix one $(u_1, u_2) \in \mathcal{J}_1'$ and let
    \begin{align*}
        \mathcal{E}(A)' = \{u_3 \in \mathsf{B}^U_1: u_1u_3, u_2u_3, u_3 \in \mathcal{E}(A)\}.
    \end{align*}
    We have $m_U(\mathcal{E}(A)') > \delta^{5\epsilon}$. 

    By part~(1) of Lemma~\ref{lem:Special Calculation1}, for all $A' \subseteq A$, we have
    \begin{align*}
        |\pi^{(1)}_{u_1}(A')|_\delta \cdot |\pi^{(1)}_{u_2}(A')|_\delta \geq |A'|_{(\pi^{(1)}_{u_1})^{-1}\mathcal{D}_\delta \vee (\pi^{(1)}_{u_2})^{-1}\mathcal{D}_\delta}
        \gg \delta^{CE_2\epsilon} |\pi_{u_1^t.\mathfrak{r}^{(1)} \oplus u_2^t. \mathfrak{r}^{(1)}}(A')|_\delta.
    \end{align*}
    The last inequality follows from the fact that $P_1(u_1, u_2) > \delta^{C\epsilon}$ for all $(u_1, u_2) \in \mathcal{J}_1$. Since $u_3 \in \mathsf{B}^U_1$, we have  
    \begin{align*}
        d_{\measuredangle}(u_3.W_1, u_3. W_2) \asymp d_{\measuredangle}(W_1, W_2)
    \end{align*}
    for any subspaces $W_1, W_2$. Therefore, \begin{align}\label{eqn:subcritical 2 transverse}
        |\pi^{(1)}_{u_1u_3}(A')|_\delta \cdot |\pi^{(1)}_{u_2u_3}(A')|_\delta 
        \gg \delta^{CE_2\epsilon} |\pi_{u_3^t.(u_1^t.\mathfrak{r}^{(1)} \oplus u_2^t. \mathfrak{r}^{(1)})}(A')|_\delta.
    \end{align}

    Since $(u_1, u_2) \in \mathcal{J}_1$,  $|S_2(u_1, u_2)| > \delta^{C\epsilon}$. Applying \cref{thm:subcritical 9 to 1} to the set $A$, $6CE_1\epsilon$ and the line $\mathfrak{r}^{(1)} \cap (u_1^t.\mathfrak{r}^{(1)} \oplus u_2^t. \mathfrak{r}^{(1)})$, we get an exceptional set $\mathcal{E}_1$ with $m_U(\mathcal{E}_1) \leq \delta^{6CE_1\epsilon}$. Since $(u_1, u_2) \in \mathcal{J}_1$,  $|S_1(u_1, u_2)| > \delta^{C\epsilon}$. Applying \cref{thm:subcritical 9 to 8} to the set $A$, $6CE_1\epsilon$ and the hyperplane $\mathfrak{r}^{(1)} + u_1^t.\mathfrak{r}^{(1)} + u_2^t. \mathfrak{r}^{(1)}$, we get an exceptional set $\mathcal{E}_2$ with $m_U(\mathcal{E}_2) \leq \delta^{6CE_1\epsilon}$. By letting $\delta \ll_\epsilon 1$, we have
    \begin{align*}
        m_U(\mathcal{E}(A)' \setminus (\mathcal{E}_1 \cup \mathcal{E}_1)) > \delta^{6\epsilon}.
    \end{align*}
    From now on we fixed some $u_3 \in \mathcal{E}(A)' \setminus (\mathcal{E}_1 \cup \mathcal{E}_1)$. Let 
    \begin{align*}
        \mathcal{P} ={}& (\pi_{u_3^t.\mathfrak{r}^{(1)}})^{-1}\mathcal{D}_\delta,\\
        \mathcal{Q} ={}& (\pi_{u_3^t u_1^t.\mathfrak{r}^{(1)} + u_3^t u_2^t. \mathfrak{r}^{(1)})})^{-1} \mathcal{D}_\delta,\\
        \mathcal{R} ={}& \mathcal{P} \vee \mathcal{Q},\\
        \mathcal{S} ={}&  (\pi_{u_3^t.(\mathfrak{r}^{(1)} \cap (u_1^t.\mathfrak{r}^{(1)} \oplus u_2^t. \mathfrak{r}^{(1)}))})^{-1} \mathcal{D}_\delta.
    \end{align*}
    Since $(u_1, u_2) \in \mathcal{J}_1$, by Lemma~\ref{lem:Nondegenerate of intersection sum} we have
    \begin{align*}
        \mathcal{R} = \mathcal{P} \vee \mathcal{Q} \overset{O(\delta^{-CE_2\epsilon})}{\sim} (\pi_{u_3^t.(\mathfrak{r}^{(1)} + u_1^t.\mathfrak{r}^{(1)} + u_2^t.\mathfrak{r}^{(1)})})^{-1}\mathcal{D}_\delta.
    \end{align*}
        
    For all $A' \subseteq A$ with $|A'|_\delta \geq \delta^\epsilon|A|_\delta$, we apply Lemma~\ref{lem:regularization set} to $A'$ and the filtration
    \begin{align*}
        \mathcal{S} \prec \mathcal{R} \prec \mathcal{D}_\delta.
    \end{align*}
    We get $A_1 \subseteq A'$ with $|A_1|_\delta \geq \delta^{\epsilon}|A'|_\delta \geq \delta^{2\epsilon}|A|_\delta$ so that $A_1$ is regular with respect to the above filtration. Moreover, $A_1$ is the intersection of $A'$ with some disjoint union of $\delta$-cubes. 

    Now, we apply \cref{lem:submodularity} to $A_1$ and the above partitions $\mathcal{P}$, $\mathcal{Q}$, $\mathcal{R}$ and $\mathcal{S}$. There exist constants $A_2 \subseteq A_1$ with $|A_2|_{\mathcal{R}} \gg |A_1|_{\mathcal{R}}$ so that
    \begin{align*}
        |A_1|_\mathcal{P} \cdot |A_1|_\mathcal{Q} \gg |A_1|_\mathcal{R} \cdot |A_2|_\mathcal{S}.
    \end{align*}

    By regularity of $A_1$, we have
    \begin{align*}
        |A_2|_\mathcal{S} \geq \frac{1}{2}\frac{|A_1|_\mathcal{S}}{|A_1|_\mathcal{R}}|A_2|_\mathcal{R} \gg |A_1|_\mathcal{S}.
    \end{align*}
    Therefore, we have
    \begin{align}\label{eqn:subcritical 9 to 3 submodularity}
        |A_1|_\mathcal{P} \cdot |A_1|_\mathcal{Q} \gg  |A_1|_\mathcal{R} \cdot |A_1|_\mathcal{S}.
    \end{align}

    Since $u_3 \in \mathcal{E}(A)'$, we have
    \begin{align}\label{eqn:subcritical 9 to 3 each copy of exceptional 2}
    \begin{aligned}
        |\pi^{(1)}_{u_1u_3}(A')|_\delta <{}& \delta^{M\epsilon}|A|_\delta^{\frac{3}{9}},\\
        |\pi^{(1)}_{u_2u_3}(A')|_\delta <{}& \delta^{M\epsilon}|A|_\delta^{\frac{3}{9}},
    \end{aligned}
    \end{align}
    and
    \begin{align}\label{eqn:subcritical 9 to 3 each copy of exceptional 1}
        |A|_\mathcal{P} = |\pi^{(1)}_{u_3}(A')|_\delta < \delta^{M\epsilon}|A|_\delta^{\frac{3}{9}}.
    \end{align}

    On the other hand, since $u_3 \notin \mathcal{E}_1 \cup \mathcal{E}_2$ and $|A_1| \geq \delta^{2\epsilon}|A|_\delta$, we have
    \begin{align}\label{eqn:subcritical 9 to 3 one copy 9 to 1}
        |A|_\mathcal{S} = |\pi_{u_3^t.(\mathfrak{r}^{(1)} \cap (u_1^t.\mathfrak{r}^{(1)} \oplus u_2^t.\mathfrak{r}^{(1)}))}(A_1)|_\delta \geq \delta^{6M_1CE_1\epsilon}|A|_\delta^{\frac{1}{9}}
    \end{align}
    and
    \begin{align}\label{eqn:subcritical 9 to 3 one copy 9 to 8}
        |\pi_{u_3^t.(\mathfrak{r}^{(1)} + u_1^t.\mathfrak{r}^{(1)} + u_2^t.\mathfrak{r}^{(1)})}(A_1)|_\delta \geq \delta^{6M_1CE_1\epsilon}|A|_\delta^{\frac{8}{9}}.
    \end{align}
    Combining \cref{eqn:subcritical 2 transverse,eqn:subcritical 9 to 3 submodularity,eqn:subcritical 9 to 3 each copy of exceptional 2,eqn:subcritical 9 to 3 each copy of exceptional 1,eqn:subcritical 9 to 3 one copy 9 to 1,eqn:subcritical 9 to 3 one copy 9 to 8}, we have
    \begin{align*}
        \delta^{3M\epsilon}|A|_\delta^{\frac{3}{9}} \cdot |A|_\delta^{\frac{3}{9}} \cdot |A|_\delta^{\frac{3}{9}} \gg \delta^{2CE_2\epsilon}\delta^{12M_1CE_1\epsilon}|A|_\delta^{\frac{1}{9}} |A|_\delta^{\frac{8}{9}}.
    \end{align*}
    By letting $M$ large enough depending $C$, $M_1$, $E_1$ and $E_2$, we get a contradiction. Note that since $M_1$ and $C$ depend only on the ambient representation, $M$ depends only on the ambient representation. 
\end{proof}

\begin{proof}[Proof of \cref{thm: subcritical 9 to 6}]
The proof follows from dualizing the argument in the above proof. For reader's convenience, we provide an outline. In what follows we will use $\{W_i\}_{i = 1, 2, 3}$ to represent three copies of $u_i^-.\mathfrak{r}^{(0)}$ for generic $u_i^-$, $i = 1, 2, 3$. 

By Lemma~\ref{lem:Special Calculation1} property~(2), we know that $W_1$ and $W_2$ are in general position, i.e., $W_1 + W_2 = \mathfrak{r}$. By sub-modularity inequality (Lemma~\ref{lem:submodularity}), we have
\begin{align*}
    |\pi_{W_1}(A)|_\delta \cdot |\pi_{W_2}(A)|_\delta \gg_{d_{\measuredangle}(W_1, W_2)} |\pi_{W_1 + W_2}(A)|_\delta |\pi_{W_1 \cap W_2}(A)|_\delta = |A|_\delta |\pi_{W_1 \cap W_2}(A)|_\delta.
\end{align*}
We now add $W_3$. Applying the sub-modularity inequality (Lemma~\ref{lem:submodularity}) again, we have
\begin{align*}
    |\pi_{W_3}(A)|_\delta \cdot |\pi_{W_1 \cap W_2}(A)|_\delta \gg |\pi_{W_3 + (W_1 \cap W_2)}(A)|_\delta \cdot |\pi_{W_1 \cap W_2 \cap W_3}(A)|_\delta.
\end{align*}
By Lemma~\ref{lem:SpecialCalculation2} property~(2), we have that $W_3 + (W_1 \cap W_2)$ is a family of $8$-dimensional subspaces and $W_1 \cap W_2 \cap W_3$ is a family of $1$-dimensional subspaces for generic choice of $W_1, W_2, W_3$. By Lemma~\ref{lem:nondegenerate intersection sum2}, they satisfy the algebraic conditions in \cref{thm:subcritical 9 to 1,thm:subcritical 9 to 8} and contribute $8/9$ and $1/9$ of the entropy respectively. Combine these estimates and the above inequalities, we prove the theorem. 
\end{proof}

\section{Optimal projections to the highest weight direction}\label{sec:optimal}
The main theorem in this section is of the following. 
\begin{theorem}\label{thm:optimal}
    Let $E \subset B^{\mathfrak{r}}_1(0)$ be a finite set. Suppose there exist $\alpha \in (0, 9)$ and $C \geq 1$ such that 
    \begin{align*}
        \mu_E(B^{\mathfrak{r}}_{\rho}(x)) \leq C \rho^{\alpha}
    \end{align*}
    for all $\rho_0 \leq \rho \leq 1$. 
    
    Then for all $\mathsf{c} > 0$, there exists $C_{\mathsf{c}}$ such that the following holds. 
    For all $\rho_0 \leq \rho \ll_{\mathsf{c}} 1$, we define the exceptional set $\mathcal{E}(E)$ to be
    \begin{align*}
        \mathcal{E}(E) = \{u \in \mathsf{B}^U_1: \exists E' \subset E \text{ with }& \mu_E(E') \geq \rho^{\mathsf{c}}\\
        \text{ and }& |\pi^{(2)}_{u}(E')|_{\rho} < C_{\mathsf{c}}^{-1}C^{-1} \rho^{- \min\{\alpha, 1\} + O(\sqrt{\mathsf{c}})}\}.
    \end{align*}

    We have
    \begin{align*}
        m_U(\mathcal{E}(E)) \leq C_{\mathsf{c}}\rho^{\mathsf{c}}.
    \end{align*}
\end{theorem}

The key ingredient of the proof is the following consequence of \cite[Theorem 2.1]{GGW24}. Let $\gamma: [-1, 1] \to \R^n$ be a curve in $\R^n$ satisfying 
\begin{align*}
    \|\gamma^{(1)}(t) \wedge \cdots \wedge \gamma^{(n)}(t)\| \geq c > 0
\end{align*}
for all $t \in [-1, 1]$. We also assume $\|\gamma^{(i)}(t)\| \leq L$ for all $i$ and $t \in [-1, 1]$. We use $\pi_t^{(i)}$ to denote the orthogonal projection to the $i$-dimensional subspace spanned by $\gamma^{(1)}, \cdots, \gamma^{(i)}$. 
\begin{theorem}\label{thm:GGW}
    Let $E \subset B^{\R^n}_1(0)$ be a finite set. Suppose there exist $\alpha \in (0, n)$ and $C \geq 1$ such that 
    \begin{align*}
        \mu_E(B^{\R^n}_{\rho}(x)) \leq C \rho^{\alpha}
    \end{align*}
    for all $\rho_0 \leq \rho \leq 1$. 
    
    Then for all $\epsilon > 0$, there exists $C_{\epsilon, c, L}$ such that the following holds. 
    For all $\rho_0 \leq \rho \ll_{\epsilon} 1$, we define the exceptional set $\mathcal{E}(E)$ to be
    \begin{align*}
        \mathcal{E}(E) = \{t \in [-1, 1]: \exists E' \subset E \text{ with }& \mu_E(E') \geq \rho^{\epsilon}\\
        \text{ and }& |\pi^{(i)}_{t}(E')|_{\rho} < C_{\epsilon, c, L}^{-1}C^{-1} \rho^{- \min\{\alpha, i\} + O(\sqrt{\epsilon})}\}.
    \end{align*}

    We have
    \begin{align*}
        |\mathcal{E}(E)| \leq C_{\epsilon, c, L}\rho^{\epsilon}.
    \end{align*}
    Moreover, the constant $C_{\epsilon, c, L}$ satisfies
    \begin{align*}
        C_{\epsilon, c, L} \ll_\epsilon c^{-\star}L^{\star}.
    \end{align*}
\end{theorem}
\begin{proof}
    The deduction from \cite[Theorem 2.1]{GGW24} to this consequence is standard, see \cite[Appendix C]{LMWY25}. The dependence of the constant in \cite[Theorem 2.1]{GGW24} to the non-degeneracy $\|\gamma^{(1)}\wedge \cdots \wedge\gamma^{(n)}\|$ is not explicitly written. Still, one can track through the proof and show that it is a polynomial on the decoupling constant of non-degenerate curve. The latter is polynomial on $\|\gamma^{(1)}\wedge \cdots \wedge\gamma^{(n)}\|^{-1}$  and $\max_{j = 1, \ldots, i}\|\gamma^{(j)}\|$. For a calculation in similar setting, see \cite[Proposition 6.13]{JL24}. 
\end{proof}

\begin{proof}[Proof of \cref{thm:optimal}]
    Recall that the set $B^{\LieU}_1(0)$ can be identified with $[-1, 1]^2$ under our parametrization $u_{r, s}$. 
    
    We first deal with the case of $\SO^+(2, 2)$ and $\mathfrak{r}^{(1)}$. Note that under the coordinate in \cref{eqn:coordinate so 2 2}, the projection $\pi_{r, s}^{(2)}$ can be viewed as taking inner product with the following vector in $\R^9$:
    \begin{align*}
        (1, r, \frac{r^2}{2}, s, sr, s\frac{r^2}{2}, \frac{s^2}{2}, \frac{s^2}{2}r, \frac{s^2}{2}\frac{r^2}{2}).
    \end{align*}
    We have the following re-parametrization of $(r, s) \in [-1, 1]^2$. We set 
    \begin{align*}
        (x, y) \mapsto (x, y + x^3).
    \end{align*}

    Note that the Jacobian of this map is $1$ and it gives a family of nondegenerate curves: 
    \begin{align*}
        \gamma_y(x)
        ={}& \Bigl(1, x, \frac{x^2}{2}, x^3 + y, x^4 + xy, \frac{1}{2}(x^5 + x^2y), \frac{1}{2}(x^6 + 2x^3y + y^2), \\&\frac{1}{2}(x^7 + 2x^4y + xy^2), \frac{1}{2}(x^8 + 2x^5y + x^2y^2)\Bigr).
    \end{align*}
    Moreover, a direct calculation shows that the non-degeneracy $\|\gamma_y \wedge \gamma_y^{(1)}\wedge \cdots \wedge\gamma_y^{(8)}\|$ is bounded from below by an absolute constant does not depend on $y$. The theorem now follows directly from \cref{thm:GGW}. 

    For the case of $\SO(3, 1)^\circ$ and $\mathfrak{r}_2$, note that by the coordinate in \cref{eqn:coordinate so 3 1}, the projection $\pi_{r, s}^{(2)}$ can be viewed as taking inner product with the following vector in $\R^9$:
    \begin{align*}
        \Biggl(1, -2r, -2s, r^2 + 3s^2, s^2 - r^2, -2rs, r(r^2 + s^2), s(r^2 + s^2), \Bigl(\frac{r^2 + s^2}{2}\Bigr)^2\Biggr).
    \end{align*}
    We use the same re-parametrization of $(r, s) \in [-1, 1]^2$. We set 
    \begin{align*}
        (x, y) \mapsto (x, y + x^3).
    \end{align*}
    and let $\gamma_y(x)$ be the re-parametrized curve. 

    A direct calculation shows that the non-degeneracy $\|\gamma_y(x) \wedge \gamma_y^{(1)}(x)\wedge \cdots \wedge\gamma_y^{(8)}(x)\|$ is a non-zero polynomial in $(x, y)$. Let $\mathcal{E}_1$ be the set of $y$ so that the coefficient of $\|\gamma_y(x) \wedge \gamma_y^{(1)}(x)\wedge \cdots \wedge\gamma_y^{(8)}(x)\|$ as polynomial of $x$ is $\geq \rho^{\mathsf{c}}$ and apply \cref{thm:GGW} to the curve $\gamma_y$. The rest follows from Fubini's theorem. 
\end{proof}

As a corollary, we prove a special case of \cref{thm:Multislicing} for $5$-tuples $\mathbf{r} = (\mathsf{r_4}, \mathsf{r_4}, \mathsf{r_4}, \mathsf{r_4}, \mathsf{r_5})$ with $0 \leq \mathsf{r_4} \leq \mathsf{r_5} \leq 1$. Note that the improvement $\hat{\varphi}(\alpha)$ here is better than the $\varphi(\alpha) = \frac{1}{36}\hat{\varphi}(\alpha)$ in the main theorem. Recall that for a dyadic cube $Q$ in $\R^n$, $\Hom_Q$ is the unique homothety that map $Q$ to $[0, 1)^n$ and 
\begin{align*}
    \hat{\varphi}(\alpha) = \min\{\alpha, 1\} - \frac{1}{9}\alpha = \begin{cases}
        \frac{8}{9} \alpha & \text{ if } 0 \leq \alpha \leq 1;\\
        1 - \frac{1}{9} \alpha & \text{ if } 1 < \alpha \leq 9.
    \end{cases}
\end{align*}

\begin{corollary}\label{cor:optimal top multislicing}
    Fix a $5$-tuples $\mathbf{r} = (\mathsf{r_4}, \mathsf{r_4}, \mathsf{r_4}, \mathsf{r_4}, \mathsf{r_5})$ with $0 \leq \mathsf{r_4} \leq \mathsf{r_5} \leq 1$. 
    
    Let $E \subset B^{\mathfrak{r}}_1(0)$ be a finite set. Suppose there exist $\alpha \in (0, 9)$ and $C \geq 1$ such that 
    \begin{align*}
        \mu_E(B^{\mathfrak{r}}_{\rho}(x)) \leq C \rho^{\alpha}
    \end{align*}
    for all $\rho_0 \leq \rho \leq 1$. 
    
    Then for all $\mathsf{c} \ll \mathsf{r_5} - \mathsf{r_4}$, there exists $C_{\mathsf{c}, \mathbf{r}} > 0$ such that the following holds. 
    For all $\rho_0 \leq \rho \ll_{\mathsf{c}} 1$, we define the exceptional set $\mathcal{E}(E)$ to be
    \begin{align*}
        \mathcal{E}(E) = \{u \in \mathsf{B}^U_1: &\exists E' \subset E \text{ with } \mu_E(E') \geq \rho^{\mathsf{c}}\\
        &\text{ and } |u.E'|_{\mathcal{D}_{\rho}^{\mathbf{r}}} < C_{\mathsf{c}, \mathbf{r}}^{-1}C^{-1} \vol(T)^{-\frac{\alpha}{9}}\rho^{-(\mathsf{r_5} - \mathsf{r_4})\hat{\varphi}(\alpha) + O(\sqrt{\mathsf{c}})}\}.
    \end{align*}

    We have
    \begin{align*}
        m_U(\mathcal{E}(E)) \leq C_{\mathsf{c}, \mathbf{r}}\rho^{\mathsf{c}}.
    \end{align*} 
\end{corollary}

\begin{proof}
    The proof is similar to the proof of \cref{thm:ImprovementSlabExpansion} assuming \cref{thm:Multislicing} in \cref{sec:Anisotropic implies expansion}. Recall that for an atom $T \in \mathcal{D}_\rho^{\mathbf{r}}$, its volume satisfies the following estimate
    \begin{align*}
        \vol(T) \sim \rho^{8\mathsf{r_4} + \mathsf{r_5}}.
    \end{align*}

    Without loss of generality, we assume that $\mathsf{r_4} < \mathsf{r_5}$. Otherwise the corollary is obvious. For simplicity, let $\rho_1 = \rho^{\mathsf{r_4}}$, $\rho_2 = \rho^{\mathsf{r_5} - \mathsf{r_4}}$ and $\mathbf{s} = (0, 0, 0, 0, \mathsf{r_5} - \mathsf{r_4})$. For all $u \in \mathsf{B}^U_1$ and all subset $F' \subseteq F$ with $\mu_F(F') \geq \rho^{\mathsf{c}}$, we have
    \begin{align*}
        |u. F'|_{\mathcal{D}_{\rho}^{\mathbf{r}}} \gg{}& |u. F'|_{\rho_1\mathcal{D}_{\rho}^{\mathbf{s}}}\\
        \gg{}& \sum_{Q \in \mathcal{D}_{\rho_1}} |u_{r, s}. F'_Q|_{\rho_1\mathcal{D}_{\rho}^{\mathbf{s}}}\\
        \gg{}& \sum_{Q \in \mathcal{D}_{\rho_1}} |u.\Hom_Q (F'_Q)|_{\mathcal{D}_{\rho}^{\mathbf{s}}}.
    \end{align*}

    Recall that for any subset $A$, we use $A^Q$ to denote the image of $A_Q = A \cap Q$ under the homothety $\Hom_Q$. It is the rescaling of $A \cap Q$ to size $1$. We now study the Frostman-type condition that $F^Q$ satisfies: 
    \begin{align*}
        \mu_{F^Q} (B^{\mathfrak{r}}_{\rho'}(x)) ={}& \frac{1}{\mu_F(Q)} \mu_F(B^{\mathfrak{r}}_{\rho_1\rho'}(x'))\\
        \leq{}& \frac{C\rho_1^\alpha}{\mu_F(Q)} (\rho')^{\alpha}
    \end{align*}
    for all $\rho' \geq \rho_1^{-1}\rho_0$. 

    Note that by our restriction to $\rho$ and $\mathsf{r_5} \leq 1$, $\rho_2 = \rho^{\mathsf{r_5} - \mathsf{r_4}} \geq \rho_1^{-1}\rho_0$. Suppose $\mathsf{c}$ is small enough so that $\tilde{c} = \frac{4}{\mathsf{r_5} - \mathsf{r_4}}\mathsf{c}$ is small enough to apply \cref{thm:optimal}. This means that $\mathsf{c} \ll \mathsf{r_5} - \mathsf{r_4}$. We use $C_{\mathsf{c}, \mathbf{r}}$ to denote $C_{\tilde{\mathsf{c}}}$ in \cref{thm:optimal}. For all $Q \in \mathcal{D}_{\rho_1}$ so that $\mu_{F_Q}(F'_Q) \geq \rho^{2\mathsf{c}}$, applying \cref{thm:optimal} to $\mu_{F^Q}$ and $\tilde{c} = \frac{4}{\mathsf{r_5} - \mathsf{r_4}}\mathsf{c}$, there exists $\mathcal{E}_Q$ with $m_U(\mathcal{E}_Q) \leq C_{\tilde{\mathsf{c}}} \rho^{4c}$, and for all $u \notin \mathcal{E}_Q$, we have \begin{align*}
        |u. \Hom_Q (F'_Q)|_{\mathcal{D}_{\rho}^{\mathbf{s}}} \geq{}& C_{\tilde{\mathsf{c}}}^{-1}\mu_F(Q)C^{-1} \rho_1^{-\alpha} \rho_2^{-\min(\alpha, 1) + O(\sqrt{\tilde{\mathsf{c}}})}\\
        ={}& \mu_F(Q)C_{\mathsf{c}, \mathbf{r}}^{-1}C^{-1}\rho^{-\mathsf{r_4}\alpha - (\mathsf{r_5} - \mathsf{r_4})\min(\alpha, 1) + O(\sqrt{(\mathsf{r_5} - \mathsf{r_4})\mathsf{c}})}\\
        ={}& \mu_F(Q)C_{\mathsf{c}, \mathbf{r}}^{-1}C^{-1}\vol(T)^{-1}\rho^{- (\mathsf{r_5} - \mathsf{r_4})\hat{\varphi}(\alpha) + O(\sqrt{(\mathsf{r_5} - \mathsf{r_4})\mathsf{c}})}.
    \end{align*}
    Similar to the proof of \cref{thm:ImprovementSlabExpansion} assuming \cref{thm:Multislicing} in \cref{sec:Anisotropic implies expansion}, we proceed by Fubini's theorem to combine those information from local pieces. Let
    \begin{align*}
        \mathcal{D}_{\rho_1}(u) = \{Q \in \mathcal{D}_{\rho_1}(F):u \in \mathcal{E}_Q\}
    \end{align*}
    and let
    \begin{align*}
        \mathcal{D}_{\rho_1}^{\mathrm{large}}(F') = \{Q \in \mathcal{D}_{\rho_1}(F):\mu_F(F' \cap Q) \geq \rho^{\mathsf{c}}\mu_F(Q)\}.
    \end{align*}
    Since $\mu_F(F') \geq \rho^{\mathsf{c}}$, we have
    \begin{align*}
        \sum_{Q \in \mathcal{D}_{\rho_1}^{\mathrm{large}}(F')} \mu_F(Q) \geq \rho^{2\mathsf{c}}.
    \end{align*}
    By Fubini's theorem, there exists $\mathcal{E} \subseteq \mathsf{B}^U_1$ with $m_U(\mathcal{E}) \ll \rho^{\mathsf{c}}$ so that for all $u \notin \mathcal{E}$, we have
    \begin{align*}
        \sum_{Q \in \mathcal{D}_{\rho_1}(u)} \mu_F(Q) \leq \rho^{3\mathsf{c}}.
    \end{align*}
    Combining all the above estimates, for all $u \notin \mathcal{E}$, we have
    \begin{align*}
        |u.F'|_{\mathcal{D}_{\rho}^{\mathbf{r}}} \gg{}& \Biggl(\sum_{Q \in \mathcal{D}_{\rho_1}^{\mathrm{large}} \setminus \mathcal{D}_{\rho_1}(u)} \mu_F(Q)\Biggr) C_{\mathsf{c}, \mathbf{r}}^{-1} C^{-1}\vol(T)^{-1} \rho^{-(\mathsf{r_5} - \mathsf{r_4})\hat{\varphi}(\alpha) + O(\sqrt{(\mathsf{r_5} - \mathsf{r_4})\mathsf{c}})}\\
        \geq{}& C_{\mathsf{c}, \mathbf{r}}^{-1} C^{-1}\vol(T)^{-1} \rho^{-(\mathsf{r_5} - \mathsf{r_4})\hat{\varphi}(\alpha) + 3\mathsf{c} +  O(\sqrt{(\mathsf{r_5} - \mathsf{r_4})\mathsf{c}})}\\
        \geq{}& C_{\mathsf{c}, \mathbf{r}}^{-1} C^{-1}\vol(T)^{-1} \rho^{-(\mathsf{r_5} - \mathsf{r_4})\hat{\varphi}(\alpha) + O(\sqrt{\mathsf{c}})}.
    \end{align*}
    This complete the proof of the corollary. 
\end{proof}

\section{Proof of Theorem~\ref{thm:Multislicing}}\label{sec:proof of multislice}
In this section we prove \cref{thm:Multislicing}. 
\subsection{Subcritical multi-slicing theorem}We first prove the following subcritical estimate for covering via tubes in $\mathcal{D}_{\rho}^{\mathbf{r}}$. Recall that we always assume the $5$-tuple $\mathbf{r} = (\mathsf{r_1}, \mathsf{r_2}, \mathsf{r_3}, \mathsf{r_4}, \mathsf{r_5})$ satisfying $0 \leq \mathsf{r_1} \leq \mathsf{r_2} \leq \mathsf{r_3} \leq \mathsf{r_4} \leq \mathsf{r_5} \leq 1$. 
\begin{proposition}\label{pro:SlabSubcritical}
    Fix a $5$-tuple $\mathbf{r} = (\mathsf{r}_1, \mathsf{r}_2, \mathsf{r}_3, \mathsf{r}_4, \mathsf{r}_5)$.     
    There exists an absolute constant $M_2 > 0$ such that the following holds for all $0 < \iota \ll_{\mathbf{r}} 1$ and $\rho \ll_{\iota, \mathbf{r}} 1$. 

    Let $A \subseteq B^{\mathfrak{r}}_1(0)$ be a set that is regular with respect to the filtration 
    \begin{align*}
        \mathcal{D}_{\rho^{\mathsf{r}_1}} \prec \mathcal{D}_{\rho^{\mathsf{r}_2}} \prec \mathcal{D}_{\rho^{\mathsf{r}_3}} \prec \mathcal{D}_{\rho^{\mathsf{r}_4}} \prec \mathcal{D}_{\rho^{\mathsf{r}_5}} \prec \mathcal{D}_{\rho}
    \end{align*}
    We define the exceptional set $\mathcal{E}(A)$ to be the following:
    \begin{align*}
        \mathcal{E}(A) = \Biggl\{u \in \mathsf{B}^U_1: \exists A' \subseteq A {}&\text{ with } |A'|_\rho \geq \rho^\iota|A|_\rho\\{}&\text{ and } |u.A'|_{\mathcal{D}_{\rho}^{\mathbf{r}}} < \rho^{M_2\iota} \prod_{i = 1}^5 |A|_{\rho^{\mathsf{r_i}}}^\frac{\mathsf{d_i}}{9}\Biggr\}.
    \end{align*}

    We have
    \begin{align*}
        m_U(\mathcal{E}(A)) \leq \rho^{\iota}.
    \end{align*}
\end{proposition}

\begin{proof}
    This follows from the proof of \cite[Proposition 2.8]{BH24} and the projection theorems in \cref{sec:Subcritical}. Replace the base case $m = 1$, $\mathsf{r_1} = 0$ in \cite[Proof of proposition 2.8]{BH24} by \cref{thm:subcritical 9 to 1,thm: subcritical 9 to 3,thm: subcritical 9 to 6,thm:subcritical 9 to 8}, the rest arguments are the same. We record the proof here for reader's convenience. 

    Let $m$ be the cardinality of the \emph{set} $\{\mathsf{r_1}, \mathsf{r_2}, \mathsf{r_3}, \mathsf{r_4}, \mathsf{r_5}\}$. The proposition is obvious when $m = 1$. We will prove it by induction when $m \geq 2$. 

    Suppose $m = 2$. Write $\mathbf{r} = (\mathsf{r_1}, \cdots, \mathsf{r_1}, \mathsf{r_2}, \cdots, \mathsf{r_2})$. Without loss of generality, we assume $\mathsf{r_2} = 1$. If not, replace $\rho$ by $\rho^{\mathsf{r_2}}$ and the proposition follows immediately. 
    
    The tuple $\mathbf{r} = (\mathsf{r_1}, \cdots, \mathsf{r_1}, \mathsf{r_2}, \cdots, \mathsf{r_2})$ corresponds to a flag $\mathfrak{r} \supset \mathfrak{r}^{(\lambda)} \supset \{0\}$ for some $\lambda$. We set $j_1 = \dim \mathfrak{r} - \dim \mathfrak{r}^{(\lambda)}$ and $j_2 = \dim \mathfrak{r}^{(\lambda)}$. For example, if $\mathbf{r} = (\mathsf{r_1}, \mathsf{r_1}, \mathsf{r_1}, \mathsf{r_2}, \mathsf{r_2})$, then $\lambda = 1$, $j_1 = 6$, $j_2 = 3$. 
    
    Let $\rho_1 = \rho^{\mathsf{r_1}}$, $\rho_2 = \rho^{\mathsf{r_2}} = \rho$ and $\mathbf{s} = (0, \cdots, 0, 1, \cdots, 1)$. We have
    \begin{align*}
        |u.A'|_{\mathcal{D}_{\rho}^{\mathbf{r}}} \gg{}& \sum_{Q \in \mathcal{D}_{\rho_{1}}} |u.A'_Q |_{\mathcal{D}_{\rho}^{\mathbf{r}}}\\
        \gg{}& \sum_{Q \in \mathcal{D}_{\rho_{1}}} |u.A'_Q |_{\mathcal{D}_{\rho_2}^{\mathbf{s}}}.
    \end{align*}
    Let $M$ be a positive constant so that the conclusions in \cref{thm:subcritical 9 to 1,thm: subcritical 9 to 3,thm: subcritical 9 to 6,thm:subcritical 9 to 8} hold. Applying one of \cref{thm:subcritical 9 to 1,thm: subcritical 9 to 3,thm: subcritical 9 to 6,thm:subcritical 9 to 8} according to the corresponding $\lambda$, we have that for all $Q \in \mathcal{D}_{\rho_{1}}$ with $|A'_Q|_{\rho_2} \geq \rho_2^{2\iota} |A_Q|_{\rho_2}$, there exists $\mathcal{E}_Q$ with $m_U(\mathcal{E}_Q) \leq \rho_2^{4\iota}$ so that for all $u \notin \mathcal{E}_Q$, we have
    \begin{align*}
        |u.A'_Q|_{\mathcal{D}_{\rho_2}^{\mathbf{s}}} \geq{}& \rho_2^{4M\iota} |A_Q|_{\rho_2}^{\frac{j_2}{9}}\\
        \geq{}& \rho_2^{5M\iota} |A|_{\rho_1}^{-\frac{j_2}{9}}|A|_{\rho_2}^{\frac{j_2}{9}}.
    \end{align*}
    The last inequality follows from the regularity of $A$. 
        
    Let 
    \begin{align*}
        \mathcal{D}_{\rho_1}^{\mathrm{large}}(A') = \{Q \in \mathcal{D}_{\rho_1}(A):|A' \cap Q|_{\rho_2} \geq \rho_2^{2\iota}|A \cap Q|_{\rho_2}\}
    \end{align*}
    and let
    \begin{align*}
        \mathcal{D}_{\rho_1}(u) = \{Q \in \mathcal{D}_{\rho_1}(A): u \in \mathcal{E}_Q\}.
    \end{align*}
    Since $|A'|_\rho \geq \rho^{\iota}|A|_\rho$, we have
    \begin{align*}
        \# \mathcal{D}_{\rho_1}^{\mathrm{large}}(A') \geq \rho^{2\iota}|A|_{\rho_1}.
    \end{align*}
    Applying Fubini's theorem, there exists $\mathcal{E} \subseteq \mathsf{B}^U_1$ with $m_U(\mathcal{E}) \leq \rho^{\iota}$ so that for all $u \notin \mathcal{E}$, we have
    \begin{align*}
        \# \mathcal{D}_{\rho_1}(u) \leq \rho^{3\iota}|A|_{\rho_1}.
    \end{align*}
    Combining all above estimates, we have
    \begin{align*}
        |u.A'|_{\mathcal{D}_{\rho}^{\mathbf{r}}} \gg{}& \Biggl(\# \mathcal{D}_{\rho_1}^{\mathrm{large}}(A') -\#\mathcal{D}_{\rho_1}(u)\Biggr)\rho^{5M\iota}|A|_{\rho_1}^{-\frac{j_2}{9}}|A|_{\rho_2}^{\frac{j_2}{9}}\\
        \geq{}& \rho^{(5M + 3)\iota}|A|_{\rho_1}^{\frac{j_1}{9}}|A|_{\rho_2}^{\frac{j_2}{9}}.
    \end{align*}
    This proves the proposition in the case where $m = 2$. 

    We now prove the inductive step. Suppose the proposition holds for $m$ and we now prove it holds for $m + 1$. As in the base case $m = 2$, we write 
    \begin{align*}
        \mathbf{r} = (\mathsf{r_1}, \ldots, \mathsf{r_1}, \mathsf{r_2}, \ldots, \mathsf{r_2}, \ldots, \mathsf{r_{m + 1}}, \ldots, \mathsf{r_{m + 1}}).
    \end{align*}
    Without loss of generality, we assume $\mathsf{r_m} = 1$. Otherwise we replace $\rho$ by $\rho^{\mathsf{r_m}}$ and the proposition follows immediately. 
    
    There is a flag associate to the tuple $\mathbf{r}$:
    \begin{align*}
        \mathfrak{r} = \mathfrak{r}^{(\lambda_1)} \supset \mathfrak{r}^{(\lambda_2)} \supset \ldots \supset \mathfrak{r}^{(\lambda_{m + 1})} \supset \{0\}
    \end{align*}
    where $\lambda_1 = -2$. Let $j_i = \dim \mathfrak{r}^{(\lambda_i)} - \dim \mathfrak{r}^{(\lambda_{i + 1})}$ for $i = 1, \ldots, m$ and $j_{m + 1} = \dim \mathfrak{r}^{(\lambda_{m+ 1})}$. For example, if $m + 1 = 5$, then $j_i = \mathsf{d_i}$. 

    Let
    \begin{align*}
        \mathbf{s} ={}& \mathbf{r} \vee (\mathsf{r_2}, \ldots, \mathsf{r_2}) = (\mathsf{r_2}, \ldots, \mathsf{r_2}, \mathsf{r_2}, \ldots, \mathsf{r_2}, \ldots, \mathsf{r_{m + 1}}, \ldots, \mathsf{r_{m + 1}}),\\
        \mathbf{t} ={}& \mathbf{r} \wedge (\mathsf{r_2}, \ldots, \mathsf{r_2}) = (\mathsf{r_1}, \ldots, \mathsf{r_1}, \mathsf{r_2}, \ldots, \mathsf{r_2},  \ldots, \mathsf{r_2}, \ldots, \mathsf{r_2}).
    \end{align*}
    Then $\mathbf{s}$ will corresponds to the flag
    \begin{align*}
        \mathfrak{r} = \mathfrak{r}^{(\lambda_1)} \supset \mathfrak{r}^{(\lambda_3)} \supset \ldots \supset \mathfrak{r}^{(\lambda_{m + 1})} \supset \{0\}
    \end{align*}
    with dimension difference corresponds to $(j_1 + j_2, j_3, \ldots, j_{m + 1})$. Similarly, $\mathbf{t}$ will corresponds to the flag
    \begin{align*}
        \mathfrak{r} = \mathfrak{r}^{(\lambda_1)} \supset \mathfrak{r}^{(\lambda_2)} \supset \{0\}
    \end{align*}
    with dimension difference corresponds to $(j_1, j_2 + \cdots + j_{m + 1})$. 
    
    We note that we have the following relations between those partitions:
    \begin{align*}
        &\mathcal{D}_{\rho}^{\mathbf{r}} \vee \mathcal{D}_{\rho^{\mathsf{r_2}}} = \mathcal{D}_{\rho}^{\mathbf{s}},\\
        &\mathcal{D}_{\rho}^{\mathbf{t}} \prec \mathcal{D}_{\rho}^{\mathbf{r}}, \mathcal{D}_{\rho}^{\mathbf{t}} \prec \mathcal{D}_{\rho^{\mathsf{r_2}}}.
    \end{align*}
    For any subset $A' \subseteq A$ with $|A'|_\rho \geq \rho^\iota|A|_{\rho}$, we apply \cref{lem:regularization set} with respect to the filtration
    \begin{align*}
        u^{-1}\mathcal{D}_{\rho}^{\mathbf{t}} \prec u^{-1}\mathcal{D}_{\rho}^{\mathbf{s}} \prec u^{-1}\mathcal{D}_{\rho}.
    \end{align*}
    This provides $A_1 \subseteq A'$ with $|A_1|_{u^{-1}\mathcal{D}_{\rho}} \geq \rho^\iota|A'|_{u^{-1}\mathcal{D}_{\rho}}$ which is regular with respect to the above filtration. Since $u \in \mathsf{B}^U_1$, we have
    \begin{align*}
        |B|_{\rho} \asymp |u.B|_\rho = |B|_{u^{-1}\mathcal{D}_{\rho}}
    \end{align*}
    for any set $B \subseteq B_{\mathfrak{r}}(0, 1)$. Therefore, 
    \begin{align*}
        |A_1|_\rho \gg \rho^\iota|A'|_\rho \geq \rho^{2\iota}|A|_\rho.
    \end{align*}
    By picking $\rho$ small enough depending only on $\iota$, we have
    \begin{align*}
        |A_1|_\rho \geq \rho^{3\iota}|A|_\rho.
    \end{align*}

    We now apply \cref{lem:submodularity} to $A_1$ and $c = \frac{1}{2}$ with respect to the partitions
    \begin{align*}
        &\mathcal{P} = u^{-1}\mathcal{D}_{\rho}^{\mathbf{r}}, \mathcal{Q} = u^{-1}\mathcal{D}_{\rho^{\mathsf{r_2}}},\\
        &\mathcal{R} = \mathcal{P} \vee \mathcal{Q} = u^{-1}\mathcal{D}_{\rho}^{\mathbf{r}}, \\
        &\mathcal{S} = u^{-1}\mathcal{D}_{\rho}^{\mathbf{t}}.
    \end{align*}
    There exists $A_1'$ with $|A_1'|_{\mathcal{R}} \gg |A_1|_{\mathcal{R}}$ so that 
    \begin{align*}
        |u.A_1|_{\mathcal{D}_{\rho}^{\mathbf{r}}} |u.A_1|_{\rho^{\mathsf{r_2}}} \gg |u.A_1|_{\mathcal{D}_{\rho}^{\mathbf{s}}}|u.A_1'|_{\mathcal{D}_{\rho}^{\mathbf{t}}}.
    \end{align*}
    Since $A_1$ is regular with respect to $u^{-1}\mathcal{D}_{\rho}^{\mathbf{t}} \prec u^{-1}\mathcal{D}_{\rho}^{\mathbf{s}}$, by \cref{lem:weaker regular} we have
    \begin{align*}
        |u.A_1'|_{\mathcal{D}_{\rho}^{\mathbf{t}}} \gg |u.A_1|_{\mathcal{D}_{\rho}^{\mathbf{t}}}.
    \end{align*}
    Therefore, we have
    \begin{align}\label{eqn:subcritical submodularity}
        |u.A'|_{\mathcal{D}_{\rho}^{\mathbf{r}}} |A'|_{\rho^{\mathsf{r_2}}} \gg{}&|u.A_1|_{\mathcal{D}_{\rho}^{\mathbf{r}}} |u.A_1|_{\rho^{\mathsf{r_2}}}\\
        \gg{}& |u.A_1|_{\mathcal{D}_{\rho}^{\mathbf{s}}}|u.A_1|_{\mathcal{D}_{\rho}^{\mathbf{t}}}.
    \end{align}
    We now estimate each term in the right side of the inequality using the inductive hypothesis and the base case. 
    
    Recall that by our construction of $A_1$, we have $|A_1|_\rho \geq \rho^{3\iota}|A|_{\rho}$. Applying the inductive hypothesis to $\mathbf{s}$, $A$, and $3\iota$, there exist $M(\mathbf{s}) > 0$ and $\mathcal{E}_{\mathbf{s}} \subset \mathsf{B}^U_1$ with $m_U(\mathcal{E}_{\mathbf{s}}) \leq \rho^{3\iota}$ so that for all $u \notin \mathcal{E}_{\mathbf{s}}$, we have
    \begin{align}\label{eqn:subcritical inductive hypothesis}
        |u.A_1|_{\mathcal{D}_{\rho}^{\mathbf{s}}} \geq \rho^{3M(\mathbf{s})\iota} |A|_{\rho^{\mathsf{r_2}}}^{\frac{j_1 + j_2}{9}}\prod_{i = 3}^{m + 1}|A|_{\rho^{\mathsf{r_i}}}^{\frac{j_i}{9}}.
    \end{align}
    Applying the base case where $m = 2$ to $\mathbf{t}$, $A$, and $3\iota$, there exist $M(\mathbf{t}) > 0$ and $\mathcal{E}_{\mathbf{t}}$ with $m_U(\mathcal{E}_{\mathbf{t}}) \leq \rho^{3\iota}$ so that for all $u \notin \mathcal{E}_{\mathbf{t}}$, we have
    \begin{align}\label{eqn:subcritical base in inductive step}
        |u.A_1|_{\mathcal{D}_{\rho}^{\mathbf{t}}} \geq \rho^{3M(\mathbf{t})\iota} |A|_{\rho^{\mathsf{r_1}}}^{\frac{j_1}{9}}|A|_{\rho^{\mathsf{r_2}}}^{\frac{9 - j_1}{9}}.
    \end{align}

    Combine \cref{eqn:subcritical submodularity,eqn:subcritical inductive hypothesis,eqn:subcritical base in inductive step}, we have
    \begin{align*}
        |u.A'|_{\mathcal{D}_{\rho}^{\mathbf{r}}} \gg \rho^{3(M(\mathbf{s}) + M(\mathbf{t}))\iota}\prod_{i = 1}^{m + 1}|A|_{\rho^{\mathsf{r_i}}}^{\frac{j_i}{9}}
    \end{align*}
    for all $u \notin \mathcal{E}_{\mathbf{s}} \cup \mathcal{E}_{\mathbf{t}}$. Let $\mathcal{E} = \mathcal{E}_{\mathbf{s}} \cup \mathcal{E}_{\mathbf{t}}$, we have
    \begin{align*}
        m_U(\mathcal{E}) \leq \rho^{\iota}
    \end{align*}
    if $\rho \ll_{\iota} 1$. For all $u \notin \mathcal{E}$, we have
    \begin{align*}
        |u.A'|_{\mathcal{D}_{\rho}^{\mathbf{r}}} \geq \rho^{3(M(\mathbf{s}) + M(\mathbf{t}) + 1)\iota}\prod_{i = 1}^{m + 1}|A|_{\rho^{\mathsf{r_i}}}^{\frac{j_i}{9}}
    \end{align*}
    if $\rho \ll_{\iota} 1$. By our construction, both $\mathbf{s}$ and $\mathbf{t}$ depends only on $\mathbf{r}$. Therefore, the new constant $M = 3(M(\mathbf{s}) + M(\mathbf{t}) + 1)$ depends only on $\mathbf{r}$. This complete the proof of the inductive step and hence the proposition. 
\end{proof}

\subsection{Proof of Theorem~\ref{thm:Multislicing}}
For simplicity, we say a measure or a set is regular in this subsection if it is regular with respect to the filtration
\begin{align*}
    \mathcal{D}_1 \prec \mathcal{D}_{\rho^{\mathsf{r}_1}} \prec \mathcal{D}_{\rho^{\mathsf{r}_2}} \prec \mathcal{D}_{\rho^{\mathsf{r}_3}} \prec \mathcal{D}_{\rho^{\mathsf{r}_4}} \prec \mathcal{D}_{\rho^{\mathsf{r}_5}} \prec \mathcal{D}_\rho.
\end{align*}

\begin{proof}[Proof of \cref{thm:Multislicing} when $\mu_F$ is regular]
This is a variant of \cite[Proof of theorem 2.1]{BH24}. The idea is straightforward. We use $\mathcal{D}_{\rho^{\mathsf{r_i}}}$ to refine $\mathcal{D}_{\rho}^{\mathbf{r}}$ and apply the sub-modularity inequality \cref{lem:submodularity}. At the end, we will end up with one partition of form $\mathcal{D}_{\rho}^{\mathbf{t}}$ with $\mathbf{t} = (\mathsf{r_4}, \mathsf{r_4}, \mathsf{r_4}, \mathsf{r_4}, \mathsf{r_5})$ and some other partitions. We apply \cref{cor:optimal top multislicing} to the former and \cref{pro:SlabSubcritical} to the latter and prove the estimate. 

Let $\epsilon \ll_{\mathbf{r}} 1$ as in \cref{pro:SlabSubcritical,cor:optimal top multislicing} and let $\rho$ small enough so that all quantity of the form $O(|\mathsf{r_i}\log \rho|)$ is dominated by $\rho^{-\epsilon}$. 

We recall that by \cref{lem:regular measure vs regular set}, the set $F$ is also regular with respect to the filtration
\begin{align*}
    \mathcal{D}_1 \prec \mathcal{D}_{\rho^{\mathsf{r}_1}} \prec \mathcal{D}_{\rho^{\mathsf{r}_2}} \prec \mathcal{D}_{\rho^{\mathsf{r}_3}} \prec \mathcal{D}_{\rho^{\mathsf{r}_4}} \prec \mathcal{D}_{\rho^{\mathsf{r}_5}} \prec \mathcal{D}_\rho.
\end{align*}
Therefore we can apply \cref{pro:SlabSubcritical} to $F$. 
    
In this case, when $\mathsf{r_4} = \mathsf{r_5}$, the theorem follows directly from \cref{pro:SlabSubcritical}. Therefore we will assume $\mathsf{r_4} < \mathsf{r_5}$. 

Let
\begin{align*}
    \mathbf{t}_1 ={}& \mathbf{r} \vee (\mathsf{r}_2, \cdots, \mathsf{r}_2)\\
    \mathbf{t}_2 ={}& \mathbf{r} \vee (\mathsf{r}_3, \cdots, \mathsf{r}_3)\\
    \mathbf{t}_3 ={}& \mathbf{r} \vee (\mathsf{r}_4, \cdots, \mathsf{r}_4) = (\mathsf{r_{4}}, \mathsf{r_{4}}, \mathsf{r_{4}}, \mathsf{r_{4}}, \mathsf{r_{5}})\\
    \mathbf{s}_1 ={}& (\mathsf{r_{1}}, \mathsf{r_{2}}, \mathsf{r_{2}}, \mathsf{r_{2}}, 
    \mathsf{r_{2}})\\
    \mathbf{s}_2 ={}& (\mathsf{r_{2}}, \mathsf{r_{2}}, \mathsf{r_{3}}, \mathsf{r_{3}}, 
    \mathsf{r_{3}})\\
    \mathbf{s}_3 ={}& (\mathsf{r_{3}}, \mathsf{r_{3}}, \mathsf{r_{3}}, \mathsf{r_{4}}, 
    \mathsf{r_{4}}).
\end{align*}
Let $M_2$ be as in \cref{pro:SlabSubcritical}. Recall that we set
\begin{align*}
    \hat{\varphi}(\alpha) = \min\{\alpha, 1\} - \frac{1}{9}\alpha
\end{align*}
and $\varphi(\alpha) = \frac{1}{36}\hat{\varphi}(\alpha)$. If for one of $i \in \{1, 2, 3, 4, 5\}$, we have
\begin{align*}
    |F|_{\rho^{\mathsf{r_{i}}}} \geq C^{-1}\rho^{-\mathsf{r_{i}}\alpha}\rho^{- \frac{\mathsf{r_5} - \mathsf{r_4}}{4\mathsf{d_{i}}}\hat{\varphi}(\alpha)-\frac{18}{\mathsf{d_{i}}}M_2\epsilon},
\end{align*}
then \cref{pro:SlabSubcritical} proves the theorem directly. Indeed, for all $F' \subseteq F$ with $\mu_F(F') \geq \rho^{\epsilon}$, we have $|F'|_\rho \geq \rho^{2\epsilon}|F|_\rho$ via \cref{lem:regular measure vs regular set}. Applying \cref{pro:SlabSubcritical} with $2\epsilon$, there exists $\mathcal{E} \subset \mathsf{B}^U_1$ with $m_U(\mathcal{E}) \leq \rho^{2\epsilon}$ so that for all $u \notin \mathcal{E}$, we have
\begin{align*}
    |u.F'|_{\mathcal{D}_\rho^{\mathbf{r}}} \geq{}& \rho^{2M_2\epsilon} \prod_{i = 1}^5 |F|_{\rho^{\mathsf{r_i}}}^\frac{\mathsf{d_i}}{9}\\
    \geq{}& \rho^{-\frac{(\mathsf{r_5} - \mathsf{r_4})}{36}\hat{\varphi}(\alpha)}\prod_{i = 1}^5 C^{-\frac{\mathsf{d_i}}{9}}\rho^{-\mathsf{r_i}\alpha\frac{\mathsf{d_i}}{9}}.
\end{align*}
Note that for an atom $T \in \mathcal{D}_{\rho}^{\mathbf{r}}$, its volume satisfies the following estimate
\begin{align*}
    \vol(T) \sim \rho^{\sum_{i = 1}^5 \mathsf{d_i r_i}}.
\end{align*}
Therefore, in this case, we have
\begin{align*}
    |u.F'|_{\mathcal{D}_\rho^{\mathbf{r}}} \geq C^{-1} \rho^{-(\mathsf{r_5 - r_4})\varphi(\alpha)} \vol(T)^{-\frac{\alpha}{9}}
\end{align*}
which proves the theorem. 

If not, we have
\begin{align*}
    C^{-1}\rho^{-\mathsf{r_{i}}\alpha} \leq |F|_{\rho^{\mathsf{r_{i}}}} \leq C^{-1}\rho^{-\mathsf{r_{i}}\alpha}\rho^{- \frac{\mathsf{r_5} - \mathsf{r_4}}{4\mathsf{d_{i}}}\hat{\varphi}(\alpha)-\frac{9}{\mathsf{d_{i}}}M_2\epsilon}
\end{align*}
for all $i = 1, \cdots, 5$. Note that since $u \in \mathsf{B}^U_1$, for all $i = 1, \cdots, 5$, we have
\begin{align}\label{eqn:multislicing dimension not too large}
    C^{-1}\rho^{-\mathsf{r_{i}}\alpha} \ll |u.F|_{\rho^{\mathsf{r_{i}}}} \ll C^{-1}\rho^{-\mathsf{r_{i}}\alpha}\rho^{- \frac{\mathsf{r_5} - \mathsf{r_4}}{4\mathsf{d_{i}}}\hat{\varphi}(\alpha)-\frac{9}{\mathsf{d_{i}}}M_2\epsilon}
\end{align}

Applying \cref{cor:optimal top multislicing} to $F$, $\mathbf{t}_3$ and $4\epsilon$, we get an exceptional set $\mathcal{E}_{\mathbf{t}_3} \subseteq \mathsf{B}^U_1$ with $m_U(\mathcal{E}_{\mathbf{t}_3}) \ll_\epsilon \rho^{6\epsilon}$. For $i = 1, 2, 3$, applying \cref{pro:SlabSubcritical} to $F$, $\mathbf{s}_i$ and $4\epsilon$, we get exceptional sets $\mathcal{E}_{\mathbf{s}_i} \subseteq \mathsf{B}^U_1$ with $m_U(\mathcal{E}_{\mathbf{s}_i}) \ll_\epsilon \rho^{6\epsilon}$. 

For all $u \notin (\cup_{i} \mathcal{E}_{\mathbf{s}_i}) \cup \mathcal{E}_{\mathbf{t}_3}$ and all $F' \subseteq F$ with $\mu_F(F') \geq \rho^{\epsilon}$, we apply \cref{lem:regularization measure} to $F'$ and the filtration
\begin{align}\label{eqn:multislicing filtration 1}
u^{-1}\mathcal{D}_{\rho}^{\mathbf{s_1}} \prec u^{-1}\mathcal{D}_{\rho}^{\mathbf{t_1}} \prec u^{-1}\mathcal{D}_{\rho^{\mathsf{r_{5}}}} \prec u^{-1}\mathcal{D}_{\rho},
\end{align}
This provides $F_1 \subseteq F'$ with $\mu_{F'}(F_1) \geq \rho^{\epsilon}$ so that $\mu_{F_1}$ is regular with respect to the above filtration. Applying \cref{lem:submodularity} to $F_1$, $u^{-1}\mathcal{D}_{\rho}^{\mathbf{r}}$, $u^{-1}\mathcal{D}_{\rho^{\mathsf{r_2}}}$, $u^{-1}\mathcal{D}_{\rho}^{\mathbf{t}_1} = u^{-1}\mathcal{D}_{\rho}^{\mathbf{r}} \vee u^{-1}\mathcal{D}_{\rho^{\mathsf{r_2}}}$, $u^{-1}\mathcal{D}_{\rho}^{\mathbf{s}_1}$ and $c = \frac{1}{2}$, there exists $F_1' \subseteq F_1$ with 
\begin{align*}
    |F_1'|_{u^{-1}\mathcal{D}_{\rho}^{\mathbf{t}_1}} \gg |F_1|_{u^{-1}\mathcal{D}_{\rho}^{\mathbf{t}_1}}.
\end{align*}
so that the following holds:
\begin{align*}
    |u.F'|_{\mathcal{D}_{\rho}^{\mathbf{r}}} \cdot |u.F'|_{\mathcal{D}_{\rho^{\mathsf{r_2}}}} \geq |u.F_1|_{\mathcal{D}_{\rho}^{\mathbf{t}_1}} \cdot |u.F_1|_{\mathcal{D}_{\rho^{\mathsf{r_2}}}} \gg |u.F_1|_{\mathcal{D}_{\rho}^{\mathbf{t}_1}} \cdot |u.F_1'|_{\mathcal{D}_{\rho}^{\mathbf{s}_1}}
\end{align*}

Since $F_1$ is regular with respect to the filtration in \cref{eqn:multislicing filtration 1}, we have
\begin{align*}
    |F_1'|_{u^{-1}\mathcal{D}_{\rho}^{\mathbf{s}_1}} \gg |F_1|_{u^{-1}\mathcal{D}_{\rho}^{\mathbf{s}_1}}.
\end{align*}
Therefore, we have
\begin{align}\label{eqn:submodularity 1}
    |u.F'|_{\mathcal{D}_{\rho}^{\mathbf{r}}} \cdot |u.F'|_{\mathcal{D}_{\rho^{\mathsf{r_2}}}} \gg |u.F_1|_{\mathcal{D}_{\rho}^{\mathbf{t}_1}} \cdot |u.F_1|_{\mathcal{D}_{\rho}^{\mathbf{s}_1}}.
\end{align}

Applying \cref{lem:regularization measure} to $F_1$ and the filtration
\begin{align}\label{eqn:multislicing filtration 2}
    u^{-1}\mathcal{D}_{\rho}^{\mathbf{s}_2} \prec u^{-1}\mathcal{D}_{\rho}^{\mathbf{t_2}} \prec u^{-1}\mathcal{D}_{\rho^{\mathsf{r_{5}}}} \prec u^{-1}\mathcal{D}_{\rho},
\end{align}
we get a subset $F_2 \subset F_1$ with $\mu_{F_1}(F_2) \geq \rho^\epsilon$ so that $\mu_{F_2}$ is regular with respect to the above filtration. Applying \cref{lem:submodularity} to $F_2$, $u^{-1}\mathcal{D}_{\rho}^{\mathbf{t}_1}$, $u^{-1}\mathcal{D}_{\rho^{\mathsf{r_3}}}$, $u^{-1}\mathcal{D}_{\rho}^{\mathbf{t}_2} = u^{-1}\mathcal{D}_{\rho}^{\mathbf{t}_1} \vee u^{-1}\mathcal{D}_{\rho^{\mathsf{r_3}}}$, $u^{-1}\mathcal{D}_{\rho}^{\mathbf{s}_2}$ and $c = \frac{1}{2}$, there exists $F_2' \subseteq F_2$ with 
\begin{align*}
    |F_2'|_{u^{-1}\mathcal{D}_{\rho}^{\mathbf{t}_2}} \gg |F_2|_{u^{-1}\mathcal{D}_{\rho}^{\mathbf{t}_2}}
\end{align*}
so that the following holds:
\begin{align*}
    |u.F_1|_{\mathcal{D}_{\rho}^{\mathbf{t}_1}} |u.F_1|_{\mathcal{D}_{\rho^{\mathsf{r_3}}}} \geq |u.F_2|_{\mathcal{D}_{\rho}^{\mathbf{t}_1}} |u.F_2|_{\mathcal{D}_{\rho^{\mathsf{r_3}}}} \geq |u.F_2|_{\mathcal{D}_{\rho}^{\mathbf{t}_2}} |u.F_2'|_{\mathcal{D}_{\rho}^{\mathbf{s}_2}}.
\end{align*}

Since $F_2$ is regular with respect to the filtration in \cref{eqn:multislicing filtration 2}, we have
\begin{align*}
    |F_2'|_{u^{-1}\mathcal{D}_{\rho}^{\mathbf{s}_2}} \gg |F_2|_{u^{-1}\mathcal{D}_{\rho}^{\mathbf{s}_2}}.
\end{align*}
Therefore, we have
\begin{align}\label{eqn:submodularity 2}
    |u.F_1|_{\mathcal{D}_{\rho}^{\mathbf{t}_1}} |u.F_1|_{\mathcal{D}_{\rho^{\mathsf{r_3}}}} \gg  |u.F_2|_{\mathcal{D}_{\rho}^{\mathbf{t}_2}} |u.F_2|_{\mathcal{D}_{\rho}^{\mathbf{s}_2}}.
\end{align}

Applying \cref{lem:regularization measure} to $F_2$ and the filtration
\begin{align}\label{eqn:multislicing filtration 3}
    u^{-1}\mathcal{D}_{\rho}^{\mathbf{s}_3} \prec u^{-1}\mathcal{D}_{\rho}^{\mathbf{t_3}} \prec u^{-1}\mathcal{D}_{\rho^{\mathsf{r_{5}}}} \prec u^{-1}\mathcal{D}_{\rho},
\end{align}
we get a subset $F_3 \subset F_2$ with $\mu_{F_2}(F_3) \geq \rho^\epsilon$ so that $\mu_{F_3}$ is regular with respect to the above filtration. Applying \cref{lem:submodularity} to $F_3$, $u^{-1}\mathcal{D}_{\rho}^{\mathbf{t}_2}$, $u^{-1}\mathcal{D}_{\rho^{\mathsf{r_4}}}$, $u^{-1}\mathcal{D}_{\rho}^{\mathbf{t}_3} = u^{-1}\mathcal{D}_{\rho}^{\mathbf{t}_2} \vee u^{-1}\mathcal{D}_{\rho^{\mathsf{r_4}}}$, $u^{-1}\mathcal{D}_{\rho}^{\mathbf{s}_3}$ and $c = \frac{1}{2}$, there exists $F_3' \subseteq F_3$ with 
\begin{align*}
    |F_3'|_{u^{-1}\mathcal{D}_{\rho}^{\mathbf{t}_3}} \gg |F_3|_{u^{-1}\mathcal{D}_{\rho}^{\mathbf{t}_3}}
\end{align*}
so that the following holds:
\begin{align*}
    |u.F_2|_{\mathcal{D}_{\rho}^{\mathbf{t}_2}} |u.F_2|_{\mathcal{D}_{\rho^{\mathsf{r_4}}}} \geq |u.F_3|_{\mathcal{D}_{\rho}^{\mathbf{t}_2}} |u.F_3|_{\mathcal{D}_{\rho^{\mathsf{r_4}}}} \geq |u.F_3|_{\mathcal{D}_{\rho}^{\mathbf{t}_3}} |u.F_3'|_{\mathcal{D}_{\rho}^{\mathbf{s}_3}}.
\end{align*}

Since $F_3$ is regular with respect to the filtration in \cref{eqn:multislicing filtration 3}, we have
\begin{align*}
    |F_3'|_{u^{-1}\mathcal{D}_{\rho}^{\mathbf{s}_3}} \gg |F_3|_{u^{-1}\mathcal{D}_{\rho}^{\mathbf{s}_3}}.
\end{align*}
Therefore, we have
\begin{align}\label{eqn:submodularity 3}
    |u.F_2|_{\mathcal{D}_{\rho}^{\mathbf{t}_2}} |u.F_2|_{\mathcal{D}_{\rho^{\mathsf{r_4}}}} \gg  |u.F_3|_{\mathcal{D}_{\rho}^{\mathbf{t}_3}} |u.F_3|_{\mathcal{D}_{\rho}^{\mathbf{s}_3}}.
\end{align}

Combining \cref{eqn:submodularity 1,eqn:submodularity 2,eqn:submodularity 3}, we have
\begin{align}\label{eqn:submodularity final}
    |u.F'|_{\mathcal{D}_{\rho}^{\mathbf{r}}} \prod_{i = 2}^4 |u.F'|_{\mathcal{D}_\rho^{\mathsf{r_i}}} \gg_\epsilon |u.F_3|_{\mathcal{D}_{\rho}^{\mathbf{t}_3}} \prod_{i = 1}^3 |u.F_i|_{\mathcal{D}_\rho^{\mathbf{s}_i}}.
\end{align}

We now apply \cref{cor:optimal top multislicing,pro:SlabSubcritical} to bound the right hand side of the above inequality. Note that since $\mu_F(F_i) \geq \rho^{4\epsilon}$, by \cref{lem:regular measure vs regular set}, we have
\begin{align*}
    |F_i|_{\rho} \gg \rho^{4\epsilon}|F|_{\rho}.
\end{align*}
By our choice of $\rho$, we have
\begin{align*}
    |F_i|_{\rho} \geq \rho^{5\epsilon}|F|_{\rho}.
\end{align*}

Recall we have
\begin{align*}
    \mathbf{t}_3 ={}& (\mathsf{r_{4}}, \mathsf{r_{4}}, \mathsf{r_{4}}, \mathsf{r_{4}}, \mathsf{r_{5}})\\
    \mathbf{s}_1 ={}& (\mathsf{r_{1}}, \mathsf{r_{2}}, \mathsf{r_{2}}, \mathsf{r_{2}}, 
    \mathsf{r_{2}})\\
    \mathbf{s}_2 ={}& (\mathsf{r_{2}}, \mathsf{r_{2}}, \mathsf{r_{3}}, \mathsf{r_{3}}, 
    \mathsf{r_{3}})\\
    \mathbf{s}_3 ={}& (\mathsf{r_{3}}, \mathsf{r_{3}}, \mathsf{r_{3}}, \mathsf{r_{4}}, 
    \mathsf{r_{4}}).
\end{align*}
The volume of an atom $T^{\mathbf{t}_3}$ in $\mathcal{D}_{\rho}^{\mathbf{t}_3}$ has the following estimate:
\begin{align*}
    \vol(T^{\mathbf{t}_3}) \sim \rho^{8\mathsf{r_4} + \mathsf{r_5}}.
\end{align*}
Since $u \notin \mathcal{E}_{\mathbf{t}_3}$, we have the following lower bound for $|u.F_3|_{\mathcal{D}_{\rho}^{\mathbf{t}_3}}$ via \cref{cor:optimal top multislicing}:
\begin{align*}
    |u.F_3|_{\mathcal{D}_{\rho}^{\mathbf{t}_3}} \geq{}& C_{\epsilon, \mathbf{r}}^{-1} C^{-1}\vol(T^{\mathbf{t}_3})^{-\frac{\alpha}{9}} \rho^{-(\mathsf{r_5} - \mathsf{r_4})\hat{\varphi}(\alpha) + O_{\mathbf{r}}(\sqrt{\epsilon})}\\
    ={}& C_{\epsilon, \mathbf{r}}^{-1} C^{-1}\rho^{-(8\mathsf{r_4} + \mathsf{r_5})\frac{\alpha}{9}} \rho^{-(\mathsf{r_5} - \mathsf{r_4})\hat{\varphi}(\alpha) + O_{\mathbf{r}}(\sqrt{\epsilon})}.
\end{align*}
Since $u \notin \cup_{i = 1}^3 \mathcal{E}_{\mathbf{s}_i}$, we have the following lower bound for $|u.F_i|_{\mathcal{D}_{\rho}^{\mathbf{s}_i}}$ via \cref{pro:SlabSubcritical}:
\begin{align*}
    |u.F_1|_{\mathcal{D}_{\rho}^{\mathbf{s}_1}} &\geq \rho^{M_2\epsilon}|F|_{\rho^{\mathsf{r_1}}}^{\frac{1}{9}} |F|_{\rho^{\mathsf{r_2}}}^{\frac{8}{9}} \geq C^{-1}\rho^{-(\mathsf{r_1} + 8\mathsf{r_2})\frac{\alpha}{9} + O(\epsilon)},\\
    |u.F_2|_{\mathcal{D}_{\rho}^{\mathbf{s}_2}} &\geq \rho^{M_2\epsilon}|F|_{\rho^{\mathsf{r_2}}}^{\frac{3}{9}} |F|_{\rho^{\mathsf{r_3}}}^{\frac{6}{9}} \geq C^{-1}\rho^{-(3\mathsf{r_2} + 6\mathsf{r_3})\frac{\alpha}{9} + O(\epsilon)},\\
    |u.F_3|_{\mathcal{D}_{\rho}^{\mathbf{s}_3}} &\geq \rho^{M_2\epsilon}|F|_{\rho^{\mathsf{r_3}}}^{\frac{6}{9}} |F|_{\rho^{\mathsf{r_4}}}^{\frac{3}{9}} \geq C^{-1}\rho^{-(6\mathsf{r_3} + 3\mathsf{r_4})\frac{\alpha}{9} + O(\epsilon)}.
\end{align*}
Recall that 
\begin{align*}
    (\mathsf{d_1}, \mathsf{d_2}, \mathsf{d_3}, \mathsf{d_4}, \mathsf{d_5}) = (1, 2, 3, 2, 1).
\end{align*}
Putting all above estimate into \cref{eqn:submodularity final}, we have
\begin{align*}
    |u.F'|_{\mathcal{D}_{\rho}^{\mathbf{r}}} \prod_{i = 2}^4 |u.F'|_{\mathcal{D}_\rho^{\mathsf{r_i}}} \gg_\epsilon C_{\epsilon, \mathbf{r}}^{-1} C^{-4}\rho^{-(\mathsf{r_2} + \mathsf{r_3} + \mathsf{r_4})\alpha}\rho^{-(\sum_{i = 1}^5 \mathsf{d_i}\mathsf{r_i})\frac{\alpha}{9}}\rho^{-(\mathsf{r_5} - \mathsf{r_4})\hat{\varphi}(\alpha) + O_{\mathbf{r}}(\sqrt{\epsilon})}.
\end{align*}
Recall that $\mathbf{r} = (\mathsf{r_1}, \mathsf{r_2}, \mathsf{r_3}, \mathsf{r_4}, \mathsf{r_5})$. For an atom $T \in \mathcal{D}_\rho^{\mathbf{r}}$, its volume has the following estimate
\begin{align*}
    \vol(T) \sim \rho^{\sum_{i = 1}^5 \mathsf{d_i}\mathsf{r_i}}.
\end{align*}
Therefore, we have
\begin{align}\label{eqn:final calculation1}
    |u.F'|_{\mathcal{D}_{\rho}^{\mathbf{r}}} \prod_{i = 2}^4 |u.F'|_{\mathcal{D}_\rho^{\mathsf{r_i}}} \gg_\epsilon C_{\epsilon, \mathbf{r}}^{-1} C^{-4}\rho^{-(\mathsf{r_2} + \mathsf{r_3} + \mathsf{r_4})\alpha}\vol(T)^{-\frac{\alpha}{9}}\rho^{-(\mathsf{r_5} - \mathsf{r_4})\hat{\varphi}(\alpha) + O_{\mathbf{r}}(\sqrt{\epsilon})}.
\end{align}

Recall that we have upper bounds for $|u.F'|_{\rho^{\mathsf{r_i}}}$ for all $i = 2, 3, 4$ as in \cref{eqn:multislicing dimension not too large}: 
\begin{align*}
    |u.F'|_{\rho^{\mathsf{r_i}}} \leq{}& |u.F|_{\rho^{\mathsf{r_i}}}\\
    \ll{}& C^{-1}\rho^{-\mathsf{r_{i}}\alpha}\rho^{- \frac{\mathsf{r_5} - \mathsf{r_4}}{4\mathsf{d_{i}}}\hat{\varphi}(\alpha)-\frac{9}{\mathsf{d_{i}}}M_2\epsilon}.
\end{align*}
Combine it with \cref{eqn:final calculation1}, we have
\begin{align*}
    |u.F'|_{\mathcal{D}_{\rho}^{\mathbf{r}}} \geq C_{\epsilon, \mathbf{r}}^{-1}C^{-1}\vol(T)^{-\frac{\alpha}{9}} \rho^{-(1 - \sum_{i = 2}^4 \frac{1}{4\mathsf{d_i}})(\mathsf{r_5} - \mathsf{r_4})\hat{\varphi}(\alpha) + O_{\mathbf{r}}(\sqrt{\epsilon})}.
\end{align*}
Recall that $(\mathsf{d_1}, \mathsf{d_2},\mathsf{d_3}, \mathsf{d_4}, \mathsf{d_5}) = (1,2,3,2,1)$, we have
\begin{align*}
    |u.F'|_{\mathcal{D}_{\rho}^{\mathbf{r}}} \geq C_{\epsilon, \mathbf{r}}^{-1}C^{-1}\vol(T)^{-\frac{\alpha}{9}} \rho^{-\frac{2}{3}(\mathsf{r_5} - \mathsf{r_4})\hat{\varphi}(\alpha) + O_{\mathbf{r}}(\sqrt{\epsilon})},
\end{align*}
which proves the theorem. 
\end{proof}

\begin{proof}[Proof of \cref{thm:Multislicing} in general case]
    Replacing the exhaustion process in \cite[Proof of theorem 2.1, general case]{BH24} by \cref{lem:regularization exhaust}, the rest arguments are the same. We just remark here that applying \cref{lem:regularization exhaust} to $F$ with $\mathsf{c} = \rho^{2\epsilon}$, the output family of subsets $\{F_j\}$ satisfies the following Frostman-type condition:
    \begin{align*}
        \mu_{F_j}(B^{\mathfrak{r}}_r(x)) \leq \rho^{-3\epsilon}Cr^{\alpha} \quad \forall r \geq \rho_0. 
    \end{align*}
\end{proof}

\part{Proof of polynomially effective equidistribution theorem}\label{part:bootstrap}
As indicated in the introduction, the framework of the proof is similar to \cite{LMWY25}. We now show how to put the new ingredients from \cref{part:projection,part:closinglemma} into this framework. 

Recall that in \cite{LMWY25} (see also \cite{LMW22,LMWY23}), the proof can be roughly divided into three phases: 
\begin{enumerate}
    \item Initial dimension from effective closing lemma;
    \item Improving dimension using ingredients from projection theorems;
    \item From large dimension to equidistribution. 
\end{enumerate}

The last phase in our setting will be exactly the same as \cite[Section 5, 9]{LMWY25}. We will only state the result and point out the corresponding changes for parameters. This is done in \cref{sec:Venkatesh,sec:From large dimension to effective equidistribution}. 

The second phase is a bootstrap process and is the core of the proof. Section~\ref{sec:bootstrap} is devoted to this phase. In each step of the bootstrap, we need an improved estimate on Margulis function from a (linear) dimension improving lemma in the transverse complement $\mathfrak{r}$. In \cite{LMWY25}, the latter was established in section 6 (see Theorem 6.1 there) and the Margulis function estimate was established in section 7 (see Lemma 7.2 there). In this paper, the dimension improving lemma in $\mathfrak{r}$ is replaced by Theorem~\ref{thm:energy Improvement} proved in \cref{part:projection} and the Margulis function estimate is recorded in \cref{prop: integrated inductive step}. 

The whole bootstrap process in \cite[Section 8]{LMWY25} was initiated with the input \cite[Proposition 4.6]{LMWY25}. Here the initiating input is replaced by Theorem~\ref{thm:closing lemma initial dim} proved in \cref{part:closinglemma}. It is slightly weaker comparing to \cite[Proposition 4.6]{LMWY25}. However, it is enough to feed into the bootstrap process and produce a suitable output which can be in turn served as an input for the last phase. Due to this difference, we provide details on this process in \cref{sec:Important bootstrap calculation}. 

Combining all the ingredients, we prove Theorem~\ref{thm:main equidistribution} in Section~\ref{sec:ProofOfMainEqui}.

\section{Mixing and Equidistribution}\label{sec:Venkatesh}
The main result of this section is Lemma~\ref{lem:Venkatesh}. It is an analog of \cite[Lemma 5.2]{LMWY25}. 
\begin{lemma}\label{lem:Venkatesh}
    There exists $\varrho_0 \in (0, 1)$ depending only on $(G, H)$ so that the following holds. Let $\delta_0 \in (0, 1)$. Let $\ell_1, \ell_2 > 0$ with $\kappa_1\ell_2 \geq \max\{\ell_1, |\log \eta|\}$ and $8\ell_2 \leq |\log \delta_0|$, and let $\varrho \in (0, \varrho_0]$. Let $\mu$ be a probability measure on $B^{\mathfrak{r}}_{\varrho}(0)$ satisfying
    \begin{align*}
        \mu(B_\delta^{\mathfrak{r}}(w)) \leq \Upsilon \delta^{\dim\mathfrak{r}} \quad \forall w \in \mathfrak{r}, \delta \geq \delta_0.
    \end{align*}
    Then for all $\phi \in \mathrm{C}^\infty_c(X) + \C\mathds{1}_X$ and all $x \in X_\eta$, we have
    \begin{align*}
        &\int_{\mathsf{B}_1^U} \int_{\mathsf{B}_1^U}\int_{\mathfrak{r}} \phi(a_{\ell_1}u_1 a_{\ell_2} u_2 \exp(w).x)\,\mathrm{d}\mu(w)\,\mathrm{d}u_2\,\mathrm{d}u_1\\
        ={}& \int_X \phi \,\mathrm{d}\mu_X + O(\mathcal{S}(\phi)(\varrho^\star + \eta + \Upsilon^{\frac{1}{2}}\varrho^{-3}e^{-\kappa_1\ell_1})).
    \end{align*}
\end{lemma}

The proof of this lemma relies on spectral gap in the ambient space $X = G/\Gamma$ and Venkatesh's argument. 

Recall the following estimate on decay of matrix coefficient for the space $X$ from \cite[Section 2.4]{KM96}. There exists $\kappa_0 \in (0, 1)$ so that 
\begin{align}\label{eqn:spectral gap}
    \Biggl|\int_X \varphi(g.x)\psi(x) \,\mathrm{d}\mu_X(x) - \int_X \varphi \,\mathrm{d}\mu_X \int_X \psi \,\mathrm{d}\mu_X \Biggr| \ll \mathcal{S}(\varphi) \mathcal{S}(\psi) e^{-\kappa_0 d(e, g)}
\end{align}
for all $\varphi, \psi \in \mathrm{C}^\infty_c(X) + \C \mathds{1}_X$. Here $\mathcal{S}(\cdot)$ is a certain Sobolev norm on $\mathrm{C}_c^\infty(X) + \C \mathds{1}_X$ so that it dominates $\|\cdot\|_{\infty}$ and the Lipschitz norm $\|\cdot\|_{\Lip}$.

The following is an analog of \cite[Proposition 5.1]{LMWY25}. 
\begin{proposition}
    There exists $\kappa_1 \in (0, 1/3)$ with $\kappa_1 \gg \kappa_0$ so that the following holds. Let $\Lambda \geq 1$ and let $\nu \ll m_G$ be a probability measure on $\mathsf{B}_1^G$ with 
    \begin{align*}
        \frac{\mathrm{d}\nu}{\mathrm{d}m_G}(g) \leq \Lambda \quad \forall g \in \supp(\nu).
    \end{align*}
    Let $\ell_1, \ell_2 > 0$ and $\eta \in (0, 1)$ satisfy the following
    \begin{align*}
        \kappa_1 \ell_2 \geq \max\{\ell_1, |\log \eta|\}.
    \end{align*}
    Then for all $x \in X_\eta$ and all $\phi \in \mathrm{C}^\infty_c(X) + \C \mathds{1}_X$, we have
    \begin{align*}
        \int_{\mathsf{B}_1^U} \int_G \phi(a_{\ell_1}u a_{\ell_2}g.x) \,\mathrm{d}\nu(g)\,\mathrm{d}u = \int_X \phi \,\mathrm{d}\mu_X + O(\mathcal{S}(\phi)(\eta + \Lambda^{\frac{1}{2}}e^{-\kappa_1\ell_1})).
    \end{align*}
\end{proposition}

\begin{proof}
    The statement can be proved by following the proof of \cite[Proposition 5.1]{LMWY25} step-by-step. 
\end{proof}

\begin{proof}[Proof of Lemma~\ref{lem:Venkatesh}]
    The statement can be proved by following the proof of \cite[Lemma 5.2]{LMWY25} step-by-step. We indicate the change of parameter here. 
    
    For the condition $8\ell_2 \leq |\log \delta_0|$, it comes from the condition $e^{2(\ell_1 + \ell_2)}\delta_0 \leq e^{-\ell_1}$. See the paragraph before \cite[Equation (5.8)]{LMWY25}. We remark that the $\mathsf{m}$ in \cite{LMWY25} is the fastest expanding rate of $a_t$ in the complement $\mathfrak{r}$ and here is replaced by $2$. 

    For the $\varrho^{-3}$ in the last error term, it comes from the fact that $m_H(B^H_{\varrho}) \asymp \varrho^{6}$. Therefore, comparing to \cite[Equation (5.10)]{LMWY25}, the corresponding mollified measure $\nu$ should satisfy
    \begin{align*}
        \nu(B_\delta^G(g)) \ll \Upsilon \varrho^{-6} \delta^{\dim G} \quad \forall g \in G, \delta \in (0, 1).
    \end{align*}
\end{proof}

\section{Margulis function estimate and dimension improvement}\label{sec:bootstrap}
The main result of this section is \cref{pro:dim improving main}. It is an analog of \cite[Proposition 8.1]{LMWY25}. We first fix the following parameters. 

Let $\epsilon_0$ be the initial dimension in Theorem~\ref{thm:closing lemma initial dim} and let $\kappa_1$ be as in Theorem~\ref{lem:Venkatesh}. Set $\theta = (\frac{\min\{\kappa_1, \epsilon_0\}}{80})^{2} \in (0, \min\{\kappa_1, \epsilon_0\})$ and $p_{\mathrm{fin}} = \lceil 6480(\frac{9 - \epsilon_0}{\theta} - 1)\rceil$. 
We choose an arithmetic progression $\{\alpha_j\}_{j = 0}^{p_{\mathrm{fin}}}$ satisfying
\begin{itemize}[label=$\bullet$]
\item $\epsilon_0 = \alpha_0 < \alpha_1 < \alpha_2 < \cdots < \alpha_{p_{\mathrm{fin}}}$,
\item $\alpha_{j} - \alpha_{j - 1} = \frac{1}{72 \cdot 9 \cdot 10}\theta$ for all $1 \leq j \leq p_{\mathrm{fin}}$,
\item $\alpha_{{p_{\mathrm{fin}}} - 1} < 9 - \theta \leq \alpha_{p_{\mathrm{fin}}} < 9$.
\end{itemize}
Let $\epsilon = 10^{-10}(\frac{3}{4})^{p_{\mathrm{fin}}} \theta$. Note that all of these constants are absolute and that $\epsilon$ is much smaller than both $\epsilon_0$ and $\theta$. Moreover, $\frac{\epsilon}{\theta}$ is much smaller than both $\epsilon_0$ and $\theta$. 

Let
\begin{equation}\label{eq:def_Nj}
\begin{split}
N_0 = 0, N_1 = \Bigl\lceil \frac{25}{2\epsilon}\Bigr\rceil, \text{ and }
N_{j} = \lceil N_1(\tfrac{3}{4})^{j - 1}\rceil \text{ for } j = 1, \cdots, p_{\mathrm{fin}}.
\end{split}
\end{equation}
Set $d = \sum_{j = 0}^{p_{\mathrm{fin}}} N_j$. Note that all of $N_j$ depends only on $(G, H, \Gamma)$. 

Let us recall the constants $\ref{a:closing lemma main} > 1$, $\ref{c:closing lemma main} > 1$, $\ref{d:closing lemma main} > 1$, $\ref{e:closing lemma main1}, \ref{e:closing lemma main2}> 1$, $\ref{m:closing lemma main} > 1$, $\epsilon_0 > 0$ from the effective closing lemma (Theorem~\ref{thm:closing lemma initial dim}). They depend only on $(G, H, \Gamma)$. Also, let us recall the constant $\ref{d:avoidance}$ from the avoidance principle (Proposition~\ref{pro:avoidance}). It also depends only on $(G, H, \Gamma)$. Fix $D = \max\{\ref{d:closing lemma main}, \ref{d:avoidance}\} + 2$ where $\ref{d:closing lemma main}$ is as in the effective closing lemma (Theorem~\ref{thm:closing lemma initial dim}) and $\ref{d:avoidance}$ is as in the avoidance principle (Proposition~\ref{pro:avoidance}). Let $M = \ref{m:closing lemma main} + \ref{c:closing lemma main} D$ be as in Theorem~\ref{thm:closing lemma initial dim}. 

Fix $R > 1$ and $t = M\log R$. We will assume $R$ to be sufficiently large depending on the space $X$. Set 
\begin{align*}
\beta = e^{-\frac{1}{10^{10}M\ref{a:closing lemma main}\ref{e:closing lemma main1}\ref{e:closing lemma main2}d^2}t}
\end{align*}
and $\ell = \frac{\epsilon}{100M\ref{a:closing lemma main}}t$. Set $\eta = \beta^{1/2}$. Note that $R \gg \eta^{-\ref{e:closing lemma main1}}$ as in Theorem~\ref{thm:closing lemma initial dim}. Let $\delta_0 = R^{-\frac{1}{\ref{a:closing lemma main}}} = e^{-\frac{t}{M\ref{a:closing lemma main}}}$ be as in Theorem~\ref{thm:closing lemma initial dim}. 

Note that $e^{-\ell}$ is a much smaller scale than $\beta$. In particular, they satisfy the following relations:
\begin{align}\label{eqn:Much smaller scale}
    e^{-\epsilon^2 \ell} \leq \beta^{10^{10}\ref{e:closing lemma main1}\ref{e:closing lemma main2}}.
\end{align}
We assume $R$ is large enough so that 
\begin{align}\label{eqn:R large enough absolutely}
    e^{-\theta\ell} \leq 10^{-10000}.
\end{align}

\begin{proposition}\label{pro:dim improving main}
Let $x_1 \in X_\eta$. 
Suppose that for all periodic orbit $H.x'$ with $\vol(H.x') \leq R$, we have
\begin{align*}
    d_X(x_1, x') > R^{-D}.
\end{align*}

Then there exist a family of sheeted sets $\bigl\{\mathcal{E}_i^{\mathrm{fin}}\bigr\}_i$ with cross-sections $\bigl\{F_i^{\mathrm{fin}}\bigr\}_i$ and associated admissible measures $\bigl\{\mu_{\mathcal{E}_i^{\mathrm{fin}}}\bigr\}_i$ satisfying the following properties.
\begin{enumerate}
\item For all $\phi \in \mathrm{C}_c^{\infty}(X) + \C \mathds{1}_X$, $\ell' \geq 0$, and $u' \in \mathsf{B}_1^U$, we have
\begin{align*}
    \int_{\mathsf{B}_1^U} \phi(a_{\ell'}u' a_{d\ell + t} u.x_1) \,\mathrm{d}u
    = \sum_i c_i\int_{\mathcal{E}_i^{\mathrm{fin}}} \phi( a_{\ell'}u'.x) \,\mathrm{d}\mu_{\mathcal{E}_i^{\mathrm{fin}}}(x) + O(\mathcal{S}(\phi)\beta^\star)
\end{align*}
for some $c_i>0$ with $\sum_i c_i=1$.
\item For all $i$, we have
\begin{align*}
\#F_i^{\mathrm{fin}} \geq \delta_0^{-\frac{4\epsilon_0}{3}} = e^{\frac{4\epsilon_0 t}{3M\ref{a:closing lemma main}}}.
\end{align*}
\item Let $\delta_{\mathrm{fin}} = \delta_0^{\frac{2\epsilon}{\theta}} = e^{-\frac{2\epsilon t}{\theta M\ref{a:closing lemma main}}}$. For all $i$ we have
\begin{align*}
f_{\mathcal{E}_i^{\mathrm{fin}},\delta_{\mathrm{fin}}}^{(9 - \theta)}(x) \leq 2^{p_{\mathrm{fin}}}e^{20\ell}\#F_i^{\mathrm{fin}} \qquad \text{for all $x \in \mathcal{E}_i^{\mathrm{fin}}$}.
\end{align*}
\end{enumerate}
\end{proposition}

\subsection{Margulis function estimate}
The following proposition provides a general iterative process for improving the dimension. It is the analog of \cite[Lemma 5.2]{LMWY25} in this setting. 
\begin{proposition}\label{prop: integrated inductive step}
Let $\delta > 0$, $\alpha\in [\epsilon_0,\dim(\mathfrak{r}))$, and $0 < \Upsilon \leq e^{1/\beta}$.
Suppose that $\mathcal{E}$ is a sheeted set with cross-section $F$ so that
\begin{align*}
f_{\mathcal{E},\delta}^{(\alpha)}(x) \leq \Upsilon \qquad \text{for all $x \in \mathcal{E}$}.
\end{align*}
Assume further $\mathcal{E}$ is assigned with an admissible measure $\mu_{\mathcal{E}}$, see \cref{sec:boxes}.
Then there exists a family of sheeted sets $\{\mathcal{E}_i\}_i$ with cross-sections $\{F_i\}_i$ and associated admissible measures $\{\mu_{\mathcal{E}_i}\}_i$ satisfying the following properties.
\begin{enumerate}
    \item For all $\phi \in \mathrm{C}_c^{\infty}(X) + \C \mathds{1}_X$, $\ell' \geq 0$, and $u' \in \mathsf{B}_1^U$, we have
    \begin{align*}
        \int_{\mathsf{B}_1^U} \int_{\mathcal{E}} \phi(a_{\ell'}u' a_{\ell} u.x) \,\mathrm{d}\mu_{\mathcal{E}}(x) \,\mathrm{d}u
        = \sum_i c_{i}\int_{\mathcal{E}_i} \phi( a_{\ell'}u'.z) \,\mathrm{d}\mu_{\mathcal{E}_i}(z) + O(\mathcal{S}(\phi)\beta^\star)
    \end{align*}
    for some $c_i > 0$ with $\sum_i c_i=1$. 
    \item For all $i$, we have
    \begin{align*}
    \beta^{29} \#F 
    \leq \#F_i
    \leq e^{2\ell}\#F.
    \end{align*}
    \item For all $i$, we have
    \begin{align*}
        f_{\mathcal{E}_i, \delta'}^{(\alpha)}(x) \leq e^{-\frac{3}{4}\varphi(\alpha)\ell}\Upsilon + e^{2\alpha\ell}\beta^{-\alpha}\#F_i \qquad \text{for all $x \in \mathcal{E}_i$}
    \end{align*}
    where $\delta' = e^{2\ell}\max\{\delta, \#F^{-\frac{1}{\alpha}}\}$ and
    \begin{align*}
        \varphi(\alpha) = \frac{1}{36}\min\Bigl\{\frac{8}{9}\alpha, 1 - \frac{1}{9}\alpha\Bigr\}
    \end{align*}
    as in Theorem~\ref{thm:energy Improvement}. 
\end{enumerate}
\end{proposition}

\begin{proof}
    The statement can be proved following the proof of \cite[Lemma 7.2]{LMWY25} step-by-step and replacing \cite[Theorem 6.1]{LMWY25} by Theorem~\ref{thm:energy Improvement}. 
\end{proof}

\subsection{Proof of Proposition~\ref{pro:dim improving main}}\label{sec:Important bootstrap calculation}

The idea of the proof of Proposition~\ref{pro:dim improving main} is rather straight-forward. First, we apply Theorem~\ref{thm:closing lemma initial dim} to gain an initial dimension. Then we apply Proposition~\ref{prop: integrated inductive step} iteratively to improve the dimension. The following lemma is a direct consequence of \cref{prop: integrated inductive step}. It says that from a good sheeted set, we can do random walk in bounded many steps using \cref{prop: integrated inductive step} to get to a family of good sheeted sets with dimension $\alpha$ in the transverse direction. 

Recall from Theorem~\ref{thm:energy Improvement} that \begin{align*}
    \varphi(\alpha) = \frac{1}{36}\min\Bigl\{\frac{8}{9}\alpha, 1 - \frac{1}{9}\alpha\Bigr\}.
\end{align*}
Let 
\begin{align*}
\bar{\varphi}(\alpha) = \frac{1}{2}\varphi(\alpha) < \frac{3}{4}\varphi(\alpha).
\end{align*}
Note that for all $\alpha \in [\epsilon_0, \alpha_{p_{\mathrm{fin}}}]$, we have
\begin{align*}
    \bar{\varphi}(\alpha) \geq \frac{1}{72 \cdot 10} \theta.
\end{align*}

\begin{lemma}\label{lem:getting dim alpha}
Suppose $\alpha \in [\epsilon_0, \alpha_{p_{\mathrm{fin}}}]$, $C \geq 1$, $\Upsilon$, $\delta \in (0, 1)$, and a sheeted subset $\mathcal{E}^{(0)}$ with cross-section $F^{(0)}$ and associated admissible measure $\mu_{\mathcal{E}}$ are given with 
\begin{align*}
f^{(\alpha)}_{\mathcal{E}^{(0)}, \delta}(z) \leq C\Upsilon \quad \forall z \in \mathcal{E}^{(0)}.
\end{align*}
    
For all integers $N\geq 1$ there exists a family of sheeted sets $\{\mathcal{E}_i\}_i$ with cross-sections $\{F_i\}_i$ and associated admissible measures $\{\mu_{\mathcal{E}_i}\}_i$ satisfying the following properties.
\begin{enumerate}
    \item For all $\phi \in \mathrm{C}_c^{\infty}(X)$, $\ell' \geq 0$, and $u' \in \mathsf{B}_1^U$ we have
    \begin{align*}
        \int_{\mathsf{B}_1^{U}} \int_X\phi(a_{\ell'}u' a_{N\ell} u.x) \,\mathrm{d}\mu_{\mathcal{E}^{(0)}}\,\mathrm{d}u
        = \sum_i c_i\int_{\mathcal{E}_i} \phi( a_{\ell'}u'.x) \,\mathrm{d}\mu_{\mathcal{E}_i}(x) + O(\mathcal{S}(\phi)N\beta^\star)
    \end{align*}
    for some $c_i>0$ with $\sum_i c_i=1$.
    \item For all $i$, we have
\begin{align*}
\beta^{29N} \#F^{(0)} \leq \#F_i \leq e^{2N\ell}\#F^{(0)}.
\end{align*}
\end{enumerate}
Moreover, if
\begin{align}\label{eqn: Number of steps}
    N \geq \Biggl\lceil \frac{1}{\bar{\varphi}(\alpha)\ell + 29\log(\beta)}\log \Biggl(\frac{\Upsilon}{e^{20\ell}\#F^{(0)}}\Biggr) \Biggr\rceil,
\end{align}
then the following holds in addition:
\begin{enumerate}
\setcounter{enumi}{2}
    \item For all $i$, we have
    \begin{align*}
f_{\mathcal{E}_i,\delta_N}^{(\alpha)}(x) \leq 2Ce^{20\ell}\#F_i \qquad \text{for all $x \in \mathcal{E}_i$}
\end{align*}
where
\begin{align*}
    \delta_N = e^{2N\ell}\max\{\delta, (\#F^{(0)})^{-\frac{1}{\alpha}}\}.
\end{align*}
\end{enumerate}

\end{lemma}
\begin{proof}
For all integers $j \geq 0$, let $\delta_j = e^{2j\ell}\max\{\delta, (\#F^{(0)})^{-\frac{1}{\alpha}}\}$. We will prove the following stronger claim. 
\begin{claim*}
For every $j\geq 0$ exists a sequence of finite families $\mathcal{F}^{(j)}$ of sheeted sets with associated admissible measures $\{\mu_{\mathcal{E}}: \mathcal{E} \in \mathcal{F}^{(j)}\}$ satisfying the following properties.
\begin{enumerate}
        \item For all $\phi \in \mathrm{C}_c^\infty(X)$, $\ell' \geq 0$, and $u' \in \mathsf{B}_1^U$, we have
    \begin{align*}
        \int_{\mathsf{B}_1^U} \int_{\mathcal{E}^{(0)}} \phi(a_{\ell'}u' a_{j\ell} u.x) \,\mathrm{d}\mu_{\mathcal{E}^{(0)}}\,\mathrm{d}u
        = \sum_{\mathcal{E} \in \mathcal{F}^{(j)}} c_\mathcal{E}\int_{\mathcal{E}} \phi( a_{\ell'}u'.x) \,\mathrm{d}\mu_{\mathcal{E}}(x) + O(\mathcal{S}(\phi)j\beta^\star)
    \end{align*}
    for some $c_\mathcal{E}>0$ with $\sum_{\mathcal{E} \in \mathcal{F}^{(j)}} c_\mathcal{E}=1$.
    \item For all $\mathcal{E} \in \mathcal{F}^{(j)}$ with cross-section $F$, we have
\begin{align*}
\beta^{29j} \#F^{(0)} \leq \#F \leq e^{2N\ell}\#F^{(0)}.
\end{align*}
    \item For all $\mathcal{E} \in \mathcal{F}^{(j)}$, we have
    \begin{align*}
f_{\mathcal{E},\delta_j}^{(\alpha)}(x) \leq 2\max\{e^{-\bar{\varphi}(\alpha)j\ell} C \Upsilon,   e^{20\ell}\#F\} \qquad \text{for all $x \in \mathcal{E}$}.
\end{align*}
\end{enumerate}
\end{claim*}

For $j = N$ as in \cref{eqn: Number of steps} we have $e^{-\bar{\varphi}(\alpha)j\ell}\Upsilon \leq e^{20\ell} \#F_i$. The lemma follows with the family $\mathcal{F}^{(N)}$ as the claim. 

We will prove the claim by induction on $j$. For $j = 0$, let $\mathcal{F}^{(0)} = \{\mathcal{E}^{(0)}\}$. Note that since $\delta_0 = \max\{\delta, (\#F)^{-\frac{1}{\alpha}}\} \geq \delta$, we have
\begin{align*}
    f_{\mathcal{E}^{(0)}, \delta_0}^{(\alpha)}(z) \leq f_{\mathcal{E}^{(0)}, \delta}^{(\alpha)}(z) \quad \forall z \in \mathcal{E}^{(0)}.
\end{align*}
The claim when $j = 0$ follows directly from the condition in the lemma. 

Assuming now the claim holds for some integer $j \geq 0$, we will apply Proposition~\ref{prop: integrated inductive step} to show that it holds for $j + 1$. For each sheeted set $\mathcal{E} \in  \mathcal{F}^{(j)}$ and its associated admissible measure $\mu_{\mathcal{E}}$, apply \cref{prop: integrated inductive step} to $\mathcal{E}$, $\mu_{\mathcal{E}}$, $\alpha$ and $\delta_j$ to obtain a family of sheeted set and their associated admissible measure. Collect them and denote this collection by $\mathcal{F}^{(j + 1)}$. The first two properties hold for this family $\mathcal{F}^{(j + 1)}$ is a direct consequence of Proposition~\ref{prop: integrated inductive step}. We now show that property~(3) also holds for this family $\mathcal{F}^{(j + 1)}$. Take a sheeted set $\mathcal{E} \in \mathcal{F}^{(j)}$ with cross-section $F$ and let $\mathcal{E}' \in \mathcal{F}^{(j + 1)}$ be one of its descendants in the above process. Let $F'$ be the cross-section of $\mathcal{E}'$. 

By the inductive hypothesis, we have
\begin{align*}
f_{\mathcal{E},\delta_j}^{(\alpha)}(x) \leq 2\max\{e^{-\bar{\varphi}(\alpha)j\ell} C \Upsilon,   e^{20\ell}\#F\} \qquad \text{for all $x \in \mathcal{E}$}.
\end{align*}
By \cref{prop: integrated inductive step}, for all $z \in \mathcal{E}'$, we have
\begin{align}\label{eqn:inner j step induction}
\begin{aligned}
f_{\mathcal{E}',\delta_{j}'}^{(\alpha)}(z) \leq{}& e^{-\frac{3}{4}\varphi(\alpha)\ell}2\max\{e^{-\bar{\varphi}(\alpha)j\ell} C \Upsilon,   e^{20\ell}\#F\} + e^{2\alpha \ell}\beta^{-\alpha} \#F'\\
\leq{}& e^{-\frac{3}{4}\varphi(\alpha)\ell}2\max\{e^{-\bar{\varphi}(\alpha)j\ell} C \Upsilon,   e^{20\ell}\#F\} + e^{20 \ell}\#F'
\end{aligned}
\end{align}
where $\delta_j' = e^{2\ell}\max\{\delta_j, (\#F)^{-\frac{1}{\alpha}}\}$. The last inequality follows from \cref{eqn:Much smaller scale}. In particular, we only use $\beta^{9} \geq e^{-\ell}$. 

We first show that $\delta_j' = \delta_{j + 1} = e^{2(j + 1)\ell} \delta_0$. By the inductive hypothesis on $\delta_j$, we have that 
\begin{align*}
    \delta_j = e^{2j\ell}\delta_0 \geq e^{2j\ell} (\#F^{(0)})^{-\frac{1}{\alpha}}. 
\end{align*}
Also, by property~(2) in the inductive hypothesis, we have
\begin{align*}
    \#F^{-\frac{1}{\alpha}} \leq \beta^{-29j/\alpha}(\#F^{(0)})^{-\frac{1}{\alpha}}.
\end{align*}
By \cref{eqn:Much smaller scale} (in particular, $e^{2\ell} \geq \beta^{-29/\epsilon_0}$), we have
\begin{align*}
    \delta_j' = e^{2\ell}\max\{\delta_j, \#F^{-\frac{1}{\alpha}}\} = e^{2\ell}\delta_j = \delta_{j + 1}.
\end{align*}

We now show property~(3) in the claim. By \cref{eqn:inner j step induction} and the above arguments, we have for all $z \in \mathcal{E}'$
\begin{align*}
    f_{\mathcal{E}', \delta_{j + 1}}^{(\alpha)}(z) \leq{}& e^{-\frac{3}{4}\varphi(\alpha)\ell}2\max\{e^{-\bar{\varphi}(\alpha)j\ell} C \Upsilon,   e^{20\ell}\#F\} + e^{20 \ell}\#F'\\
    \leq{}& 2e^{-\frac{1}{4}\varphi(\alpha)\ell} \max\big\{e^{-\bar{\varphi}(\alpha)(j + 1)\ell} C\Upsilon,   e^{-\bar{\varphi}(\alpha)\ell}e^{20\ell}\#F\big\} \\
    &\quad+ e^{20\ell}\#F'\\
    \leq{}& 2\max\{e^{-\bar{\varphi}(\alpha)(j + 1)\ell}C \Upsilon,   e^{20\ell}\#F'\}.
\end{align*}
For the last inequality, we used \cref{eqn:R large enough absolutely,eqn:Much smaller scale}. In particular, we only use $e^{-\frac{1}{4}\varphi(\alpha)\ell} \leq 1/2$ and $e^{-\bar{\varphi}(\alpha)\ell}\beta^{-29} \leq 1$. This proves the claim. 
\end{proof}

We now apply the above lemma to the sequence $\{\alpha_j\}_{j = 0}^M$ to prove \cref{pro:dim improving main}. Before we proceed the proof, let us recall the following lemma. Recall that
\begin{align*}
\nu_t = (a_t)_\ast m_{\mathsf{B}_1^U}
\end{align*}
and $\lambda$ is the normalized Haar measure on
\begin{align*}
    \mathsf{B}^{s, H}_{\beta + 100 \beta^2} = \mathsf{B}_{\beta + 100 \beta^2}^{U^-} \mathsf{B}_{\beta + 100 \beta^2}^{M_0A}.
\end{align*}

\begin{lemma}\label{lem:Folner H}
For all $\phi \in \mathrm{C}_c^{\infty}(X) + \C \mathds{1}_X$, $x \in X$, $t_1, t_2 > 0$ with $e^{-t_1} \leq \beta$, we have
\begin{align*}
\Biggl|\int_X \phi \,\mathrm{d}(\nu_{t_2 + t_1} \ast \delta_x) - \int_X \phi \,\mathrm{d}(\nu_{t_2} \ast \lambda \ast \nu_{t_1} \ast \delta_x)\biggr| \ll \mathcal{S}(\phi) \beta^\star.
\end{align*}
\end{lemma}

\begin{proof}
This is a direct consequence of the F{\o}lner property of $U$. See \cite[Lemma 7.4]{LMW22}. 
\end{proof}

\begin{proof}[Proof of \cref{pro:dim improving main}]
Recall that we chose an arithmetic progression $\{\alpha_j\}_{j = 0}^{p_{\mathrm{fin}}}$ satisfying
\begin{itemize}[label=$\bullet$]
\item $\epsilon_0 = \alpha_0 < \alpha_1 < \alpha_2 < \cdots < \alpha_{p_{\mathrm{fin}}}$,
\item $\alpha_{j} - \alpha_{j - 1} = \frac{1}{72 \cdot 9 \cdot 10}\theta$ for all $1 \leq j \leq p_{\mathrm{fin}}$,
\item $\alpha_{p_{\mathrm{fin}} - 1} \leq 9 - \theta < \alpha_{p_{\mathrm{fin}}} < 9$.
\end{itemize}
Note that for all $j$, we have
\begin{align*}
    \bar{\varphi}(\alpha_j) \geq \frac{1}{720} \theta.
\end{align*}

Let $d_j = \sum_{i = 0}^j N_i$. Note that $d = d_{p_{\mathrm{fin}}}$. For all $j = 0, 1, \ldots, p_{\mathrm{fin}}$, let $\delta_j = e^{2d_j\ell}\delta_0$. Recall that $\delta_0 = R^{-\frac{1}{\ref{a:closing lemma main}}} = e^{-\frac{t}{M\ref{a:closing lemma main}}}$. We will apply Lemma~\ref{lem:getting dim alpha} to obtain sheeted sets with dimension $\alpha_j$ at scale $\delta_j$. 

Applying Theorem~\ref{thm:closing lemma initial dim} to the initial point $x_1 \in X_\eta$, we get a family $\mathcal{F}^{\mathrm{ini}}$ of sheeted sets and associsted admissible measures $\{\mu_{\mathcal{E}^{\mathrm{ini}}}: \mathcal{E}^{\mathrm{ini}} \in \mathcal{F}^{\mathrm{ini}}\}$. For each $\mathcal{E}^{\mathrm{ini}} \in \mathcal{F}^{\mathrm{ini}}$, we claim the following. 

\begin{claim*}
For all $j = 0, 1, \ldots, p_{\mathrm{fin}}$, there exists a sequence of family of sheeted sets $\mathcal{F}^{(j)}$ and associated admissible measures with the following properties. For all $\mathcal{E} \in \mathcal{F}^{(j)}$, we use $F$ to denote its cross-section and $\mu_{\mathcal{E}}$ to denote the associated admissible measure. 
\begin{enumerate}
\item For all $\phi \in \mathrm{C}_c^\infty(X)$, $\ell' \geq 0$, and $u' \in \mathsf{B}_1^U$, we have
\begin{align*}
    \int_{\mathsf{B}_1^U} \int_{X} \phi(a_{\ell'}u'a_{d_j \ell} u.x) \,\mathrm{d}\mu_{\mathcal{E}^{\mathrm{ini}}}\,\mathrm{d}u
    = \sum_{\mathcal{E} \in \mathcal{F}^{(j)}} c_\mathcal{E}\int_{\mathcal{E}} \phi( a_{\ell'}u'.x) \,\mathrm{d}\mu_{\mathcal{E}}(x) + O(\mathcal{S}(\phi)d_j\beta^\star)
\end{align*}
for some $c_\mathcal{E}>0$ with $\sum_{\mathcal{E} \in \mathcal{F}^{(j)}} c_\mathcal{E}=1$.
\item For all $\mathcal{E} \in \mathcal{F}^{(j)}$, we have
\begin{align*}
\#F \geq \beta^{29 d_j} \delta_0^{-\frac{3\epsilon_0}{2}}.
\end{align*}
\item For all $\mathcal{E} \in \mathcal{F}^{(j)}$, we have
\begin{align*}
f_{\mathcal{E},\delta_j}^{(\alpha_j)}(x) \leq 2^{j}e^{20\ell}\#F \qquad \text{for all $x \in \mathcal{E}$}
\end{align*}
where $\delta_j = e^{2d_j \ell}\delta_0$.
\end{enumerate}
\end{claim*}

Let us first conclude the proposition assuming the claim. Let $\{\mathcal{E}_i^{\mathrm{fin}}\}_i = \mathcal{F}^{(p_{\mathrm{fin}})}$ and $\{\mu_{\mathcal{E}_i^{\mathrm{fin}}}\}$ be those associated admissible measure produced by the claim. As usual, we use $F_i^{\mathrm{fin}}$ to denote the cross-section of $\mathcal{E}_i^{\mathrm{fin}}$. We will show that the proposition holds for this family $\{\mathcal{E}_i^{\mathrm{fin}}\}_i$. We will first show property~(2) and (3) from the claim and show property~(1) by Theorem~\ref{thm:closing lemma initial dim}, and Lemma~\ref{lem:Folner H}. 

By property~(2) in the claim, we have
\begin{align}\label{eqn:final number of sheet estimate}
    \#F_i^{\mathrm{fin}} \geq \beta^{29d} \#F^{\mathrm{ini}}_i \geq \beta^{29d} e^{\frac{3\epsilon_0 t}{2M\ref{a:closing lemma main}}} \geq \delta_0^{-\frac{4\epsilon_0}{3}}.
\end{align}
The last inequality follows from \cref{eqn:Much smaller scale} and the fact that $\epsilon$ is much smaller than $\epsilon_0$. In particular, we use $\beta^{29d} \geq e^{-\frac{\epsilon_0 t}{6M\ref{a:closing lemma main}}}$. This shows property~(2) in the proposition. 

We now estimate $\delta_{p_{\mathrm{fin}}}$. We have
\begin{align*}
    \delta_{p_{\mathrm{fin}}} = e^{2d\ell}\delta_0 \leq e^{2p_{\mathrm{fin}}\ell} e^{2N_1\ell(\sum_{j = 0}^{p_{\mathrm{fin}} - 1} (\frac{3}{4})^j)} \delta_0 = e^{2p_{\mathrm{fin}}\ell}e^{8N_1\ell} e^{-8(\frac{3}{4})^{p_{\mathrm{fin}}}N_1\ell}\delta_0.
\end{align*}
Note that
\begin{align*}
    e^{8N_1\ell}\delta_0 \leq e^{8\ell} e^{8\frac{25}{2\epsilon}\ell} e^{-\frac{t}{M\ref{a:closing lemma main}}} \leq e^{8\ell}.
\end{align*}
The last inequality follows from the definition $\ell = \frac{\epsilon t}{100M\ref{a:closing lemma main}}$. Therefore, 
\begin{align*}
    \delta_{p_{\mathrm{fin}}} \leq e^{2p_{\mathrm{fin}}\ell}e^{8\ell}e^{-8(\frac{3}{4})^{p_{\mathrm{fin}}}N_1\ell} \leq e^{2p_{\mathrm{fin}}\ell}e^{8\ell}e^{-8(\frac{3}{4})^{p_{\mathrm{fin}}}\frac{25}{2\epsilon}\ell}.
\end{align*}
The last inequality follows from the definition $N_1 = \lceil \frac{25}{2\epsilon} \rceil$. Recall that $p_{\mathrm{fin}} = \lceil 6480(\frac{9 - \epsilon_0}{\theta} - 1)\rceil \leq \frac{10^6}{\theta}$ and $\epsilon = 10^{-10}(\frac{3}{4})^{p_{\mathrm{fin}}}\theta$. We have
\begin{align*}
\delta_{p_{\mathrm{fin}}} \leq e^{-\frac{200}{\theta}\ell} = \delta_0^{\frac{2\epsilon}{\theta}}.
\end{align*}
This shows property~(3) in the proposition. 

We now show property~(1). Fix $\phi \in \mathrm{C}_c^{\infty}(X) + \C \mathds{1}_X$, $\ell' \geq 0$, and $u' \in \mathsf{B}_1^U$. Since $e^{-t} \ll \beta $, by Lemma~\ref{lem:Folner H} we have
\begin{align}\label{eqn:Folner in dim improvement}
\begin{aligned}
&\int_{\mathsf{B}_1^U} \phi(a_{\ell'}u' a_{d\ell + t} u.x_1) \,\mathrm{d}u\\
={}& 
\int_{\mathsf{B}_1^U}
\int_H \int_{\mathsf{B}_1^U} \phi(a_{\ell'}u' a_{d\ell}u_2 \mathsf{h} a_{t} u_1.x_1) \,\mathrm{d}u_1\,\mathrm{d}\lambda(\mathsf{h})\,\mathrm{d}u_2 + O(\mathcal{S}(\phi)\beta^\star).
\end{aligned}
\end{align}
By Theorem~\ref{thm:closing lemma initial dim}, we have
\begin{align}\label{eqn:Closing lemma in dim improvement measure}
\begin{aligned}    {}&\int_{\mathsf{B}_1^U}
\int_H \int_{\mathsf{B}_1^U} \phi(a_{\ell'}u' a_{d\ell}u_2 \mathsf{h} a_{t} u_1.x_1) \,\mathrm{d}u_1\,\mathrm{d}\lambda(\mathsf{h})\,\mathrm{d}u_2\\
={}& \sum_{\mathcal{E}^{\mathrm{ini}}} c_{\mathcal{E}^{\mathrm{ini}}} \int_{\mathsf{B}_1^U} \int_{\mathcal{E}^{\mathrm{ini}}} \phi( a_{\ell'}u'a_{d\ell}u_2.x) \,\mathrm{d}\mu_{\mathcal{E}^{\mathrm{ini}}}(x) \,\mathrm{d}u_2 + O(\mathcal{S}(\phi)\beta^\star).
\end{aligned}
\end{align}
for some $c_{\mathcal{E}^{\mathrm{ini}}}>0$ with $\sum_{\mathcal{E}^{\mathrm{ini}}} c_{\mathcal{E}^{\mathrm{ini}}}=1$. Combine \cref{eqn:Folner in dim improvement,eqn:Closing lemma in dim improvement measure} with property~(1) in the claim, we prove property~(1) in the proposition. 

\begin{proof}[Proof of the claim]
For $j = 0$, let $\mathcal{F}^{(0)}$ consist of the single initial sheeted set $\mathcal{E}^{\mathrm{ini}}$. Let $F^{\mathrm{ini}}$ be its cross-section. It suffices to show property~(2) and (3). By Theorem~\ref{thm:closing lemma initial dim} property~(2), we have
\begin{align*}
    \#F^{\mathrm{ini}} \geq \beta^{29}\delta^{-2\epsilon_0} \geq \delta^{-\frac{3\epsilon_0}{2}}.
\end{align*}
The last inequality follows from \cref{eqn:Much smaller scale}. In particular, we use $\delta_0^{\frac{\epsilon_0}{2}} = e^{-\frac{\epsilon_0 t}{2M\ref{a:closing lemma main}}} \leq \beta^{29}$. By Theorem~\ref{thm:closing lemma initial dim} property~(3), for all $x \in \mathcal{E}^{\mathrm{ini}}$ we have
\begin{align*}
f_{\mathcal{E}^{\mathrm{ini}},\delta_0}^{(\alpha_0)}(x) \leq \beta^{-\ref{e:closing lemma main2}}\#F^{\mathrm{ini}} \leq e^{20\ell}\#F^{\mathrm{ini}}.
\end{align*}
The last inequality follows from \cref{eqn:Much smaller scale}. In particular, we use $e^{20\ell} = e^{\frac{\epsilon t}{5M\ref{a:closing lemma main}}} \leq \beta^{\ref{e:closing lemma main2}}$. 

For $j \geq 1$, we will prove the claim by induction on $j$. For $j = 1$, fix $\mathcal{E}^{(0)} \in \mathcal{F}^{(0)}$ with cross-section $F^{(0)}$ and associated admissible measure $\mu_{\mathcal{E}^{(0)}}$. By the previous case with $j = 0$, we have
\begin{align*}
    f_{\mathcal{E}^{(0)}, \delta_0}^{(\alpha_0)}(x) \leq e^{20\ell}\#F^{(0)} \qquad \text{for all $x \in \mathcal{E}^{(0)}$}.
\end{align*}
By definition of the Margulis function in \cref{subsec:Margulis function}, we have for all $x \in \mathcal{E}^{(0)}$
\begin{align*}
    f_{\mathcal{E}^{(0)}, \delta_0}^{(\alpha_{1})}(x) &= \sum_{w \in I_{\mathcal{E}^{(0)}}(x) \setminus \{0\}} \max\{\|w\|, \delta_0\}^{-\alpha_{1}}\\
    &\leq \delta_0^{-(\alpha_{1} - \alpha_0)} \sum_{w \in I_{\mathcal{E}^{(0)}}(x) \setminus \{0\}} \max\{\|w\|, \delta_0\}^{-\alpha_{0}} \\ 
    &\leq  e^{20\ell}\delta_0^{-\frac{\theta}{72\cdot 9 \cdot 10}}\#F^{(0)} .
\end{align*}
Applying \cref{lem:getting dim alpha} for all $\mathcal{E}^{(0)} \in \mathcal{F}^{(0)}$, $\alpha = \alpha_{1}$, $\delta = \delta_0$, $C = 1$, $N = N_1$, and
\begin{align*}
    \Upsilon = \delta_0^{-\frac{\theta}{72 \cdot 9 \cdot 10}} e^{20\ell}\#F^{(0)},
\end{align*}
we obtain a new family $\mathcal{F}^{(1)}$ of sheeted sets. Properties~(1) and (2) follow directly from \cref{lem:getting dim alpha}; it remains to prove Property~(3).
Notice that by definition
\begin{align*}
    N_1 = \left\lceil \frac{25}{2\epsilon}\right\rceil 
    \geq \left\lceil\frac{{\frac{\theta}{72\cdot 9\cdot10}}|\log(\delta_0)|}{\frac{4\theta}{5 \cdot72\cdot 9}\ell}\right\rceil 
    \geq  \Biggl\lceil \frac{1}{\bar{\varphi}(\alpha)\ell + 29\log(\beta)}\log\Biggl(\frac{\Upsilon}{e^{20\ell}\#F^{(0)}}\Biggr) \Biggr\rceil.
\end{align*}
where the last inequality follows from \eqref{eqn:Much smaller scale}. In particular, we use $e^{-\frac{\theta}{720}\ell} \leq \beta^{29}$ and $\epsilon$ is much smaller than $\theta$. Notice also that $\#F^{(0)} \geq \delta_0^{-\frac{3}{2}\alpha_0}$, we have $\delta_0 = \max\{\delta_0,(\#F^{(0)})^{-\frac{1}{\alpha_1}}\}$. Therefore, Lemma~\ref{lem:getting dim alpha} implies that the Margulis functions for any point in the new good sheeted sets satisfy the desired bound at scale $\delta_1 = e^{2N_1\ell}\delta_0$.

Assuming the claim holds for $j$, we will use \cref{prop: integrated inductive step} to show the claim holds for $j + 1$. For all sheeted set $\mathcal{E} \in \mathcal{F}^{(j)}$, let $F$ be its cross-section and $\mu_{\mathcal{E}}$ be the associated admissible measure. 
By the inductive hypothesis, we have for all $\mathcal{E} \in \mathcal{F}^{(j)}$
\begin{align*}
    f_{\mathcal{E}, \delta_j}^{(\alpha_j)}(x) \leq 2^{j}e^{20\ell}\#F \qquad \text{for all $x \in \mathcal{E}$}.
\end{align*}
By definition of the Margulis function, for all $x \in \mathcal{E}$, we have
\begin{align*}
    f_{\mathcal{E}, \delta_j}^{(\alpha_{j + 1})}(x) ={}& \sum_{w \in I_{\mathcal{E}}(x) \setminus \{0\}} \max\{\|w\|, \delta_j\}^{-\alpha_{j + 1}}\\
    \leq{}& \delta_j^{-(\alpha_{j + 1} - \alpha_j)} \sum_{w \in I_{\mathcal{E}}(x) \setminus \{0\}} \max\{\|w\|, \delta_j\}^{-\alpha_{j}}
    \\\leq{}& 2^{j} e^{20\ell} \delta_j^{-\frac{\theta}{72 \cdot 9 \cdot 10}} \#F.
\end{align*}
Applying \cref{lem:getting dim alpha} for all $\mathcal{E} \in \mathcal{F}^{(j)}$ (with cross-section $F$ and associated admissible measure $\mu_{\mathcal{E}}$), $\alpha = \alpha_{j + 1}$, $\delta = \delta_j$, $C = 2^{j}$, $N_{j + 1}$ and
\begin{align*}
    \Upsilon = e^{20\ell}\delta_j^{-\frac{\theta}{72 \cdot 9 \cdot 10}} \#F,
\end{align*}
we have a family $\mathcal{F}^{(j + 1)}$ of good sheeted sets. Properties~(1) and (2) follows directly from \cref{lem:getting dim alpha}. Note that we have
\begin{align*}
\frac{\log(\delta_j^{-\frac{\theta}{72\cdot 9 \cdot10}})}{\frac{4\theta}{5 \cdot72\cdot 9}\ell}
&\leq \frac{-2d_j \frac{\theta}{72\cdot 9 \cdot10} \ell + \frac{\theta}{72 \cdot 9 \cdot 10A_0 M}t}{\frac{4\theta}{5 \cdot72\cdot 9}\ell}
= -\frac{1}{4} d_j + \frac{25}{2\epsilon}\\
&\leq -\tfrac{1 }{4} N_1\sum_{0\leq i < j} (\tfrac{3}{4})^i + \frac{25}{2\epsilon}
= - N_1 \big(1-(\tfrac{3}{4})^j\big) + \frac{25}{2\epsilon} \leq N_1 (\tfrac{3}{4})^j \leq N_{j+1}.
\end{align*}
Hence, by \eqref{eqn:Much smaller scale} we may apply (3) in \cref{lem:getting dim alpha} with $N_{j+1}$. For all new sheeted set $\mathcal{E}' \in \mathcal{F}^{(j + 1)}$ with cross-section $F'$, we have
\begin{align*}
f_{\mathcal{E}',\delta_j'}^{(\alpha_{j + 1})}(z) \leq 2^{j + 1}e^{20\ell}\#F' \qquad \text{for all $z \in \mathcal{E}'$}
\end{align*}
where $\delta_j' = e^{2N_{j + 1}\ell}\max\{\delta_j, (\#F)^{-\frac{1}{\alpha_{j + 1}}}\}$. It suffices to show that $\delta_j \geq (\#F)^{-\frac{1}{\alpha_{j + 1}}}$. 
By property~(2) in the inductive hypothesis, we have 
\begin{align*}
\#F \geq \beta^{29d_j}\delta_0^{-\frac{3}{2}\epsilon_0} \geq \beta^{29d}\delta_0^{-\frac{3}{2}\epsilon_0} \geq \delta_0^{-\frac{4}{3}\epsilon_0}
\end{align*}
where the last inequality follows from \cref{eqn:final number of sheet estimate}. 
By the definition of $d_j$ and $\ell$, we have
\begin{align*}
2d_j \ell \geq 2\Biggl(\sum_{i = 0}^{j - 1} \Biggl(\frac{25}{2\epsilon}\Biggr)\Biggl(\frac{3}{4}\Biggr)^j\Biggr)\ell = \frac{1}{M\ref{a:closing lemma main}}\Biggl(1 - \Biggl(\frac{3}{4}\Biggr)^j\Biggr).
\end{align*}
Therefore, 
\begin{align*}
\delta_j = e^{2d_j\ell}\delta_0 \geq \delta_0^{(\frac{3}{4})^j} \geq \delta_0^{\frac{4}{3}\frac{\epsilon_0}{\alpha_{j + 1}}} \geq (\#F)^{-\frac{1}{\alpha_{j + 1}}}.
\end{align*}
The middle inequality follows from our definition $\alpha_{j + 1} = \epsilon_0 + \frac{\theta}{72 \cdot 9 \cdot 10} (j + 1)$ and the fact that $\theta$ is much smaller than $\epsilon_0$. (In fact here we only need $\theta < \frac{\epsilon_0}{4}$.) Thus property~(3) in the claim holds for all sheeted sets in $\mathcal{F}^{(j + 1)}$. The proof of the claim is complete. 
\end{proof}
The proof of the proposition is complete. 
\end{proof}

\section{From large dimension to effective equidistribution}\label{sec:From large dimension to effective equidistribution}
The main result of this proposition is Proposition~\ref{pro:Final phase}. This is an anologue of \cite[Proposition 9.1]{LMWY25}. It allow us to get effective equidistribution from high transverse dimension. Let us recall the following parameters from the previous sections. 

Recall the constants $\ref{a:closing lemma main} > 1$, $\ref{c:closing lemma main} > 1$, $\ref{d:closing lemma main} > 1$, $\ref{e:closing lemma main1}, \ref{e:closing lemma main2}> 1$, $\ref{m:closing lemma main} > 1$, $\epsilon_0 > 0$ from the effective closing lemma (Theorem~\ref{thm:closing lemma initial dim}) and the constant $\ref{d:avoidance}$ from the avoidance principle (Proposition~\ref{pro:avoidance}). They depend only $(G, H, \Gamma)$. Recall $D = \max\{\ref{d:closing lemma main}, \ref{d:avoidance}\} + 2$ and $M = \ref{m:closing lemma main} + \ref{c:closing lemma main} D$ be as in Theorem~\ref{thm:closing lemma initial dim}. Let $R > 1$ and $t = M\log R$. We will assume that $R$ is large enough depending on the space $X$. Recall from Theorem~\ref{thm:closing lemma initial dim} that $\delta_0 = R^{-\frac{1}{\ref{a:closing lemma main}}} = e^{-\frac{t}{M\ref{a:closing lemma main}}}$.

Recall from the previous section that we have $\theta = (\frac{\min\{\kappa_1, \epsilon_0\}}{80})^{2} \in (0, \min\{\kappa_1, \epsilon_0\})$, $p_{\mathrm{fin}} = \lceil 6480(\frac{9 - \epsilon_0}{\theta} - 1)\rceil$, and $\epsilon = 10^{-10}(\frac{3}{4})^{p_{\mathrm{fin}}} \theta$. Recall that $\frac{\epsilon}{\theta}$ is much smaller than both $\kappa_1$ and $\epsilon_0$. Let $\alpha = \dim(\mathfrak{r}) - \theta = 9 - \theta$. 

Let $\beta$, $\eta$, and $\ell$ be as in the previous section. We recall that $e^{-\ell}$ is a much smaller scale than $\beta$. Recall that we pick $R$ large enough so that $e^{-\theta\ell}$ is a small scale. In particular, let us recall \cref{eqn:Much smaller scale,eqn:R large enough absolutely} in the following inequalities:
\begin{align}
    &e^{-\epsilon^2 \ell} \leq \beta^{10^{10}\ref{e:closing lemma main1}\ref{e:closing lemma main2}},\\
    &e^{-\theta\ell} \leq 10^{-10000}.
\end{align}

\begin{proposition}\label{pro:Final phase}
    Let $F \subset B_{\beta}^{\mathfrak{r}}$ be a finite set with $\#F \geq \delta_0^{-\frac{4\epsilon_0}{3}} = e^{\frac{4\epsilon_0 t}{3M\ref{a:closing lemma main}}}$. Let 
    \begin{align*}
        \mathcal{E} = \mathsf{E} \exp(F).y \subset X_\eta
    \end{align*}
    be a sheeted set equipped with an admissible measure $\mu_{\mathcal{E}}$. Assume further that the following is satisfied. For all $z = \mathsf{h}\exp(w).y$ with $\mathsf{h} \in \mathsf{E} \setminus \partial_{10\beta^2}\mathsf{E}$, 
    \begin{align}\label{eqn:Margulis function estimate final phase}
        f^{(\alpha)}_{\mathcal{E}, \delta_{\mathrm{fin}}}(z) \leq e^{20\ell} \#F \text{ where } \delta_{\mathrm{fin}} = \delta_0^{\frac{2\epsilon}{\theta}}.
    \end{align}
    Let $\tau$ be a parameter with $\frac{1}{16}|\log \delta_{\mathrm{fin}}|\leq \tau \leq \frac{1}{8}|\log \delta_{\mathrm{fin}}|$. Then we have
    \begin{align*}
        \Biggl|\int_{\mathsf{B}_1^U} \int_X \phi(a_\tau u.z) \,\mathrm{d}\mu_{\mathcal{E}}(z) \,\mathrm{d}u - \int_X \phi \,\mathrm{d}\mu_X\Biggr| \ll \mathcal{S}(\phi)\beta^\star
    \end{align*}
    for all $\phi \in \mathrm{C}_c^{\infty}(X) + \C \mathds{1}_X$. 
\end{proposition}

\begin{proof}
The statement can be proved following the proof of \cite[Proposition 9.1]{LMWY25} step-by-step. We present the necessary change of parameter for reader's convenience. 

Write $\tau = \ell_1 + \ell_2$ where 
\begin{align}\label{eqn:final phase setting two push}
\ell_2 = \frac{\tau}{1 + \kappa_1} \text{ and } \ell_1 = \kappa_1 \ell_2.
\end{align}
We have $8\ell_2 \leq 8\tau \leq |\log \delta_{\mathrm{fin}}|$, $\ell_1 \leq \kappa_1 \ell_2$. Recall from \cref{eqn:Much smaller scale} that $\beta^{10^{10}} \geq e^{-\frac{\epsilon t}{\kappa_1 M\ref{a:closing lemma main}}}$. We have $|\log \eta| \leq \frac{\kappa_1}{1 + \kappa_1}\tau = \kappa_1 \ell_2$. Therefore, we have as in Lemma~\ref{lem:Venkatesh} $\kappa_1\ell_2 \geq \max\{\ell_1, |\log \eta|\}$ and $8\ell_2 \leq |\log \delta_0|$. 

Note that for all $\phi \in \mathrm{C}_c^\infty(X) + \C\mathds{1}_X$, we have
\begin{align*}
&\int_{\mathsf{B}_1^U} \int_X \phi(a_\tau u.z) \,\mathrm{d}\mu_{\mathcal{E}}(z) \,\mathrm{d}u\\
=&
\int_{\mathsf{B}_1^U} \int_{\mathsf{B}_1^U} \int_X \phi(a_{\ell_1} u_1 a_{\ell_2} u_2.z) \,\mathrm{d}\mu_{\mathcal{E}}(z) \,\mathrm{d}u_2 \,\mathrm{d}u_1 + O(\mathcal{S}(\phi)e^{-\ell_2}).
\end{align*}
It suffices to estimate
\begin{align*}
\int_{\mathsf{B}_1^U} \int_{\mathsf{B}_1^U} \int_X \phi(a_{\ell_1} u_1 a_{\ell_2} u_2.z) \,\mathrm{d}\mu_{\mathcal{E}}(z) \,\mathrm{d}u_2 \,\mathrm{d}u_1.
\end{align*}

Disintegrating the measure $\mu_{\mathcal{E}}$ as in \cite[Section 9.2]{LMWY25}, for all $\mathsf{h} \in \hat{\mathsf{E}} = \overline{\mathsf{E} \setminus \partial_{20\beta^2}\mathsf{E}}$, there exists $\hat{\mu}^{\mathsf{h}}$ supported on a finite set $F^{\mathsf{h}}$ with the following properties. 
\begin{enumerate}
    \item For all $w \in F^{\mathsf{h}}$, we have
    \begin{align}\label{eqn:Close to uniform measure}
        \hat{\mu}^{\mathsf{h}}(\{w\}) \asymp \#F^{-1} \asymp (\#F^{\mathsf{h}})^{-1}.
    \end{align}
    \item We have the following estimate on the (modified) $\alpha$-energy of $F^{\mathsf{h}}$:
    \begin{align}\label{eqn:energy in final phase}
        \mathcal{G}_{F^{\mathsf{h}}, \delta_{\mathrm{fin}}}^{(\alpha)}(w) \ll e^{20\ell} \#F^{\mathsf{h}} \qquad \forall w \in F^{\mathsf{h}}.
    \end{align}
\end{enumerate}
We remark that \cref{eqn:Close to uniform measure} follows from the fact that $\mu_{\mathcal{E}}$ is an admissible measure and \cref{eqn:energy in final phase} follows from \cref{eqn:Margulis function estimate final phase} and \cite[Lemma 7.1]{LMWY25}. 

Moreover, it suffices to estimate
\begin{align*}
    \int_{\mathsf{B}_1^U} \int_{\mathsf{B}_1^U} \int_{\mathfrak{r}} \phi(a_{\ell_1} u_1 a_{\ell_2} u_2 \exp(w)\mathsf{h}.z) \,\mathrm{d}\hat{\mu}^{\mathsf{h}}(w) \,\mathrm{d}u_1\,\mathrm{d}u_2
\end{align*}
for some $z \in \exp(F).y$. See \cite[Section 9.2]{LMWY25}. 

Since $\#F \geq \delta_0^{-\frac{4\epsilon_0}{3}}$, the scale $\delta_{\mathrm{fin}}$ satisfies
\begin{align*}
\delta_{\mathrm{fin}} = \delta_0^{\frac{2\epsilon}{\theta}} \geq \#F^{-\frac{1}{9}} \geq \#F^{-1/\alpha}
\end{align*}
The first inequality follows from the fact that $\frac{\epsilon}{\theta}$ is much smaller than $\epsilon_0$. Therefore, by property~(2), the measure $\hat{\mu}^{\mathsf{h}}$ satisfies the following Frostman-type condition:
\begin{align*}
\hat{\mu}^{\mathsf{h}} (B^{\mathfrak{r}}_{\delta}(w)) \ll e^{20\ell} \delta^\alpha \leq e^{20\ell}\delta_{\mathrm{fin}}^{-\theta} \delta^{\dim(\mathfrak{r})} \quad \forall w\in \mathfrak{r}, \delta \geq \delta_{\mathrm{fin}}.
\end{align*}

Apply Lemma~\ref{lem:Venkatesh} with $\hat{\mu}^{\mathsf{h}}$, $\varrho = \beta$, scale $\delta_{\mathrm{fin}}$, 
\begin{align*}
\Upsilon \asymp \delta_{\mathrm{fin}}^{-\theta}e^{20\ell}
\end{align*}
and $\ell_1$, $\ell_2$ as in \cref{eqn:final phase setting two push}. We have
\begin{align*}
&\int_{\mathsf{B}_1^U} \int_{\mathsf{B}_1^U}\int_{\mathfrak{r}} \phi(a_{\ell_1}u_1 a_{\ell_2} u_2 \exp(w).x)\,\mathrm{d}\hat{\mu}^{\mathsf{h}}(w)\,\mathrm{d}u_2\,\mathrm{d}u_1\\
={}& \int_X \phi \,\mathrm{d}\mu_X + O(\mathcal{S}(\phi)(\beta^\star + \eta + \Upsilon^{\frac{1}{2}}\beta^{-3}e^{-\kappa_1\ell_1})).
\end{align*}
Since $\eta = \beta^{\frac{1}{2}}$, it suffices to estimate the last term. We have
\begin{align*}
\Upsilon^{\frac{1}{2}}\beta^{-3}e^{-\kappa_1\ell_1} = \delta_{\mathrm{fin}}^{-\theta/2}e^{10\ell}\beta^{-3}e^{-\frac{\kappa_1^2}{1 + \kappa_1}\tau}
\leq \delta_{\mathrm{fin}}^{-\theta/2}e^{11\ell}\delta_{\mathrm{fin}}^{\frac{\kappa_1^2}{(1 + \kappa_1)16}}.
\end{align*}
The last inequality follows from the fact that $e^{-\ell}$ is a much smaller scale than $\beta$ and $\tau \geq \frac{1}{16}|\log \delta_{\mathrm{fin}}|$. Since $\theta \leq (\frac{\kappa_1}{80})^2$, we have
\begin{align*}
\Upsilon^{\frac{1}{2}}\beta^{-3}e^{-\kappa_1\ell_1} \leq \delta_{\mathrm{fin}}^{-\theta/2}e^{11\ell}\delta_{\mathrm{fin}}^{\frac{\kappa_1^2}{(1 + \kappa_1)16}} \leq e^{11\ell}\delta_{\mathrm{fin}}^{\theta}.
\end{align*}
Recall that $\delta_{\mathrm{fin}} = \delta_0^{\frac{2\epsilon}{\theta}} = e^{-\frac{2\epsilon t}{\theta M\ref{a:closing lemma main}}}$ and $\ell = \frac{\epsilon t}{100 M\ref{a:closing lemma main}}$, the above error term is bounded by $e^{-\ell} \leq \beta$. This completes the proof of the proposition. 
\end{proof}

\section{Proof of Theorem~\ref{thm:main equidistribution}}\label{sec:ProofOfMainEqui}
This section is devoted to the proof of Theorem~\ref{thm:main equidistribution}. Before we proceed the proof, let us recall the constants and parameters from previous sections needed in the proof. 

Recall the constants $\ref{a:closing lemma main} > 1$, $\ref{c:closing lemma main} > 1$, $\ref{d:closing lemma main} > 1$, $\ref{e:closing lemma main1}, \ref{e:closing lemma main2}> 1$, $\ref{m:closing lemma main} > 1$, $\epsilon_0 > 0$ from the effective closing lemma (Theorem~\ref{thm:closing lemma initial dim}). They depend only on $(G, H, \Gamma)$. Also, let us recall the constants $\mathsf{m}$, $s_0$, $\ref{a:avoidance}$, $\ref{c:avoidance}$, and $\ref{d:avoidance}$ depending only on $(G, H, \Gamma)$ from the avoidance principle (Proposition~\ref{pro:avoidance}). Fix $D = \max\{\ref{d:closing lemma main}, \ref{d:avoidance}\} + 2$ where $\ref{d:closing lemma main}$ is as in the effective closing lemma (Theorem~\ref{thm:closing lemma initial dim}) and $\ref{d:avoidance}$ is as in the avoidance principle (Proposition~\ref{pro:avoidance}). Let $M = \ref{m:closing lemma main} + \ref{c:closing lemma main} D$ be as in Theorem~\ref{thm:closing lemma initial dim}. 

Recall $\kappa_1$ from Theorem~\ref{lem:Venkatesh}. Recall that we set $\theta = (\frac{\min\{\kappa_1, \epsilon_0\}}{80})^{2} \in (0, \min\{\kappa_1, \epsilon_0\})$ and $p_{\mathrm{fin}} = \lceil 6480(\frac{9 - \epsilon_0}{\theta} - 1)\rceil$ in Section~\ref{sec:bootstrap}. 
Recall we set $\epsilon = 10^{-10}(\frac{3}{4})^{p_{\mathrm{fin}}} \theta$ and 
\begin{align*}
    d = \sum_{j = 0}^{p_\mathrm{fin} - 1} \Biggl\lceil\Biggl\lceil\frac{25}{2\epsilon}\Bigl(\frac{3}{4}\Bigr)^j\Biggr\rceil\Biggr\rceil.
\end{align*}
Note that those are constants depending only on $(G, H, \Gamma)$. 

\begin{proof}[Proof of Theorem~\ref{thm:main equidistribution}]
Recall from Subsection~\ref{subsection:different form on balls} that Theorem~\ref{thm:main equidistribution different form} is equivalent to Theorem~\ref{thm:main equidistribution}. We prove Theorem~\ref{thm:main equidistribution different form} here, i.e., the effective equidistribution theorem for $a_{\log T} \mathsf{B}^U_1$. 

Let $\ref{a:main equidistribution1} = 10M(10\ref{a:avoidance} + 1)$ and $\ref{a:main equidistribution2} = 10\ref{a:avoidance}$. Note that $\ref{a:main equidistribution1} > \ref{a:main equidistribution2} \geq 1$. Fix $x_0 \in X$. 
Suppose
\begin{align}\label{eqn:FinalProofConditionR}
    R \geq \max\{(\inj(x_0))^{-10\ref{a:avoidance}}, (2\ref{d:avoidance})^{\ref{d:avoidance}\ref{a:main equidistribution2}}, \ref{c:avoidance}, e^{s_0}, 10^{10^7\frac{\ref{a:closing lemma main}}{\theta\epsilon}}\}
\end{align}
and $T \geq R^{\ref{a:main equidistribution1}}$. Suppose case~(2) in the statement does not hold for the initial point $x_0$. Then for all $x$ so that $H.x$ is periodic with $\vol(H.x) \leq R$, we have
\begin{align*}
    d_X(x_0, x) > T^{-\frac{1}{\ref{a:main equidistribution2}}}.
\end{align*}

Set $t = M\log R$. We will assume $R$ to be sufficiently large depending on the space $X$. Set 
\begin{align*}
\beta = e^{-\frac{1}{10^{10}M\ref{a:closing lemma main}\ref{e:closing lemma main1}\ref{e:closing lemma main2}d^2}t}
\end{align*}
and $\ell = \frac{\epsilon}{100M\ref{a:closing lemma main}}t$. Set $\eta = \beta^{1/2}$. Note that both $\beta$ and $\eta$ are of size $R^{-\star}$. Let $\delta_0 = R^{-\frac{1}{\ref{a:closing lemma main}}} = e^{-\frac{t}{M\ref{a:closing lemma main}}}$ be as in Theorem~\ref{thm:closing lemma initial dim} and $\delta_{\mathrm{fin}} = \delta_0^{\frac{2\epsilon}{\theta}}$ as in both Proposition~\ref{pro:dim improving main} and Proposition~\ref{pro:Final phase}. To apply results in Section~\ref{sec:bootstrap} and Section~\ref{sec:From large dimension to effective equidistribution}, the parameters needs to satisfy \cref{eqn:Much smaller scale,eqn:R large enough absolutely}, i.e., $e^{-\ell}$ needs to be a small scale absolutely and also much smaller than $\beta$. By the last condition on $R$, i.e., $R \geq 10^{10^7\frac{\ref{a:closing lemma main}}{\theta\epsilon}}$ in \cref{eqn:FinalProofConditionR} and \cref{eqn:R large enough absolutely} hold. The second condition holds automatically since our choice of parameters is exactly the same as in Section~\ref{sec:bootstrap}. 

We now cut $\log T = t_3 + t_2 + t_1 + t_0$ as the following. Let $t_1 = t$, $t_2 = d\ell$, and 
\begin{align*}
    t_3 = \frac{\epsilon t}{4\theta M\ref{a:closing lemma main}}
\end{align*}
and $t_0 = \log T - (t_3 + t_2 + t_1)$. 
They satisfy the following conditions. The length of the last step $t_3$ satisfies $t_3 = \frac{1}{8}|\log \delta_{\mathrm{fin}}|$ as in Proposition~\ref{pro:Final phase}. The parameters $t_2 = d\ell$ and $t_1 = t = M\log R$ are as in Proposition~\ref{pro:dim improving main}. We have the following estimate:
\begin{align}\label{eqn:ProofMainEquiFinalTwoStepEstimate}
t_2 + t_3 \leq \frac{2t}{M\ref{a:closing lemma main}} \leq t = M\log R.
\end{align}

Using the F{\o}lner property of $U$ (see Lemma~\ref{lem:Folner1} or Lemma~\ref{lem:Folner H}), for all $\phi \in \mathrm{C}_c^\infty(X)$ we have
\begin{align}\label{eqn:ProofOfMainEquiFolner}
\begin{aligned}
    &\int_{\mathsf{B}_1^U} \phi(a_{\log T} u.x_0) \,\mathrm{d}u\\
    ={}& \int_{\mathsf{B}_1^U} \int_{\mathsf{B}_1^U} \int_{\mathsf{B}_1^U} \int_{\mathsf{B}_1^U} \phi(a_{t_3}u_3 a_{t_2} u_2 a_{t_1} u_1 a_{t_0} u_0.x_0) \,\mathrm{d}u_0\,\mathrm{d}u_1\,\mathrm{d}u_2\,\mathrm{d}u_3 + O(\mathcal{S}(\phi) \beta^\star).
\end{aligned}
\end{align}

Let $R_1 = R$ and $R_2 = T^{\frac{2}{\ref{a:main equidistribution2}}}$. By our assumption, the initial point $x_0$ satisfies
\begin{align*}
d_X(x_0, x) > T^{-\frac{1}{\ref{a:main equidistribution2}}} = R_2^{-\frac{1}{2}} \geq (\log R_2)^{\ref{d:avoidance}} R_2^{-1}.
\end{align*}
The last inequality follows from the fact that $R_2 \geq R^{\frac{2}{\ref{a:main equidistribution2}}} \geq (2\ref{d:avoidance})^{2\ref{d:avoidance}}$. 

We claim that $t_0 \geq \ref{a:avoidance}\max\{\log R_2, |\log \inj(x_0)|\} + s_0$. Indeed, for the right hand side of the inequality, we have
\begin{align*}
\log R_2 = \frac{2}{\ref{a:main equidistribution2}}\log T \geq \log R \geq \max\{|\log \inj(x_0)|, s_0.
\end{align*}
Therefore, it suffices to show that
\begin{align*}
t_0 \geq 2\ref{a:avoidance}  \log R_2 = \frac{4\ref{a:avoidance}}{\ref{a:main equidistribution2}}\log T = \frac{2}{5}\log T.
\end{align*}
By definition of $t_0$, we have
\begin{align*}
t_0 = \log T - (t_1 + t_2 + t_3) \geq \log T - 2M\log R.
\end{align*}
Recall that $\ref{a:main equidistribution1} = 10M(10\ref{a:avoidance} + 1)$ and $\log T \geq \ref{a:main equidistribution1}\log R$, we have
\begin{align*}
t_0 \geq \log T - 2M\log R \geq 8M(10\ref{a:avoidance} + 1) = \frac{4}{5} \log T > \frac{2}{5} \log T.
\end{align*}
Therefore, we have
\begin{align*}
t_0 \geq \ref{a:avoidance}\max\{\log R_2, |\log \inj(x_0)|\} + s_0.
\end{align*}

Let
\begin{align*}
    \mathsf{B}_1^{U, \mathrm{WA}} = \Biggl\{u \in \mathsf{B}_1^U: \begin{aligned}{}&\inj(a_{t_0}u.x_0) \leq \eta \text{ or }\exists x \text { with }\vol(H.x) \leq R_1\\ {}&\text { and } d_X(a_{t_0} u. x_0, x) \leq R_1^{-D}\end{aligned}\Biggr\}
\end{align*}
and $\mathsf{B}_1^{U, \mathrm{Dio}} = \mathsf{B}_1^{U} \setminus \mathsf{B}_1^{U, \mathrm{WA}}$. Since $D \geq \ref{d:avoidance} + 2$ and $R \geq \ref{c:avoidance}$, Proposition~\ref{pro:avoidance} implies
\begin{align*}
    |\mathsf{B}_1^{U, \mathrm{WA}}| \ll \eta^{\frac{1}{\mathsf{m}}}.
\end{align*}
Here we apply $R^{-1} \leq \eta$. Therefore, 
\begin{align}\label{eqn:ProofOfMainEquiAvoidance}
\begin{aligned}
&\int_{\mathsf{B}_1^U} \int_{\mathsf{B}_1^U} \int_{\mathsf{B}_1^U} \int_{\mathsf{B}_1^U} \phi(a_{t_3}u_3 a_{t_2} u_2 a_{t_1} u_1 a_{t_0} u_0.x_0) \,\mathrm{d}u_0\,\mathrm{d}u_1\,\mathrm{d}u_2\,\mathrm{d}u_3\\   
={}& \int_{\mathsf{B}_1^U} \int_{\mathsf{B}_1^U} \int_{\mathsf{B}_1^U} \int_{\mathsf{B}_1^{U, \mathrm{Dio}}} \phi(a_{t_3}u_3 a_{t_2} u_2 a_{t_1} u_1 a_{t_0} u_0.x_0) \,\mathrm{d}u_0\,\mathrm{d}u_1\,\mathrm{d}u_2\,\mathrm{d}u_3 + O(\mathcal{S}(\phi)\beta^\star).
\end{aligned}
\end{align}
It suffices to estimate
\begin{align*}
\int_{\mathsf{B}_1^U} \int_{\mathsf{B}_1^U} \int_{\mathsf{B}_1^U} \phi(a_{t_3}u_3 a_{t_2} u_2 a_{t_1} u_1 x_1) \,\mathrm{d}u_1\,\mathrm{d}u_2\,\mathrm{d}u_3.
\end{align*}
for all $x_1 \in a_{t_0} \mathsf{B}_1^{U, \mathrm{Dio}}.x_0$. Note that such $x_1$ satisfies the following.
\begin{enumerate}
    \item The point $x_1 \in X_\eta$.
    \item For all $x' \in X$ so that $H.x'$ periodic with $\vol(H.x') \leq R_1 = R$, we have
    \begin{align*}
        d_X(x_1, x') > R^{-D}.
    \end{align*}
\end{enumerate}

Recall that we picked $t_1 = t$ and $t_2 = d\ell$ exactly as in Proposition~\ref{pro:dim improving main}. Applying Proposition~\ref{pro:dim improving main} to such $x_1$, there exists a family of sheeted sets $\{\mathcal{E}^{\mathrm{fin}}_i\}_i$ with cross-section $\{F^{\mathrm{fin}}_i\}_i$, associated admissible measures $\{\mu_{\mathcal{E}^{\mathrm{fin}}_i}\}_i$ and $\{c_i\}_i$ satisfying $c_i > 0$ and $\sum_i c_i = 1$ so that the following holds. 
\begin{enumerate}
\item We have
\begin{align}\label{eqn:ProofOfMainEquiBootstrap}
\begin{aligned}
&\int_{\mathsf{B}_1^U} \int_{\mathsf{B}_1^U} \int_{\mathsf{B}_1^U} \phi(a_{t_3}u_3 a_{t_2} u_2 a_{t_1} u_1 x_1) \,\mathrm{d}u_1\,\mathrm{d}u_2\,\mathrm{d}u_3\\
={}& \sum_i c_i\int_{\mathsf{B}_1^U}\int_{\mathcal{E}_i^{\mathrm{fin}}} \phi(a_{t_3}u_3.x) \,\mathrm{d}\mu_{\mathcal{E}_i^{\mathrm{fin}}}(x) \,\mathrm{d}u_3 + O(\mathcal{S}(\phi)\beta^\star).
\end{aligned}
\end{align}
\item For all $i$, we have
\begin{align*}
\#F_i^{\mathrm{fin}} \geq \delta_0^{-\frac{4\epsilon_0}{3}} = e^{\frac{4\epsilon_0 t}{3M\ref{a:closing lemma main}}}.
\end{align*}
\item Let $\delta_{\mathrm{fin}} = \delta_0^{\frac{2\epsilon}{\theta}} = e^{-\frac{2\epsilon t}{\theta M\ref{a:closing lemma main}}}$. For all $i$ we have
\begin{align*}
f_{\mathcal{E}_i^{\mathrm{fin}},\delta_{\mathrm{fin}}}^{(9 - \theta)}(x) \leq 2^{p_{\mathrm{fin}}}e^{20\ell}\#F_i^{\mathrm{fin}} \qquad \text{for all $x \in \mathcal{E}_i^{\mathrm{fin}}$}.
\end{align*}
\end{enumerate}
By property~(2) and (3) and the fact $t_3 = \frac{1}{8}|\log \delta_{\mathrm{fin}}|$, all conditions in Proposition~\ref{pro:Final phase} are satisfied. Apply Proposition~\ref{pro:Final phase} to each sheeted set $\mathcal{E}^{\mathrm{fin}}_i$ and their associated admissible measures, we have
\begin{align}\label{eqn:ProofOfMainEquiFinalPhase}
\int_{\mathsf{B}_1^U}\int_{\mathcal{E}_i^{\mathrm{fin}}} \phi(a_{t_3}u_3.x) \,\mathrm{d}\mu_{\mathcal{E}_i^{\mathrm{fin}}}(x)\,\mathrm{d}u_3 = \int_X \phi \,\mathrm{d}\mu_X + O(\mathcal{S}(\phi)\beta^\star).
\end{align}
Recall that $\beta$ is of size $R^{-\star}$. Combining \cref{eqn:ProofOfMainEquiFolner,eqn:ProofOfMainEquiAvoidance,eqn:ProofOfMainEquiBootstrap,eqn:ProofOfMainEquiFinalPhase}, we have
\begin{align*}
    \Biggl|\int_{\mathsf{B}_1^U} \phi(a_{\log T} u.x_0) \,\mathrm{d}u - \int_X \phi\,\mathrm{d}\mu_X\Biggr| \ll \mathcal{S}(\phi) R^{-\star},
\end{align*}
where the implied constants depend only on $(G, H, \Gamma)$. The proof is complete. 
\end{proof}

\nocite{*}
\bibliographystyle{alpha_name-year-title}
\bibliography{References}

\newcommand{\etalchar}[1]{$^{#1}$}
\begin{thebibliography}{LMWY25}

\bibitem[And75]{And75}
Peter~P. Andre.
\newblock {$k$}-regular elements in semisimple algebraic groups.
\newblock {\em Trans. Amer. Math. Soc.}, 201:105--124, 1975.

\bibitem[Bou89]{Bou89}
Nicolas Bourbaki.
\newblock {\em Lie groups and {L}ie algebras. {C}hapters 1--3}.
\newblock Elements of Mathematics (Berlin). Springer-Verlag, Berlin, 1989.
\newblock Translated from the French, Reprint of the 1975 edition.

\bibitem[Bou05]{Bou05}
Nicolas Bourbaki.
\newblock {\em Lie groups and {L}ie algebras. {C}hapters 7--9}.
\newblock Elements of Mathematics (Berlin). Springer-Verlag, Berlin, 2005.
\newblock Translated from the 1975 and 1982 French originals by Andrew Pressley.

\bibitem[Bou10]{Bou10}
Jean Bourgain.
\newblock The discretized sum-product and projection theorems.
\newblock {\em J. Anal. Math.}, 112:193--236, 2010.

\bibitem[BGHM22]{BGHM22}
P.~Buterus, F.~G\"otze, T.~Hille, and G.~Margulis.
\newblock Distribution of values of quadratic forms at integral points.
\newblock {\em Invent. Math.}, 227(3):857--961, 2022.

\bibitem[BH24]{BH24}
Timothée Bénard and Weikun He.
\newblock Multislicing and effective equidistribution for random walks on some homogeneous spaces, 2024.

\bibitem[Dan81]{Dan81}
S.~G. Dani.
\newblock Invariant measures and minimal sets of horospherical flows.
\newblock {\em Invent. Math.}, 64(2):357--385, 1981.

\bibitem[DM89]{DM89}
S.~G. Dani and G.~A. Margulis.
\newblock Values of quadratic forms at primitive integral points.
\newblock {\em Invent. Math.}, 98(2):405--424, 1989.

\bibitem[DM90]{DM90}
S.~G. Dani and G.~A. Margulis.
\newblock Orbit closures of generic unipotent flows on homogeneous spaces of {${\mathrm{SL}}(3,{\mathbb{R}})$}.
\newblock {\em Math. Ann.}, 286(1-3):101--128, 1990.

\bibitem[EMV09]{EMV09}
M.~Einsiedler, G.~Margulis, and A.~Venkatesh.
\newblock Effective equidistribution for closed orbits of semisimple groups on homogeneous spaces.
\newblock {\em Invent. Math.}, 177(1):137--212, 2009.

\bibitem[EMM98]{EMM98}
Alex Eskin, Gregory Margulis, and Shahar Mozes.
\newblock Upper bounds and asymptotics in a quantitative version of the {O}ppenheim conjecture.
\newblock {\em Ann. of Math. (2)}, 147(1):93--141, 1998.

\bibitem[EMM05]{EMM05}
Alex Eskin, Gregory Margulis, and Shahar Mozes.
\newblock Quadratic forms of signature {$(2,2)$} and eigenvalue spacings on rectangular 2-tori.
\newblock {\em Ann. of Math. (2)}, 161(2):679--725, 2005.

\bibitem[GGW24]{GGW24}
Shengwen Gan, Shaoming Guo, and Hong Wang.
\newblock A restricted projection problem for fractal sets in {$\mathbb{R}^n$}.
\newblock {\em Camb. J. Math.}, 12(3):535--561, 2024.

\bibitem[He20]{He20}
Weikun He.
\newblock Orthogonal projections of discretized sets.
\newblock {\em J. Fractal Geom.}, 7(3):271--317, 2020.

\bibitem[JL24]{JL24}
Ben Johnsrude and Zuo Lin.
\newblock Restricted projections and fourier decoupling in $\mathbb{Q}_p^n$, 2024.

\bibitem[Kat23]{Kat23}
Asaf Katz.
\newblock Margulis' inequality for translates of horospherical orbits and applications to equidistribution, 2023.

\bibitem[KS19]{KS19}
Tam\'as Keleti and Pablo Shmerkin.
\newblock New bounds on the dimensions of planar distance sets.
\newblock {\em Geom. Funct. Anal.}, 29(6):1886--1948, 2019.

\bibitem[Kim24]{Kim24}
Wooyeon Kim.
\newblock Moments of margulis functions and indefinite ternary quadratic forms, 2024.

\bibitem[KM96]{KM96}
D.~Y. Kleinbock and G.~A. Margulis.
\newblock Bounded orbits of nonquasiunipotent flows on homogeneous spaces.
\newblock In {\em Sina\u i's {M}oscow {S}eminar on {D}ynamical {S}ystems}, volume 171 of {\em Amer. Math. Soc. Transl. Ser. 2}, pages 141--172. Amer. Math. Soc., Providence, RI, 1996.

\bibitem[KM98]{KM98}
D.~Y. Kleinbock and G.~A. Margulis.
\newblock Flows on homogeneous spaces and {D}iophantine approximation on manifolds.
\newblock {\em Ann. of Math. (2)}, 148(1):339--360, 1998.

\bibitem[KT07]{KT07}
Dmitry Kleinbock and George Tomanov.
\newblock Flows on {$S$}-arithmetic homogeneous spaces and applications to metric {D}iophantine approximation.
\newblock {\em Comment. Math. Helv.}, 82(3):519--581, 2007.

\bibitem[LM23]{LM23}
E.~Lindenstrauss and A.~Mohammadi.
\newblock Polynomial effective density in quotients of {$\mathbb{H}^3$} and {$\mathbb{H}^2\times\mathbb{H}^2$}.
\newblock {\em Invent. Math.}, 231(3):1141--1237, 2023.

\bibitem[LMMS24]{LMMS}
Elon Lindenstrauss, Gregorii Margulis, Amir Mohammadi, and Nimish~A. Shah.
\newblock Quantitative behavior of unipotent flows and an effective avoidance principle.
\newblock {\em J. Anal. Math.}, 153(1):1--61, 2024.

\bibitem[LM14]{LM14}
Elon Lindenstrauss and Gregory Margulis.
\newblock Effective estimates on indefinite ternary forms.
\newblock {\em Israel J. Math.}, 203(1):445--499, 2014.

\bibitem[LMM{\etalchar{+}}24]{LMMSW24}
Elon Lindenstrauss, Gregory Margulis, Amir Mohammadi, Nimish Shah, and Andreas Wieser.
\newblock An effective closing lemma for unipotent flows, 2024.

\bibitem[LMW22]{LMW22}
Elon Lindenstrauss, Amir Mohammadi, and Zhiren Wang.
\newblock Polynomial effective equidistribution, 2022.

\bibitem[LMWY23]{LMWY23}
Elon Lindenstrauss, Amir Mohammadi, Zhiren Wang, and Lei Yang.
\newblock An effective version of the oppenheim conjecture with a polynomial error rate, 2023.

\bibitem[LMWY25]{LMWY25}
Elon Lindenstrauss, Amir Mohammadi, Zhiren Wang, and Lei Yang.
\newblock Effective equidistribution in rank 2 homogeneous spaces and values of quadratic forms, 2025.

\bibitem[\L{}o59]{Loj59}
S.~\L{}ojasiewicz.
\newblock Sur le probl\`eme de la division.
\newblock {\em Studia Math.}, 18:87--136, 1959.

\bibitem[Mar89]{Mar89}
G.~A. Margulis.
\newblock Indefinite quadratic forms and unipotent flows on homogeneous spaces.
\newblock In {\em Dynamical systems and ergodic theory ({W}arsaw, 1986)}, volume~23 of {\em Banach Center Publ.}, pages 399--409. PWN, Warsaw, 1989.

\bibitem[Mar90]{Mar90}
G.~A. Margulis.
\newblock Orbits of group actions and values of quadratic forms at integral points.
\newblock In {\em Festschrift in honor of {I}. {I}. {P}iatetski-{S}hapiro on the occasion of his sixtieth birthday, {P}art {II} ({R}amat {A}viv, 1989)}, volume~3 of {\em Israel Math. Conf. Proc.}, pages 127--150. Weizmann, Jerusalem, 1990.

\bibitem[Moh23]{Moh23}
Amir Mohammadi.
\newblock Finitary analysis in homogeneous spaces.
\newblock In {\em I{CM}---{I}nternational {C}ongress of {M}athematicians. {V}ol. 5. {S}ections 9--11}, pages 3530--3551. EMS Press, Berlin, [2023] \copyright 2023.

\bibitem[Mor05]{Mor05}
Dave~Witte Morris.
\newblock {\em Ratner's theorems on unipotent flows}.
\newblock Chicago Lectures in Mathematics. University of Chicago Press, Chicago, IL, 2005.

\bibitem[Mos55]{Mos55}
G.~D. Mostow.
\newblock Some new decomposition theorems for semi-simple groups.
\newblock {\em Mem. Amer. Math. Soc.}, 14:31--54, 1955.

\bibitem[OL25]{OL25}
K.~W. Ohm and Z.~Lin.
\newblock Projection theorems in the presence of expansions, 2025.

\bibitem[OS25]{OS25}
Ko~W. Ohm and Anthony Sanchez.
\newblock Quantitative finiteness of hyperplanes in hybrid manifolds, 2025.

\bibitem[Rat90a]{Rat90a}
Marina Ratner.
\newblock On measure rigidity of unipotent subgroups of semisimple groups.
\newblock {\em Acta Math.}, 165(3-4):229--309, 1990.

\bibitem[Rat90b]{Rat90b}
Marina Ratner.
\newblock Strict measure rigidity for unipotent subgroups of solvable groups.
\newblock {\em Invent. Math.}, 101(2):449--482, 1990.

\bibitem[Rat91a]{Rat91a}
Marina Ratner.
\newblock On {R}aghunathan's measure conjecture.
\newblock {\em Ann. of Math. (2)}, 134(3):545--607, 1991.

\bibitem[Rat91b]{Rat91b}
Marina Ratner.
\newblock Raghunathan's topological conjecture and distributions of unipotent flows.
\newblock {\em Duke Math. J.}, 63(1):235--280, 1991.

\bibitem[SS24]{SS24}
Anthony Sanchez and Juno Seong.
\newblock An avoidance principle and {M}argulis functions for expanding translates of unipotent orbits.
\newblock {\em J. Mod. Dyn.}, 20:409--439, 2024.

\bibitem[Sha96]{Sh96}
Nimish~A. Shah.
\newblock Limit distributions of expanding translates of certain orbits on homogeneous spaces.
\newblock {\em Proc. Indian Acad. Sci. Math. Sci.}, 106(2):105--125, 1996.

\bibitem[Shm23a]{Shm23a}
Pablo Shmerkin.
\newblock A nonlinear version of bourgain’s projection theorem.
\newblock {\em J. Eur. Math. Soc.(JEMS)}, 25(10):4155--4204, 2023.

\bibitem[Shm23b]{Shm23b}
Pablo Shmerkin.
\newblock Slices and distances: on two problems of {F}urstenberg and {F}alconer.
\newblock In {\em I{CM}---{I}nternational {C}ongress of {M}athematicians. {V}ol. 4. {S}ections 5--8}, pages 3266--3290. EMS Press, Berlin, [2023] \copyright 2023.

\bibitem[Sol91]{Sol91}
Pablo Solern\'o.
\newblock Effective {{\L}}ojasiewicz inequalities in semialgebraic geometry.
\newblock {\em Appl. Algebra Engrg. Comm. Comput.}, 2(1):2--14, 1991.

\bibitem[Yan25]{Yan24}
Lei Yang.
\newblock Effective version of {R}atner's equidistribution theorem for {$\mathrm{SL}(3,\mathbb{R})$}.
\newblock {\em Ann. of Math. (2)}, 202(1):189--264, 2025.

\end{thebibliography}
\end{document}